\documentclass[a4,11pt]{amsart}
\usepackage{amsmath,amsthm,amssymb,amscd,color}
\usepackage[all,arc,curve,color,frame,poly]{xy}
\usepackage{enumerate,extarrows}
\usepackage[dvips]{graphicx}
\usepackage{latexsym}
\usepackage{lscape}
\usepackage{pstricks}
\usepackage{pst-node}
\usepackage{verbatim}
\usepackage{bm}
\usepackage[dvips,pdfstartview={FitH -32768},bookmarks=false,colorlinks=false]{hyperref}
\theoremstyle{definition}
\numberwithin{equation}{section}
\newtheorem{thm}{Theorem}[section]
\newtheorem{df}[thm]{Definition}
\newtheorem{ex}[thm]{Example}
\newtheorem{prop}[thm]{Proposition}
\newtheorem{cor}[thm]{Corollary}
\newtheorem{lem}[thm]{Lemma}

\newtheorem{rmk}[thm]{Remark}

\def\dim{\mathop{\mathrm{dim}}\nolimits}
\def\udim{\mathop{\underline{\mathrm{dim}}}\nolimits}

\def\Im{\mathop{\mathrm{Im}}\nolimits}

\def\Coker{\mathop{\mathrm{Coker}}\nolimits}
\def\Hom{\mathop{\mathrm{Hom}}\nolimits}
\def\End{\mathop{\mathrm{End}}\nolimits}
\def\Ext{\mathop{\mathrm{Ext}}\nolimits}
\def\Tor{\mathop{\mathrm{Tor}}\nolimits}

\def\Spec{\mathop{\mathrm{Spec}}\nolimits}

\def\Irr{\mathop{\mathrm{Irr}}\nolimits}
\def\mod{\mathop{\mathrm{mod}}\nolimits}

\def\RHom{\mathop{\mathbb R\mathrm{Hom}}\nolimits}

\def\add{\mathop{\mathrm{add}}\nolimits}

\newcommand{\op}{\mathrm{op}}

\newcommand{\e}{{\mathbf{e}}}

\newcommand{\m}{{\mathfrak{m}}}

\newcommand{\C}{{\mathbb{C}}}
\newcommand{\Q}{{\mathbb{Q}}}
\newcommand{\R}{{\mathbb{R}}}
\newcommand{\PP}{{\mathbb{P}}}
\newcommand{\OO}{{\mathcal{O}}}
\newcommand{\Z}{{\mathbb{Z}}}
\newcommand{\Hilb}[2]{#1\text{-Hilb}(#2)}
\newcommand{\M}{{\mathcal{M}}}
\newcommand{\SO}{SO}
\newcommand{\SL}{SL}
\newcommand{\GL}{GL}

\newcommand{\Stab}[3]{\mathcal S_{#1 #3}(#2)}  
\newcommand{\Mod}{{\mathrm{Mod}}}
\newcommand{\Ltensor}{\stackrel{\mathbb{L}}{\otimes}}

\title[Flops and mutations for polyhedral singularities]{Flops and mutations for crepant resolutions of polyhedral singularities}
\author{\'Alvaro Nolla de Celis}
\author{Yuhi Sekiya}
\date{}
\thanks{{\em Affiliation:} Graduate School of Mathematics. Nagoya University, Chikusa-ku Nagoya 464-8602, Japan. \\{\em \'Alvaro Nolla de Celis:} alnolla@gmail.com, {\em Yuhi Sekiya:} yuhi-sekiya@math.nagoya-u.ac.jp}

\begin{document}

\begin{abstract}
Let $G$ be a polyhedral group $G\subset\SO(3)$ of types $\Z/n\Z$, $D_{2n}$ and $\mathbb{T}$. We prove that there exists a one-to-one correspondence between flops of $\Hilb{G}{\C^3}$ and mutations of the McKay quiver with potential which do not mutate the trivial vertex. This correspondence provides two equivalent methods to construct every projective crepant resolution for the singularities $\C^3/G$, which are constructed as moduli spaces $\mathcal{M}_C$ of quivers with potential for some chamber $C$ in the space $\Theta$ of stability conditions. In addition, we study the relation between the exceptional locus in $\mathcal{M}_C$ with the corresponding quiver $Q_C$, and we describe explicitly the part of the chamber structure in $\Theta$ where every such resolution can be found.
\end{abstract}

\maketitle

\setcounter{tocdepth}{1}
\tableofcontents

\section{Introduction}




This paper focuses on the problem of describing every projective crepant resolution of the quotient $\C^3/G$ for a given finite subgroup $G$ of $\SL(3,\C)$. In particular, we consider the case when $G$ belongs to the special orthogonal group $\SO(3)$, also called {\em polyhedral subgroups}, classified into five types: cyclic $\Z/n\Z$, dihedral $D_{2n}$, tetrahedral  $\mathbb{T}$, octahedral $\mathbb{O}$ and icosahedral $\mathbb{I}$. 

It is well known that every such crepant resolution is related by a sequence of {\em flops}. The purpose of this work is to construct explicitly every projective crepant resolution as certain moduli space of quiver representations and describe how can we perform {\em flops} between them in two different and equivalent ways: by changing the stability condition keeping the original quiver, or by changing the quiver by {\em mutation} but keeping the original stability condition. 

The equivariant Hilbert scheme \Hilb{$G$}{$\C^3$}, or moduli space of {\em $G$-clusters}, is the distinguished candidate of projective crepant resolution to start with (recall that a $G$-cluster is a 0-dimensional subscheme $\mathcal{Z}\subset\C^3$ such that $\OO_\mathcal{Z}\cong\C[G]$ the regular representation of $G$ as $\C[G]$-modules). In one hand, by \cite{BKR} it is known that \Hilb{$G$}{$\C^3$} is always a projective crepant resolution of $\C^3/G$. For polyhedral subgroups \Hilb{$G$}{$\C^3$} was first studied by Gomi, Nakamura and Shinoda in \cite{GNS1,GNS2}, showing also that the fibre over the origin $E:=\pi^{-1}(0)$ of $\pi:\Hilb{G}{\C^3}\to\C^3/G$ has dimension one and there is a one to one correspondence between smooth rational curves in $E$ and nontrivial irreducible representations of $G$ (see also \cite{BS}). 

In addition, by \cite{IN00} we know that we can interpret $\Hilb{G}{\C^3}$ as the moduli space $\M_{\theta,{\bf d}}$ of $\theta$-stable representations of dimension ${\bf d} := (\dim V_i)_{V_i\in\Irr G}$ of the {\em McKay quiver} $Q$ with suitable relations, for a particular choice of generic $\theta$ in the space of stability conditions $\Theta_{\bf d} := \{\theta\in{\Hom_{\Z}(\Z^{Q_0}},\Z)\otimes\Q~|~\theta\cdot{\bf{d}}=0\}\subset\Q^{|Q_0|}$. By \cite{BSW}, the relations in $Q$ are obtained as derivations of a potential $W$, so we consider quivers in this paper as {\em quivers with potential} (QP for short). 

The space of generic parameters in $\Theta_{\bf d}$ (or simply $\Theta$) is the disjoint union of finitely many convex polyhedral cones called {\em chambers} where the moduli space is constant, that is, $\M_{\theta,{\bf d}}\cong\M_{\theta',{\bf d}}$ where $\theta,\theta'\in C$ for any chamber $C\subset\Theta$. Calling this moduli space $\M_C$, if follows from \cite{BKR} that in fact $\M_C$ is a projective crepant resolution of $\C^3/G$ for any finite subgroup $G\subset\SL(3,\C)$ and any chamber $C\subset\Theta$. In the opposite direction, it was conjectured (and proved in the Abelian case) by Craw and Ishii in \cite{CI} that every projective crepant resolution is isomorphic to a $\M_C$ for some $C\subset\Theta$. 

In terms of the moduli spaces $\M_C$, the operation of a flop corresponds to vary the stability parameter to cross a wall in $\Theta$ to an adjacent chamber $C'$, obtaining a new moduli space $\M_{C'}$. It should be pointed out that not every wall crossing in $\Theta$ produces a flop since it may happen that $\M_C\cong\M_{C'}$. Therefore, one strategy to obtain our goal is to start from \Hilb{$G$}{$\C^3$} and perform a flop for every possible floppable rational curve contained in the exceptional divisor $E\subset\Hilb{G}{\C^3}$. By constructing explicitly the other side of the flop and iterating the process we eventually obtain every projective crepant resolution of $\C^3/G$. 

On the other hand, {\em non-commutative crepant resolutions} (NCCRs) of $R:=\C[x,y,z]^G$ are considered to be the non-commutative analogue of crepant resolutions of $\C^3/G$. They are algebras of the form $\Lambda:=\End_R(M)$ where $M$ is a reflexive $R$-module, $\Lambda$ has finite global dimension and it is a (maximal) Cohen-Macaulay $R$-module (see \cite{VdB}). Indeed, if $\Gamma$ is a NCCR of $R$ then for any generic stability condition $\theta$ the moduli space $\M_{\theta,{\bf d}}(\Gamma)$ of $\theta$-stable $\Gamma$-modules of dimension vector ${\bf d}$ is a crepant resolution of $\Spec R$. For instance, the skew group algebra $S\ast G$ is an NCCR of $R$, which is Morita equivalent to the {\em Jacobian algebra} $\mathcal{P}(Q,W):=\C Q/\langle \partial_a W \mid a \in Q_1 \rangle$ of the McKay QP. Then $\M_{C_0,{\bf d}}(\mathcal{P}(Q,W))\cong\Hilb{G}{\C^3}$ where $C_0$ is the chamber containing the {\em 0-generated} stability condition $\theta^0$ (the one such that $\theta^0_i>0$ for $i\neq0$).

From this point of view, a common operation to obtain a new NCCR from a given one is by {\em mutation}. For the groups $G$ treated in this paper, we take the Jacobian algebra $\mathcal{P}(Q,W)$ of the McKay QP of $G$, and consider mutations of quivers with potential at suitable vertices $k$ in $Q$ without loops. We obtain in this manner new QPs denoted by $\mu_k\mathcal{P}(Q,R)$, starting the iterative procedure which turns out to cover every projective crepant resolution of $\C^3/G$. We say that a quiver with potential $(Q', W')$ is an {\em iterated mutation} of $(Q,W)$ if there are quivers with potentials $(Q^{(i)},W^{(i)})$ for $0\leq i\leq n$ such tat $(Q,W)=(Q^{(0)},W^{(0)})$, $(Q',W')=(Q^{(n)},W^{(n)})$ and $(Q^{(i+1)},W^{(i+1)})=\mu_{k_i}(Q^{(i)},W^{(i)})$ where $k_i$ is a vertex of $Q^{(i)}$ without loops. See \ref{defn:mut} for the precise definition of the mutation that we use in this paper, and Section \ref{sect:mutNCCRs} for a discussion about the meaning of mutation at loops in our setting.

In this paper we prove that the two strategies above explained are equivalent for polyhedral subgroups $G\subset\SO(3)$ of types $\Z/n\Z$, $D_{2n}$ and $\mathbb{T}$. For all these cases every mutation at a vertex with a loop is trivial, restricting our the study to mutations only at vertices without loops. For subgroups $G \subset SO(3)$ of types $\mathbb O$ and $\mathbb I$ there are vertices with loops in some iterated quiver QP $(Q',W')$ for which the mutation is not trivial. As it is explained in Section \ref{sect:mutNCCRs}, this fact is encoded locally in the factor algebra $\Lambda/\Lambda(1-e_i)\Lambda$ where $\Lambda=\mathcal{P}(Q',W')$, which in these cases turns out to have finite dimension. Geometrically, this means that there are a priori floppable $(-2,0)$ and $(-3,1)$-curves in the fibre of origin of some crepant resolution of $\mathbb C^3/G$. Even though following \cite{Wem14} we can ensure that the correspondence of both approaches also holds for types $\mathbb O$ and $\mathbb I$, because of the different nature of this cases with respect to explicit computations (namely the presence of high rank modules in the McKay quiver and mutations of QPs at vertices with loops) we leave the treatment of this cases for a future work.

\subsection{Statements of results and corollaries} 

The main theorem of the paper is the following.

\begin{thm}\label{thm:main} Let $G\subset\SO(3)$ of types $\Z/n\Z$, $D_{2n}$ or $\mathbb{T}$, and let $(Q,W)$ be the McKay quiver with potential. Then there exists a one-to-one correspondence between flops of $\Hilb{G}{\C^3}$ and mutations of $(Q,W)$ which do not mutate the trivial vertex.
\end{thm}

Therefore, to a given projective crepant resolution $\mathcal{M}_C$ for some $C\subset\Theta$ we can associate a QP $(Q_C, W_C)$ obtained as an iterated mutation from the McKay QP. 

The theorem is proved in Section \ref{Proof:Main} and it is done by direct comparison. On one hand we calculate every possible mutation at non-trivial vertices of the McKay QP according to Definition \ref{defn:mut}, and on the other hand we construct an explicit open cover of every projective crepant resolution of $\C^3/G$ obtained by a sequence of flops from $\Hilb{G}{\C^3}$. It turns out that every open cover consists of a finite number of copies of $\C^3$ (see Theorems \ref{OpensDnOdd}, \ref{OpensDnEven} and \ref{OpensE6}) and in every step only $(-1,-1)$-curves are floppable. The last fact is proved in Lemma \ref{floppable} using Reid's {\em width} for $(-2,0)$-curves and $S$-equivalence classes for $(-3,1)$-curves (see also Section \ref{sect:contralg} for an alternative approach using {\em contraction algebras}). In fact, the direct comparison shows that by mutating at non-trivial vertex $k$ in $Q_C$ without loops we match what is happening geometrically when flopping a rational curve $E_k\subset\M_C$.

By construction, every such resolution is described as a moduli space of the McKay quiver for some chamber $C\subset\Theta$, which means that for this groups the Craw-Ishii conjecture holds:

\begin{cor} Let $G\subset\SO(3)$ be a finite subgroup of type $\Z/n\Z$, $D_{2n}$ or $\mathbb{T}$. Then every projective crepant resolution of $\C^3/G$ is isomorphic to $\mathcal{M}_C$ for some chamber $C\subset\Theta$.
\end{cor}

As the next corollary shows, the relation between $\mathcal{M}_C$ and $(Q_C, W_C)$ goes one step further:

\begin{cor}\label{cor:DualGraph} Let $G$ be as above, let $\pi_C:\mathcal{M}_C\to\C^3/G$ be the projective crepant resolution for some chamber $C\subset\Theta$ and let $Q_C$ be the corresponding iterated quiver. 
\begin{itemize}
\item[(i)] The dual graph of $\pi_C^{-1}(0)$ is the same as the graph of $Q_C$ removing the trivial vertex. 
\item[(ii)] The number of loops at a vertex $i$ of the quiver $Q_C$ determines the degree of the normal bundle of the corresponding rational curve $E_i\subset\pi_C^{-1}(0)\subset\M_C$. More precisely, we have the following one-to-one correspondences:
\[
\begin{array}{rcl}
\{ \text{$(-1,-1)$-curves in $\mathcal{M}_C$} \} & \xlongleftrightarrow{\text{}} & \{ \text{non-trivial vertices in $Q_C$ with no loops} \} \\
\{ \text{$(-2,0)$-curves in $\mathcal{M}_C$} \} & \xlongleftrightarrow{\text{}} & \{ \text{non-trivial vertices in $Q_C$ with one loop} \} \\
\{ \text{$(-3,1)$-curves in $\mathcal{M}_C$} \} & \xlongleftrightarrow{\text{}} & \{ \text{non-trivial vertices in $Q_C$ with two loops} \}
\end{array}
\]
\end{itemize}
\end{cor}

Although there is no relation with irreducible representations of $G$ except in the case when $\M_C\cong\Hilb{G}{\C^3}$, this corollary extends the McKay correspondence for finite subgroups in $\GL(2,\C)$ of Wemyss \cite{Wem08}. We would also like to note that the one-to-one correspondences in (ii) are expected since the dimension of the fibre over the origin $0\in\C^3/G$ has dimension one (see Remark \ref{rem:LoopDeg}).

The way of finding the projective crepant resolution $\mathcal{M}_C$ in the corresponding QP $(Q_C,W_C)$ is shown in the next result (= Theorem \ref{horizontal}), which states that $\mathcal{M}_C$ is the moduli space of representations of $(Q_C,W_C)$ of dimension vector $\omega\bf{d}$ and the 0-generated stability condition $\theta^0$ (see Section \ref{Sect:Stability} for the precise definition of $\omega$). In the opposite direction, i.e.\ starting from an iterated QP $(Q,W)$ and its corresponding moduli space $X:=\M_{\theta^0}(\mathcal{P}(Q,W))$, it also provides the way of finding the stability parameter $\theta$ such that the moduli space of McKay quiver representations $\mathcal{M}_\theta$ is isomorphic to $X$. The result was motivated by the work of \cite{SY} in dimension 2.

\begin{thm}
Let $G\subset\SO(3)$ be a finite subgroup of type $\Z/n\Z$, $D_{2n}$ or $\mathbb{T}$, and let $X:=\mathcal{M}_C$ be an projective crepant resolution of $C^3/G$. Then \[X \cong \mathcal{M}_{\theta^0,\omega \mathbf d}(\mathcal P(Q_C,W_C)).\] 
Moreover, there exists a corresponding sequence of wall crossings from $\mathcal{M}_{C_0}\cong\Hilb{G}{\C^3}$ which leads to $X \cong \mathcal M_{\theta,\mathbf d}(\Lambda)$ where $\Lambda$ is the Jacobi algebra associated to the McKay QP, and the chamber $C \subset \Theta_{\mathbf d}$ containing $\theta$ is given by the inequalities $\theta(\omega^{-1}\e_i)>0$ for any $i \neq 0$.
\end{thm}

In relation to the space of stability conditions $\Theta$ we describe explicitly the part of the chamber structure that contains every moduli space $\M_C$ constructed in Section \ref{sect:opens} (see Theorem \ref{stability}). In other words, considering the dual graph $\mathcal{T}$ of $\Theta$, that is, one vertex for each chamber and an edge between two vertices if the corresponding chambers are separated by a wall, then we can state the following corollary (= Corollary \ref{FRegion}):

\begin{cor}\label{Intro:FRegion} Let $G\subset\SO(3)$ be a finite subgroup of type $\Z/n\Z$, $D_{2n}$ or $\mathbb{T}$. There exists a path in $\mathcal{T}$ containing the chamber $C_0$ where every crepant resolution of $\C^3/G$ can be found and such that every wall crossing in $\mathcal{T}$ corresponds to a flop.
\end{cor}

This nice distribution contrast for example with the general case for Abelian groups in $\SL(3,\C)$, where it can happen that finitely many wall crossings (of {\em Types} $0$ or $III$) are needed to connect two crepant resolutions related by a single flop. See \cite{CI} for more details. \\

The paper is organized as follows. In Section \ref{Sect:Subgroups} we make a brief introduction to the finite subgroups of $\SO(3)$ and their irreducible representations. In Section \ref{Sect:McKayQP-SO(3)} we describe the McKay QP $(Q,W)$ using the \cite{BSW} method for every polyhedral subgroups in $\SO(3)$. Section \ref{Sect:Mutation} describes mutations of quiver with potentials and Section calculates every possible mutation of the McKay QP at non-trivial vertices for subgroups $G\subset\SO(3)$ of types $\Z/n\Z$, $D_{2n}$ and $\mathbb{T}$. In Section \ref{sect:opens} we describe explicitly every projective crepant resolution of $\C^3/G$ with $G\subset\SO(3)$ of types $\Z/n\Z$, $D_{2n}$ and $\mathbb{T}$ as moduli spaces $\mathcal{M}_C$ of representations of the McKay QP. Section \ref{Sect:Stability} is dedicated to the space of stability conditions $\Theta$ for the moduli spaces $\mathcal{M}_C$ and the relation between changing the stability condition and mutating at a vertex $k\in Q$. Finally, in Section \ref{Sect:Floppable} we prove the lemma which allows us to calculate explicitly every crepant resolution by flopping only at $(-1,-1)$-curves and we describe explicitly the contraction algebra for the tetrahedral subgroup $\mathbb{T}$. \\

The authors would like to thank Alastair Craw for his suggestion to study the polyhedral subgroups to the second author when he was visiting Glasgow. We are also grateful to Osamu Iyama and Michael Wemyss for invaluable comments and improvements of this manuscript, Akira Ishii and Kota Yamaura for many useful discussions. 
Finally, we would also like to thank Yukari Ito for bringing us together. 

The first author is supported by FY2009 JSPS Fellowship for Foreign Researchers and JSPS grant No.\ 09F09768, and the second author by JSPS Fellowship for Young Scientists No.\ 21-6922.

\subsection{Conventions} We always take $\C$ as ground field although everything can be done in any algebraically closed field of characteristic 0. 

Abusing the notation, we indistinguishably use vertices in a quiver $Q$ and their corresponding vector spaces in a representations of $Q$. If in addition $Q$ is the McKay quiver we also treat them as irreducible representations of $Q$.

Since we use GIT methods, by crepant resolution $\pi:Y\to X$ we always mean projective.

\section{Finite subgroups of $SO(3)$}\label{Sect:Subgroups}

Let $G$ be a finite subgroup of $\SO(3)$ which consists of rotations about $0\in\R^3$. These groups are the so called {\em polyhedral groups} and are classified into five cases: cyclic, dihedral, tetrahedral, octahedral and icosahedral (see Table \ref{subSO3}). 

\begin{table}[htdp]
\begin{center}\begin{tabular}{|c|c|c|c|}\hline 
Polyhedral group  & Isomorphic group & Order \\\hline 
Cyclic & $\mathbb Z/n\mathbb Z$ & $n$ \\
Dihedral $D_{2n}$ & $\mathbb Z/n\mathbb Z \rtimes \mathbb Z/2\mathbb Z$ & $2n$ \\
Tetrahedral $\mathbb T$ & $A_4$ & 12  \\ 
Octahedral $\mathbb O$ & $S_4$ & 24\\
Icosahedral $\mathbb I$ & $A_5$ & 60\\\hline 
\end{tabular} 
\caption{Finite subgroups of $SO(3)$}
\label{subSO3}
\end{center}
\end{table}

\subsection{The cyclic group of order $n+1$}

Let $G$ be the cyclic subgroup of $SO(3)$ of order $n+1$. Then $G$ is of the form:
\[
G = \frac{1}{n+1}(1,n,0):= \langle \sigma = \left(\begin{array}{ccc} \epsilon &0&0\\ 0& \epsilon^{-1} &0\\ 0&0&1 \end{array}\right) \rangle, \text{ where } \epsilon = e^{2\pi i/(n+1)}.
\]
The character table of $G$ is given in Table \ref{CharCyclic}.

{\renewcommand{\arraystretch}{1.25}
\begin{table}[htdp]
\begin{center}\begin{tabular}{|c|cccccc|}\hline 
Conjugacy classes & 1 & $\sigma$ & $\cdots$ & $\sigma^i$ & $\cdots$ & $\sigma^{n}$  \\\hline 
Number of elements & 1 & 1 & $\cdots$ & 1 & $\cdots$ & 1 \\\hline 
$V_{j}$ & 1 & $\epsilon^j$ & $\cdots$ & $\epsilon^{ij}$ & $\cdots$ & $\epsilon^{jn}$ \\
\hline
\end{tabular}
\caption{Characters of $G$ of type $\Z/n\Z$ with $0\leq j\leq n$.}
\label{CharCyclic}
\end{center}
\end{table}}

\subsection{The dihedral group of order $2n$}
Let $n$ be a positive integer and $G$ be the dihedral subgroup $D_{2n}$ in $SO(3)$ of order $2n$. Then $G$ is generated by:
\[
G = \langle \sigma=\left(\begin{array}{ccc} \epsilon &0&0\\ 0& \epsilon^{n-1} &0\\ 0&0&1 \end{array}\right), \tau=\left(\begin{array}{ccc} 0&1&0\\ 1&0&0\\ 0&0&-1 \end{array}\right) \rangle, \text{ where } \epsilon = e^{2\pi i/n}.
\]
These groups are divided into two cases depending on the parity of $n$. 

For the case $n=2m$ even, the group has four 1-dimensional irreducible  representations $V_0$, $V_{0'}$, $V_m$ and $V_{m'}$, and $m-1$ irreducible representations $V_j$ of dimension 2 for $1\leq j \leq m-1$. The character table is given in Table \ref{D-even-table}. 

{\renewcommand{\arraystretch}{1.25}
\begin{table}[htdp]
\begin{center}\begin{tabular}{|c|ccccc|}\hline 
c.c. & 1 & -1 & $\tau$ & $\tau\sigma$ & $\sigma^i$ \\\hline 
$\#$  & 1 &  1 & $m$ & $m$ & 2 \\\hline 
$V_{0}$ & 1 & 1 & 1 & 1 & 1 \\
$V_{0'}$ & $1$ & $1$ & $-1$ & $-1$ & $1$ \\
$V_j$ & $2$ & $(-1)^j2$ & $0$ & $0$ & $\epsilon^{ij}+\epsilon^{-ij}$ \\
$V_{m}$ & $1$ & $(-1)^m$ & $1$ & $-1$ & $(-1)^i$\\
$V_{m'}$ & $1$ & $(-1)^m$ & $-1$ & $1$ & $(-1)^i$\\
\hline
\end{tabular}
\caption{Characters of $D_{2n}$, with $n=2m$ even and $1\leq i,j \leq m-1$.}
\label{D-even-table}
\end{center}
\end{table}}

For the case $n=2m+1$ odd, the group has two 1-dimensional representations $V_0$ and $V_{0'}$, and $m$ 2-dimensional representations $V_j$ for $1\leq j \leq m$. The character table is given in Table \ref{D-odd-table}. 

{\renewcommand{\arraystretch}{1.25}
\begin{table}[htdp]
\begin{center}\begin{tabular}{|c|ccc|}\hline 
c.c. & 1 & $\tau$ & $\sigma^i$ \\\hline 
$\#$ & 1 & $2m+1$ & 2 \\\hline 
$V_0$ & 1 & 1 & 1 \\
$V_{0'}$ & $1$ & $-1$ & $1$ \\
$V_j$ & $2$ & $0$ & $\epsilon^{ij}+\epsilon^{-ij}$\\
\hline
\end{tabular}
\caption{Characters of $D_{2n}$, with $n=2m+1$ odd and $1\leq i,j \leq m$.}
\label{D-odd-table}
\end{center}
\end{table}}

In both cases, the representations $V_j$ are realized as $V_j(\sigma)= \left(\!\begin{smallmatrix}\epsilon^j &0\\ 0&\epsilon^{-j}\\ \end{smallmatrix}\!\right), V_j(\tau)= \left(\!\begin{smallmatrix} 0&1\\1&0\\ \end{smallmatrix}\!\right)$.

\subsection{The tetrahedral group}
Let $G$ be the tetrahedral subgroup $\mathbb T$ of $SO(3)$. Then
\[
G = \langle \sigma=\left(\begin{array}{ccc} -1&0&0\\ 0&-1&0\\ 0&0&1 \end{array}\right), \tau=\left(\begin{array}{ccc} 0&1&0\\ 0&0&1\\ 1&0&0 \end{array}\right) \rangle.
\]
The group $G$ is isomorphic to the alternating group $A_4$ and it has order $12$. This group is also called {\em trihedral group of order 12} and the character table of $G$ is shown in Table \ref{Char-T}.

{\renewcommand{\arraystretch}{1.25}
\begin{table}[htdp]
\begin{center}\begin{tabular}{|c|cccc|}\hline 
c.c. & 1 & $\sigma$ & $\tau$ & $\tau^2$ \\\hline 
$\#$ & 1 & 3 & 4 & 4 \\\hline 
$V_0$ & 1 & 1 & 1 & 1\\
$V_1$ & $1$ & $1$ & $\omega$ & $\omega^2$ \\
$V_2$ & $1$ & $1$ & $\omega^2$ & $\omega$ \\
$V_3$ & $3$ & $-1$ & 0 & 0\\\hline
\end{tabular}
\caption{Characters of tetrahedral group of order 12.}
\label{Char-T}
\end{center}
\end{table}}

\subsection{The octahedral group}
Let $G$ be the octahedral subgroup $\mathbb O$ of $SO(3)$. Then:
\[
G = \langle \sigma=\left(\begin{array}{ccc} 0&-1&0\\ 1&0&0\\ 0&0&1 \end{array}\right), \tau=\left(\begin{array}{ccc} 0&1&0\\ 0&0&1\\ 1&0&0 \end{array}\right) \rangle.
\]
The group $G$ is isomorphic to the symmetric group $S_4$ and its order is $24$. The character table of $G$ is given in Table \ref{Char-E7}.

{\renewcommand{\arraystretch}{1.25}
\begin{table}[htdp]
\begin{center}\begin{tabular}{|c|ccccc|}\hline 
c.c. & 1 & $\sigma^2$ & $\tau$ & $\sigma$ & $\sigma\tau\sigma^2$ \\\hline 
$\#$ & 1 & 3 & 8 & 6 & 6 \\\hline 
$V_0$ & 1 & 1 & 1 & 1 & 1\\
$V_1$ & $1$ & $1$ & $1$ & $-1$ & $-1$ \\
$V_2$ & $2$ & $2$ & $-1$ & 0 & 0 \\
$V_3$ & $3$ & $-1$ & 0 & 1 & $-1$\\
$V_4$ & $3$ & $-1$ & 0 & $-1$ & $1$\\\hline
\end{tabular}
\caption{Characters of octahedral group}
\label{Char-E7}
\end{center}
\end{table}}

The representation $V_{3}$ is the representation given by the inclusion of $G$ into $SO(3)$, i.e.\ it is the natural representation. The irreducible representations $V_2$ and $V_{4}$ are realized as $V_2(\sigma)= \left(\begin{smallmatrix}0&1\\1&0\\ \end{smallmatrix}\right)$, $V_2(\tau)= \left(\begin{smallmatrix} \omega&0\\0&\omega^2\\ \end{smallmatrix}\right)$, $V_{4}(\sigma)=-\sigma, V_{4}(\tau)=\tau$.

\subsection{The icosahedral group}
Let $G$ be the icosahedral subgroup $\mathbb I$ of $SO(3)$:
\[
G = \langle \sigma=\left(\begin{array}{ccc} 1&0&0\\ 0&\epsilon&0\\ 0&0&\epsilon^4 \end{array}\right), \tau=\displaystyle{\frac{1}{\sqrt{5}}}\left(\begin{array}{ccc} 1&1&1\\ 2&s&t\\ 2&t&s \end{array}\right) \rangle,
\]
where $\epsilon = e^{2\pi i/5}$, $s = \epsilon^2 + \epsilon^3 = \frac{-1-\sqrt{5}}{2}$ and $s = \epsilon + \epsilon^4 = \frac{-1+\sqrt{5}}{2}$. Note that $v = \left(\begin{smallmatrix} -1&0&0\\ 0&0&-1\\ 0&-1&0 \end{smallmatrix}\right) = \sigma^3\tau$ (compare \cite[3.1]{GNS2} and \cite{YY}). The group $G$ is isomorphic to the alternating group $A_5$ and its order is $60$. The character table of $G$ is given in Table \ref{Char-G60}.

{\renewcommand{\arraystretch}{1.25}
\begin{table}[htdp]
\begin{center}\begin{tabular}{|c|ccccc|}\hline 
c.c. & 1 & $\sigma\tau$ & $\tau$ & $\sigma$ & $\sigma^2$ \\\hline 
$\#$ & 1 & 20 & 15 & 12 & 12 \\\hline 
$V_0$ & 1 & 1 & 1 & 1 & 1\\
$V_1$ & $3$ & $0$ & -1 & $-s$ & $-t$\\
$V_2$ & $3$ & $0$ & -1 & $-t$ & $-s$\\
$V_3$ & $4$ & $1$ & 0 & $-1$ & $-1$\\
$V_4$ & $5$ & $-1$ & 1 & $0$ & $0$\\\hline
\end{tabular}
\caption{Characters of icosahedral group}
\label{Char-G60}
\end{center}
\end{table}}

The natural representation is $V_{1}$, the irreducible representation $V_{2}$ is realized as $V_{2}(\sigma)=\sigma^2, V_{2}(\tau)=\tau$, and the 4-dimensional irreducible representation $V_3$ is obtained by removing the unit representation from the permutation representation of $A_5$ on $\{a,b,c,d,e\}$. If we take a suitable basis of $V_3$, it is realized as:
\[
V_3(\sigma)= \left(\begin{array}{cccc}\epsilon&0&0&0\\0&\epsilon^2&0&0\\0&0&\epsilon^3&0\\0&0&0&\epsilon^4 \end{array}\right), V_3(\tau)=\displaystyle{\frac{1}{\sqrt{5}}}\left(\begin{array}{cccc} 1&t&-s&-1\\t&-1&1&-s\\-s&1&-1&t\\-1&-s&t&1 \end{array}\right).
\]
The 5-dimensional irreducible representation $V_4$ is a representation obtained by removing the unit representation from the permutation representation of $A_5$ on the set of its 6 subgroups of order 5. Taking a suitable basis, it is realized as:
\[
V_4(\sigma)= \left(\begin{array}{ccccc}1&0&0&0&0\\0&\epsilon&0&0&0\\0&0&\epsilon^2&0&0\\0&0&0&\epsilon^3&0\\0&0&0&0&\epsilon^4 \end{array}\right), V_4(\tau)=\displaystyle{\frac{1}{5}}\left(\begin{array}{ccccc} -1&-6&-6&-6&-6\\-1&1-t&2s&2t&1-s\\-1&2s&1-s&1-t&2t\\-1&2t&1-t&1-s&2s\\-1&1-s&2t&2s&1-t\\ \end{array}\right).
\]

\section{McKay quivers with potential for $G \subset SO(3)$}
\label{Sect:McKayQP-SO(3)}
 
Let $Q$ be an arbitrary finite connected quiver (possibly with loops and 2-cycles) with a vertex set $Q_0$ and an arrow set $Q_1$. For an arrow $a\in Q_1$, denote by $ha$ and $ta$ the head and tail of $a$ respectively.
Let $\C Q$ be the path algebra and denote by $\C Q_i$ the $\C$-vector space with basis $Q_i$ consisting of paths of length $i$ in $Q$, and by $\C Q_{i,\mathrm{cyc}}$ the subspace of $\C Q_i$ spanned by all cycles.
A {\it quiver with potential} ({\it QP} for short) is a pair $(Q,W)$ consisting of $Q$ and an element $W \in \bigoplus_{i \geq 2} \C Q_{i,\mathrm{cyc}}$ called {\it potential}. 
For an arrow $a \in Q_1$, the cyclic derivative $\partial_a W$ is defined by $\partial_a(a_1\cdots a_{\ell}) = \sum_{a_i=a}a_{i+1}\cdots a_{\ell}a_1\cdots a_{i-1}$ and extended linearly.
The {\it Jacobian algebra} of a QP $(Q,W)$ is defined by
\[
\mathcal P(Q,W) := \C Q/\langle \partial_a W \mid a \in Q_1 \rangle.
\]

Let $Q$ be the McKay quiver of $G$, that is the quiver such that the vertex set is the set of irreducible representations $V_i$ of $G$ and we draw $a_{ij}$ arrows from $V_i$ to $V_j$ where $a_{ij} := \dim_{\C } \Hom_{\C G}(V_i,V^{*}\otimes V_j)$. In this case it is well known that $\mathcal{P}(Q,W)$ is Morita equivalent to the skew group algebra $S\ast G$, where $S*G$ is a free $S$-module $S\otimes_{\C } G$ with basis $G$ with multiplications given by $(s\otimes g)(s'\otimes g')=sg(s')\otimes gg'$ for any $s,s' \in S$ and $g,g'\in G$. 

Let us now restrict to the case of polyhedral subgroups $G\subset\SO(3)$. The description of the potential $W$ can be calculated following the method provided in \cite{BSW}, which we now briefly sketch.

Take the standard basis $v_1,v_2$ and $v_3$ of $V=\C^3$. Note that in this case $G$ acts on $V$ naturally and dually on the polynomial ring $S:=\C [V]$. Then $\Hom_{\C G}(V_i,V^{*}\otimes V_j)$ is isomorphic to $\Hom_{\C G}(V_j,V\otimes V_i)$ as a $\C$-vector space. For each arrow $a \in Q_1$ we consider the $G$-equivariant homomorphism $\varphi_{a} : V_{t(a)} \to V\otimes V_{h(a)}$.
If $p=abc$ is a closed path of length $3$, then by Schur's lemma, the composition of maps
\[\xymatrix{
V_{t(a)} \ar[r]^(.4){\varphi_{a}} & 
*++{V^{*}\otimes V_{t(b)}} \ar[r]^(.45){id_{V} \otimes \varphi_{b}} &
*++{V^{\otimes 2}\otimes V_{t(c)}} \ar[r]^(.5){id_{V^{\otimes 2}} \otimes \varphi_{c}} & 
*++{V^{\otimes 3}\otimes V_{h(c)}} \ar[r]^(.5){\alpha \otimes id_{V_{h(c)}}} & 
\bigwedge^3 V\otimes V_{h(c)} \ar[r]^{} \ar[r]^(.5){\sim} & 
V_{h(c)=t(a)} 
}\]
is a constant denoted by $c_p$. Note that in the above sequence $\alpha : V^{\otimes 3} \to \bigwedge^3 V$ is the antisymmetrizer, the last map is the isomorphism given by the composition of $\bigwedge^3 V \to V_0$ defined by $v_1\wedge v_2\wedge v_3 \to \ell_{0}$ and $V_0\otimes V_{h(c)} \to V_{h(c)}$ defined by $\ell_{0}\otimes v \to v$, where $\ell_0$ is the basis for $V_0$ (recall that $\bigwedge^3 V\cong V_0$ since $G\subset\SL(3,\C)$). 

\begin{thm}[{\cite[Theorem 3.2]{BSW}}]
If we take $W=\sum_{|p|=3} c_p(\dim V_{h(p)}) p$ then $\mathcal P(Q,W)$ is Morita equivalent to $S*G$.
\end{thm}

In the following sections we give the explicit description of the McKay QP $(Q,W)$ for every finite subgroup $G\subset\SO(3)$. For simplicity we write $Q_0 = \{ 0,\ldots,n \}$ where $i \in Q_0$ corresponds to the irreducible representation $V_i$. In particular, $0$ corresponds to the trivial representation $V_0$.

\subsection{The cyclic group of order $n+1$}
\label{QP-cyclic}

Let $G$ be a finite cyclic subgroup of $SO(3)$ of order $n+1$. The McKay quiver $Q$ of $G$ is as follows:

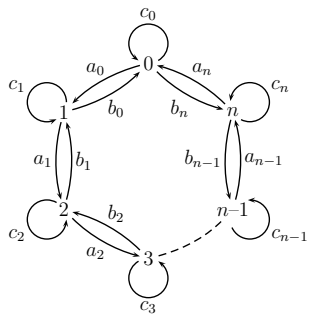
\begin{center}
\scalebox{.65}{
\begin{pspicture}(-2,-3)(2,3.15)
	\psset{arcangle=-15,nodesep=2pt}
\rput(0,0){
	\rput(0,2){\rnode{0}{\large $0$}}
	\rput(-1.73,1){\rnode{1}{\large $1$}}
	\rput(-1.73,-1){\rnode{2}{\large $2$}}
	\rput(0,-2){\rnode{3}{\large $3$}}
	\rput(1.73,-1){\rnode{4}{\large $n\text{--}1$}}
	\rput(1.73,1){\rnode{5}{\large $n$}}
	\ncarc{->}{0}{1}\Bput[0.05]{\large $a_0$}
	\ncarc{->}{1}{2}\Bput[0.05]{\large $a_1$}
	\ncarc{->}{2}{3}\Bput[0.05]{\large $a_2$}
	\ncarc[linestyle=dashed,nodesep=3pt]{-}{3}{4}
	\ncarc{->}{4}{5}\Bput[0.05]{\large $a_{n-1}$}
	\ncarc{->}{5}{0}\Bput[0.05]{\large $a_{n}$}
	\ncarc{->}{0}{5}\Bput[0.05]{\large $b_{n}$}
	\ncarc{->}{1}{0}\Bput[0.05]{\large $b_0$}
	\ncarc{->}{2}{1}\Bput[0.05]{\large $b_1$}
	\ncarc{->}{3}{2}\Bput[0.05]{\large $b_2$}
	\ncarc{->}{5}{4}\Bput[0.05]{\large $b_{n-1}$}
	\nccircle[nodesep=3pt]{->}{0}{.4cm}\Bput[0.05]{\large $c_0$}
	\nccircle[angleA=60,nodesep=3pt]{->}{1}{.4cm}\Bput[0.05]{\large $c_1$}
	\nccircle[angleA=120,nodesep=3pt]{->}{2}{.4cm}\Bput[0.05]{\large $c_2$} 
	\nccircle[angleA=180,nodesep=3pt]{->}{3}{.4cm}\Bput[0.05]{\large $c_3$}
	\nccircle[angleA=240,nodesep=1pt]{->}{4}{.4cm}\Bput[0.05]{\large $c_{n-1}$}
	\nccircle[angleA=300,nodesep=3pt]{->}{5}{.4cm}\Bput[0.05]{\large $c_{n}$}
	}
\end{pspicture}
}
\end{center}

For each arrow, the corresponding $G$-equivariant homomorphism is given by:
\begin{align*}
\varphi_{a_i}: & V_{i} \to V\otimes V_{i+1} & \varphi_{b_i}: & V_{i+1} \to V\otimes V_{i} & \varphi_{c_i}: & V_{i} \to V\otimes V_{i}\\
& \ell_{i} \mapsto v_2\otimes \ell_{i+1} & & \ell_{i+1} \mapsto v_1\otimes \ell_{i} & & \ell_{i} \mapsto v_3\otimes \ell_{i}
\end{align*}
where $\ell_i$ denotes a basis of $V_{i}$ for any $i=0,\ldots,n$. For any $i=0,\ldots,n$,  it follows that $c_p=-1$ if $p=c_ia_ib_i$ and  $c_p=1$ if $p=c_ib_{i-1}a_{i-1}$, where $a_{-1}=a_{n}$ and $b_{-1}=b_{n}$. By definition of $c_p$, for all other 3-cycles $p$ we have $c_p=0$.
Hence the McKay potential is given by
\[
W = -\sum_{i=0}^{n}a_ib_ic_i +\sum_{i=0}^{n}b_{i-1}a_{i-1}c_i.
\]

\subsection{The dihedral group of order $2n$ ($n$ even)}
\label{QP-D-even}

Let $G$ be a dihedral group $D_{2n}$ where $n=2m$ for some positive integer $m$. The McKay quiver $Q$ of $G$ is as follows:

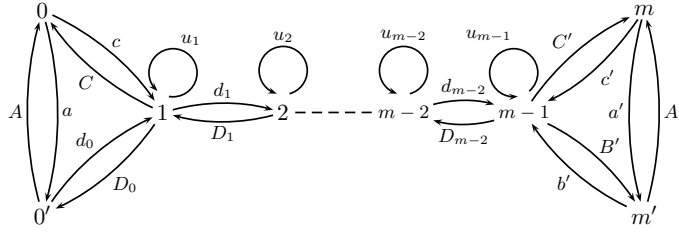
\begin{center}
\scalebox{.8}{
\begin{pspicture}(-1,-2)(10,2)
	\psset{arcangle=15,nodesep=2pt}
\rput(0,0){
	\rput(-0,1.7){\rnode{0}{$0$}}
	\rput(-0,-1.7){\rnode{1}{$0'$}}
	\rput(2,0){\rnode{2}{$1$}}
	\rput(4,0){\rnode{3}{$2$}}
	\rput(6,0){\rnode{4}{\footnotesize$m-2$}}
	\rput(8,0){\rnode{5}{\footnotesize$m-1$}}
	\rput(10,1.7){\rnode{6}{$m$}}
	\rput(10,-1.7){\rnode{7}{$m'$}}	
	\rput(2,0.3){\rnode{u1}{}}	
	\rput(4,0.3){\rnode{u2}{}}	
	\rput(6,0.3){\rnode{u3}{}}	
	 \rput(8,0.3){\rnode{u4}{}}	
	\ncarc{->}{0}{1}\Aput[0.05]{\footnotesize $a$}
	\ncarc{->}{1}{0}\Aput[0.05]{\footnotesize $A$}
	\ncarc{->}{1}{2}\Aput[0.05]{\footnotesize $d_0$}
	\ncarc{->}{2}{1}\Aput[0.05]{\footnotesize $D_0$}
	\ncarc{->}{0}{2}\Aput[0.05]{\footnotesize $c$}
	\ncarc{->}{2}{0}\Aput[0.05]{\footnotesize $C$}
	\ncarc{->}{2}{3}\Aput[0.05]{\footnotesize $d_1$}
	\ncarc{->}{3}{2}\Aput[0.05]{\footnotesize $D_1$}
	\ncline[linestyle=dashed,nodesep=3pt]{-}{3}{4}
	\ncarc{->}{4}{5}\Aput[0.05]{\footnotesize $d_{m-2}$}
	\ncarc{->}{5}{4}\Aput[0.05]{\footnotesize $D_{m-2}$}
	\ncarc{->}{5}{6}\Aput[0.05]{\footnotesize $C'$}
	\ncarc{->}{6}{5}\Aput[0.05]{\footnotesize $c'$}
	\ncarc{->}{5}{7}\Aput[0.05]{\footnotesize $B'$}
	\ncarc{->}{7}{5}\Aput[0.05]{\footnotesize $b'$}
	\ncarc{<-}{6}{7}\Aput[0.05]{\footnotesize $A'$}
	\ncarc{<-}{7}{6}\Aput[0.05]{\footnotesize $a'$}
	 \nccircle[angleA=-25,nodesep=3pt]{->}{u1}{.4cm}\Bput[0.05]{\footnotesize $u_1$}
	 \nccircle[nodesep=3pt]{->}{u2}{.4cm}\Bput[0.05]{\footnotesize $u_2$}
	 \nccircle[nodesep=3pt]{->}{u3}{.4cm}\Bput[0.05]{\footnotesize $u_{m-2}$}
	 \nccircle[angleA=25,nodesep=3pt]{->}{u4}{.4cm}\Bput[0.05]{\footnotesize $u_{m-1}$}
	}
\end{pspicture}
}
\end{center}

Abusing the notation, the corresponding $G$-equivariant maps in matrix form are: 

$A=a=A'=a'=v_3$, 
$d_0 = \begin{pmatrix} v_2, -v_1 \end{pmatrix}$, 
$D_0 = \begin{pmatrix} v_1\\ -v_2 \end{pmatrix}$,
$c = \begin{pmatrix} v_2, v_1 \end{pmatrix}$, 
$C = \begin{pmatrix} v_1\\ v_2 \end{pmatrix}$,
$b' = \begin{pmatrix} -v_1, -v_2 \end{pmatrix}$, 
$B' = \begin{pmatrix} -v_2\\ -v_1 \end{pmatrix}$,
$c' = \begin{pmatrix} v_1, v_2 \end{pmatrix}$, 
$C' = \begin{pmatrix} v_2\\ v_1 \end{pmatrix}$,
$d_i = \begin{pmatrix} v_2& 0\\ 0&v_1 \end{pmatrix}$ and
$D_i = \begin{pmatrix} v_1& 0\\ 0&v_2 \end{pmatrix}$ for $0\leq i \leq m-2$,
and $u_i = \begin{pmatrix} v_3& 0\\ 0&-v_3 \end{pmatrix}$ for $0\leq i \leq m-1$.

With this notation, $d_0 = \begin{pmatrix} v_2, -v_1 \end{pmatrix}$ means the $G$-equivariant map $\varphi_{d_0} : V_{0'} \to V\otimes V_{1}$ is defined by $\ell_{0'} \mapsto v_2\otimes \ell_1^1 -v_1\otimes \ell_1^2$, where $\ell_{0'}$ is the basis of $V_{0'}$ and $\{\ell_1^1,\ell_1^2\}$ is the basis of $V_1$ given in the previous section. Note that the above description depends on the choice of basis for the $V_i$'s.

For the above equivariant maps, one can calculate the potential to obtain:
\begin{align*}
W/2 &= -ad_0C-cD_0A +u_1D_0d_0+u_1Cc -\sum_{i=1}^{m-2}u_id_{i}D_{i}+\sum_{i=2}^{m-1}u_{i}D_{i-1}d_{i-1}\\
&-u_{m-1}B'b'-u_{m-1}C'c'-a'b'C'-c'B'A'.
\end{align*}

\subsection{The dihedral group of order $2n$ ($n$ odd)}
\label{QP-D-odd}

Let $G$ be a dihedral group $D_{2n}$ where $n=2m+1$ for some positive integer $m$. The McKay quiver $Q$ of $G$ is as follows:

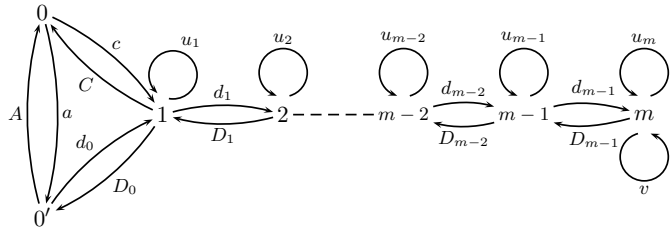
\begin{center}
\scalebox{.8}{
\begin{pspicture}(-1,-2)(10,2)
	\psset{arcangle=15,nodesep=2pt}
\rput(0,0){
	\rput(-0,1.7){\rnode{0}{$0$}}
	\rput(-0,-1.7){\rnode{1}{$0'$}}
	\rput(2,0){\rnode{2}{$1$}}
	\rput(4,0){\rnode{3}{$2$}}
	\rput(6,0){\rnode{4}{\footnotesize$m-2$}}
	\rput(8,0){\rnode{5}{\footnotesize$m-1$}}
	\rput(10,0){\rnode{6}{$m$}}
	\rput(2,0.3){\rnode{u1}{}}	
	\rput(4,0.3){\rnode{u2}{}}	
	\rput(6,0.3){\rnode{u3}{}}	
	\rput(8,0.3){\rnode{u4}{}}	
	\rput(10,0.3){\rnode{u5}{}}	
	\rput(10,-0.3){\rnode{u6}{}}	
	\ncarc{->}{0}{1}\Aput[0.05]{\footnotesize $a$}
	\ncarc{->}{1}{0}\Aput[0.05]{\footnotesize $A$}
	\ncarc{->}{1}{2}\Aput[0.05]{\footnotesize $d_0$}
	\ncarc{->}{2}{1}\Aput[0.05]{\footnotesize $D_0$}
	\ncarc{->}{0}{2}\Aput[0.05]{\footnotesize $c$}
	\ncarc{->}{2}{0}\Aput[0.05]{\footnotesize $C$}
	\ncarc{->}{2}{3}\Aput[0.05]{\footnotesize $d_1$}
	\ncarc{->}{3}{2}\Aput[0.05]{\footnotesize $D_1$}
	\ncline[linestyle=dashed]{-}{3}{4}
	\ncarc{->}{4}{5}\Aput[0.05]{\footnotesize $d_{m-2}$}
	\ncarc{->}{5}{4}\Aput[0.05]{\footnotesize $D_{m-2}$}
	\ncarc{->}{5}{6}\Aput[0.05]{\footnotesize $d_{m-1}$}
	\ncarc{->}{6}{5}\Aput[0.05]{\footnotesize $D_{m-1}$}
	 \nccircle[angleA=-25,nodesep=3pt]{->}{u1}{.4cm}\Bput[0.05]{\footnotesize $u_1$}
	 \nccircle[nodesep=3pt]{->}{u2}{.4cm}\Bput[0.05]{\footnotesize $u_2$}
	 \nccircle[nodesep=3pt]{->}{u3}{.4cm}\Bput[0.05]{\footnotesize $u_{m-2}$}
	 \nccircle[nodesep=3pt]{->}{u4}{.4cm}\Bput[0.05]{\footnotesize $u_{m-1}$}
	 \nccircle[nodesep=3pt]{->}{u5}{.4cm}\Bput[0.05]{\footnotesize $u_{m}$}
	 \nccircle[angleA=180,nodesep=3pt]{->}{u6}{.4cm}\Bput[0.05]{\footnotesize $v$}
	}
\end{pspicture}
}
\end{center}

and the corresponding $G$-equivariant maps are 

$A = a = v_3$, 
$d_0 = \begin{pmatrix} v_2, -v_1 \end{pmatrix}$, 
$D_0 = \begin{pmatrix} v_1\\ -v_2 \end{pmatrix}$,
$c = \begin{pmatrix} v_2, v_1 \end{pmatrix}$, 
$C = \begin{pmatrix} v_1\\ v_2 \end{pmatrix}$,
$d_i = \begin{pmatrix} v_2& 0\\ 0&v_1 \end{pmatrix}$ and
$D_i = \begin{pmatrix} v_1& 0\\ 0&v_2 \end{pmatrix}$ for $0\leq i \leq m-1$,
$u_i = \begin{pmatrix} v_3& 0\\ 0&-v_3 \end{pmatrix}$ for $0\leq i \leq m$, and
$v = \begin{pmatrix} 0& v_2\\ v_1&0 \end{pmatrix}$.

For the above choice of equivariant maps, the potential is given by
\begin{align*}
W/2 &= -ad_0C-cD_0A +u_1D_0d_0+u_1Cc -\sum_{i=1}^{m-1}u_id_{i}D_{i}+\sum_{i=2}^{m}u_{i}D_{i-1}d_{i-1}-u_mv^2.
\end{align*}

\subsection{The tetrahedral group}
\label{QP-tetra}

Let $G$ be the tetrahedral group of order 12. The McKay quiver $Q$ of $G$ is the following:

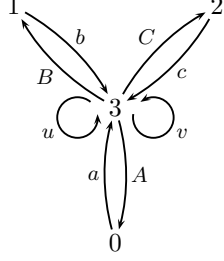
\begin{center}
\scalebox{.9}{
\begin{pspicture}(-2,-2)(2,1.5)
	\psset{arcangle=15,nodesep=2pt}
\rput(0,0){
	\rput(0,-2){\rnode{0}{$0$}}
	\rput(0,0){\rnode{3}{$3$}}
	\rput(-1.5,1.5){\rnode{1}{$1$}}
	\rput(1.5,1.5){\rnode{2}{$2$}}
	\rput(-0.3,0){\rnode{u}{}}	
	\rput(0.3,0){\rnode{v}{}}	
	\ncarc{->}{0}{3}\Aput[0.05]{\footnotesize $a$}
	\ncarc{->}{3}{0}\Aput[0.05]{\footnotesize $A$}
	\ncarc{->}{3}{1}\Aput[0.05]{\footnotesize $B$}
	\ncarc{->}{1}{3}\Aput[0.05]{\footnotesize $b$}
	\ncarc{->}{3}{2}\Aput[0.05]{\footnotesize $C$}
	\ncarc{->}{2}{3}\Aput[0.05]{\footnotesize $c$}
	\nccircle[angleA=120,nodesep=3pt]{->}{u}{.3cm}\Bput[0.05]{\footnotesize $u$}
	 \nccircle[angleA=240,nodesep=3pt]{->}{v}{.3cm}\Bput[0.05]{\footnotesize $v$}
	}
\end{pspicture}
}
\end{center}

and the corresponding $G$-equivariant maps are:

$a = \begin{pmatrix} v_1, v_2, v_3 \end{pmatrix}$, 
$A = \begin{pmatrix} v_1 \\ v_2 \\ v_3 \end{pmatrix}$, 
$b = \begin{pmatrix} v_1, \omega v_2, \omega^2 v_3 \end{pmatrix}$, 
$B = \begin{pmatrix} v_1 \\ \omega^2 v_2 \\ \omega v_3 \end{pmatrix}$, 
$c = \begin{pmatrix} v_1, \omega^2 v_2, \omega v_3 \end{pmatrix}$, 
$C = \begin{pmatrix} v_1 \\ \omega v_2 \\ \omega^2 v_3 \end{pmatrix}$, 
$u = \begin{pmatrix} 0&0&v_2 \\ v_3&0&0 \\ 0&v_1&0 \end{pmatrix}$ and
$v = \begin{pmatrix} 0&v_3&0 \\ 0&0&v_1 \\ v_2&0&0 \end{pmatrix}$.

For the above choice of equivariant maps, the potential is given by
\begin{align*}
W/3 &= uAa+\omega uBb +\omega^2 uCc -\frac{1}{3}u^3  -vAa-\omega^2 vBb-\omega vCc +\frac{1}{3}v^3.
\end{align*}

\subsection{The octahedral group}
\label{QP-octa}

Let $G$ be the tetrahedral group. The McKay quiver $Q$ of $G$ is the following:

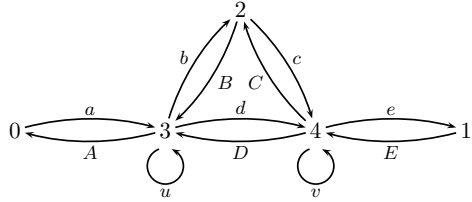
\begin{center}
\scalebox{.8}{
\begin{pspicture}(0,-1)(6,2.1)
	\psset{arcangle=15,nodesep=2pt}
\rput(0,0){
	\rput(0,0){\rnode{0}{$0$}}
	\rput(2.5,0){\rnode{1}{$3$}}
	\rput(5,0){\rnode{3}{$4$}}
	\rput(3.75,2){\rnode{2}{$2$}}
	\rput(7.5,0){\rnode{4}{$1$}}
	\rput(2.5,-.3){\rnode{u}{}}	
	\rput(5,-.3){\rnode{v}{}}	
	\ncarc{->}{0}{1}\Aput[0.05]{\footnotesize $a$}
	\ncarc{->}{1}{0}\Aput[0.05]{\footnotesize $A$}
	\ncarc{->}{1}{2}\Aput[0.05]{\footnotesize $b$}
	\ncarc{->}{2}{1}\Aput[0.05]{\footnotesize $B$}
	\ncarc{->}{2}{3}\Aput[0.05]{\footnotesize $c$}
	\ncarc{->}{3}{2}\Aput[0.05]{\footnotesize $C$}
	\ncarc{->}{1}{3}\Aput[0.05]{\footnotesize $d$}
	\ncarc{->}{3}{1}\Aput[0.05]{\footnotesize $D$}
	\ncarc{->}{3}{4}\Aput[0.05]{\footnotesize $e$}
	\ncarc{->}{4}{3}\Aput[0.05]{\footnotesize $E$}
	\nccircle[angleA=180,nodesep=3pt]{->}{u}{.3cm}\Bput[0.05]{\footnotesize $u$}
	 \nccircle[angleA=180,nodesep=3pt]{->}{v}{.3cm}\Bput[0.05]{\footnotesize $v$}
	}
\end{pspicture}
}
\end{center}

and the $G$-equivariant maps are:

$a = E = \begin{pmatrix} v_1, v_2, v_3 \end{pmatrix}$,
$A = e = \begin{pmatrix} v_1\\ v_2\\ v_3 \end{pmatrix}$,
$b = \begin{pmatrix} v_1&\omega v_1\\ \omega v_2&v_2 \\ \omega^2 v_3&v_3 \end{pmatrix}$,
$B = \begin{pmatrix} v_1&\omega^2 v_2&\omega v_3\\ \omega^2 v_1&v_2&\omega v_3 \end{pmatrix}$,
$c = \begin{pmatrix} v_1&\omega^2 v_2&\omega v_3\\ -\omega^2 v_1&-v_2&-\omega v_3 \end{pmatrix}$,
$C = \begin{pmatrix} v_1&-\omega v_1\\ \omega v_2&-v_2 \\ \omega^2 v_3&-v_3 \end{pmatrix}$,
$d = D = \begin{pmatrix} 0&v_3&v_2\\ v_3&0&v_1\\ v_2&v_1&0 \end{pmatrix}$ and
$u = v = \begin{pmatrix} 0&-v_3&v_2\\ v_3&0&-v_1\\ -v_2&v_1&0 \end{pmatrix}$.

For the above equivariant maps, the potential is given by
\[
W/6 = uAa-ubB-udD-1/3u^3 +veE-vCc-vDd+1/3v^3 +(w^2-w)dCB+(w^2-w)Dbc.
\]

\subsection{The icosahedral group}
\label{QP-icosa}

Let $G$ be the tetrahedral group. The McKay quiver $Q$ of $G$ is as follows:

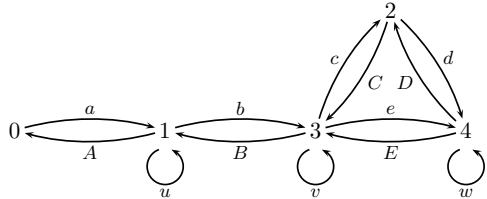
\begin{center}
\scalebox{.8}{
\begin{pspicture}(0,-1)(8,2.1)
	\psset{arcangle=15,nodesep=2pt}
\rput(0,0){
	\rput(0,0){\rnode{0}{$0$}}
	\rput(2.5,0){\rnode{1}{$1$}}
	\rput(5,0){\rnode{2}{$3$}}
	\rput(6.25,2){\rnode{3}{$2$}}
	\rput(7.5,0){\rnode{4}{$4$}}
	\rput(2.5,-.3){\rnode{u}{}}	
	\rput(5,-.3){\rnode{v}{}}	
	\rput(7.5,-.3){\rnode{w}{}}	
	\ncarc{->}{0}{1}\Aput[0.05]{\footnotesize $a$}
	\ncarc{->}{1}{0}\Aput[0.05]{\footnotesize $A$}
	\ncarc{->}{1}{2}\Aput[0.05]{\footnotesize $b$}
	\ncarc{->}{2}{1}\Aput[0.05]{\footnotesize $B$}
	\ncarc{->}{2}{3}\Aput[0.05]{\footnotesize $c$}
	\ncarc{->}{3}{2}\Aput[0.05]{\footnotesize $C$}
	\ncarc{->}{3}{4}\Aput[0.05]{\footnotesize $d$}
	\ncarc{->}{4}{3}\Aput[0.05]{\footnotesize $D$}
	\ncarc{->}{2}{4}\Aput[0.05]{\footnotesize $e$}
	\ncarc{->}{4}{2}\Aput[0.05]{\footnotesize $E$}
	\nccircle[angleA=180,nodesep=3pt]{->}{u}{.3cm}\Bput[0.05]{\footnotesize $u$}
	 \nccircle[angleA=180,nodesep=3pt]{->}{v}{.3cm}\Bput[0.05]{\footnotesize $v$}
	 \nccircle[angleA=180,nodesep=3pt]{->}{w}{.3cm}\Bput[0.05]{\footnotesize $w$}
	}
\end{pspicture}
}
\end{center}

and the $G$-equivariant maps are:

{\footnotesize
$a = \begin{pmatrix} 2v_1, v_3, v_2 \end{pmatrix}$,
$A = \begin{pmatrix} v_1 \\ v_2 \\ v_3 \end{pmatrix}$,
$u = \begin{pmatrix} 0&v_3&-v_2\\ 2v_2&-2v_1&0\\ -2v_3&0&2v_1 \end{pmatrix}$,
$b = \begin{pmatrix} -2v_1&3v_3&0&0&3v_2\\ v_2&6v_1&6v_3&0&0\\ v_3&0&0&6v_2&6v_1 \end{pmatrix}$,

$B = \begin{pmatrix} -4v_1&v_3&v_2\\ v_2&v_1&0\\ 0&v_2&0\\ 0&0&v_3\\ v_3&0&v_1 \end{pmatrix}$,
$v = \begin{pmatrix} 0&6v_3&0&0&-6v_2\\ v_2&2v_1&-2v_3&0&0\\ 0&-2v_2&4v_1&0&0\\ 0&0&0&-4v_1&2v_3 \\-v_3&0&0&2v_2&-2v_1 \end{pmatrix}$,
$c = \begin{pmatrix} 6v_1&0&0\\ v_2&v_3&0\\ 0&v_1&v_3\\ 0&v_2&v_1\\ v_3&0&v_2 \end{pmatrix}$,

$C = \begin{pmatrix} v_1&v_3&0&0&v_2\\ 0&2v_2&2v_1&2v_3&0\\ 0&0&2v_2&2v_1&2v_3 \end{pmatrix}$,
$d = \begin{pmatrix} v_3&0&0&v_2\\ -v_2&-v_1&v_3&0\\ 0&v_2&-2v_1&-v_3 \end{pmatrix}$,
$D = \begin{pmatrix} -v_2&v_3&0\\ 0&2v_1&-v_3\\ 0&-v_2&2v_1\\ -2v_3&0&v_2 \end{pmatrix}$,

$e = \begin{pmatrix} 6v_3&0&0&-6v_2\\-4v_1&2v_3&0&0\\v_2&2v_1&3v_3&0\\0&-3v_2&-2v_1&-v_3\\0&0&-2v_2&4v_1 \end{pmatrix}$,
$E = \begin{pmatrix} v_2&-4v_1&v_3&0&0\\0&2v_2&2v_1&-3v_3&0\\0&0&3v_2&-2v_1&-2v_3\\-v_3&0&0&-v_2&4v_1 \end{pmatrix}$ and
$w = \begin{pmatrix} v_1&v_3&0&0\\v_2&-v_1&0&0\\0&0&v_1&-v_3\\0&0&-v_2&-v_1\end{pmatrix}$.
}

For the above equivariant maps, the potential is given by
\[
W/12 = -uAa+5ubB-\frac{2}{3}u^3 +15vBb-5vcC-20vEe+\frac{10}{3}v^3+5weE-wDd-\frac{1}{3}w^3 -5dec+10CED
\]

\section{Mutations of quivers with potentials}
\label{Sect:Mutation}

For a QP $(Q,W)$ let $k$ be a vertex in $Q$ with no loops (but possibly lying on a 2-cycle). We define the {\em mutation} of $(Q,W)$ by first constructing the QP $\widetilde \mu_k(Q,W)$ in the following way:

\begin{enumerate}
\item Let $Q'$ be the quiver obtained from $Q$ by the following steps:
\begin{enumerate}
\item Replace the vertex $k$ in $Q$ by a new vertex $k^*$.
\item Add new arrows $[ab] : i \to j$ for each pair of arrows $a : i \to k$ and $b : k \to j$ in $Q$.
\item Replace each arrow $a : i \to k$ in $Q$ by a new arrow $a^* : k^* \to i$.
\item Replace each arrow $b : k \to j$ in $Q$ by a new arrow $b^* : j \to k^*$.
\end{enumerate}
\item Let $W':=[W]+\Delta$ where 
\begin{enumerate}
\item $[W]$ is obtained by substituting $[ab]$ for each factor $ab$ in $W$ with $ha=k=tb$.
\item $\displaystyle \Delta = \!\!\!\!\sum_{a,b \in Q_1, ha=k=tb}\!\!\!\![ab]b^*a^*$.
\end{enumerate}
\end{enumerate}

This mutation is obtained from the original \cite{DWZ} with the difference that $(Q,W)$ may have loops and 2-cycles. This situation is quite natural in some geometric contexts as the one treated in this paper.

A QP $(Q,W)$ is called {\it reduced} if $W \in \bigoplus_{i \geq 3} \C Q_{i,\mathrm{cyc}}$. Given a non-reduced QP $(Q,W)$, if there is a reduced QP $(Q',W')$ such that $\mathcal P(Q,W) \simeq \mathcal P(Q',W')$, then we say that $(Q',W')$ is the reduced part of $(Q,W)$.

Before going further, let us consider now {\em graded} quivers with potentials. Given $(Q,W)$ a QP, we can define a map $\deg:Q_1\to\Z$ which extends naturally on $\C Q$. We say that the potential $W$ is {\em homogeneous} of degree $d$ if all terms in $W$ are of degree $d$. 

\begin{df}
We say that a QP $(Q,W)$ is {\em graded} if there exist a grading in $Q$ such that $W$ is homogeneous of degree $d$.
\end{df} 

Trivially the Jacobian algebra of a graded QP becomes a graded algebra, and $\widetilde \mu_k(Q,W)$ is naturally graded with degree $d$ (see \cite{Mizuno}, Section 3).

\begin{lem}
If a QP $(Q,W)$ is graded, then there exists a reduced graded QP $(Q_{\mathrm{red}},W_{\mathrm{red}})$ such that $\mathcal P(Q,W) \simeq \mathcal P(Q_{\mathrm{red}},W_{\mathrm{red}})$.
\end{lem}

\begin{proof}
If $(Q,W)$ is reduced, we have nothing to do. So we assume that there is a 2-cycle $ab$ which appears in $W$. Write $W = c_{ab}ab + W'$ where $c_{ab} \in \C$ is a non-zero element. Then $\partial_a W = c_{ab}b + \partial_a W'$ and $\partial_b W = c_{ab}a + \partial_b W'$. Without loss of generality we can assume that $\deg a \geq \deg b$. Then $\deg \partial_a W' = \deg b$, so $\partial_a W'$ does not contain $a$ or $b$. Thus $(Q,W)$ reduce to a QP whose potential does not contain $a$ or $b$ and Jacobian algebras are isomorphic, and the new QP obtained is graded. By repeating this operation, we obtain the required reduced QP.
\end{proof}

\begin{df}\label{defn:mut}
We define the mutation $\mu_k(Q,W)$ of the quiver with potential $(Q,W)$ to be the reduced part of $\widetilde \mu_k(Q,W)$.
\end{df}

\begin{ex}\label{exa:D6} Consider the group $G=D_{2n}$ for $n=6$. In this case there are 16 non-equivalent QPs, shown in Figure \ref{D-12}.

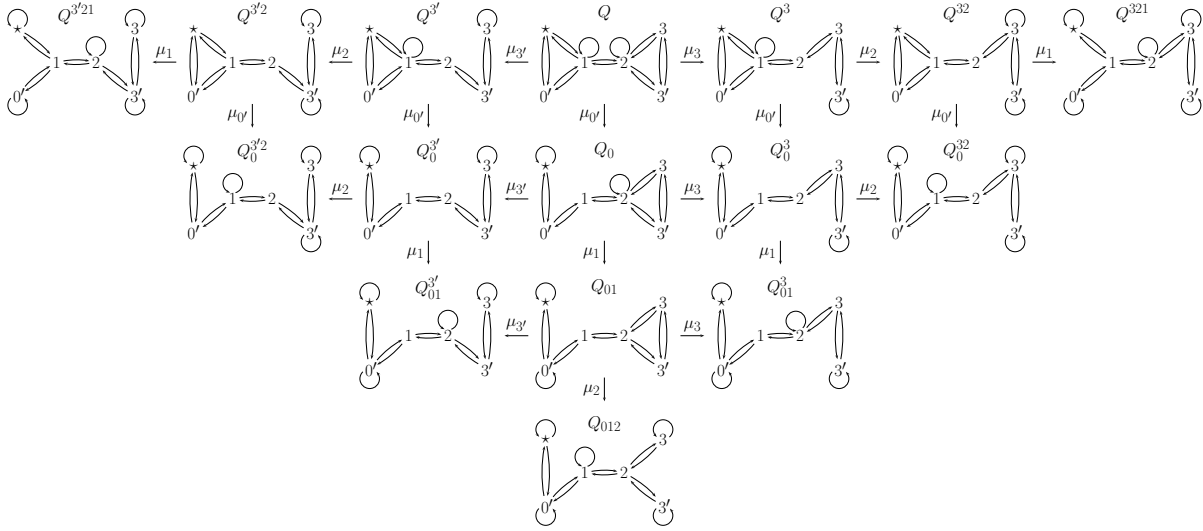
\begin{figure}[htbp]
\begin{center}
\scalebox{0.26}{
\begin{pspicture}(-20,-23)(26,2.25)
	\psset{arcangle=10,nodesep=4pt,linewidth=1.5pt}

\rput(0,0){
	\rput(3,2.5){\Huge\bf $Q$}
	\rput(-0,1.7){\rnode{0}{\Huge $\star$}}
	\rput(-0,-1.7){\rnode{0'}{\Huge $0'$}}
	\rput(2,0){\rnode{1}{\Huge $1$}}
	\rput(4,0){\rnode{2}{\Huge $2$}}
	\rput(6,1.7){\rnode{3}{\Huge $3$}}
	\rput(6,-1.7){\rnode{3'}{\Huge $3'$}}	
	\rput(2,0.3){\rnode{u1}{}}	
	\rput(4,0.3){\rnode{u2}{}}	
	\ncarc{->}{0}{0'}\ncarc{->}{0'}{0}\ncarc{->}{0'}{1}\ncarc{->}{1}{0'}
	\ncarc{->}{0}{1}\ncarc{->}{1}{0}\ncarc{->}{1}{2}\ncarc{->}{2}{1}
	\ncarc{->}{2}{3}\ncarc{->}{3}{2}\ncarc{->}{2}{3'}\ncarc{->}{3'}{2}
	\ncarc{<-}{3}{3'}\ncarc{<-}{3'}{3}
	\nccircle[angleA=-25,nodesep=4pt]{->}{u1}{.5cm}
	\nccircle[angleA=25,nodesep=5pt]{->}{u2}{.5cm}
}

\rput(0,-7){
	\rput(3,2.5){\Huge\bf $Q_0$}
	\rput(-0,1.7){\rnode{0}{\Huge $\star$}}
	\rput(-0,-1.7){\rnode{0'}{\Huge $0'$}}
	\rput(2,0){\rnode{1}{\Huge $1$}}
	\rput(4,0){\rnode{2}{\Huge $2$}}
	\rput(6,1.7){\rnode{3}{\Huge $3$}}
	\rput(6,-1.7){\rnode{3'}{\Huge $3'$}}	
	\rput(2,0.3){\rnode{u1}{}}	
	\rput(4,0.3){\rnode{u2}{}}	
	\ncarc{->}{0}{0'}\ncarc{->}{0'}{0}\ncarc{->}{0'}{1}\ncarc{->}{1}{0'}
	\ncarc{->}{1}{2}\ncarc{->}{2}{1}\ncarc{->}{2}{3}\ncarc{->}{3}{2}
	\ncarc{->}{2}{3'}\ncarc{->}{3'}{2}\ncarc{<-}{3}{3'}\ncarc{<-}{3'}{3}
	\nccircle[angleA=0,nodesep=4pt]{->}{0}{.5cm}
	\nccircle[angleA=25,nodesep=5pt]{->}{u2}{.5cm}
}

\rput(0,-14){
	\rput(3,2.5){\Huge\bf $Q_{01}$}
	\rput(-0,1.7){\rnode{0}{\Huge $\star$}}
	\rput(-0,-1.7){\rnode{0'}{\Huge $0'\!$}}
	\rput(2,0){\rnode{1}{\Huge $1$}}
	\rput(4,0){\rnode{2}{\Huge $2$}}
	\rput(6,1.7){\rnode{3}{\Huge $3$}}
	\rput(6,-1.7){\rnode{3'}{\Huge $3'$}}	
	\rput(2,0.3){\rnode{u1}{}}	
	\rput(4,0.3){\rnode{u2}{}}	
	\ncarc{->}{0}{0'}\ncarc{->}{0'}{0}\ncarc{->}{0'}{1}\ncarc{->}{1}{0'}
	\ncarc{->}{1}{2}\ncarc{->}{2}{1}\ncarc{->}{2}{3}\ncarc{->}{3}{2}
	\ncarc{->}{2}{3'}\ncarc{->}{3'}{2}\ncarc{<-}{3}{3'}\ncarc{<-}{3'}{3}
	\nccircle[angleA=0,nodesep=4pt]{->}{0}{.5cm}
	\nccircle[angleA=180,nodesep=4pt]{->}{0'}{.5cm}
}

\rput(0,-21){
	\rput(3,2.5){\Huge\bf $Q_{012}$}
	\rput(-0,1.7){\rnode{0}{\Huge $\star$}}
	\rput(-0,-1.7){\rnode{0'}{\Huge $0'\!$}}
	\rput(2,0){\rnode{1}{\Huge $1$}}
	\rput(4,0){\rnode{2}{\Huge $2$}}
	\rput(6,1.7){\rnode{3}{\Huge $3$}}
	\rput(6,-1.7){\rnode{3'}{\Huge $3'\!$}}	
	\rput(2,0.3){\rnode{u1}{}}	
	\rput(4,0.3){\rnode{u2}{}}	
	\ncarc{->}{0}{0'}\ncarc{->}{0'}{0}\ncarc{->}{0'}{1}\ncarc{->}{1}{0'}
	\ncarc{->}{1}{2}\ncarc{->}{2}{1}\ncarc{->}{2}{3}\ncarc{->}{3}{2}
	\ncarc{->}{2}{3'}\ncarc{->}{3'}{2}
	\nccircle[angleA=0,nodesep=4pt]{->}{0}{.5cm}
	\nccircle[angleA=180,nodesep=4pt]{->}{0'}{.5cm}
	\nccircle[angleA=0,nodesep=4pt]{->}{u1}{.5cm}
	\nccircle[angleA=0,nodesep=4pt]{->}{3}{.5cm}
	\nccircle[angleA=180,nodesep=4pt]{->}{3'}{.5cm}
}

\rput(9,0){
	\rput(3,2.5){\Huge\bf $Q^3$}
	\rput(-0,1.7){\rnode{0}{\Huge $\star$}}
	\rput(-0,-1.7){\rnode{0'}{\Huge $0'$}}
	\rput(2,0){\rnode{1}{\Huge $1$}}
	\rput(4,0){\rnode{2}{\Huge $2$}}
	\rput(6,1.7){\rnode{3}{\Huge $3$}}
	\rput(6,-1.7){\rnode{3'}{\Huge $3'\!$}}	
	\rput(2,0.3){\rnode{u1}{}}	
	\rput(4,0.3){\rnode{u2}{}}	
	\ncarc{->}{0}{0'}\ncarc{->}{0'}{0}\ncarc{->}{0'}{1}\ncarc{->}{1}{0'}
	\ncarc{->}{0}{1}\ncarc{->}{1}{0}\ncarc{->}{1}{2}\ncarc{->}{2}{1}
	\ncarc{->}{2}{3}\ncarc{->}{3}{2}\ncarc{<-}{3}{3'}\ncarc{<-}{3'}{3}
	\nccircle[angleA=-25,nodesep=4pt]{->}{u1}{.5cm}
	\nccircle[angleA=180,nodesep=4pt]{->}{3'}{.5cm}
}

\rput(18,0){
\rput(3,2.5){\Huge\bf $Q^{32}$}
	\rput(-0,1.7){\rnode{0}{\Huge $\star$}}
	\rput(-0,-1.7){\rnode{0'}{\Huge $0'$}}
	\rput(2,0){\rnode{1}{\Huge $1$}}
	\rput(4,0){\rnode{2}{\Huge $2$}}
	\rput(6,1.7){\rnode{3}{\Huge $3$}}
	\rput(6,-1.7){\rnode{3'}{\Huge $3'\!$}}	
	\rput(2,0.3){\rnode{u1}{}}	
	\rput(4,0.3){\rnode{u2}{}}	
	\ncarc{->}{0}{0'}\ncarc{->}{0'}{0}\ncarc{->}{0'}{1}\ncarc{->}{1}{0'}
	\ncarc{->}{0}{1}\ncarc{->}{1}{0}\ncarc{->}{1}{2}\ncarc{->}{2}{1}
	\ncarc{->}{2}{3}\ncarc{->}{3}{2}\ncarc{<-}{3}{3'}\ncarc{<-}{3'}{3}
	\nccircle[angleA=180,nodesep=4pt]{->}{3'}{.5cm}
	\nccircle[angleA=0,nodesep=4pt]{->}{3}{.5cm}
}

\rput(27,0){
	\rput(3,2.5){\Huge\bf $Q^{321}$}
	\rput(-0,1.7){\rnode{0}{\Huge $\star$}}
	\rput(-0,-1.7){\rnode{0'}{\Huge $0'\!$}}
	\rput(2,0){\rnode{1}{\Huge $1$}}
	\rput(4,0){\rnode{2}{\Huge $2$}}
	\rput(6,1.7){\rnode{3}{\Huge $3$}}
	\rput(6,-1.7){\rnode{3'}{\Huge $3'\!$}}	
	\rput(2,0.3){\rnode{u1}{}}	
	\rput(4,0.3){\rnode{u2}{}}	
	\ncarc{->}{0'}{1}\ncarc{->}{1}{0'}
	\ncarc{->}{0}{1}\ncarc{->}{1}{0}\ncarc{->}{1}{2}\ncarc{->}{2}{1}
	\ncarc{->}{2}{3}\ncarc{->}{3}{2}\ncarc{<-}{3}{3'}\ncarc{<-}{3'}{3}
	\nccircle[angleA=25,nodesep=5pt]{->}{u2}{.5cm}
	\nccircle[angleA=180,nodesep=4pt]{->}{3'}{.5cm}
	\nccircle[angleA=180,nodesep=4pt]{->}{0'}{.5cm}
	\nccircle[angleA=0,nodesep=4pt]{->}{0}{.5cm}
	\nccircle[angleA=0,nodesep=4pt]{->}{3}{.5cm}
}

\rput(9,-7){
	\rput(3,2.5){\Huge\bf $Q_0^3$}
	\rput(-0,1.7){\rnode{0}{\Huge $\star$}}
	\rput(-0,-1.7){\rnode{0'}{\Huge $0'$}}
	\rput(2,0){\rnode{1}{\Huge $1$}}
	\rput(4,0){\rnode{2}{\Huge $2$}}
	\rput(6,1.7){\rnode{3}{\Huge $3$}}
	\rput(6,-1.7){\rnode{3'}{\Huge $3'\!$}}	
	\rput(2,0.3){\rnode{u1}{}}	
	\rput(4,0.3){\rnode{u2}{}}	
	\ncarc{->}{0'}{1}\ncarc{->}{1}{0'}
	\ncarc{->}{0}{0'}\ncarc{->}{0'}{0}\ncarc{->}{1}{2}\ncarc{->}{2}{1}
	\ncarc{->}{2}{3}\ncarc{->}{3}{2}\ncarc{<-}{3}{3'}\ncarc{<-}{3'}{3}
	\nccircle[angleA=180,nodesep=4pt]{->}{3'}{.5cm}
	\nccircle[angleA=0,nodesep=4pt]{->}{0}{.5cm}
}

\rput(18,-7){
	\rput(3,2.5){\Huge\bf $Q_0^{32}$}
	\rput(-0,1.7){\rnode{0}{\Huge $\star$}}
	\rput(-0,-1.7){\rnode{0'}{\Huge $0'$}}
	\rput(2,0){\rnode{1}{\Huge $1$}}
	\rput(4,0){\rnode{2}{\Huge $2$}}
	\rput(6,1.7){\rnode{3}{\Huge $3$}}
	\rput(6,-1.7){\rnode{3'}{\Huge $3'\!$}}	
	\rput(2,0.3){\rnode{u1}{}}	
	\rput(4,0.3){\rnode{u2}{}}	
	\ncarc{->}{0'}{1}\ncarc{->}{1}{0'}
	\ncarc{->}{0}{0'}\ncarc{->}{0'}{0}\ncarc{->}{1}{2}\ncarc{->}{2}{1}
	\ncarc{->}{2}{3}\ncarc{->}{3}{2}\ncarc{<-}{3}{3'}\ncarc{<-}{3'}{3}
	\nccircle[angleA=180,nodesep=4pt]{->}{3'}{.5cm}
	\nccircle[angleA=0,nodesep=4pt]{->}{u1}{.5cm}
	\nccircle[angleA=0,nodesep=4pt]{->}{0}{.5cm}
	\nccircle[angleA=0,nodesep=4pt]{->}{3}{.5cm}
}

\rput(9,-14){
	\rput(3,2.5){\Huge\bf $Q_{01}^3$}
	\rput(-0,1.7){\rnode{0}{\Huge $\star$}}
	\rput(-0,-1.7){\rnode{0'}{\Huge $0'\!$}}
	\rput(2,0){\rnode{1}{\Huge $1$}}
	\rput(4,0){\rnode{2}{\Huge $2$}}
	\rput(6,1.7){\rnode{3}{\Huge $3$}}
	\rput(6,-1.7){\rnode{3'}{\Huge $3'\!$}}	
	\rput(2,0.3){\rnode{u1}{}}	
	\rput(4,0.3){\rnode{u2}{}}	
	\ncarc{->}{0'}{1}\ncarc{->}{1}{0'}
	\ncarc{->}{0}{0'}\ncarc{->}{0'}{0}\ncarc{->}{1}{2}\ncarc{->}{2}{1}
	\ncarc{->}{2}{3}\ncarc{->}{3}{2}\ncarc{<-}{3}{3'}\ncarc{<-}{3'}{3}
	\nccircle[angleA=180,nodesep=4pt]{->}{3'}{.5cm}
	\nccircle[angleA=180,nodesep=4pt]{->}{0'}{.5cm}
	\nccircle[angleA=0,nodesep=4pt]{->}{0}{.5cm}
	\nccircle[angleA=25,nodesep=5pt]{->}{u2}{.5cm}
}

\rput(-9,0){
	\rput(3,2.5){\Huge\bf $Q^{3'}$}
	\rput(-0,1.7){\rnode{0}{\Huge $\star$}}
	\rput(-0,-1.7){\rnode{0'}{\Huge $0'$}}
	\rput(2,0){\rnode{1}{\Huge $1$}}
	\rput(4,0){\rnode{2}{\Huge $2$}}
	\rput(6,1.7){\rnode{3}{\Huge $3$}}
	\rput(6,-1.7){\rnode{3'}{\Huge $3'$}}	
	\rput(2,0.3){\rnode{u1}{}}	
	\rput(4,0.3){\rnode{u2}{}}	
	\ncarc{->}{0}{0'}\ncarc{->}{0'}{0}\ncarc{->}{0'}{1}\ncarc{->}{1}{0'}
	\ncarc{->}{0}{1}\ncarc{->}{1}{0}\ncarc{->}{1}{2}\ncarc{->}{2}{1}
	\ncarc{->}{2}{3'}\ncarc{->}{3'}{2}\ncarc{<-}{3}{3'}\ncarc{<-}{3'}{3}
	\nccircle[angleA=-25,nodesep=4pt]{->}{u1}{.5cm}
	\nccircle[angleA=0,nodesep=4pt]{->}{3}{.5cm}
}

\rput(-18,0){
	\rput(3,2.5){\Huge\bf $Q^{3'2}$}
	\rput(-0,1.7){\rnode{0}{\Huge $\star$}}
	\rput(-0,-1.7){\rnode{0'}{\Huge $0'$}}
	\rput(2,0){\rnode{1}{\Huge $1$}}
	\rput(4,0){\rnode{2}{\Huge $2$}}
	\rput(6,1.7){\rnode{3}{\Huge $3$}}
	\rput(6,-1.7){\rnode{3'}{\Huge $3'\!$}}	
	\rput(2,0.3){\rnode{u1}{}}	
	\rput(4,0.3){\rnode{u2}{}}	
	\ncarc{->}{0}{0'}\ncarc{->}{0'}{0}\ncarc{->}{0'}{1}\ncarc{->}{1}{0'}
	\ncarc{->}{0}{1}\ncarc{->}{1}{0}\ncarc{->}{1}{2}\ncarc{->}{2}{1}
	\ncarc{->}{2}{3'}\ncarc{->}{3'}{2}\ncarc{<-}{3}{3'}\ncarc{<-}{3'}{3}
	\nccircle[angleA=180,nodesep=4pt]{->}{3'}{.5cm}
	\nccircle[angleA=0,nodesep=4pt]{->}{3}{.5cm}
}

\rput(-27,0){
	\rput(3,2.5){\Huge\bf $Q^{3'21}$}
	\rput(-0,1.7){\rnode{0}{\Huge $\star$}}
	\rput(-0,-1.7){\rnode{0'}{\Huge $0'\!$}}
	\rput(2,0){\rnode{1}{\Huge $1$}}
	\rput(4,0){\rnode{2}{\Huge $2$}}
	\rput(6,1.7){\rnode{3}{\Huge $3$}}
	\rput(6,-1.7){\rnode{3'}{\Huge $3'\!$}}	
	\rput(2,0.3){\rnode{u1}{}}	
	\rput(4,0.3){\rnode{u2}{}}	
	\ncarc{->}{0'}{1}\ncarc{->}{1}{0'}
	\ncarc{->}{0}{1}\ncarc{->}{1}{0}\ncarc{->}{1}{2}\ncarc{->}{2}{1}
	\ncarc{->}{2}{3'}\ncarc{->}{3'}{2}\ncarc{<-}{3}{3'}\ncarc{<-}{3'}{3}
	\nccircle[angleA=0,nodesep=5pt]{->}{u2}{.5cm}
	\nccircle[angleA=180,nodesep=4pt]{->}{3'}{.5cm}
	\nccircle[angleA=180,nodesep=4pt]{->}{0'}{.5cm}
	\nccircle[angleA=0,nodesep=4pt]{->}{0}{.5cm}
	\nccircle[angleA=0,nodesep=4pt]{->}{3}{.5cm}
}

\rput(-9,-7){
	\rput(3,2.5){\Huge\bf $Q_0^{3'}$}
	\rput(-0,1.7){\rnode{0}{\Huge $\star$}}
	\rput(-0,-1.7){\rnode{0'}{\Huge $0'$}}
	\rput(2,0){\rnode{1}{\Huge $1$}}
	\rput(4,0){\rnode{2}{\Huge $2$}}
	\rput(6,1.7){\rnode{3}{\Huge $3$}}
	\rput(6,-1.7){\rnode{3'}{\Huge $3'$}}	
	\rput(2,0.3){\rnode{u1}{}}	
	\rput(4,0.3){\rnode{u2}{}}	
	\ncarc{->}{0'}{1}\ncarc{->}{1}{0'}
	\ncarc{->}{0}{0'}\ncarc{->}{0'}{0}\ncarc{->}{1}{2}\ncarc{->}{2}{1}
	\ncarc{->}{2}{3'}\ncarc{->}{3'}{2}\ncarc{<-}{3}{3'}\ncarc{<-}{3'}{3}
	\nccircle[angleA=0,nodesep=4pt]{->}{3}{.5cm}
	\nccircle[angleA=0,nodesep=4pt]{->}{0}{.5cm}
}

\rput(-18,-7){
	\rput(3,2.5){\Huge\bf $Q_0^{3'2}$}
	\rput(-0,1.7){\rnode{0}{\Huge $\star$}}
	\rput(-0,-1.7){\rnode{0'}{\Huge $0'$}}
	\rput(2,0){\rnode{1}{\Huge $1$}}
	\rput(4,0){\rnode{2}{\Huge $2$}}
	\rput(6,1.7){\rnode{3}{\Huge $3$}}
	\rput(6,-1.7){\rnode{3'}{\Huge $3'\!$}}	
	\rput(2,0.3){\rnode{u1}{}}	
	\rput(4,0.3){\rnode{u2}{}}	
	\ncarc{->}{0'}{1}\ncarc{->}{1}{0'}
	\ncarc{->}{0}{0'}\ncarc{->}{0'}{0}\ncarc{->}{1}{2}\ncarc{->}{2}{1}
	\ncarc{->}{2}{3'}\ncarc{->}{3'}{2}\ncarc{<-}{3}{3'}\ncarc{<-}{3'}{3}
	\nccircle[angleA=0,nodesep=4pt]{->}{3}{.5cm}
	\nccircle[angleA=0,nodesep=4pt]{->}{u1}{.5cm}
	\nccircle[angleA=0,nodesep=4pt]{->}{0}{.5cm}
	\nccircle[angleA=180,nodesep=4pt]{->}{3'}{.5cm}
}

\rput(-9,-14){
	\rput(3,2.5){\Huge\bf $Q_{01}^{3'}$}
	\rput(-0,1.7){\rnode{0}{\Huge $\star$}}
	\rput(-0,-1.7){\rnode{0'}{\Huge $0'\!$}}
	\rput(2,0){\rnode{1}{\Huge $1$}}
	\rput(4,0){\rnode{2}{\Huge $2$}}
	\rput(6,1.7){\rnode{3}{\Huge $3$}}
	\rput(6,-1.7){\rnode{3'}{\Huge $3'$}}	
	\rput(2,0.3){\rnode{u1}{}}	
	\rput(4,0.3){\rnode{u2}{}}	
	\ncarc{->}{0'}{1}\ncarc{->}{1}{0'}
	\ncarc{->}{0}{0'}\ncarc{->}{0'}{0}\ncarc{->}{1}{2}\ncarc{->}{2}{1}
	\ncarc{->}{2}{3'}\ncarc{->}{3'}{2}\ncarc{<-}{3}{3'}\ncarc{<-}{3'}{3}
	\nccircle[angleA=0,nodesep=4pt]{->}{3}{.5cm}
	\nccircle[angleA=0,nodesep=4pt]{->}{0}{.5cm}
	\nccircle[angleA=180,nodesep=4pt]{->}{0'}{.5cm}
	\nccircle[angleA=0,nodesep=5pt]{->}{u2}{.5cm}
}

\rput(6.75,0){\rput(0,0){\rnode{0}{}}\rput(1.5,0){\rnode{1}{}}\ncline{->}{0}{1}\Aput{\Huge $\mu_3$}}
\rput(6.76,-7){\rput(0,0){\rnode{0}{}}\rput(1.5,0){\rnode{1}{}}\ncline{->}{0}{1}\Aput{\Huge $\mu_3$}}
\rput(6.75,-14){\rput(0,0){\rnode{0}{}}\rput(1.5,0){\rnode{1}{}}\ncline{->}{0}{1}\Aput{\Huge $\mu_3$}}
\rput(15.75,0){\rput(0,0){\rnode{0}{}}\rput(1.5,0){\rnode{1}{}}\ncline{->}{0}{1}\Aput{\Huge $\mu_2$}}
\rput(15.75,-7){\rput(0,0){\rnode{0}{}}\rput(1.5,0){\rnode{1}{}}\ncline{->}{0}{1}\Aput{\Huge $\mu_2$}}
\rput(24.75,0){\rput(0,0){\rnode{0}{}}\rput(1.5,0){\rnode{1}{}}\ncline{->}{0}{1}\Aput{\Huge $\mu_1$}}
\rput(-2.25,0){\rput(0,0){\rnode{0}{}}\rput(1.5,0){\rnode{1}{}}\ncline{<-}{0}{1}\Aput{\Huge $\mu_{3'}$}}
\rput(-2.25,-7){\rput(0,0){\rnode{0}{}}\rput(1.5,0){\rnode{1}{}}\ncline{<-}{0}{1}\Aput{\Huge $\mu_{3'}$}}
\rput(-2.25,-14){\rput(0,0){\rnode{0}{}}\rput(1.5,0){\rnode{1}{}}\ncline{<-}{0}{1}\Aput{\Huge $\mu_{3'}$}}
\rput(-11.25,0){\rput(0,0){\rnode{0}{}}\rput(1.5,0){\rnode{1}{}}\ncline{<-}{0}{1}\Aput{\Huge $\mu_2$}}
\rput(-11.25,-7){\rput(0,0){\rnode{0}{}}\rput(1.5,0){\rnode{1}{}}\ncline{<-}{0}{1}\Aput{\Huge $\mu_2$}}
\rput(-20.25,0){\rput(0,0){\rnode{0}{}}\rput(1.5,0){\rnode{1}{}}\ncline{<-}{0}{1}\Aput{\Huge $\mu_1$}}
\rput(-15,-2){\rput(0,0){\rnode{0}{}}\rput(0,-1.5){\rnode{1}{}}\ncline{->}{0}{1}\Bput{\Huge $\mu_{0'}$}}
\rput(-6,-2){\rput(0,0){\rnode{0}{}}\rput(0,-1.5){\rnode{1}{}}\ncline{->}{0}{1}\Bput{\Huge $\mu_{0'}$}}
\rput(3,-2){\rput(0,0){\rnode{0}{}}\rput(0,-1.5){\rnode{1}{}}\ncline{->}{0}{1}\Bput{\Huge $\mu_{0'}$}}
\rput(12,-2){\rput(0,0){\rnode{0}{}}\rput(0,-1.5){\rnode{1}{}}\ncline{->}{0}{1}\Bput{\Huge $\mu_{0'}$}}
\rput(21,-2){\rput(0,0){\rnode{0}{}}\rput(0,-1.5){\rnode{1}{}}\ncline{->}{0}{1}\Bput{\Huge $\mu_{0'}$}}
\rput(-6,-9){\rput(0,0){\rnode{0}{}}\rput(0,-1.5){\rnode{1}{}}\ncline{->}{0}{1}\Bput{\Huge $\mu_1$}}
\rput(3,-9){\rput(0,0){\rnode{0}{}}\rput(0,-1.5){\rnode{1}{}}\ncline{->}{0}{1}\Bput{\Huge $\mu_1$}}
\rput(12,-9){\rput(0,0){\rnode{0}{}}\rput(0,-1.5){\rnode{1}{}}\ncline{->}{0}{1}\Bput{\Huge $\mu_1$}}
\rput(3,-16){\rput(0,0){\rnode{0}{}}\rput(0,-1.5){\rnode{1}{}}\ncline{->}{0}{1}\Bput{\Huge $\mu_2$}}
\end{pspicture}
}
\caption{Mutations of type $D_{12}$.}
\label{D-12}
\end{center}
\end{figure}

We demonstrate how to calculate the mutation of $(Q,W)$ at the vertex $0'$. 
First add new arrows $[aA], [D_0A], [ad_0], [D_0d_0]$, replace $a,A,d_0,D_0$ by $a^*,A^*,d_0^*,D_0^*$ respectively as shown below, and denote the new quiver by $\widetilde\mu_{0'}(Q)$.

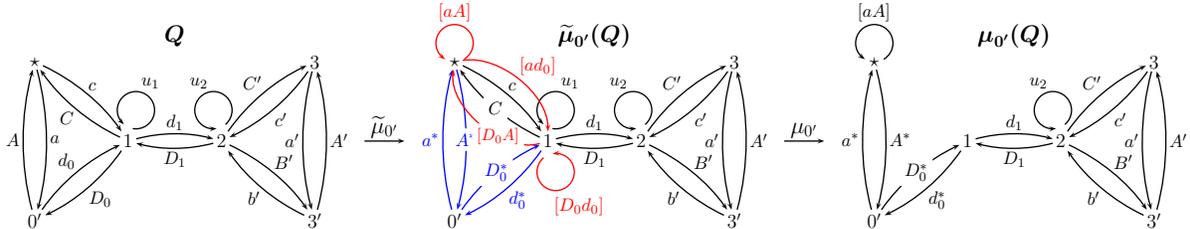
\begin{center}
\scalebox{.62}{
\begin{pspicture}(-1,-2)(26,2.5)
	\psset{arcangle=15,nodesep=2pt}
\rput(0,0){
\rput(3,2.25){{\Large $\bm{Q}$}}
	\rput(0,1.7){\rnode{0}{\large $\star$}}
	\rput(0,-1.7){\rnode{0'}{\large $0'$}}
	\rput(2,0){\rnode{1}{\large $1$}}
	\rput(4,0){\rnode{2}{\large $2$}}
	\rput(6,1.7){\rnode{3}{\large $3$}}
	\rput(6,-1.7){\rnode{3'}{\large $3'$}}	
	\rput(2,0.3){\rnode{u1}{}}	
	\rput(4,0.3){\rnode{u2}{}}	
	\ncarc{->}{0}{0'}\Aput[0.05]{$a$}
	\ncarc{->}{0'}{0}\Aput[0.05]{$A$}
	\ncarc{->}{0'}{1}\Aput[0.05]{$d_0$}
	\ncarc{->}{1}{0'}\Aput[0.05]{$D_0$}
	\ncarc{->}{0}{1}\Aput[0.05]{$c$}
	\ncarc{->}{1}{0}\Aput[0.05]{$C$}
	\ncarc{->}{1}{2}\Aput[0.05]{$d_1$}
	\ncarc{->}{2}{1}\Aput[0.05]{$D_1$}
	\ncarc{->}{2}{3}\Aput[0.05]{$C'$}
	\ncarc{->}{3}{2}\Aput[0.05]{$c'$}
	\ncarc{->}{2}{3'}\Aput[0.05]{$B'$}
	\ncarc{->}{3'}{2}\Aput[0.05]{$b'$}
	\ncarc{<-}{3}{3'}\Aput[0.05]{$A'$}
	\ncarc{<-}{3'}{3}\Aput[0.05]{$a'$}
	 \nccircle[angleA=-25,nodesep=3pt]{->}{u1}{.4cm}\Bput[0.05]{$u_1$}
	 \nccircle[angleA=25,nodesep=3pt]{->}{u2}{.4cm}\Bput[0.05]{$u_2$}
}
\rput(7,0){
\rput(5,2.25){{\Large {$\bm{\widetilde\mu_{0'}(Q)}$}}}
	\rput(0,0){\rnode{0}{}}
	\rput(1,0){\rnode{1}{}}
	\ncline{->}{0}{1}\Aput[0.05]{\Large $\widetilde\mu_{0'}$}
}
\rput(9,0){
	\rput(-0,1.7){\rnode{0}{\large $\star$}}
	\rput(-0,-1.7){\rnode{0'}{\large $0'$}}
	\rput(2,0){\rnode{1}{\large $1$}}
	\rput(4,0){\rnode{2}{\large $2$}}
	\rput(6,1.7){\rnode{3}{\large $3$}}
	\rput(6,-1.7){\rnode{3'}{\large $3'$}}	
	\rput(2,0.3){\rnode{u1}{}}
	\rput(4,0.3){\rnode{u2}{}}
	\rput(2,-0.3){\rnode{u1'}{}}
	\rput(4,-0.3){\rnode{u2'}{}}
	\ncarc[linecolor=blue]{->}{0}{0'}\mput*[0.05]{$\textcolor{blue}{A^*}$}
	\ncarc[linecolor=blue]{->}{0'}{0}\Aput[0.05]{$\textcolor{blue}{a^*}$}
	\ncarc[linecolor=blue]{->}{0'}{1}\mput*[0.05]{$\textcolor{blue}{D_0^*}$}
	\ncarc[linecolor=blue]{->}{1}{0'}\Aput[0.05]{$\textcolor{blue}{d_0^*}$}
	\ncarc{->}{0}{1}\Aput[0.05]{$c$}
	\ncarc{->}{1}{2}\Aput[0.05]{$d_1$}
	\ncarc{->}{2}{1}\Aput[0.05]{$D_1$}
	\ncarc{->}{2}{3}\Aput[0.05]{$C'$}
	\ncarc{->}{3}{2}\Aput[0.05]{$c'$}
	\ncarc{->}{2}{3'}\Aput[0.05]{$B'$}
	\ncarc{->}{3'}{2}\Aput[0.05]{$b'$}
	\ncarc{<-}{3}{3'}\Aput[0.05]{$A'$}
	\ncarc{<-}{3'}{3}\Aput[0.05]{$a'$}
	\nccircle[angleA=-25,nodesep=3pt]{->}{u1}{.4cm}\Bput[0.05]{$u_1$}
	\nccircle[angleA=25,nodesep=3pt]{->}{u2}{.4cm}\Bput[0.05]{$u_2$}
	\ncarc[linecolor=red,arcangle=50]{->}{0}{1}\Aput[0.05]{$\textcolor{red}{[ad_0]}$}
	\ncarc[linecolor=red,arcangle=57]{->}{1}{0}\mput*[0.05,npos=0.35]{\small $\textcolor{red}{[D_0A]}$}
	\ncarc{->}{1}{0}\mput*[0.05]{$C$}
	\nccircle[linecolor=red,angleA=0,nodesep=3pt]{->}{0}{.4cm}\Bput[0.05]{$\textcolor{red}{[aA]}$}
	\nccircle[linecolor=red,angleA=-155,nodesep=3pt]{->}{u1'}{.4cm}\Bput[0.05]{$\textcolor{red}{[D_0d_0]}$}
}
\rput(16,0){
\rput(5,2.25){{\Large {$\bm{\mu_{0'}(Q)}$}}}
	\rput(0,0){\rnode{0}{}}
	\rput(1,0){\rnode{1}{}}
	\ncline{->}{0}{1}\Aput[0.05]{\Large $\mu_{0'}$}
}
\rput(18,0){
	\rput(-0,1.7){\rnode{0}{\large $\star$}}
	\rput(-0,-1.7){\rnode{0'}{\large $0'$}}
	\rput(2,0){\rnode{1}{\large $1$}}
	\rput(4,0){\rnode{2}{\large $2$}}
	\rput(6,1.7){\rnode{3}{\large $3$}}
	\rput(6,-1.7){\rnode{3'}{\large $3'$}}	
	\rput(2,0.3){\rnode{u1}{}}
	\rput(4,0.3){\rnode{u2}{}}
	\rput(2,-0.3){\rnode{u1'}{}}
	\rput(4,-0.3){\rnode{u2'}{}}
	\ncarc{->}{0}{0'}\Aput[0.05]{$A^*$}
	\ncarc{->}{0'}{0}\Aput[0.05]{$a^*$}
	\ncarc{->}{0'}{1}\mput*[0.05]{$D_0^*$}
	\ncarc{->}{1}{0'}\Aput[0.05]{$d_0^*$}
	\ncarc{->}{1}{2}\Aput[0.05]{$d_1$}
	\ncarc{->}{2}{1}\Aput[0.05]{$D_1$}
	\ncarc{->}{2}{3}\Aput[0.05]{$C'$}
	\ncarc{->}{3}{2}\Aput[0.05]{$c'$}
	\ncarc{->}{2}{3'}\Aput[0.05]{$B'$}
	\ncarc{->}{3'}{2}\Aput[0.05]{$b'$}
	\ncarc{<-}{3}{3'}\Aput[0.05]{$A'$}
	\ncarc{<-}{3'}{3}\Aput[0.05]{$a'$}
	\nccircle[angleA=25,nodesep=3pt]{->}{u2}{.4cm}\Bput[0.05]{$u_2$}
	\nccircle[angleA=0,nodesep=3pt]{->}{0}{.4cm}\Bput[0.05]{$[aA]$}
}
\end{pspicture}
}
\end{center}

Then the potential $[W] + \Delta$ is given by
\begin{align*}
[W] + \Delta =& -[ad_0]C-c[D_0A] +u_1Cc+u_1[D_0d_0] -u_1d_1D_1 + u_2d_1D_1 -u_{2}B'b'-u_{2}C'c' \\
&-a'b'C'-A'c'b' +a^*[aA]A^* + a^*[ad_0]d_0^* + D_0^*[D_0A]A^* + D_0^*[D_0d_0]d_0^*,
\end{align*}
which is non-reduced. By taking derivations we have the following equalities 
\[
\begin{array}{lll}
\partial_C = -[ad_0] + cu_1, 
&\partial_c = -[D_0A] + u_1C,
&\partial_{u_1} = [D_0d_0] - d_1D_1 + Cc, \\
\partial_{[ad_0]} = -C + d_0^*a^*,
&\partial_{[D_0A]} = -c + A^*D_0^*,
&\partial_{[D_0d_0]} = u_1 + d_0^*D_0^*. 
\end{array}
\]
allowing us to obtain the reduced expression $W_{0'}$ of $[W]+\Delta$:
\begin{align*}
W_{0'}  = -(D_0^*d_0^*)^2a^*A^* +d_0^*D_0^*d_1D_1 + u_2d_1D_1 -u_{2}B'b'-u_{2}C'c' -a'b'C'-A'c'B'  + a^*[aA]A^*.
\end{align*}
If $\mu_{0'}(Q)$ is a quiver obtained from $\widetilde\mu_{0'}(Q)$ by removing $c,C,[ad_0],[D_0A],u_1,[D_0d_0]$, then $\mu_{0'}(Q,W)$ $:= (\mu_{0'}(Q),W_{0})$ is a reduced QP and we have $\mathcal P(\widetilde\mu_{0'}(Q,W)) \simeq \mathcal P(\mu_{0'}(Q,W))$. The rest of mutations in Figure \ref{D-12} are obtained similarly.
\end{ex}

\subsection{Mutations of McKay quivers with potentials for $G \subset SO(3)$}
\label{Sect:Mut-SO(3)}

Let $G \subset SO(3)$ be a finite subgroup of type $\Z/n\Z$, $D_{2n}$ or $\mathbb T$. Denote by $\star$ the vertex $0$ corresponding trivial representation of $G$. We do not consider mutations at vertex $\star$ and by Definition \ref{defn:mut} we do not mutate at any vertex which has a loop (see Section \ref{sect:mutNCCRs} for the justification of these two assumptions). 

\subsubsection{The cyclic group of order $n+1$}

Let $(Q,W)$ be the McKay QP of $G$ obtained in \ref{QP-cyclic}. Notice that every vertex in $Q$ has a loop, so according to our definition there are no mutations in this case.

\subsubsection{The dihedral group of order $2n$ ($n$ even)}\label{SubSect:MutDnEven}
Let $n=2m$ with $m \geq 2$ and $(Q,W)$ be the McKay QP of $G$ obtained in \ref{QP-D-even}. Because of the symmetry between the vertices $m$ and $m'$ we only write down mutations of $(Q,W)$ with respect to $0'$ and $m$. Mutations with respect to $m'$ are done in the same way.

We first fix some notations. Note that we do not mutate at the vertex $\star$ so for simplicity we denote by $Q_0$ the quiver obtained by mutating at vertex $0'$. The quiver $Q_{0\ldots i}^{m\ldots(m-j)}$ denotes the iterated quiver obtained by mutating vertices $0,1,\ldots,i$ followed by mutations at vertices $m,m-1,\ldots, j$. The order in both sequences $0,1,\ldots,i$ and $m,m-1,\ldots,j$ are necessary due to the fact that we only mutate a vertices without loops. Moreover, upper indices and lower indices are independent, i.e.\ mutation at 0' followed by mutation at $m$ is the same ate mutation at $m$ followed by mutation at 0' (both give $Q_0^m$).

By abusing the notation, in the calculation of $\mu_k(Q,W)$ for some QP $(Q,W)$ we do not use the notations $(-)^*$ and $[-]$: for example, starting from the McKay quiver $Q$ with notations as in Section \ref{QP-D-even}, arrows $a$, $A$ and $u_{\star}$ in the quiver $Q_{0\cdots i}$ actually mean $A^*$, $a^*$ and $[aA]$ respectively. Finally, we write 
\begin{align*}
X_j &= u_jd_jD_j \text{for $j=0,\ldots,m-2$}, &X_{m-1} &= u_{m-1}C'c', &X_{m-1}' &= u_{m-1}B'b'\\
Y_j &= u_jD_{j-1}d_{j-1} \text{ for $j=1,\ldots,m-1$}, &Y_{m} &= u_mc'C', &Y_{m}'&=u_mb'B' \\
Z_j &= d_jD_jD_{j-1}d_{j-1} \text{ for $j=1,\ldots,m-2$}, &Z_{m-1}&=C'c'D_{m-2}d_{m-2}, &Z_{m-1}'&=B'b'D_{m-2}d_{m-2}.
\end{align*}

The ingredients of the calculations in the general case are the same as in Example \ref{exa:D6}. We summarize the result in the following proposition.

\begin{prop} For dihedral groups of order $2n$ ($n$ even) there are $(m+1)^2$ non-equivalent mutation QPs of the form $(Q_{0\ldots i}^{m\ldots(m-j)},W_{0\ldots i}^{m\ldots(m-j)})$ obtained from the McKay QP $(Q,W)$. \\The list of every possible $(Q_{0\ldots i}^{m\ldots(m-j)},W_{0\ldots i}^{m\ldots(m-j)})$ is the following:

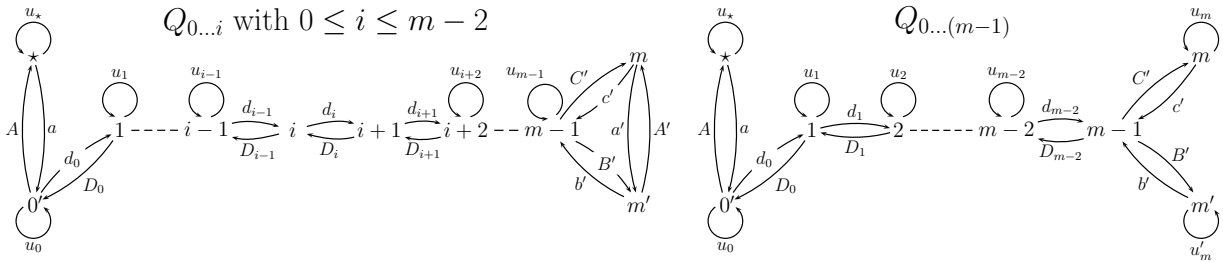
\begin{center}
\scalebox{0.575}{
\begin{pspicture}(-1,-2.75)(20,3.5)
	\psset{arcangle=15,nodesep=2pt}
\rput(2.75,2.5){\huge $Q_{0\ldots i}$ with $0 \leq i \leq m-2$}
\rput(17.25,2.5){\huge $Q_{0\ldots (m-1)}$}
\rput(-4,0){
	\rput(-0,1.7){\rnode{0}{\Large $\star$}}
	\rput(-0,-1.7){\rnode{0'}{\Large $0'\!$}}
	\rput(2,0){\rnode{1}{\Large $1$}}
	\rput(4,0){\rnode{i-1}{\Large $i-1$}}
	\rput(6,0){\rnode{i}{\Large $~i~$}}
	\rput(8,0){\rnode{i+1}{\Large $i+1$}}
	\rput(10,0){\rnode{i+2}{\Large $i+2$}}
	\rput(12,0){\rnode{m-1}{\Large $m-1$}}
	\rput(14,1.7){\rnode{m}{\Large $m$}}
	\rput(14,-1.7){\rnode{m'}{\Large $m'$}}	
	\rput(2,0.3){\rnode{u1}{}}	
	\rput(4,0.3){\rnode{ui-1}{}}	
	\rput(10,0.3){\rnode{ui+2}{}}	
	 \rput(12,0.3){\rnode{um-1}{}}	
	\ncarc{->}{0}{0'}\Aput[0.05]{\large $a$}
	\ncarc{->}{0'}{0}\Aput[0.05]{\large $A$}
	\ncarc{->}{0'}{1}\mput*[0.05]{\large $d_0$}
	\ncarc{->}{1}{0'}\Aput[0.05]{\large $D_0$}
	\ncline[linestyle=dashed,nodesep=4pt]{-}{1}{i-1}
	\ncarc{->}{i-1}{i}\Aput[0.05]{\large $d_{i-1}$}
	\ncarc{->}{i}{i-1}\Aput[0.05]{\large $D_{i-1}$}
	\ncarc{->}{i}{i+1}\Aput[0.05]{\large $d_{i}$}
	\ncarc{->}{i+1}{i}\Aput[0.05]{\large $D_{i}$}
	\ncarc{->}{i+1}{i+2}\Aput[0.05]{\large $d_{i+1}$}
	\ncarc{->}{i+2}{i+1}\Aput[0.05]{\large $D_{i+1}$}
	\ncline[linestyle=dashed,nodesep=4pt]{-}{i+2}{m-1}
	\ncarc{->}{m-1}{m}\Aput[0.05]{\large $C'$}
	\ncarc{->}{m}{m-1}\mput*[0.05]{\large $c'$}
	\ncarc{->}{m-1}{m'}\mput*[0.05]{\large $B'$}
	\ncarc{->}{m'}{m-1}\Aput[0.05]{\large $b'$}
	\ncarc{<-}{m}{m'}\Aput[0.05]{\large $A'$}
	\ncarc{<-}{m'}{m}\Aput[0.05]{\large $a'$}
	\nccircle[angleA=0,nodesep=3pt]{->}{0}{.4cm}\Bput[0.05]{\large $u_{\star}$}
	\nccircle[angleA=180,nodesep=3pt]{->}{0'}{.4cm}\Bput[0.05]{\large $u_0$}
	\nccircle[angleA=0,nodesep=3pt]{->}{u1}{.4cm}\Bput[0.05]{\large $u_1$}
	 \nccircle[nodesep=3pt]{->}{ui-1}{.4cm}\Bput[0.05]{\large $u_{i-1}$}
	 \nccircle[nodesep=3pt]{->}{ui+2}{.4cm}\Bput[0.05]{\large $u_{i+2}$}
	 \nccircle[angleA=25,nodesep=3pt]{->}{um-1}{.4cm}\Bput[0.05]{\large $u_{m-1}$}
	}
	
\rput(12,0){
	\rput(-0,1.7){\rnode{0}{\Large $\star$}}
	\rput(-0,-1.7){\rnode{0'}{\Large $0'\!$}}
	\rput(2,0){\rnode{1}{\Large $1$}}
	\rput(4,0){\rnode{2}{\Large $2$}}
	\rput(6.5,0){\rnode{m-2}{\Large $m-2$}}
	\rput(9,0){\rnode{m-1}{\Large $m-1$}}
	\rput(11,1.7){\rnode{m}{\Large $m$}}
	\rput(11,-1.7){\rnode{m'}{\Large $m'\!$}}	
	\rput(2,0.3){\rnode{u1}{}}	
	\rput(4,0.3){\rnode{ui-1}{}}	
	\rput(6.5,0.3){\rnode{ui+2}{}}	
	 \rput(9,0.3){\rnode{um-1}{}}	
	\ncarc{->}{0}{0'}\Aput[0.05]{\large $a$}
	\ncarc{->}{0'}{0}\Aput[0.05]{\large $A$}
	\ncarc{->}{0'}{1}\mput*[0.05]{\large $d_0$}
	\ncarc{->}{1}{0'}\Aput[0.05]{\large $D_0$}
	\ncarc{->}{1}{2}\Aput[0.05]{\large $d_{1}$}
	\ncarc{->}{2}{1}\Aput[0.05]{\large $D_{1}$}
	\ncline[linestyle=dashed,nodesep=4pt]{-}{2}{m-2}
	\ncarc{->}{m-2}{m-1}\Aput[0.05]{\large $d_{m-2}$}
	\ncarc{->}{m-1}{m-2}\Aput[0.05]{\large $D_{m-2}$}
	\ncarc{->}{m-1}{m}\Aput[0.05]{\large $C'$}
	\ncarc{->}{m}{m-1}\Aput[0.05]{\large $c'$}
	\ncarc{->}{m-1}{m'}\Aput[0.05]{\large $B'$}
	\ncarc{->}{m'}{m-1}\Aput[0.05]{\large  $b'$}
	\nccircle[angleA=0,nodesep=3pt]{->}{0}{.4cm}\Bput[0.05]{\large $u_{\star}$}
	\nccircle[angleA=180,nodesep=3pt]{->}{0'}{.4cm}\Bput[0.05]{\large $u_0$}
	\nccircle[angleA=0,nodesep=3pt]{->}{u1}{.4cm}\Bput[0.05]{\large $u_1$}
	 \nccircle[nodesep=3pt]{->}{ui-1}{.4cm}\Bput[0.05]{\large $u_2$}
	 \nccircle[nodesep=3pt]{->}{ui+2}{.4cm}\Bput[0.05]{\large $u_{m-2}$}
	\nccircle[angleA=0,nodesep=3pt]{->}{m}{.4cm}\Bput[0.05]{\large $u_m$}
	\nccircle[angleA=180,nodesep=3pt]{->}{m'}{.4cm}\Bput[0.05]{\large $u_{m}'$}
	}

\end{pspicture}
}
\end{center}

{\footnotesize
\[
\begin{array}{rl}
W_{0\ldots i}= &
aAu_{\star}-Aau_{0}^2 +\sum_{j=0}^{i-1}X_j-\sum_{j=1}^{i-1}Y_j -Z_i +Z_{i+1}+\sum_{j=i+2}^{m-2}X_j - \sum_{j=i+2}^{m-1}Y_j \\
& +X_{m-1}+X_{m-1}'+a'b'C'+c'B'A'. \\
W_{0\ldots (m-1)}= &
aAu_{\star} -Aau_{0}^2 +\sum_{j=0}^{m-3}X_j -\sum_{j=1}^{m-2}Y_j  +Z_m +Z_m' +Y_{m} +Y_{m}'.
\end{array}
\]
}


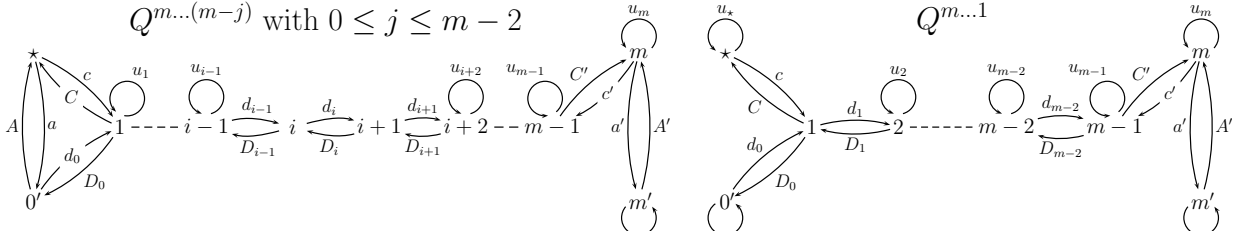
\begin{center}
\scalebox{0.575}{
\begin{pspicture}(-1,-2.75)(20,3.5)
	\psset{arcangle=15,nodesep=2pt}
\rput(2.75,2.5){\huge $Q^{m \ldots (m-j)}$ with $0 \leq j \leq m-2$}
\rput(17.25,2.5){\huge $Q^{m\ldots 1}$}
\rput(-4,0){
	\rput(-0,1.7){\rnode{0}{\Large $\star$}}
	\rput(-0,-1.7){\rnode{0'}{\Large $0'$}}
	\rput(2,0){\rnode{1}{\Large $1$}}
	\rput(4,0){\rnode{i-1}{\Large $i-1$}}
	\rput(6,0){\rnode{i}{\Large $~i~$}}
	\rput(8,0){\rnode{i+1}{\Large $i+1$}}
	\rput(10,0){\rnode{i+2}{\Large $i+2$}}
	\rput(12,0){\rnode{m-1}{\Large $m-1$}}
	\rput(14,1.7){\rnode{m}{\Large $m$}}
	\rput(14,-1.7){\rnode{m'}{\Large $m'\!$}}	
	\rput(2,0.3){\rnode{u1}{}}	
	\rput(4,0.3){\rnode{ui-1}{}}	
	\rput(10,0.3){\rnode{ui+2}{}}	
	 \rput(12,0.3){\rnode{um-1}{}}	
	\ncarc{->}{0}{0'}\Aput[0.05]{\large $a$}
	\ncarc{->}{0'}{0}\Aput[0.05]{\large $A$}
	\ncarc{->}{0'}{1}\mput*[0.05]{\large $d_0$}
	\ncarc{->}{1}{0'}\Aput[0.05]{\large $D_0$}
	\ncarc{->}{0}{1}\Aput[0.05]{\large $c$}
	\ncarc{->}{1}{0}\mput*[0.05]{\large $C$}
	\ncline[linestyle=dashed,nodesep=4pt]{-}{1}{i-1}
	\ncarc{->}{i-1}{i}\Aput[0.05]{\large $d_{i-1}$}
	\ncarc{->}{i}{i-1}\Aput[0.05]{\large $D_{i-1}$}
	\ncarc{->}{i}{i+1}\Aput[0.05]{\large $d_{i}$}
	\ncarc{->}{i+1}{i}\Aput[0.05]{\large $D_{i}$}
	\ncarc{->}{i+1}{i+2}\Aput[0.05]{\large $d_{i+1}$}
	\ncarc{->}{i+2}{i+1}\Aput[0.05]{\large $D_{i+1}$}
	\ncline[linestyle=dashed,nodesep=4pt]{-}{i+2}{m-1}
	\ncarc{->}{m-1}{m}\Aput[0.05]{\large $C'$}
	\ncarc{->}{m}{m-1}\mput*[0.05]{\large $c'$}
	\ncarc{<-}{m}{m'}\Aput[0.05]{\large $A'$}
	\ncarc{<-}{m'}{m}\Aput[0.05]{\large $a'$}
	\nccircle[angleA=-25,nodesep=3pt]{->}{u1}{.4cm}\Bput[0.05]{\large $u_1$}
	 \nccircle[nodesep=3pt]{->}{ui-1}{.4cm}\Bput[0.05]{\large $u_{i-1}$}
	 \nccircle[nodesep=3pt]{->}{ui+2}{.4cm}\Bput[0.05]{\large $u_{i+2}$}
	 \nccircle[angleA=25,nodesep=3pt]{->}{um-1}{.4cm}\Bput[0.05]{\large $u_{m-1}$}
	\nccircle[angleA=0,nodesep=3pt]{->}{m}{.4cm}\Bput[0.05]{\large $u_m$}
	\nccircle[angleA=180,nodesep=3pt]{->}{m'}{.4cm}\Bput[0.05]{\large $u_{m}'$}
	}
\rput(12,0){
	\rput(-0,1.7){\rnode{0}{\Large $\star$}}
	\rput(-0,-1.7){\rnode{0'}{\Large $0'\!$}}
	\rput(2,0){\rnode{1}{\Large $1$}}
	\rput(4,0){\rnode{2}{\Large $2$}}
	\rput(6.5,0){\rnode{m-2}{\Large $m-2$}}
	\rput(9,0){\rnode{m-1}{\Large $m-1$}}
	\rput(11,1.7){\rnode{m}{\Large $m$}}
	\rput(11,-1.7){\rnode{m'}{\Large $m'\!$}}	
	\rput(2,0.3){\rnode{u1}{}}	
	\rput(4,0.3){\rnode{ui-1}{}}	
	\rput(6.5,0.3){\rnode{ui+2}{}}	
	 \rput(9,0.3){\rnode{um-1}{}}	
	\ncarc{->}{0}{1}\Aput[0.05]{\large $c$}
	\ncarc{->}{1}{0}\Aput[0.05]{\large $C$}
	\ncarc{->}{0'}{1}\Aput[0.05]{\large $d_0$}
	\ncarc{->}{1}{0'}\Aput[0.05]{\large $D_0$}
	\ncarc{->}{1}{2}\Aput[0.05]{\large $d_{1}$}
	\ncarc{->}{2}{1}\Aput[0.05]{\large $D_{1}$}
	\ncline[linestyle=dashed,nodesep=4pt]{-}{2}{m-2}
	\ncarc{->}{m-2}{m-1}\Aput[0.05]{\large $d_{m-2}$}
	\ncarc{->}{m-1}{m-2}\Aput[0.05]{\large $D_{m-2}$}
	\ncarc{->}{m-1}{m}\Aput[0.05]{\large $C'$}
	\ncarc{->}{m}{m-1}\mput*[0.05]{\large $c'$}
	\ncarc{<-}{m}{m'}\Aput[0.05]{\large $A'$}
	\ncarc{<-}{m'}{m}\Aput[0.05]{\large $a'$}
	\nccircle[angleA=0,nodesep=3pt]{->}{0}{.4cm}\Bput[0.05]{\large $u_{\star}$}
	\nccircle[angleA=180,nodesep=3pt]{->}{0'}{.4cm}\Bput[0.05]{\large $u_0$}
	 \nccircle[nodesep=3pt]{->}{ui-1}{.4cm}\Bput[0.05]{\large $u_2$}
	 \nccircle[nodesep=3pt]{->}{ui+2}{.4cm}\Bput[0.05]{\large $u_{m-2}$}
	 \nccircle[angleA=25,nodesep=3pt]{->}{um-1}{.4cm}\Bput[0.05]{\large $u_{m-1}$}
	\nccircle[angleA=0,nodesep=3pt]{->}{m}{.4cm}\Bput[0.05]{\large $u_m$}
	\nccircle[angleA=180,nodesep=3pt]{->}{m'}{.4cm}\Bput[0.05]{\large $u_{m}'$}
	}

\end{pspicture}
}
\end{center}
{\footnotesize
\[
\begin{array}{rl}
W^{m \ldots (m-j)}=&
ad_0C+cD_0A -u_1D_0d_0-u_1Cc 
+\sum_{j=1}^{i-1}X_j +\sum_{j=2}^{i-1}Y_j  +Z_i -Z_{i+1} -
\sum_{j=i+2}^{m-1}X_j +\sum_{j=i+2}^{m}Y_j \\ 
& -u_{m}^2A'a'+A'a'u_{m}'. \\
W^{m\ldots 1} =& u_{\star}cC+u_{0}d_0D_0 -Ccd_1D_1 -D_0d_0d_1D_1
-\sum_{j=2}^{m-1}X_j +\sum_{j=3}^{m}Y_j -A'a'u_{m}^2 +a'A'u_{m}'.
\end{array}
\]
}


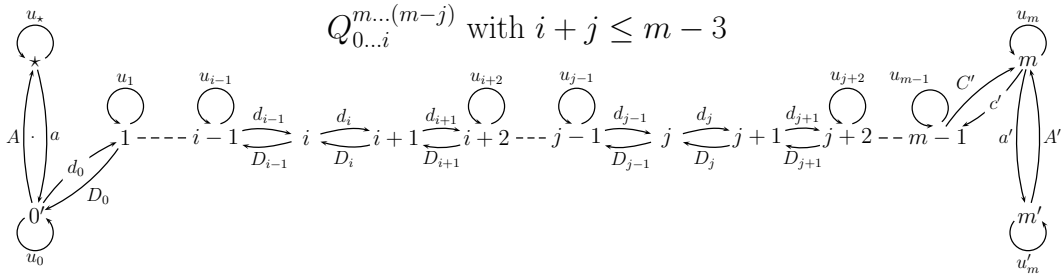
\begin{center}
\scalebox{0.6}{
\begin{pspicture}(-1,-3)(23,3.25)
	\psset{arcangle=15,nodesep=2pt}
\rput(11,2.5){\huge $Q_{0\ldots i}^{m \ldots (m-j)}$ with $i+j \leq m-3$}.
\rput(0,0){
	\rput(-0,1.7){\rnode{0}{\Large $\star$}}
	\rput(-0,-1.7){\rnode{0'}{\Large $0'\!$}}
	\rput(2,0){\rnode{1}{\Large $1$}}
	\rput(4,0){\rnode{i-1}{\Large $i-1$}}
	\rput(6,0){\rnode{i}{\Large $~i~$}}
	\rput(8,0){\rnode{i+1}{\Large $i+1$}}
	\rput(10,0){\rnode{i+2}{\Large $i+2$}}
	\rput(12,0){\rnode{j-1}{\Large $j-1$}}
	\rput(14,0){\rnode{j}{\Large $~j~$}}
	\rput(16,0){\rnode{j+1}{\Large $j+1$}}
	\rput(18,0){\rnode{j+2}{\Large $j+2$}}
	\rput(20,0){\rnode{m-1}{\Large $m-1$}}
	\rput(22,1.7){\rnode{m}{\Large $m$}}
	\rput(22,-1.7){\rnode{m'}{\Large $m'\!$}}	
	\rput(2,0.3){\rnode{u1}{}}	
	\rput(4,0.3){\rnode{ui-1}{}}	
	\rput(10,0.3){\rnode{ui+2}{}}	
	\rput(12,0.3){\rnode{uj-1}{}}	
	\rput(18,0.3){\rnode{uj+2}{}}	
	 \rput(20,0.3){\rnode{um-1}{}}	
	\ncarc{->}{0}{0'}\Aput[0.05]{\large $a$}
	\ncarc{->}{0'}{0}\Aput[0.05]{\large $A$}
	\ncarc{->}{0'}{1}\mput*[0.05]{\large $d_0$}
	\ncarc{->}{1}{0'}\Aput[0.05]{\large $D_0$}
	\ncline[linestyle=dashed,nodesep=4pt]{-}{1}{i-1}
	\ncarc{->}{i-1}{i}\Aput[0.05]{\large $d_{i-1}$}
	\ncarc{->}{i}{i-1}\Aput[0.05]{\large $D_{i-1}$}
	\ncarc{->}{i}{i+1}\Aput[0.05]{\large $d_{i}$}
	\ncarc{->}{i+1}{i}\Aput[0.05]{\large $D_{i}$}
	\ncarc{->}{i+1}{i+2}\Aput[0.05]{\large $d_{i+1}$}
	\ncarc{->}{i+2}{i+1}\Aput[0.05]{\large $D_{i+1}$}
	\ncline[linestyle=dashed,nodesep=4pt]{-}{i+2}{j-1}
	\ncarc{->}{j-1}{j}\Aput[0.05]{\large $d_{j-1}$}
	\ncarc{->}{j}{j-1}\Aput[0.05]{\large $D_{j-1}$}
	\ncarc{->}{j}{j+1}\Aput[0.05]{\large $d_{j}$}
	\ncarc{->}{j+1}{j}\Aput[0.05]{\large $D_{j}$}
	\ncarc{->}{j+1}{j+2}\Aput[0.05]{\large $d_{j+1}$}
	\ncarc{->}{j+2}{j+1}\Aput[0.05]{\large $D_{j+1}$}
	\ncline[linestyle=dashed,nodesep=4pt]{-}{j+2}{m-1}
	\ncarc{->}{m-1}{m}\Aput[0.05]{\large $C'$}
	\ncarc{->}{m}{m-1}\mput*[0.05]{\large $c'$}
	\ncarc{<-}{m}{m'}\Aput[0.05]{\large $A'$}
	\ncarc{<-}{m'}{m}\Aput[0.05]{\large $a'$}
	\nccircle[angleA=0,nodesep=3pt]{->}{0}{.4cm}\Bput[0.05]{\large $u_{\star}$}
	\nccircle[angleA=180,nodesep=3pt]{->}{0'}{.4cm}\Bput[0.05]{\large $u_0$}
	\nccircle[angleA=0,nodesep=3pt]{->}{u1}{.4cm}\Bput[0.05]{\large $u_1$}
	 \nccircle[nodesep=3pt]{->}{ui-1}{.4cm}\Bput[0.05]{\large $u_{i-1}$}
	 \nccircle[nodesep=3pt]{->}{ui+2}{.4cm}\Bput[0.05]{\large $u_{i+2}$}
	 \nccircle[nodesep=3pt]{->}{uj-1}{.4cm}\Bput[0.05]{\large $u_{j-1}$}
	 \nccircle[nodesep=3pt]{->}{uj+2}{.4cm}\Bput[0.05]{\large $u_{j+2}$}
	 \nccircle[angleA=25,nodesep=3pt]{->}{um-1}{.4cm}\Bput[0.05]{\large $u_{m-1}$}
	\nccircle[angleA=0,nodesep=3pt]{->}{m}{.4cm}\Bput[0.05]{\large $u_m$}
	\nccircle[angleA=180,nodesep=3pt]{->}{m'}{.4cm}\Bput[0.05]{\large $u_{m}'$}
	}
\end{pspicture}
}
\end{center}
{\footnotesize
\[
\begin{array}{rl}
W_{0\ldots i}^{m \ldots (m-j)} = &
aAu_{\star}-Aau_{0}^2 
+\sum_{k=0}^{i-1}X_k -\sum_{k=1}^{i-1}Y_k -Z_{i} +Z_{i+1} +\sum_{k=i+2}^{j-1}X_k -\sum_{k=i+2}^{j-1}Y_k\\
& +Z_{j} -Z_{j+1} -\sum_{k=j+2}^{m-1}X_k +\sum_{k=j+2}^{m}Y_k -a'A'u_{m}^2 +A'a'u_{m}'.
\end{array}
\]
}

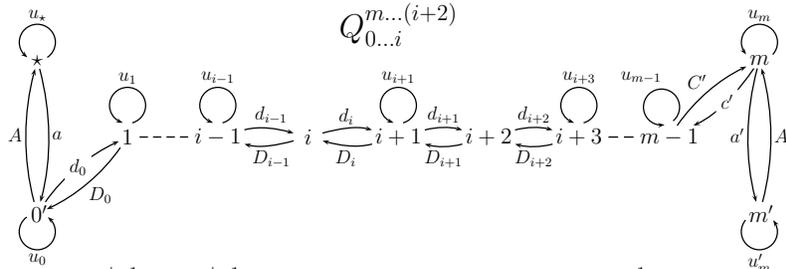
\begin{center}
\scalebox{0.6}{
\begin{pspicture}(-1,-2.4)(17,3.25)
	\psset{arcangle=15,nodesep=2pt}
\rput(8,2.5){\huge $Q_{0\ldots i}^{m \ldots (i+2)}$}	
\rput(0,0){
	\rput(-0,1.7){\rnode{0}{\Large $\star$}}
	\rput(-0,-1.7){\rnode{0'}{\Large $0'\!$}}
	\rput(2,0){\rnode{1}{\Large $1$}}
	\rput(4,0){\rnode{i-1}{\Large $i-1$}}
	\rput(6,0){\rnode{i}{\Large $~i~$}}
	\rput(8,0){\rnode{i+1}{\Large $i+1$}}
	\rput(10,0){\rnode{i+2}{\Large $i+2$}}
	\rput(12,0){\rnode{i+3}{\Large $i+3$}}
	\rput(14,0){\rnode{m-1}{\Large $m-1$}}
	\rput(16,1.7){\rnode{m}{\Large $m$}}
	\rput(16,-1.7){\rnode{m'}{\Large $m'\!$}}	
	\rput(2,0.3){\rnode{u1}{}}	
	\rput(4,0.3){\rnode{ui-1}{}}	
	\rput(8,0.3){\rnode{ui+1}{}}	
	\rput(12,0.3){\rnode{ui+3}{}}	
	 \rput(14,0.3){\rnode{um-1}{}}	
	\ncarc{->}{0}{0'}\Aput[0.05]{\large $a$}
	\ncarc{->}{0'}{0}\Aput[0.05]{\large $A$}
	\ncarc{->}{0'}{1}\mput*[0.05]{\large $d_0$}
	\ncarc{->}{1}{0'}\Aput[0.05]{\large $D_0$}
	\ncline[linestyle=dashed,nodesep=4pt]{-}{1}{i-1}
	\ncarc{->}{i-1}{i}\Aput[0.05]{\large $d_{i-1}$}
	\ncarc{->}{i}{i-1}\Aput[0.05]{\large $D_{i-1}$}
	\ncarc{->}{i}{i+1}\Aput[0.05]{\large $d_{i}$}
	\ncarc{->}{i+1}{i}\Aput[0.05]{\large $D_{i}$}
	\ncarc{->}{i+1}{i+2}\Aput[0.05]{\large $d_{i+1}$}
	\ncarc{->}{i+2}{i+1}\Aput[0.05]{\large $D_{i+1}$}
	\ncarc{->}{i+2}{i+3}\Aput[0.05]{\large $d_{i+2}$}
	\ncarc{->}{i+3}{i+2}\Aput[0.05]{\large $D_{i+2}$}
	\ncline[linestyle=dashed,nodesep=4pt]{-}{i+3}{m-1}
	\ncarc{->}{m-1}{m}\Aput[0.05]{\large $C'$}
	\ncarc{->}{m}{m-1}\mput*[0.05]{\large $c'$}
	\ncarc{<-}{m}{m'}\Aput[0.05]{\large $A'$}
	\ncarc{<-}{m'}{m}\Aput[0.05]{\large $a'$}
	\nccircle[angleA=0,nodesep=3pt]{->}{0}{.4cm}\Bput[0.05]{\large $u_{\star}$}
	\nccircle[angleA=180,nodesep=3pt]{->}{0'}{.4cm}\Bput[0.05]{\large $u_0$}
	\nccircle[angleA=0,nodesep=3pt]{->}{u1}{.4cm}\Bput[0.05]{\large $u_1$}
	 \nccircle[nodesep=3pt]{->}{ui-1}{.4cm}\Bput[0.05]{\large $u_{i-1}$}
	 \nccircle[nodesep=3pt]{->}{ui+1}{.4cm}\Bput[0.05]{\large $u_{i+1}$}
	 \nccircle[nodesep=3pt]{->}{ui+3}{.4cm}\Bput[0.05]{\large $u_{i+3}$}
	 \nccircle[angleA=25,nodesep=3pt]{->}{um-1}{.4cm}\Bput[0.05]{\large $u_{m-1}$}
	\nccircle[angleA=0,nodesep=3pt]{->}{m}{.4cm}\Bput[0.05]{\large $u_m$}
	\nccircle[angleA=180,nodesep=3pt]{->}{m'}{.4cm}\Bput[0.05]{\large $u_{m}'$}
	}
\end{pspicture}
}
\end{center}
{\footnotesize
\begin{align*}
W_{0\ldots i}^{m \ldots (i+2)} &=
aAu_{\star}-Aau_{0}^2 
+\sum_{k=0}^{i-1}X_k -\sum_{k=1}^{i-1}Y_k  -Z_{i} +Y_{i+1} +X_{i+1}  -Z_{i+2}  -\sum_{k=i+3}^{m-1}X_k +\sum_{k=i+3}^{m}Y_k -a'A'u_{m}^2 +A'a'u_{m}'
\end{align*}
}
\end{prop}

\subsubsection{The dihedral group of order $2n$ ($n$ odd)}

Let $n=2m+1$ with $m \geq 2$ and $(Q,W)$ be the McKay QP of $G$ obtained in \ref{QP-D-odd}. In the following proposition we write down all QPs which are obtained from the McKay QP not mutating at the vertex $\star$.

\begin{prop} For the dihedral group of order $2n$ ($n$ odd) there are $m+1$ non-equivalent mutation QPs of the form $(Q_{0\ldots i},W_{0\ldots i})$ which obtained from the McKay QP $(Q,W)$. The list of every possible $(Q_{0\ldots i},W_{0\ldots i})$ is the following:

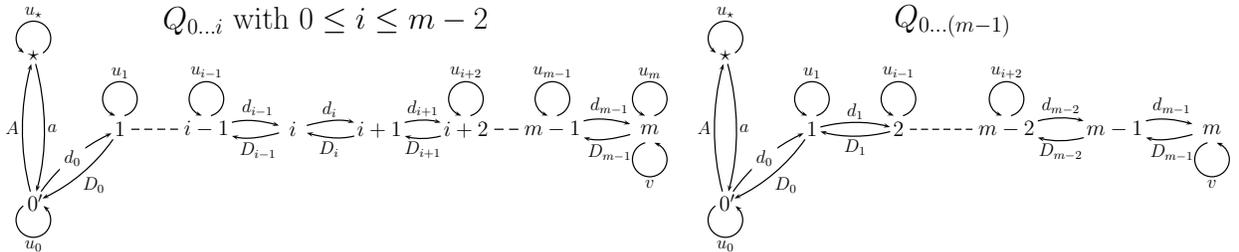
\begin{center}
\scalebox{0.575}{
\begin{pspicture}(-1,-2.75)(20,3.5)
	\psset{arcangle=15,nodesep=2pt}
\rput(2.75,2.5){\huge $Q_{0\ldots i}$ with $0 \leq i \leq m-2$}
\rput(17.25,2.5){\huge $Q_{0\ldots (m-1)}$}
\rput(-4,0){
	\rput(-0,1.7){\rnode{0}{\Large $\star$}}
	\rput(-0,-1.7){\rnode{0'}{\Large $0'\!$}}
	\rput(2,0){\rnode{1}{\Large $1$}}
	\rput(4,0){\rnode{i-1}{\Large $i-1$}}
	\rput(6,0){\rnode{i}{\Large $~i~$}}
	\rput(8,0){\rnode{i+1}{\Large $i+1$}}
	\rput(10,0){\rnode{i+2}{\Large $i+2$}}
	\rput(12,0){\rnode{m-1}{\Large $m-1$}}
	\rput(14.25,0){\rnode{m}{\Large $~m~$}}
	\rput(2,0.3){\rnode{u1}{}}	
	\rput(4,0.3){\rnode{ui-1}{}}	
	\rput(10,0.3){\rnode{ui+2}{}}	
	\rput(12,0.3){\rnode{um-1}{}}	
	\rput(14.25,0.3){\rnode{um}{}}	
	\rput(14.25,-0.3){\rnode{um'}{}}	
	\ncarc{->}{0}{0'}\Aput[0.05]{\large $a$}
	\ncarc{->}{0'}{0}\Aput[0.05]{\large $A$}
	\ncarc{->}{0'}{1}\mput*[0.05]{\large $d_0$}
	\ncarc{->}{1}{0'}\Aput[0.05]{\large $D_0$}
	\ncline[linestyle=dashed,nodesep=4pt]{-}{1}{i-1}
	\ncarc{->}{i-1}{i}\Aput[0.05]{\large $d_{i-1}$}
	\ncarc{->}{i}{i-1}\Aput[0.05]{\large $D_{i-1}$}
	\ncarc{->}{i}{i+1}\Aput[0.05]{\large $d_{i}$}
	\ncarc{->}{i+1}{i}\Aput[0.05]{\large $D_{i}$}
	\ncarc{->}{i+1}{i+2}\Aput[0.05]{\large $d_{i+1}$}
	\ncarc{->}{i+2}{i+1}\Aput[0.05]{\large $D_{i+1}$}
	\ncline[linestyle=dashed,nodesep=4pt]{-}{i+2}{m-1}
	\ncarc{->}{m-1}{m}\Aput[0.05]{\large $d_{m-1}$}
	\ncarc{->}{m}{m-1}\Aput[0.05]{\large $D_{m-1}$}
	\nccircle[angleA=0,nodesep=3pt]{->}{0}{.4cm}\Bput[0.05]{\large $u_{\star}$}
	\nccircle[angleA=180,nodesep=3pt]{->}{0'}{.4cm}\Bput[0.05]{\large $u_0$}
	\nccircle[angleA=0,nodesep=3pt]{->}{u1}{.4cm}\Bput[0.05]{\large $u_1$}
	 \nccircle[nodesep=3pt]{->}{ui-1}{.4cm}\Bput[0.05]{\large $u_{i-1}$}
	 \nccircle[nodesep=3pt]{->}{ui+2}{.4cm}\Bput[0.05]{\large $u_{i+2}$}
	 \nccircle[angleA=0,nodesep=3pt]{->}{um-1}{.4cm}\Bput[0.05]{\large $u_{m-1}$}
	 \nccircle[angleA=0,nodesep=3pt]{->}{um}{.4cm}\Bput[0.05]{\large $u_{m}$}
	 \nccircle[angleA=180,nodesep=3pt]{->}{um'}{.4cm}\Bput[0.05]{\large $v$}
	}
	\rput(12,0){
	\rput(-0,1.7){\rnode{0}{\Large $\star$}}
	\rput(-0,-1.7){\rnode{0'}{\Large $0'\!$}}
	\rput(2,0){\rnode{1}{\Large $1$}}
	\rput(4,0){\rnode{2}{\Large $2$}}
	\rput(6.5,0){\rnode{m-2}{\Large $m-2$}}
	\rput(9,0){\rnode{m-1}{\Large $m-1$}}
	\rput(11.25,0){\rnode{m}{\Large $~m~$}}
	\rput(2,0.3){\rnode{u1}{}}	
	\rput(4,0.3){\rnode{ui-1}{}}	
	\rput(6.5,0.3){\rnode{ui+2}{}}	
	\rput(9,0.3){\rnode{um-1}{}}	
	\rput(11.25,0.3){\rnode{um}{}}	
	\rput(11.25,-0.3){\rnode{um'}{}}	
	\ncarc{->}{0}{0'}\Aput[0.05]{\large $a$}
	\ncarc{->}{0'}{0}\Aput[0.05]{\large $A$}
	\ncarc{->}{0'}{1}\mput*[0.05]{\large $d_0$}
	\ncarc{->}{1}{0'}\Aput[0.05]{\large $D_0$}
	\ncarc{->}{1}{2}\Aput[0.05]{\large $d_{1}$}
	\ncarc{->}{2}{1}\Aput[0.05]{\large $D_{1}$}
	\ncline[linestyle=dashed,nodesep=4pt]{-}{2}{m-2}
	\ncarc{->}{m-2}{m-1}\Aput[0.05]{\large $d_{m-2}$}
	\ncarc{->}{m-1}{m-2}\Aput[0.05]{\large $D_{m-2}$}
	\ncarc{->}{m-1}{m}\Aput[0.05]{\large $d_{m-1}$}
	\ncarc{->}{m}{m-1}\Aput[0.05]{\large $D_{m-1}$}
	\nccircle[angleA=0,nodesep=3pt]{->}{0}{.4cm}\Bput[0.05]{\large $u_{\star}$}
	\nccircle[angleA=180,nodesep=3pt]{->}{0'}{.4cm}\Bput[0.05]{\large $u_0$}
	\nccircle[angleA=0,nodesep=3pt]{->}{u1}{.4cm}\Bput[0.05]{\large $u_1$}
	 \nccircle[nodesep=3pt]{->}{ui-1}{.4cm}\Bput[0.05]{\large $u_{i-1}$}
	 \nccircle[nodesep=3pt]{->}{ui+2}{.4cm}\Bput[0.05]{\large $u_{i+2}$}
	 \nccircle[angleA=180,nodesep=3pt]{->}{um'}{.4cm}\Bput[0.05]{\large $v$}
	}

\end{pspicture}
	}
\end{center}
{\footnotesize
\[
\begin{array}{rl}
W_{0\ldots i} = & 
u_{\star}aA -aAu_{0}^2 
+\sum_{j=0}^{i-1}d_{j}D_{j}u_{j} -\sum_{j=1}^{i-1}D_{j-1}d_{j-1}u_{j} 
 -D_{i-1}d_{i-1}d_{i}D_{i}+D_{i}d_{i}d_{i+1}D_{i+1} \\
& +\sum_{j=i+2}^{m-1}d_{j}D_{j}u_{j} -\sum_{j=i+2}^{m}D_{j-1}d_{j-1}u_{j} 
+u_mv^2. \\
W_{0\ldots (m-1)} = & 
aAu_{\star} -aAu_{0}^2 +\sum_{j=0}^{m-2}d_{j}D_{j}u_{j} -\sum_{j=1}^{m-2}D_{j-1}d_{j-1}u_{j} -D_{m-2}d_{m-2}d_{m-1}D_{m-1} +D_{m-1}d_{m-1}v^2.
\end{array}
\]
}
\end{prop}

\subsubsection{The tetrahedral group}

Let $(Q,W)$ be the McKay QP of $G$ obtained in \ref{QP-tetra}. In this case there are $5$ non-equivalent mutation QPs which are equivalent to $(Q,W)$ (see Figure \ref{E_6}). 

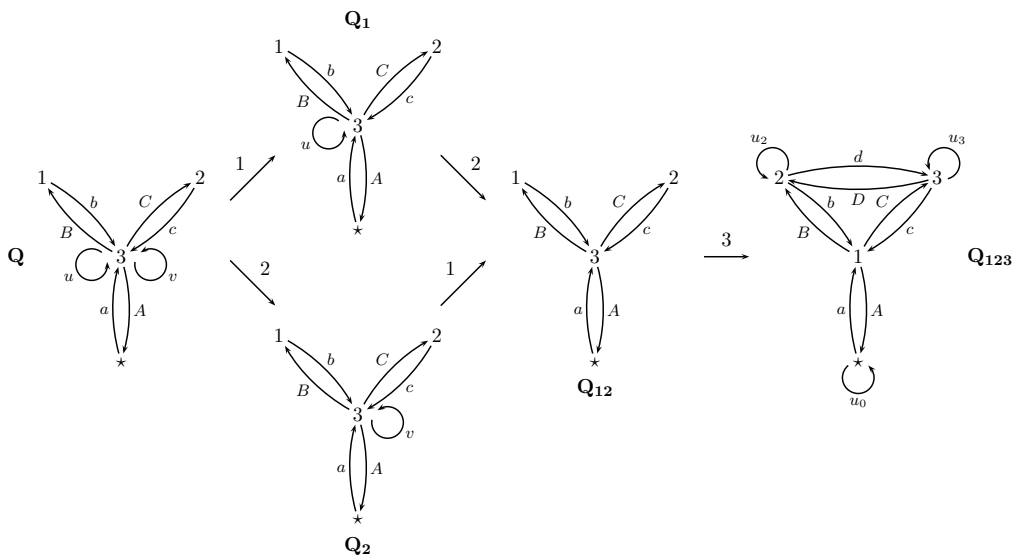
\begin{figure}[htbp]
\begin{center}
\begin{pspicture}(-2,-3.5)(12,3.5)
	\psset{arcangle=15,nodesep=2pt}
\scalebox{0.7}{

\rput(0,0){
	\rput(-2,0){$\bf Q$}
	\rput(0,-2){\rnode{0}{$\star$}}
	\rput(0,0){\rnode{1}{$3$}}
	\rput(-1.5,1.5){\rnode{2}{$1$}}
	\rput(1.5,1.5){\rnode{3}{$2$}}
	\rput(-0.3,0){\rnode{u}{}}	
	\rput(0.3,0){\rnode{v}{}}	
	\ncarc{->}{0}{1}\Aput[0.05]{\footnotesize $a$}
	\ncarc{->}{1}{0}\Aput[0.05]{\footnotesize $A$}
	\ncarc{->}{1}{2}\Aput[0.05]{\footnotesize $B$}
	\ncarc{->}{2}{1}\Aput[0.05]{\footnotesize $b$}
	\ncarc{->}{1}{3}\Aput[0.05]{\footnotesize $C$}
	\ncarc{->}{3}{1}\Aput[0.05]{\footnotesize $c$}
	\nccircle[angleA=120,nodesep=3pt]{->}{u}{.3cm}\Bput[0.05]{\footnotesize $u$}
	 \nccircle[angleA=240,nodesep=3pt]{->}{v}{.3cm}\Bput[0.05]{\footnotesize $v$}
}

\rput(4.5,2.5){
	\rput(0,2){$\bf Q_1$}
	\rput(0,-2){\rnode{0}{$\star$}}
	\rput(0,0){\rnode{1}{$3$}}
	\rput(-1.5,1.5){\rnode{2}{$1$}}
	\rput(1.5,1.5){\rnode{3}{$2$}}
	\rput(-0.3,0){\rnode{u}{}}	
	\rput(0.3,0){\rnode{v}{}}	
	\ncarc{->}{0}{1}\Aput[0.05]{\footnotesize $a$}
	\ncarc{->}{1}{0}\Aput[0.05]{\footnotesize $A$}
	\ncarc{->}{1}{2}\Aput[0.05]{\footnotesize $B$}
	\ncarc{->}{2}{1}\Aput[0.05]{\footnotesize $b$}
	\ncarc{->}{1}{3}\Aput[0.05]{\footnotesize $C$}
	\ncarc{->}{3}{1}\Aput[0.05]{\footnotesize $c$}
	\nccircle[angleA=120,nodesep=3pt]{->}{u}{.3cm}\Bput[0.05]{\footnotesize $u$}
}

\rput(4.5,-3){
	\rput(0,-2.5){$\bf Q_2$}
	\rput(0,-2){\rnode{0}{$\star$}}
	\rput(0,0){\rnode{1}{$3$}}
	\rput(-1.5,1.5){\rnode{2}{$1$}}
	\rput(1.5,1.5){\rnode{3}{$2$}}
	\rput(-0.3,0){\rnode{u}{}}	
	\rput(0.3,0){\rnode{v}{}}	
	\ncarc{->}{0}{1}\Aput[0.05]{\footnotesize $a$}
	\ncarc{->}{1}{0}\Aput[0.05]{\footnotesize $A$}
	\ncarc{->}{1}{2}\Aput[0.05]{\footnotesize $B$}
	\ncarc{->}{2}{1}\Aput[0.05]{\footnotesize $b$}
	\ncarc{->}{1}{3}\Aput[0.05]{\footnotesize $C$}
	\ncarc{->}{3}{1}\Aput[0.05]{\footnotesize $c$}
	 \nccircle[angleA=240,nodesep=3pt]{->}{v}{.3cm}\Bput[0.05]{\footnotesize $v$}
}

\rput(9,0){
	\rput(0,-2.5){$\bf Q_{12}$}
	\rput(0,-2){\rnode{0}{$\star$}}
	\rput(0,0){\rnode{1}{$3$}}
	\rput(-1.5,1.5){\rnode{2}{$1$}}
	\rput(1.5,1.5){\rnode{3}{$2$}}
	\rput(-0.3,0){\rnode{u}{}}	
	\rput(0.3,0){\rnode{v}{}}	
	\ncarc{->}{0}{1}\Aput[0.05]{\footnotesize $a$}
	\ncarc{->}{1}{0}\Aput[0.05]{\footnotesize $A$}
	\ncarc{->}{1}{2}\Aput[0.05]{\footnotesize $B$}
	\ncarc{->}{2}{1}\Aput[0.05]{\footnotesize $b$}
	\ncarc{->}{1}{3}\Aput[0.05]{\footnotesize $C$}
	\ncarc{->}{3}{1}\Aput[0.05]{\footnotesize $c$}
}

\rput(14,0){
	\rput(2.5,0){$\bf Q_{123}$}
	\rput(0,-2){\rnode{0}{$\star$}}
	\rput(0,0){\rnode{1}{$1$}}
	\rput(-1.5,1.5){\rnode{2}{$2$}}
	\rput(1.5,1.5){\rnode{3}{$3$}}
	\ncarc{->}{0}{1}\Aput[0.05]{\footnotesize $a$}
	\ncarc{->}{1}{0}\Aput[0.05]{\footnotesize $A$}
	\ncarc{->}{1}{2}\Aput[0.05]{\footnotesize $B$}
	\ncarc{->}{2}{1}\Aput[0.05]{\footnotesize $b$}
	\ncarc{->}{1}{3}\Aput[0.05]{\footnotesize $C$}
	\ncarc{->}{3}{1}\Aput[0.05]{\footnotesize $c$}
	\ncarc{->}{2}{3}\Aput[0.05]{\footnotesize $d$}
	\ncarc{->}{3}{2}\Aput[0.05]{\footnotesize $D$}
	\nccircle[angleA=25,nodesep=3pt]{->}{2}{.3cm}\Bput[0.05]{\footnotesize $u_2$}
	 \nccircle[angleA=-25,nodesep=3pt]{->}{3}{.3cm}\Bput[0.05]{\footnotesize $u_3$}
	 \nccircle[angleA=180,nodesep=3pt]{->}{0}{.3cm}\Bput[0.05]{\footnotesize $u_0$}
}

\rput(2,1){\rput(0,0){\rnode{0}{}}\rput(1,1){\rnode{1}{}}\ncline{->}{0}{1}\Aput{$1$}}
\rput(2,0){\rput(0,0){\rnode{0}{}}\rput(1,-1){\rnode{1}{}}\ncline{->}{0}{1}\Aput{$2$}}
\rput(6,2){\rput(0,0){\rnode{0}{}}\rput(1,-1){\rnode{1}{}}\ncline{->}{0}{1}\Aput{$2$}}
\rput(6,-1){\rput(0,0){\rnode{0}{}}\rput(1,1){\rnode{1}{}}\ncline{->}{0}{1}\Aput{$1$}}
\rput(11,0){\rput(0,0){\rnode{0}{}}\rput(1,0){\rnode{1}{}}\ncline{->}{0}{1}\Aput{$3$}}

}
\end{pspicture}
\caption{Mutations of type $\mathbb T$.}
\label{E_6}
\end{center}
\end{figure}

The potentials in these cases are: 
\[
\begin{array}{rl}
W = &uAa+\omega uBb +\omega^2 uCc -\frac{1}{3}u^3  -vAa-\omega^2 vBb-\omega vCc +\frac{1}{3}v^3. \\
W_{1} = &(1-\omega^2)Aau-\omega AaBb+(\omega^2-1)Ccu-\omega^2 CcbB+\omega^2 Bbu^2+\omega(Bb)^2u+\frac{1}{3}(Bb)^3.\\
W_{2} = &(\omega^2-1)Aav-\omega AaCc+(1-\omega^2)Bbv-\omega^2 BbCc+\omega^2 Ccv^2+\omega(Cc)^2v+\frac{1}{3}(Cc)^3.\\
W_{12} = &AaBb+AaCc-(Bb)^2Cc-Bb(Cc)^2.\\
W_{123} = &aAu_0 -AaBb -AaCc +bBu_2 -dDu_2 + cCu_3 -Ddu_3 +Bdc +CDb.
\end{array}
\]

We note that in $(Q_1,W_1)$ and $(Q_2,W_2)$ there is a relation of the form $v=\omega^2u-\omega Bb$ and $u=\omega^2v-\omega Cc$ respectively, so $u$ and $v$ are symmetric and $Q_{12}=Q_{21}$.

\subsection{Mutations of quivers with potential and NCCRs}
\label{sect:mutNCCRs}

In this section we explain why we do not consider mutations at the trivial vertex $\star$ and the reason why we do not take into account mutations at vertices with loops. We finish the section with a brief discussion about the number of NCCRs for the cases we treat in this paper in the complete local setting. 

We first start by explaining the equivalence between mutation of QP of Section \ref{Sect:Mutation} and mutations of NCCRs (also called tilting mutation). Let $\Lambda = \mathcal P(Q,W)=\bigoplus_{i\in Q_0}P_i$ be a NCCR and $P_k$ the projective right $\Lambda$-module associated to the vertex $k$. Let $f : P_k \to X$ be a left $\add \Lambda/P_k$-approximation and take $K_k := \Coker f$. If $f$ is injective, then $\mu_k\Lambda:=\Lambda/P_k \oplus K_k$ is a tilting $\Lambda$-module and $\End_{\Lambda}(\mu_k\Lambda)$ is also a NCCR. By a similar strategy of \cite{BIRS}, it can be shown that if a QP $(Q,W)$ is gradable and $\Lambda$ is a 3-Calabi-Yau algebra, then the tilting mutation coincides with the mutation of QP. In other words, we have an isomorphism $\End_{\Lambda}(\mu_k\Lambda) \simeq \mathcal P(\mu_k(Q,W))$. In our case, since the potential $W$ is homogeneous of degree 3 the McKay QP $(Q,W)$ is graded, so all mutations $\mu(Q,W)$ obtained from it are gradable. Also, note that the algebras $P(\mu_k(Q,W))$ are 3-Calabi-Yau (3-CY for short) since $S*G$ is 3-CY and the property of 3-CY is closed under Morita equivalences and mutations, so the result applies. \\

The reason why we do not mutate at the trivial vertex $\star$ can be explain in two different ways. In one hand this vertex corresponds to the trivial representation of $G$, so it does not correspond to any exceptional curve in the fibre over the origin. Thus geometrically there is no reason to mutate (or equivalently flop) at the vertex $\star$. On the other hand, in the context of NCCRs we are dealing with algebras of the form $\End_{R}(M)$ for some reflexive module $M$. Then it is easy to see that if $R$ is a Gorenstein ring then $M$ is a Cohen-Macaulay $R$-module if and only if $M$ contains $R$ as a direct summand. Therefore, the mutation at $\star$ would replace $R$ by a different module, losing the Cohen-Macaulay condition. \\

With respect to the mutation at vertices with loops, let us consider first the following result of Iyama and Wemyss.

\begin{thm}[\cite{IW10}, 6.13]\label{MutTriv}
Suppose $R$ is a complete local normal three-dimensional Gorenstein ring. Denote $\Lambda= \End_R(M)$, let $M_k$ be an indecomposable summand of $M$ and consider $\Lambda_k := \Lambda/\Lambda(1- e_k)\Lambda$ where $e_k$ is the idempotent in $\Lambda$ corresponding to $M_k$. Then if $\dim_\C \Lambda_k = \infty$ then $\mu_k\Lambda \simeq \Lambda$.
\end{thm}

If we remove the complete local condition this result is still true but up to additive closure, that is, one can show that if $\dim_\C \Lambda_k = \infty$ then $\add(M/M_k \oplus K_k)=\add M$. In some respect this is enough for our purposes, since at the level of modules we have that $\End_R(M/M_k \oplus K_k)$ and $\End_R(M)$ are Morita equivalent, which induces an isomorphism of the corresponding moduli spaces.

For the groups treated in this paper, the following proposition states that every iterated QP obtained from the McKay QP verifies that the dimension of the factor algebra $\Lambda_k$ at every vertex $k$ with a loop is infinity. In fact, every vertex with a loop corresponds to a non-floppable curve in the corresponding moduli space (see Section \ref{Sect:Floppable}), allowing us to consider only mutations at vertices without loops. 

\begin{prop}
Let $G\subset\SO(3)$ of type $\Z/n\Z$, $D_{2n}$ and $\mathbb{T}$. Let $(Q',W')$ be an iterated QP from the McKay QP $(Q,W)$ and let $\Lambda = \mathcal P(Q',W')$ be the Jacobian algebra. Then for any vertex $i \in Q'_0$ with a loop we have that $\dim_{\C} \Lambda_i = \infty$.
\end{prop}

\begin{proof}
In the Abelian case every vertex $i$ has a loop and $\Lambda_{i} = \bigoplus_{\ell \geq 0} \C c_{i}^{\ell}$, thus $\dim_{\C} \Lambda_i = \infty$. For the dihedral groups it is straightforward to check that for every vertex $i\in Q'$ there is no relation derived from the potential $W'$ which identifies $u_i^p$ as a combination of other paths for some $p>0$. 

For $G=\mathbb{T}$, notice that in the McKay QP we have relations of the form $u^2=3(Aa+\omega Bb+\omega^2Cc)$ and $v^2=3(Aa+\omega^2Bb+\omega Cc)$ but there is no relation involving $uv$, so $\dim_\C\Lambda_3=\infty$.
\end{proof}

However, for subgroups $G \subset SO(3)$ of type $\mathbb O$ and $\mathbb I$, there are vertices $i$ in the McKay QP with a loop such that $\dim_{\C} \Lambda_i < \infty$. For instance, in $\mathbb O$ we have relations of the form $u^2=3(Aa+bB+dD)$ and $v^2=3(Cc+Dd+eE)$ so $\dim_\C\Lambda_3=\dim_\C\Lambda_4=2$. This fact predicts that the mutation at this vertices is not trivial and also that the $(-2,0)$-curves corresponding to this vertices are floppable. At present we do not have a definition of mutation of quivers with potential at a vertex with a loop, so these cases will be treated in a future work. \\

We finish this section by considering the number of NNCRs for the cases treated in this paper. We restrict ourselves to the complete case, that is, let $\widehat R$ be the completion of $R:=S^G$ and we want to look for all possible NCCRs over $\widehat R$. We say that $\Lambda:=\End_{\widehat R}(M)$ is a Cohen-Macaulay (CM) NCCR if $M$ is Cohen-Macaulay. If the set of NCCRs obtained by a sequence of mutations at non-trivial vertices is finite, then this set contains all possible CM NCCRs (see \cite{IW13}, Theorem 1.9). In our case we proved that the list of QPs obtained by mutating at non-trivial vertices is finite (notice that by Theorem \ref{MutTriv} mutating at loops in the complete setting is trivial), so we obtain the following result.

\begin{thm}\label{NCCRsFinite}
Let $G$ be a finite subgroup of $SO(3)$ of types $\Z/n\Z$, $D_{2n}$ and $\mathbb T$. The number of mutations of the McKay QP at non-trivial vertices is finite up to isomorphisms. Moreover the number of CM NCCRs of $\widehat{R}$ is finite up to Morita equivalences.
\end{thm}

\section{Explicit description of the moduli spaces $\mathcal{M}_\theta$}
\label{sect:opens}

In this section we describe the explicit structure of every projective crepant resolution $\pi:X\to\C^3/G$ for $G\subset\SO(3)$ of types $\Z/n\Z$, $D_{2n}$ and $\mathbb{T}$. We do not describe the cyclic case $G\cong\Z/n\Z$ since $\Hilb{G}{\C^3}$ is the unique crepant resolution and it is already treated in \cite{CR02} and \cite{Nak01}. 

The results are summarized in Theorems \ref{OpensDnOdd}, \ref{OpensDnEven} and \ref{OpensE6} for the subgroups of type $D_{2n}$ with $n$ odd, $D_{2n}$ with $n$ even and $\mathbb{T}$ respectively. This explicit description allows us to conclude that every projective crepant resolution $X$ is isomorphic to a moduli space $\mathcal{M}_\theta$ of $\theta$-stable representations of the McKay quiver with relations $(Q, R)$ for some $\theta\in\Theta$, and $X$ consists of a finite union of copies of $\C^3$. Moreover, we give local coordinates of every open set and the degrees of the normal bundles of every exceptional rational curve in the fibre $\pi^{-1}(0)$. 

The proof is done by the explicit calculation of every case. The strategy used is the following:
\begin{enumerate}
\item[{\bf Step 1}] We obtain first the crepant resolution $X:=\Hilb{G}{\C^3}\cong\mathcal{M}_{\theta^0}$, where $\theta^0$ is the so called {\em $0$-generated stability condition} (i.e.\ $\theta^0$ is generic and $\theta^0_i>0$ for every $i\neq0$), which from \cite{IN00} is well known to be contained in the chamber corresponding to $\Hilb{G}{\C^3}$. In particular, this choice of stability implies that for every $\rho\in\Irr G\backslash\{\rho_0\}$ there exist $\dim(\rho)$ linearly independent paths from $\rho_0$ to $\rho$.
\item[{\bf Step 2}] By calculating the gluings between the open sets in $X$ we obtain the degrees of the normal bundle of every rational exceptional curve in $\pi^{-1}(0)$. This is done by using the quiver of the endomorphism algebra $\End_{S^G}(\bigoplus_{\rho\in\Irr G}S_\rho)$ where $S_\rho$ are the CM $S^G$-modules $S_\rho:=(S\otimes\rho)^G$, which coincides with the McKay quiver. This identification allows us to give local coordinates at each open set in terms of the original coordinates $x,y$ and $z$ in $\C^3$. Then by Lemma \ref{floppable} we know that only the rational curves $E\subset X$ of type $(-1,-1)$ are floppable, which gives us a finite number of possible flops of $X$.
\item[{\bf Step 3}] We calculate the flop $X'$ of $X$ by only modifying the open sets containing the curve which is flopped, obtaining $X'$ again as a moduli space $\mathcal{M}_{C'}$. Here we use a representation space which dominates both sides of the flop. 
\item[{\bf Step 4}] We repeat the process until we find all possible flops of $(-1,-1)$-curves.
\end{enumerate}
 
As a consequence, every projective crepant resolution $X$ of $\C^3/G$ for $G\subset\SO(3)$ of types $\Z/n\Z$, $D_{2n}$ and $\mathbb T$ is described by an affine open cover $X\cong\bigcup U_i$ with $U_i\cong\C^3$ (see part (1) in Theorems \ref{OpensDnOdd}, \ref{OpensDnEven} and \ref{OpensE6}). 
 
In this section $S:=\C[x,y,z]$, $(Q,W)$ always denote the McKay quiver potential as in Section \ref{Sect:McKayQP-SO(3)}, $\Lambda:= \mathcal{P}(Q,W)$ and ${\bf{d}}=(d_i)_{i\in Q_0}:=(\dim\rho)_{\rho\in\Irr G}$. Thus we write simply $\mathcal{M}_\theta$ or $\mathcal{M}_C$ to denote the moduli space $\mathcal{M}_{\theta,{\bf{d}}}(\Lambda)$, where $\theta\in C$ for some $C\subset\Theta$. 

The notion of stability of a representation of $(Q,W)$ is defined as follows (cf.\ \cite{King}): let $M$ be a representation of $Q$ of dimension vector $\textbf{d}$, let $\theta\in\Theta_{\bf d}$ and define $\theta(M):=\sum\theta_id_i$. Then $M$ is $\theta$-(semi)stable if $\theta(M')>0=\theta(M)$ for $0\subsetneq M'\subsetneq M$ (with the usual $\geq$ for semistability). The stability parameter $\theta\in\Theta_{\bf{d}}$ is said to be {\em generic} if every $\theta$-semistable representation is $\theta$-stable. 

For dimension vectors ${\bf d}$ of representations (and subrepresentations) of $Q$ we adopt the notation ${\bf d}={\Large{\begin{smallmatrix}d_0\\d_{0'}\end{smallmatrix}}}d_1\ldots d_m$, where $d_i=\dim V_i$ and $V_i$ is the vector space associated to the vertex $i\in Q_0$.

\subsection{The dihedral group of order $2n$ ($n$ odd)}\label{sect:Dnodd}

With the potential given in Section \ref{QP-D-odd} the relations derived by the potential $W$ in this case are:
\[
\begin{small}
\begin{array}{rl}
\partial a, \partial A, \partial b, \partial B, \partial c, \partial C: & bC=0, cB =0, Ca=u_1B, Ac=bu_1, BA=u_1C, ab=cu_1. \\
\partial d_1,\ldots , \partial d_{m-1}: & D_1u_1=u_2D_1, \ldots, D_{m-1}u_{m-1}=u_{m}D_{m-1}. \\
\partial D_1, \ldots,\partial D_{m-1}: & u_1d_1=d_1u_2, \ldots, u_{m-1}d_{m-1}=d_{m-1}u_m. \\ 
\partial u_1: & Bb+Cc=d_1D_1. \\
\partial u_2, \ldots, \partial u_{m-1}: & d_2D_2=D_1d_1, \ldots, d_{m-1}D_{m-1}=D_{m-2}d_{m-2}. \\
\partial u_m : & D_{m-1}d_{m-1}=v^2. \\
\partial v: & u_mv+vu_m=0. 

\end{array}
\end{small}
\]

We introduce the following notation for the arrows of $Q$ as linear maps between vector spaces: $a:=a$, $A:=A$, $b:=(b_1,b_2)$, $B:=\left(\begin{smallmatrix}B_1\\B_2\end{smallmatrix}\right)$, $c:=(c_1,c_2)$, $C=\left(\begin{smallmatrix}C_1\\C_2\end{smallmatrix}\right)$, $d_i:=\left(\begin{smallmatrix}d^i_{11}&d^i_{12}\\d^i_{21}&d^i_{22}\end{smallmatrix}\right)$, $D_i:=\left(\begin{smallmatrix}D^i_{11}&D^i_{12}\\D^i_{21}&D^i_{22}\end{smallmatrix}\right)$, $u_j:=\left(\begin{smallmatrix}u^j_{11}&u^j_{12}\\u^j_{21}&u^j_{22}\end{smallmatrix}\right)$ and $v:=\left(\begin{smallmatrix}v_{11}&v_{12}\\v_{21}&v_{22}\end{smallmatrix}\right)$, where $1\leq i\leq m-1$, $1\leq j\leq m$, and every entry belongs to $\C$. We may drop the upper indices in the entries of the matrices for $d_i$, $D_i$ and $u_i$ if the context allows us to do so, and we may also drop the lower indices when there exist a unique element in the matrix which is non-zero. 

\begin{thm}\label{OpensDnOdd} Let $G=D_{2n}\subset\SO(3)$ with $n$ odd. Let $f_1:=x^{2m+1}+y^{2m+1}$ and $f_2:=x^{2m+1}-y^{2m+1}$. Then,
\begin{enumerate}
\item There are $m+1$ crepant resolutions of $\pi_i:X_{0\ldots i}\to\C^3/G$ for $-1\leq i\leq m$, given by the open covers:
\[
X_{0\ldots i} = \bigcup_{k=1}^{i+1}U''_k\cup U'_{i+2}\cup\bigcup_{k=i+3}^{m+2}U_i 
\]
where $U_i, U'_i, U''_i\cong\C^3$ for all $i$, with corresponding local coordinates
{\renewcommand{\arraystretch}{1.5}
\[
\begin{array}{l}
U'_1 \cong \C^3_{a,b,C}= \Spec\C[\frac{z}{f_2}, \text{\scriptsize $xy$}, \text{\scriptsize $f_1$}]. \\
U_i \cong \C^3_{d,D,u}= \Spec\C[\frac{(xy)^{i-1}z}{f_2}, \frac{f_2}{(xy)^{i-2}z}, \frac{zf_1}{f_2}], \text{ for $i\leq m+1$}.  \\
U_{m+2} \cong \C^3_{v,V,u}= \Spec\C[ \text{\scriptsize $z^2$},\frac{f_1}{(xy)^{m}}, \frac{f_2}{(xy)^mz}]. \\
U'_i \cong \C^3_{a,d,D}= \Spec\C[\frac{(xy)^{i-1}z}{f_2}, \frac{f_1}{(xy)^{i-1}},\text{\scriptsize $xy$}], \text{ for $i\leq m+1$}.  \\
U''_i \cong \C^3_{a,d,D}= \Spec\C[\frac{zf_1}{f_2}, \frac{(xy)^i}{f_1},\frac{f_1}{(xy)^{i-1}}], \text{ for $i\leq m$}.  
\end{array}
\]}
For $i=-1$ we have $X\cong\Hilb{G}{\C^3}$.
\item The degrees of the normal bundles $\mathcal{N}_{X\slash E}$ of the exceptional rational curves $E\subset X_{0\ldots i}$ are 
\[
{\small
\begin{array}{|c|l|}
\hline
\text{Open cover of $E$} & \text{Degree of $\mathcal{N}_{X\slash E}$} \\
\hline
U_i\cup U_{i+1} 	& \text{$(-2,0)$ for $i=2,\ldots,m$} \\
				& \text{$(-3,1)$ for $i=m+1$} \\
\hline
U'_i\cup U_{i+1} 	& \text{$(-1,-1)$ for $i=1,\ldots,m$} \\
				& \text{$(-2,0)$ for $i=m+1$} \\
\hline
U''_i\cup U'_{i+1} 	& \text{$(-1,-1)$ for $i=1,\ldots,m-1$} \\				
\hline
U''_i\cup U''_{i+1} 	& \text{$(-2,0)$ for $i=1,\ldots,m$} \\
\hline
\end{array}}
\]
\item Let $\pi_{i}:X_{0\ldots i}\to\C^3/G$ a crepant resolution. Then the dual graph of $\pi_i^{-1}(0)$ is of the form:
\begin{center}
\begin{pspicture}(0,-.025)(5,0.5)
	\psset{arcangle=15,nodesep=2pt}
\rput(0,0){\rnode{0}{$\bullet$}}
\rput(1,0){\rnode{1}{$\bullet$}}
\rput(2,0){\rnode{2}{$\cdots$}}
\rput(3,0){\rnode{3}{$\bullet$}}
\rput(4,0){\rnode{4}{$\bullet$}}
\ncline{-}{0}{1}\ncline{-}{1}{2}\ncline{-}{2}{3}\ncline{-}{3}{4}
\end{pspicture}
\end{center}
\end{enumerate}
\end{thm}

\begin{proof} {\em Step 1}. Let $X:=\Hilb{G}{\C^3}\cong\mathcal{M}_{\theta}$ for the $0$-generated stability condition $\theta$ and dimension vector ${\bf d}={\Large{\begin{smallmatrix}1\\1\end{smallmatrix}}}2\ldots2$. Then we can construct the following open cover of $X$:

\[\Hilb{G}{\C^3}\cong U'_1\cup\bigcup_{i=2}^{m+2}U_i\]

We divide the proof of this fact in 4 steps:
\begin{itemize}
\item[(i)] We can always choose $c=(1,0)$. Otherwise, by the $0$-generated stability we need to have $ab=(1,0)$, so that the relation $ab=cu_1$ implies that $A(c_1)^2+u^1_{21}c_2=1$. But then $c_1$ and $c_2$ cannot be both zero at the same time, which means that we can change basis at the vertex $1$ to obtain $c=(1,0)$. 
\item[(ii)] Similarly, we can change basis to choose $d_i:=\left(\begin{smallmatrix}1&0\\d^i_{21}&d^i_{22}\end{smallmatrix}\right)$ for every $i$. Therefore it is remaining to generate the basis element $(0,1)$ at every 2-dimensional vertex. 
\item[(iii)] The open conditions to generate the 2-dimensional vertices can be done involving only the maps $d_i$ and $D_i$. Indeed, if we suppose that $u_i=\left(\begin{smallmatrix}0&1\\u^i_{21}&u^i_{22}\end{smallmatrix}\right)$ then using the relations involving the vertex $i$ we can conclude that $D^i_{22}\neq0$, thus we can choose $D^i_{21}=0$ and $D^i_{22}=1$ instead. 
\item[(iv)] By the stability condition we need to reach the vertex $0'$ with a nonzero map from $0$, thus we have that 
\[
\begin{array}{rl}
\text{either} & cd_1d_2\cdots d_iD_iD_{i-1}\cdots D_1B\neq0 \\
\text{or} & a\neq0
\end{array}
\]
Consider the first case and suppose initially that $i=1$, which after a change of basis is equivalent to say that $cd_1D_1B=1$. In particular we can choose $B_2=1$ and $D_1:=\left(\begin{smallmatrix}0&1\\D^1_{21}&D^1_{22}\end{smallmatrix}\right)$, which leads to a contradiction when applying the relation $Bb+Cc=d_1D_1$. Similarly, for $i>0$ we obtain contradiction with the relation $D_{i-1}d_{i-1}=d_iD_i$. Therefore we either have $cd_1\cdots d_{m-1}u_mD_{m-1}\cdots D_1B\neq0$ or $cd_1\cdots d_{m-1}vD_{m-1}\cdots D_1B\neq0$. By (ii) we can always choose the second option, which gives the open set 
\[
U'_1:cd_1\cdots d_{m-1}vD_{m-1}\cdots D_1B=1
\] 
By the relation $cB=0$ we get $B_1=0$, so by changing basis we can always take $c=(1,0)$, $d_i:=\left(\begin{smallmatrix}1&0\\d^i_{21}&d^i_{22}\end{smallmatrix}\right)$, $D_i:=\left(\begin{smallmatrix}d^i_{11}&d^i_{12}\\0&1\end{smallmatrix}\right)$ for all $i$, $v=\left(\begin{smallmatrix}0&1\\v_{21}&v_{22}\end{smallmatrix}\right)$ and $B:=\left(\begin{smallmatrix}0\\1\end{smallmatrix}\right)$, and using the relations we obtain the representation space for $U'_1$ shown in Figure \ref{DnoddU1'}. If we suppose that $a\neq0$, by (ii) and the usual change of basis at every 2-dimensional vertex $i$ we reach the standard basis element $(1,0)$ by the path $cd_1\cdots d_{i}$. Then, the rest of possibilities for the open sets are
\[
\begin{array}{rl}
U_2: & cd_1\cdots d_{m-1}vD_{m-1}\cdots D_1=(0,1), a=1 \\
U_3: & cd_1\cdots d_{m-1}vD_{m-1}\cdots D_2=(0,1), ab=1 \\
U_i: & cd_1\cdots d_{m-1}vD_{m-1}\cdots D_i=(0,1), abd_1\cdots d_{i-1}=(0,1) \text{ for $4\leq i<m$} \\
U_{m+1}: & cd_1\cdots d_{m-1}v=(0,1), abd_1\cdots d_{m-2}=(0,1) \\
U_{m+2}: & cd_1\cdots d_{m-1}=(1,0), abd_1\cdots d_{m-1}=(0,1)
\end{array}
\]

\begin{figure}[h]
\begin{pspicture}(0,-1.25)(10,2.25)
	\psset{arcangle=15,nodesep=2pt}
\rput(-2.75,0){
\rput(3,1.85){$U'_1\cong\C^3_{a,b,C}$}
\rput(4,-1){\tiny $A=aC^2-ab^{2m+1}$}
\scalebox{0.8}{
	\rput(-0,1.7){\rnode{0}{$0$}}
	\rput(-0,-1.7){\rnode{1}{$0'$}}
	\rput(2,0){\rnode{2}{$1$}}
	\rput(4,0){\rnode{3}{$2$}}
	\rput(6,0){\rnode{5}{\footnotesize$m\!-\!1$}}
	\rput(8,0){\rnode{6}{$m$}}
	\rput(2,0.3){\rnode{u1}{}}	
	\rput(4,0.3){\rnode{u2}{}}	
	\rput(6,0.3){\rnode{u3}{}}	
	\rput(8,0.3){\rnode{u5}{}}	
	\rput(8,-0.3){\rnode{u6}{}}	
	\ncarc{->}{0}{1}\Aput[0.05]{\footnotesize $a$}		
	\ncarc{->}{1}{0}\Aput[0.05]{\footnotesize $A$}		
	\ncarc{->}{1}{2}\mput*[0.05]{\scriptsize $(\text{-}C,b)$}		
	\ncarc[linecolor=red]{->}{2}{1}\Aput[0.05]{\large $\left(\begin{smallmatrix}0\\1\end{smallmatrix}\right)$}		
	\ncarc[linecolor=red]{->}{0}{2}\Aput[0]{\scriptsize $(1,0)$}		
	\ncarc{->}{2}{0}\mput*[0.05]{\large $\left(\!\begin{smallmatrix}b\\C\end{smallmatrix}\!\right)$}	
	\ncarc[linecolor=red]{->}{2}{3}\Aput[0.05]{$\left(\begin{smallmatrix}1&0\\0&b\end{smallmatrix}\right)$}		
	\ncarc[linecolor=red]{->}{3}{2}\Aput[0.05]{$\left(\begin{smallmatrix}b&0\\0&1\end{smallmatrix}\right)$}		
	\ncline[linestyle=dashed]{-}{3}{5}
	\ncarc[linecolor=red]{->}{5}{6}\Aput[0.05]{$\left(\begin{smallmatrix}1&0\\0&b\end{smallmatrix}\right)$}		
	\ncarc[linecolor=red]{->}{6}{5}\Aput[0.05]{$\left(\begin{smallmatrix}b&0\\0&1\end{smallmatrix}\right)$}		
	 \nccircle[angleA=-25,nodesep=3pt]{->}{u1}{.3cm}\Bput[0.05]{$\left(\!\begin{smallmatrix}\text{-}aC&ab\\\text{-}ab^{2m}&aC\end{smallmatrix}\!\right)$}	
	 \nccircle[angleA=-25,nodesep=3pt]{->}{u2}{.3cm}\Bput[-0.075]{$\left(\!\begin{smallmatrix}\text{-}aC&ab^2\\\text{-}ab^{2m\text{-}1}&aC\end{smallmatrix}\!\right)$}	
	 \nccircle[nodesep=3pt]{->}{u5}{.3cm}\Bput[0.05]{$\left(\!\begin{smallmatrix}\text{-}aC&ab^{m}\\\text{-}ab^{m+1}&aC\end{smallmatrix}\!\right)$}		
	 \nccircle[linecolor=red,angleA=180,nodesep=3pt]{->}{u6}{.3cm}\Bput[0.05]{$\left(\begin{smallmatrix}0&1\\b&0\end{smallmatrix}\right)$}	
	}}

\rput(5.25,0){
\rput(3,1.85){$U_2\cong\C^3_{b,B,u}$}
\rput(4,-1){\tiny $A=b^2-u^2(uB)^{2m-1}$}
\scalebox{0.8}{
	\rput(-0,1.7){\rnode{0}{$0$}}
	\rput(-0,-1.7){\rnode{1}{$0'$}}
	\rput(2,0){\rnode{2}{$1$}}
	\rput(4,0){\rnode{3}{$2$}}
	\rput(6,0){\rnode{5}{\footnotesize$m\!-\!1$}}
	\rput(8,0){\rnode{6}{$m$}}
	\rput(2,0.3){\rnode{u1}{}}	
	\rput(4,0.3){\rnode{u2}{}}	
	\rput(6,0.3){\rnode{u3}{}}	
	\rput(8,0.3){\rnode{u5}{}}	
	\rput(8,-0.3){\rnode{u6}{}}	
	\ncarc[linecolor=red]{->}{0}{1}\Aput[0.05]{\footnotesize $1$}		
	\ncarc{->}{1}{0}\Aput[0.05]{\footnotesize $A$}		
	\ncarc{->}{1}{2}\mput*[0.05]{\scriptsize $(b,u)$}		
	\ncarc{->}{2}{1}\Aput[0.05]{\large $\left(\begin{smallmatrix}0\\B\end{smallmatrix}\right)$}		
	\ncarc[linecolor=red]{->}{0}{2}\Aput[0.05]{\scriptsize $(1,0)$}		
	\ncarc{->}{2}{0}\mput*[0.5]{\large $\left(\!\begin{smallmatrix}uB\\-bB\end{smallmatrix}\!\right)$}	
	\ncarc[linecolor=red]{->}{2}{3}\Aput[0.05]{$\left(\begin{smallmatrix}1&0\\0&uB\end{smallmatrix}\right)$}	
	\ncarc[linecolor=red]{->}{3}{2}\Aput[0.05]{$\left(\begin{smallmatrix}uB&0\\0&1\end{smallmatrix}\right)$}	
	\ncline[linestyle=dashed]{-}{3}{5}
	\ncarc[linecolor=red]{->}{5}{6}\Aput[0.05]{$\left(\begin{smallmatrix}1&0\\0&uB\end{smallmatrix}\right)$}	
	\ncarc[linecolor=red]{->}{6}{5}\Aput[0.05]{$\left(\begin{smallmatrix}uB&0\\0&1\end{smallmatrix}\right)$}	
	 \nccircle[angleA=-25,nodesep=3pt]{->}{u1}{.3cm}\Bput[-0.05]{$\left(\!\begin{smallmatrix}b&u\\\text{-}u(\!uB\!)^{2m\text{-}1}&\text{-}b\end{smallmatrix}\!\right)$}		
	 \nccircle[angleA=-35,nodesep=3pt]{->}{u2}{.3cm}\Bput[0]{$\left(\!\begin{smallmatrix}b&u(\!uB\!)\\\text{-}u(\!uB\!)^{2m\text{-}2}&\text{-}b\end{smallmatrix}\!\right)$}		
	 \nccircle[angleA=-10,nodesep=3pt]{->}{u5}{.3cm}\Bput[-0.1]{$\left(\!\begin{smallmatrix}b&u(\!uB\!)^{m\text{-}1}\\\text{-}u(\!uB\!)^{m}&\text{-}b\end{smallmatrix}\!\right)$}		
	 \nccircle[linecolor=red,angleA=180,nodesep=3pt]{->}{u6}{.3cm}\Bput[0.05]{$\left(\begin{smallmatrix}0&1\\uB&0\end{smallmatrix}\right)$}	
	}}
\end{pspicture}
\caption{Open sets $U_1'$ and $U_2$ in $\Hilb{D_{2n}}{\C^3}$.}
\label{DnoddU1'}
\end{figure}

It follows that $X=\Hilb{G}{\C^3}$ is covered by the union of $m+2$ open sets isomorphic each of them to $\C^3$.
\end{itemize}

{\em Step 2.} For every $\rho\in\Irr G$ let $S_\rho:=(S\otimes\rho)^G$ be the CM $S^G$-modules. Writing down irreducible maps between these modules we obtain the McKay quiver shown in Figure \ref{McKayCMDodd}.

\begin{figure}[h]
\begin{pspicture}(0,-1.5)(10,1.75)
	\psset{arcangle=15,nodesep=2pt}
\rput(0,0){
\scalebox{0.8}{
	\rput(0,1.7){\rnode{0}{$S_{0}$}}
	\rput(0,-1.7){\rnode{1}{$S_{0'}$}}
	\rput(2.75,0){\rnode{2}{$S_1$}}
	\rput(5.5,0){\rnode{3}{$S_2$}}
	\rput(8.25,0){\rnode{5}{$S_{m\!-\!1}$}}
	\rput(11,0){\rnode{6}{$S_m$}}
	\rput(2.75,0.3){\rnode{u1}{}}	
	\rput(5.5,0.3){\rnode{u2}{}}	
	\rput(8.25,0.3){\rnode{u3}{}}	
	\rput(11,0.3){\rnode{u5}{}}	
	\rput(11,-0.3){\rnode{u6}{}}	
	\ncarc{->}{0}{1}\Aput[0.05]{$z$}		
	\ncarc{->}{1}{0}\Aput[0.05]{$z$}		
	\ncarc{->}{1}{2}\mput*[0.05]{\small $(x,\text{-}y)$}		
	\ncarc{->}{2}{1}\Aput[0.05]{\Large $\left(\begin{smallmatrix}y\\\text{-}x\end{smallmatrix}\right)$}		
	\ncarc{->}{0}{2}\aput[0](0.4){\small $(x,y)$}		
	\ncarc{->}{2}{0}\mput*[0.05]{\Large $\left(\begin{smallmatrix}y\\x\end{smallmatrix}\right)$}	
	\ncarc{->}{2}{3}\Aput[0.05]{\Large $\left(\begin{smallmatrix}x&0\\0&y\end{smallmatrix}\right)$}		
	\ncarc{->}{3}{2}\Aput[0.05]{\Large $\left(\begin{smallmatrix}y&0\\0&x\end{smallmatrix}\right)$}		
	\ncline[nodesep=4pt,linestyle=dashed]{-}{3}{5}
	\ncarc{->}{5}{6}\Aput[0.05]{\Large $\left(\begin{smallmatrix}x&0\\0&y\end{smallmatrix}\right)$}		
	\ncarc{->}{6}{5}\Aput[0.05]{\Large $\left(\begin{smallmatrix}y&0\\0&x\end{smallmatrix}\right)$}		
	 \nccircle[angleA=-20,nodesep=3pt]{->}{u1}{.3cm}\Bput[0.05]{\Large $\left(\!\begin{smallmatrix}z&0\\0&-z\end{smallmatrix}\!\right)$}	
	 \nccircle[angleA=-30,nodesep=3pt]{->}{u2}{.3cm}\Bput[0.05]{\Large $\left(\!\begin{smallmatrix}z&0\\0&-z\end{smallmatrix}\!\right)$}	
	 \nccircle[angleA=-25,nodesep=3pt]{->}{u5}{.3cm}\Bput[0.05]{\Large $\left(\!\begin{smallmatrix}z&0\\0&-z\end{smallmatrix}\!\right)$}		
	 \nccircle[angleA=200,nodesep=3pt]{->}{u6}{.3cm}\Bput[0.05]{\Large $\left(\begin{smallmatrix}0&x\\y&0\end{smallmatrix}\right)$}	
	}}
\end{pspicture}
\caption{Quiver of CM-modules $S_\rho$ for $G=D_{2n}$, $n$ odd.}
\label{McKayCMDodd}
\end{figure}
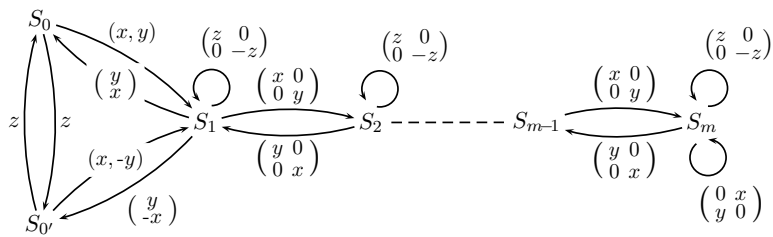

Every point in an open set $U\subset X$ is a representation of $Q$ of dimension vector $d=\begin{smallmatrix}1\\1\end{smallmatrix}\text{\small 2\ldots2}$ generated by a subset of linearly independent {\em distinguished arrows}. We choose basis elements for the vector spaces at every vertex of $Q$ by following the distinguished arrows of $U$ in the quiver between the CM $S^G$-modules shown above. For instance, in $U'_1$ with coordinates $a$, $b$ and $C$, the distinguished arrows are $c$, $d_1,\ldots,d_{m-1}$, $v$, $D_{m-1},\ldots,D_{1}$ and $B$, so we choose the following basis elements:

{\renewcommand{\arraystretch}{1.15}
\[
\begin{array}{|c|c||c|c|}
\hline
V_0 	& 1 			& V_2 & (x^2,y^2)=e_1 \\
V_{0'} & x^7-y^7 	& 	& (y^5,x^5)=e_2 \\
V_1 & (x,y)=e_1 	& V_3 & (x^3,y^3)=e_1 \\
	& (y^6,x^6)=e_2 &  & (y^4,x^4)=e_2 \\
\hline
\end{array}
\]}

For instance, this choice of basis elements implies that taking the paths in the quiver $a$, $cC$ and $cd_1\cdots d_{m-1}vD_{m-1}\cdots D_1C$ we obtain the identities $z=a(x^7-y^7)$, $2xy=b\cdot1$ and $x^7+y^7=C\cdot1$, which by rescaling the coefficients gives us the coordinates of the open set $U'_1$: 
\[
a=z/(x^7-y^7), b=xy, C=x^7+y^7.
\] 
The same method gives $U_2\cong\C^3_{b,B,u}$ with $b=xyz/f_2$, $B=f_2/z$ and $u=zf_1/f_2$. The gluing between $U'_1$ and $U_2$ is given by $(a,b,C)\mapsto(ab,a_1,aC)$ thus the exceptional $\mathbb{P}^1_{(z:x^7-y^7)}$ has a normal bundle of degree $(-1,-1)$. In other words, knowing the representation space of an open set $U\subset\mathcal{M}_{\theta^0}$ we can recover the local coordinates for $U$, and the gluing between open sets determines the degrees of the normal bundles of the exceptional curves. 

It follows that in $X$ we have $\pi^{-1}(0)=\bigcup_{i=0}^mE_i$ where $E_i\cong\mathbb{P}^1$ with coordinates $(xy)^iz:x^{2m+1}-y^{2m+1}$ for $i=0,\ldots,m$, intersecting pairwise according to the dual graph shown below (see also \cite{GNS2} \S3.5). 

\begin{center}
\begin{pspicture}(0,0)(5,0.35)
	\psset{arcangle=15,nodesep=2pt}
\rput(0,0){\rnode{0}{$\bullet$}}\rput(0,0.35){$E_0$}
\rput(1,0){\rnode{1}{$\bullet$}}\rput(1,0.35){$E_1$}
\rput(2,0){\rnode{2}{$\cdots$}}
\rput(3,0){\rnode{3}{$\bullet$}}\rput(3,0.35){$E_{m-1}$}
\rput(4,0){\rnode{4}{$\bullet$}}\rput(4,0.35){$E_m$}
\ncline{-}{0}{1}\ncline{-}{1}{2}\ncline{-}{2}{3}\ncline{-}{3}{4}
\end{pspicture}
\end{center}

Therefore $E_0$ is a $(-1,-1)$-curve, $E_m$ is a $(-3,1)$-curve, and the rest are $(-2,0)$-curves. By Lemma \ref{floppable} only $E_0$ can be floppable, which gives us $X_0$. 

{\em Step 3.} Recall that any $(-1,-1)$-curve $E\subset X$ is floppable. In fact, there exists $F\cong\PP^1\times\PP^1\subset Y$ fitting in the following diagram

\begin{center}
\begin{pspicture}(0,0)(4,2)
	\psset{arcangle=15,nodesep=2pt}
\rput(0,0){\rnode{p0}{$E\subset X$}}
\rput(3,0){\rnode{p1}{$X'\supset E'$}}
\rput(1.5,1.5){\rnode{p2}{$F\subset Y$}}
\ncline[linestyle=dashed]{->}{p0}{p1}
\ncline{->}{p2}{p0}\Bput[0.05]{$\sigma$}
\ncline{->}{p2}{p1}\Aput[0.05]{$\sigma'$}
\end{pspicture}
\end{center}

\noindent where $F$ dominates $E$ and $E'$ in two different ways as the exceptional locus for $\sigma$ and $\sigma'$ (see for instance \cite{Pagoda} \S5.5 and references therein). In our case, the curve $E_0\subset X$ is covered by $U'_1$ and $U_2$, and to calculate the new open sets in $U''_1,U'_2\subset X_0$ covering the flopped curve $E'_0$ we look at the following representation space:

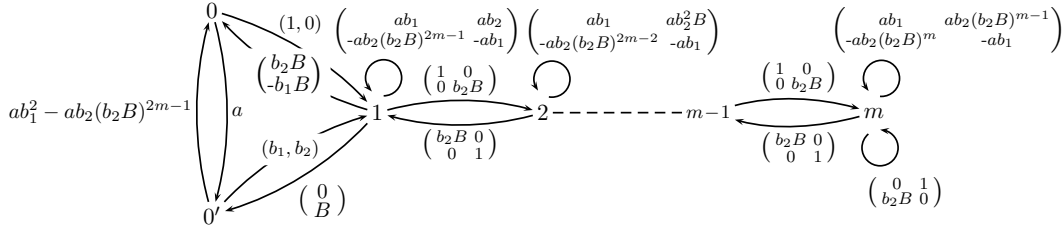
\begin{center}
\begin{pspicture}(0,-1.25)(10,1.5)
	\psset{arcangle=15,nodesep=2pt}
\rput(0,0){
\scalebox{0.8}{
	\rput(0,1.7){\rnode{0}{$0$}}
	\rput(0,-1.7){\rnode{1}{$0'$}}
	\rput(2.75,0){\rnode{2}{$1$}}
	\rput(5.5,0){\rnode{3}{$2$}}
	\rput(8.25,0){\rnode{5}{\footnotesize$m\!-\!1$}}
	\rput(11,0){\rnode{6}{$m$}}
	\rput(2.75,0.3){\rnode{u1}{}}	
	\rput(5.5,0.3){\rnode{u2}{}}	
	\rput(8.25,0.3){\rnode{u3}{}}	
	\rput(11,0.3){\rnode{u5}{}}	
	\rput(11,-0.3){\rnode{u6}{}}	
	\ncarc{->}{0}{1}\Aput[0.05]{\small $a$}		
	\ncarc{->}{1}{0}\Aput[0.05]{\small $ab_1^2-ab_2(b_2B)^{2m-1}$}		
	\ncarc{->}{1}{2}\mput*[0.05]{\footnotesize $(b_1,b_2)$}		
	\ncarc{->}{2}{1}\Aput[0.05]{\Large $\left(\begin{smallmatrix}0\\B\end{smallmatrix}\right)$}		
	\ncarc{->}{0}{2}\aput[0](0.4){\footnotesize $(1,0)$}		
	\ncarc{->}{2}{0}\mput*[0.05]{\Large $\left(\!\begin{smallmatrix}b_2B\\\text{-}b_1B\end{smallmatrix}\!\right)$}	
	\ncarc{->}{2}{3}\Aput[0.05]{\large $\left(\begin{smallmatrix}1&0\\0&b_2B\end{smallmatrix}\right)$}		
	\ncarc{->}{3}{2}\Aput[0.05]{\large $\left(\begin{smallmatrix}b_2B&0\\0&1\end{smallmatrix}\right)$}		
	\ncline[linestyle=dashed]{-}{3}{5}
	\ncarc{->}{5}{6}\Aput[0.05]{\large $\left(\begin{smallmatrix}1&0\\0&b_2B\end{smallmatrix}\right)$}		
	\ncarc{->}{6}{5}\Aput[0.05]{\large $\left(\begin{smallmatrix}b_2B&0\\0&1\end{smallmatrix}\right)$}		
	 \nccircle[angleA=-20,nodesep=3pt]{->}{u1}{.3cm}\Bput[-0.25]{\large $\left(\!\begin{smallmatrix}ab_1&ab_2\\\text{-}ab_2(b_2B)^{2m-1}&\text{-}ab_1\end{smallmatrix}\!\right)$}	
	 \nccircle[angleA=-30,nodesep=3pt]{->}{u2}{.3cm}\Bput[-0.2]{\large $\left(\!\begin{smallmatrix}ab_1&ab_2^2B\\\text{-}ab_2(b_2B)^{2m-2}&\text{-}ab_1\end{smallmatrix}\!\right)$}	
	 \nccircle[angleA=-25,nodesep=3pt]{->}{u5}{.3cm}\Bput[-0.3]{\large $\left(\!\begin{smallmatrix}ab_1&ab_2(b_2B)^{m-1}\\\text{-}ab_2(b_2B)^{m}&\text{-}ab_1\end{smallmatrix}\!\right)$}		
	 \nccircle[angleA=200,nodesep=3pt]{->}{u6}{.3cm}\Bput[0.05]{\large $\left(\begin{smallmatrix}0&1\\b_2B&0\end{smallmatrix}\right)$}	
	}}
\end{pspicture}
\end{center}

It is the representation space obtained by taking the distinguished nonzero arrows which are common in $U'_1$ and $U_2$. This space dominates both sides of the flop, in the sense that if we set $B=1$ we get $U'_1$ and if we set $a=1$ we get $U_2$, covering $E_0$; if $b_1=1$ we get $U''_1$ and if $b_2=1$ we get $U'_2$, covering the flop of $E_0$. The coordinates of $F\cong\PP^1\times\PP^1$ are $(a:B;b_1:b_2)$.

{\em Step 4.} For the rest of crepant resolutions $X_{0\ldots i}$ we argue in the same way, i.e.\ we repeat this process in every $(-1,-1)$-curve to produce every open set $U\subset\mathcal{M}_{\theta}$. We show in Figure \ref{DnoddUi'Ui''} the representation spaces for the open sets $U_i'$ and $U_i''$, the rest are done similarly.  

\begin{figure}[h]
\begin{center}
\begin{pspicture}(-1,0)(15,7)
	\psset{arcangle=15,nodesep=2pt}


\rput(1,4.75){
\rput(-1,2){$U'_i\cong\C^3_{a,d,D}$}
\rput(-1,1.5){\tiny $3\leq i\leq m\!+\!1$}
\rput(5.5,-1){\tiny $A=ad^2-aD^{2\alpha+3}$}
\scalebox{0.8}{
	\rput(-0,1.7){\rnode{0}{$0$}}
	\rput(-0,-1.7){\rnode{0'}{$0'$}}
	\rput(2,0){\rnode{1}{$1$}}
	\rput(4,0){\rnode{2}{$2$}}
	\rput(6,0){\rnode{i-3}{\footnotesize $i\!-\!3$}}
	\rput(8,0){\rnode{i-2}{\footnotesize $i\!-\!2$}}
	\rput(10,0){\rnode{i-1}{\footnotesize $i\!-\!1$}}
	\rput(12,0){\rnode{i}{\footnotesize $~i~$}}
	\rput(14,0){\rnode{m-1}{\footnotesize$m\!-\!1$}}
	\rput(16,0){\rnode{m}{$~m~$}}
	\rput(2,0.3){\rnode{u1}{}}	
	\rput(6,0.3){\rnode{ui-3}{}}	
	\rput(8,0.3){\rnode{ui-2}{}}	
	\rput(10,0.3){\rnode{ui-1}{}}	
	\rput(12,0.3){\rnode{ui}{}}	
	\rput(14,0.3){\rnode{um-1}{}}	
	\rput(16,0.3){\rnode{um}{}}	
	\rput(16,-0.3){\rnode{um'}{}}	
	\ncarc{->}{0}{0'}\Aput[0.05]{\footnotesize $a$}		
	\ncarc{->}{0'}{0}\Aput[0.05]{\footnotesize $A$}		
	\ncarc[linecolor=red]{->}{0'}{1}\mput*[0.05]{\scriptsize $(0,1)$}		
	\ncarc{->}{1}{0'}\Aput[0.05]{\large $\left(\begin{smallmatrix}0\\D\end{smallmatrix}\right)$}		
	\ncarc[linecolor=red]{->}{0}{1}\Aput[0.05]{\scriptsize $(1,0)$}		
	\ncarc{->}{1}{0}\mput*[0.05]{\large $\left(\!\begin{smallmatrix}D\\0\end{smallmatrix}\!\right)$}	
	\ncline[linestyle=dashed]{-}{2}{i-3}
	\ncarc[linecolor=red]{->}{1}{2}\Aput[0.05]{$\left(\begin{smallmatrix}1&0\\0&1\end{smallmatrix}\right)$}		
	\ncarc{->}{2}{1}\Aput[0.05]{$\left(\begin{smallmatrix}D&0\\0&D\end{smallmatrix}\right)$}		
	\ncarc[linecolor=red]{->}{i-3}{i-2}\Aput[0.05]{$\left(\begin{smallmatrix}1&0\\0&1\end{smallmatrix}\right)$}	
	\ncarc{->}{i-2}{i-3}\Aput[0.05]{$\left(\begin{smallmatrix}D&0\\0&D\end{smallmatrix}\right)$}	
	\ncarc[linecolor=red]{->}{i-2}{i-1}\Aput[0.05]{$\left(\begin{smallmatrix}1&0\\d&1\end{smallmatrix}\right)$}	
	\ncarc{->}{i-1}{i-2}\Aput[0.05]{$\left(\begin{smallmatrix}D&0\\\text{-}dD&D\end{smallmatrix}\right)$}	
	\ncarc[linecolor=red]{->}{i-1}{i}\Aput[0.05]{$\left(\begin{smallmatrix}1&0\\0&D\end{smallmatrix}\right)$}		
	\ncarc[linecolor=red]{->}{i}{i-1}\Aput[0.05]{$\left(\begin{smallmatrix}D&0\\0&1\end{smallmatrix}\right)$}		
	\ncline[linestyle=dashed]{-}{i}{m-1}
	\ncarc[linecolor=red]{->}{m-1}{m}\Aput[0.05]{$\left(\begin{smallmatrix}1&0\\0&D\end{smallmatrix}\right)$}		
	\ncarc[linecolor=red]{->}{m}{m-1}\Aput[0.05]{$\left(\begin{smallmatrix}D&0\\0&1\end{smallmatrix}\right)$}		
	\nccircle[angleA=-15,nodesep=3pt]{->}{u1}{.3cm}\Bput[0.05]{$\left(\begin{smallmatrix}0&a\\A&0\end{smallmatrix}\right)$}	
	 \nccircle[nodesep=3pt]{->}{ui-2}{.3cm}\Bput[0.05]{$\left(\begin{smallmatrix}0&a\\A&0\end{smallmatrix}\right)$}	
 	 \nccircle[nodesep=3pt]{->}{ui-1}{.3cm}\Bput[0.05]{$\left(\begin{smallmatrix}ad&a\\\text{-}aD^{2\alpha+3}&\text{-}ad\end{smallmatrix}\right)$}	
	 \nccircle[angleA=-15,nodesep=3pt]{->}{ui}{.3cm}\Bput[-0.1]{$\left(\begin{smallmatrix}ad&aD\\\text{-}aD^{2\alpha+2}&\text{-}ad\end{smallmatrix}\right)$}		
	 \nccircle[angleA=0,nodesep=3pt]{->}{um}{.3cm}\Bput[0.05]{$\left(\begin{smallmatrix}ad&aD^\alpha\\\text{-}aD^{\alpha+2}&\text{-}ad\end{smallmatrix}\right)$}		
	 \nccircle[linecolor=red,angleA=180,nodesep=3pt]{->}{um'}{.3cm}\Bput[0.05]{$\left(\begin{smallmatrix}0&1\\D&0\end{smallmatrix}\right)$}		
	}}


\rput(1,1.25){
\rput(-1,1.75){$U''_i\cong\C^3_{a,d,D}$}
\rput(-1.25,1.25){\tiny $2\leq i\leq m$}
\rput(5.5,-1){\tiny $A=a+ad^2(dD)^{2\alpha+1}$}
\scalebox{0.8}{
	\rput(-0,1.7){\rnode{0}{$0$}}
	\rput(-0,-1.7){\rnode{0'}{$0'$}}
	\rput(2,0){\rnode{1}{$1$}}
	\rput(4,0){\rnode{2}{$2$}}
	\rput(6,0){\rnode{i-2}{\footnotesize $i\!-\!2$}}
	\rput(8,0){\rnode{i-1}{\footnotesize $i\!-\!1$}}
	\rput(10,0){\rnode{i}{\footnotesize $~i~$}}
	\rput(12,0){\rnode{i+1}{\footnotesize $i\!+\!1$}}
	\rput(14,0){\rnode{m-1}{\footnotesize$m\!-\!1$}}
	\rput(16,0){\rnode{m}{$~m~$}}
	\rput(2,0.3){\rnode{u1}{}}	
	\rput(6,0.3){\rnode{ui-2}{}}	
	\rput(8,0.3){\rnode{ui-1}{}}	
	\rput(10,0.3){\rnode{ui}{}}	
	\rput(12,0.3){\rnode{ui+1}{}}	
	\rput(14,0.3){\rnode{um-1}{}}	
	\rput(16,0.3){\rnode{um}{}}	
	\rput(16,-0.3){\rnode{um'}{}}	
	\ncarc{->}{0}{0'}\Aput[0.05]{\footnotesize $a$}		
	\ncarc{->}{0'}{0}\Aput[0.05]{\footnotesize $A$}		
	\ncarc[linecolor=red]{->}{0'}{1}\mput*[0.05]{\scriptsize $(0,1)$}		
	\ncarc{->}{1}{0'}\Aput[0.05]{\large $\left(\begin{smallmatrix}0\\dD\end{smallmatrix}\right)$}		
	\ncarc[linecolor=red]{->}{0}{1}\Aput[0.05]{\scriptsize $(1,0)$}		
	\ncarc{->}{1}{0}\mput*[0.05]{\large $\left(\!\begin{smallmatrix}0\\dD\end{smallmatrix}\!\right)$}	
	\ncline[linestyle=dashed]{-}{2}{i-2}
	\ncarc[linecolor=red]{->}{1}{2}\Aput[0.05]{$\left(\begin{smallmatrix}1&0\\0&1\end{smallmatrix}\right)$}		
	\ncarc{->}{2}{1}\Aput[0.05]{$\left(\begin{smallmatrix}dD&0\\0&dD\end{smallmatrix}\right)$}		
	\ncarc[linecolor=red]{->}{i-2}{i-1}\Aput[0.05]{$\left(\begin{smallmatrix}1&0\\0&1\end{smallmatrix}\right)$}	
	\ncarc{->}{i-1}{i-2}\Aput[0.05]{$\left(\begin{smallmatrix}dD&0\\0&dD\end{smallmatrix}\right)$}	
	\ncarc[linecolor=red]{->}{i-1}{i}\Aput[0.05]{$\left(\begin{smallmatrix}1&0\\1&d\end{smallmatrix}\right)$}	
	\ncarc{->}{i}{i-1}\Aput[0.05]{$\left(\begin{smallmatrix}dD&0\\\text{-}D&D\end{smallmatrix}\right)$}	
	\ncarc[linecolor=red]{->}{i}{i+1}\Aput[0.05]{$\left(\begin{smallmatrix}1&0\\0&dD\end{smallmatrix}\right)$}		
	\ncarc[linecolor=red]{->}{i+1}{i}\Aput[0.05]{$\left(\begin{smallmatrix}dD&0\\0&1\end{smallmatrix}\right)$}		
	\ncline[linestyle=dashed]{-}{i+1}{m-1}
	\ncarc[linecolor=red]{->}{m-1}{m}\Aput[0.05]{$\left(\begin{smallmatrix}1&0\\0&dD\end{smallmatrix}\right)$}		
	\ncarc[linecolor=red]{->}{m}{m-1}\Aput[0.05]{$\left(\begin{smallmatrix}dD&0\\0&1\end{smallmatrix}\right)$}		
	\nccircle[angleA=-15,nodesep=3pt]{->}{u1}{.3cm}\Bput[0.05]{$\left(\begin{smallmatrix}0&a\\A&0\end{smallmatrix}\right)$}	
	 \nccircle[nodesep=3pt]{->}{ui-1}{.3cm}\Bput[0.05]{$\left(\begin{smallmatrix}0&a\\A&0\end{smallmatrix}\right)$}	
 	 \nccircle[angleA=5,nodesep=3pt]{->}{ui}{.3cm}\Bput[-0.05]{$\left(\begin{smallmatrix}a&ad\\\text{-}ad(dD)^{2\alpha+1}&\text{-}a\end{smallmatrix}\right)$}	
	 \nccircle[angleA=-15,nodesep=3pt]{->}{ui+1}{.3cm}\Bput[-0.15]{$\left(\begin{smallmatrix}a&ad(aD)\\\text{-}ad(dD)^{2\alpha}&\text{-}a\end{smallmatrix}\right)$}		
	 \nccircle[angleA=0,nodesep=3pt]{->}{um}{.3cm}\Bput[0.05]{$\left(\begin{smallmatrix}a&ad(dD)^\alpha\\\text{-}ad(dD)^{\alpha+1}&\text{-}a\end{smallmatrix}\right)$}		
	 \nccircle[linecolor=red,angleA=180,nodesep=3pt]{->}{um'}{.3cm}\Bput[0.05]{$\left(\begin{smallmatrix}0&1\\dD&0\end{smallmatrix}\right)$}		
	}}
\end{pspicture}
\caption{Open sets $U'_i$ and $U''_i$ contained in crepant resolutions of $\C^3/{D_{2n}}$ with $n$ odd.}
\label{DnoddUi'Ui''} 
\end{center}
\end{figure}
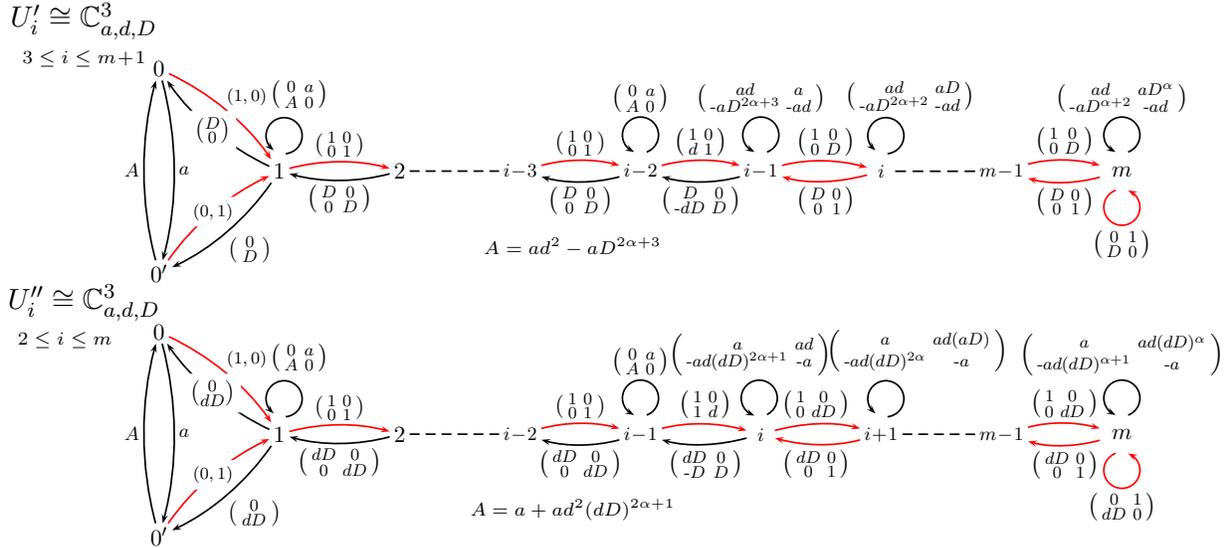

In each of $X_{0\ldots i}$ with $0<i<m$ there are precisely two $(-1,-1)$-curves, whose flops give $X_{0\ldots(i-1)}$ and $X_{0\ldots(i+1)}$ respectively. In the case $i=m$, the 3-fold $X_{0\ldots m}$ has only $E_{m-1}$ as a $(-1,-1)$-curve and the flop produces the previous $X_{0\ldots m-1}$, so we stop. This proves (3).

The gluings between the open sets shown below, from which we can read off the degrees of the normal bundles in (2).
\[
\begin{small}
\begin{array}{rl}
U'_1\to U_2: & (a,b,C)\mapsto(ab,a^{-1},aC) \\
U_i\to U_{i+1}: & (d,D,u)\mapsto(d^2D,d^{-1},u), \text{ for $i=2,\ldots,m+1$} \\
U_{m+1}\to U_{m+2}: & (d,D,u)\mapsto(u^2-4d^3D,d^{-1}u,d^{-1}) \\
U'_i\to U_{i+1}: & (a,d,D)\mapsto(aD,a^{-1},ad), \text{ for $i=2,\ldots,m$} \\
U'_{m+1}\to U_{m+2}: & (a,d,D)\mapsto(a^2d^2-4a^2D,d,a^{-1}) \\
U''_i\to U'_{i+1}: & (a,d,D)\mapsto(ad,d^{-1},dD), \text{ for $i=1,\ldots,m-1$}  \\
U''_i\to U''_{i+1}: & (a,d,D)\mapsto(a,d^2D,d^{-1}), \text{ for $i=1,\ldots,m$}
\end{array}
\end{small}
\]
In every crepant resolution the dual graph of the fibre over the origin $0\in\C^3/G$ is of type $A$, they only differ in the degrees of the normal bundles over $E_i$ given in $(2)$.
\end{proof}


\subsection{The dihedral group of order $2n$ ($n$ even)}

With the potential given in Section \ref{QP-D-even} the list of relations in this case is the following:
\[
\begin{small}
\begin{array}{rl}
\partial a, \partial A, \partial b, \partial B, \partial c, \partial C: & bC=0, cB =0, Ca=u_1B, Ac=bu_1, BA=u_1C, ab=cu_1. \\
\partial a', \partial A', \partial b', \partial B', \partial c', \partial C': & b'C'=0, c'B' =0, C'a'=u_{m-1}B', \\
& A'c'=b'u_{m-1}, B'A'=u_{m-1}C', a'b'=c'u_{m-1}. \\
\partial d_1,\ldots , \partial d_{m-2}: & D_1u_1=u_2D_1, \ldots, D_{m-2}u_{m-2}=u_{m-1}D_{m-2}. \\
\partial D_1, \ldots,\partial D_{m-2}: & u_1d_1=d_1u_2, \ldots, u_{m-2}d_{m-2}=d_{m-2}u_{m-1}. \\ 
\partial u_1: & Bb+Cc=d_1D_1. \\
\partial u_2, \ldots, \partial u_{m-2}: & d_2D_2=D_1d_1, \ldots, d_{m-2}D_{m-2}=D_{m-3}d_{m-3}. \\
\partial u_m : & B'b'+C'c'=D_{m-2}d_{m-2}. 
\end{array}
\end{small}
\]
We consider the same notation as in \ref{sect:Dnodd} for the arrows of $Q$ as linear maps, adding in this case $a':=a'$, $A':=A'$, $b':=(b'_1,b'_2)$, $B':=\left(\begin{smallmatrix}B'_1\\B'_2\end{smallmatrix}\right)$, $c':=(c'_1,c'_2)$ and $C':=\left(\begin{smallmatrix}C'_1\\C'_2\end{smallmatrix}\right)$.

\begin{thm}\label{OpensDnEven} Let $G=D_{2n}\subset\SO(3)$ with $n$ even. Let $f_1:=x^{m}+y^{m}$ and $f_2:=x^{m}-y^{m}$. Then,
\begin{enumerate}
\item There are $(m+1)(m+2)/2$ non-isomorphic crepant resolutions $\pi_{ij}:X_{0\ldots i}^{m\ldots(m-j)} \to\C^3/G$ given by the open covers:  
\[
X_{0\ldots i}^{m\ldots(m-j)} = \bigcup_{k=1}^{i+1}U''_k\cup U'_{i+2}\cup\bigcup_{k=i+3}^{m-j}U_{k}\cup V'_{m-j+1}\cup\bigcup_{k=m-j+2}^{m+3}V''_{k}
\]
for $-1\leq i,j\leq m-1$, where $U_k, U'_k, U''_k,V'_k,V''_k\cong\C^3$ for all $k$, with local coordinates 
{\renewcommand{\arraystretch}{1.5}
\[
\begin{array}{l|l}
\underset{(i\leq m)}{U_i\cong}\C^3_{d,D,u}=\Spec\C[\frac{(xy)^{i-1}z}{f_1f_2}, \frac{f_1f_2}{(xy)^{i-2}z}, \frac{zf_1}{f_2}] & \underset{i\leq m}{V'_i\cong}\C^3_{a',u,C'}=\Spec\C[\frac{(xy)^{i-2}z^2}{f_2^2}, \text{\scriptsize $xy$}, \frac{f_1f_2}{(xy)^{i-2}z}]   \\

U_{m+1}\cong\C^3_{C',c',B'}=\Spec\C[\frac{zf_2}{f_1}, \frac{f_1f_2}{(xy)^{m-1}z}, \frac{zf_1}{f_2}] & V'_{m+1}\cong\C^3_{a',b',C'}=\Spec\C[\frac{(xy)^{m-1}z^2}{f_2^2}, \frac{f_2^2}{(xy)^{m-1}}, \frac{f_1f_2}{(xy)^{m-1}z}] \\

\underset{(i\leq m)}{U'_i\cong}\C^3_{a,d,D}=\Spec\C[\frac{(xy)^{i-1}z}{f_1f_2}, \frac{f_1^2}{(xy)^{i-1}}, \text{\scriptsize $xy$}]  & V'_{m+2}\cong\C^3_{a',c',C'}=\Spec\C[\text{\scriptsize $z^2$}, \frac{f_2^2}{(xy)^{m-1}}, \frac{f_1}{zf_2}] \\

U'_{m+1}\cong\C^3_{a,B',c'}=\Spec\C[\frac{zf_2}{f_1}, \frac{f_1^2}{f_2^2}, \frac{f_2^2}{(xy)^{m-1}}] & \underset{i\leq m+1}{V''_i\cong}\C^3_{d,D,C'}=\Spec\C[\frac{(xy)^{i-2}z^2}{f_2^2}, \frac{f_2^2}{(xy)^{i-3}z^2}, \frac{zf_1}{f_2}] \\

\underset{(i\leq m-1)}{U''_i\cong}\C^3_{a,d,D}=\Spec\C[\frac{zf_1}{f_2}, \frac{x^iy^i}{f_1^2}, \frac{f_1^2}{(xy)^{i-1}}]  & V''_{m+2}\cong\C^3_{A,c',C'}=\Spec\C[\text{\scriptsize $z^2$}, \frac{f_2^2}{(xy)^{m-1}z^2}, \frac{zf_1}{f_2}] \\

U''_{m}\cong\C^3_{a,c',C'}=\Spec\C[\frac{zf_1}{f_2}, \frac{f_1^2}{(xy)^{m-1}}, \frac{f_2^2}{f_1^2}] & V''_{m+3}\cong\C^3_{A',b',B'}=\Spec\C[\text{\scriptsize $z^2$}, \frac{f_1^2}{(xy)^{m-1}}, \frac{f_2}{zf_1}] \\
\end{array}
\]
}
When $i=j=-1$ we have $X\cong\Hilb{G}{\C^3}$. 

\item The degrees of the normal bundles $\mathcal{N}_{X\slash E}$ of the exceptional rational curves $E\subset X_{0\ldots i}^{m\ldots(m-j)}$ are:
{\renewcommand{\arraystretch}{1.15}
\[
{\small
\begin{array}{|c|l||c|l|}
\hline
\text{Open cover of $E$} & \text{Degree of $\mathcal{N}_{X\slash E}$} & \text{Open cover of $E$} & \text{Degree of $\mathcal{N}_{X\slash E}$} \\
\hline
U_i\cup U_{i+1}	& \text{$(-2,0)$ for $2\leq i\leq m$} & U'_i\cup V'_{i+1} 	& \text{$(-2,0)$ for $1\leq i\leq m$} \\
U_i\cup V'_{i+1}	& \text{$(-1,-1)$ for $2\leq i\leq m+1$} & U''_i\cup U'_{i+1} 	& \text{$(-1,-1)$ for $1\leq i\leq m$} \\	
U_{m+1}\cup V''_{m+3}	& \text{$(-1,-1)$} & U''_i\cup U''_{i+1} 	& \text{$(-2,0)$ for $1\leq i\leq m-1$} \\
U'_i\cup U_{i+1} 	& \text{$(-1,-1)$ for $1\leq i\leq m$} & V'_i\cup V''_{i+1} 	& \text{$(-1,-1)$ for $2\leq i\leq m+1$} \\
				& \text{$(-2,0)$ for $i=m+1$} & V''_i\cup V''_{i+1} 	& \text{$(-2,0)$ for $3\leq i\leq m+2$} \\
\hline
\end{array}}
\]}

\item The dual graph of $\pi_{ij}^{-1}(0)$ is:

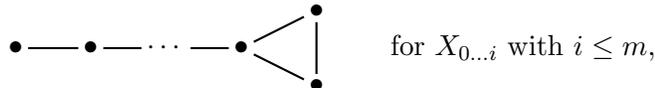
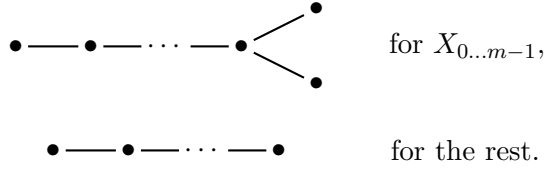
\begin{center}
\begin{pspicture}(0,-.75)(8,0.75)
	\psset{arcangle=15,nodesep=2pt}
\rput(0,0){\rnode{0}{$\bullet$}}
\rput(1,0){\rnode{1}{$\bullet$}}
\rput(2,0){\rnode{2}{$\cdots$}}
\rput(3,0){\rnode{3}{$\bullet$}}
\rput(4,0.5){\rnode{4}{$\bullet$}}
\rput(4,-0.5){\rnode{4'}{$\bullet$}}
\ncline{-}{0}{1}\ncline{-}{1}{2}\ncline{-}{2}{3}\ncline{-}{3}{4}\ncline{-}{3}{4'}\ncline{-}{4}{4'}
\rput(6.745,0){for $X_{0\ldots i}$ with $i\leq m$,}
\end{pspicture}
\end{center}
\begin{center}
\begin{pspicture}(0,-.5)(8,0.75)
	\psset{arcangle=15,nodesep=2pt}
\rput(0,0){\rnode{0}{$\bullet$}}
\rput(1,0){\rnode{1}{$\bullet$}}
\rput(2,0){\rnode{2}{$\cdots$}}
\rput(3,0){\rnode{3}{$\bullet$}}
\rput(4,0.5){\rnode{4}{$\bullet$}}
\rput(4,-0.5){\rnode{4'}{$\bullet$}}
\ncline{-}{0}{1}\ncline{-}{1}{2}\ncline{-}{2}{3}\ncline{-}{3}{4}\ncline{-}{3}{4'}
\rput(6,0){for $X_{0\ldots m-1}$,}
\end{pspicture}
\end{center}
\begin{center}
\begin{pspicture}(0,0)(7,0.75)
	\psset{arcangle=15,nodesep=2pt}
\rput(0,0){\rnode{0}{$\bullet$}}
\rput(1,0){\rnode{1}{$\bullet$}}
\rput(2,0){\rnode{2}{$\cdots$}}
\rput(3,0){\rnode{3}{$\bullet$}}
\ncline{-}{0}{1}\ncline{-}{1}{2}\ncline{-}{2}{3}
\rput(5.477,0){for the rest.}
\end{pspicture}
\end{center}
\end{enumerate}
\end{thm}

\begin{proof} {\em Step 1.} As in the proof of Theorem \ref{OpensDnOdd} we start by calculating explicitly $X:=\Hilb{G}{\C^3}\cong\mathcal{M}_{\theta^0}$ for the $0$-generated stability condition $\theta^0$. In this case the open cover is given by

\[\Hilb{G}{\C^3}\cong U'_1\cup\bigcup_{k=2}^{m+1}U_k\cup V'_{m+2}\cup V''_{m+3}\]

From Section \ref{Sect:McKayQP-SO(3)} we can see that the McKay quiver in this case only differs from the case when $n$ is odd in the vertices $m-1$, $m$ and $m'$, thus the argument is very similar to the proof of Step 1 in Theorem \ref{OpensDnOdd}. In particular, we can choose $c=(1,0)$, $d_i:=\left(\begin{smallmatrix}1&0\\d^i_{21}&d^i_{22}\end{smallmatrix}\right)$ and we cannot have a path $cd_i\cdots d_iD_i\cdots D_1B\neq0$ for any $i$. Therefore, we have three possibilities to reach the second linearly independent vector, which we may choose to be $(0,1)$, at the vector space at the vertex $m+1$. Namely $cd_i\cdots d_{m-2}B'b'=(0,1)$, $cd_i\cdots d_{m-2}C'c'=(0,1)$ or $abd_i\cdots d_{m-2}=(0,1)$. By symmetry the first two are equivalent, so we can assume that $C'_1=1$ and $c'=(0,1)$. In other words, we have that
\[
\begin{array}{rl}
\text{either} & cd_1d_2\cdots d_{m-2}C'c'=(0,1) \\
\text{or} & abd_i\cdots d_{m-2}=(0,1)
\end{array}
\]
Let us consider the first case. To reach the 1-dimensional vector space at the vertex $m'$ we can always choose $B'_1=1$. Indeed, by the relations $c'B'=0$ and $C'a'=u_{m-1}B'$ we obtain the equality $a'=u_{11}B'_1$. This means that if we choose $a'\neq0$ then $B'_1\neq0$, and we can change basis to consider $B'_1=1$ instead. 

Thus we obtain the following open sets
\[
\begin{array}{rl}
U'_1: & cd_1\cdots d_{m-2}C'=1,B'=1,c'D_{m-2}\cdots D_1B=1 \\
U_2: &  cd_1\cdots d_{m-2}C'=1,B'=1, c'D_{m-2}\cdots D_1=(0,1), a=1  \\
U_3: &  cd_1\cdots d_{m-2}C'=1,B'=1, c'D_{m-2}\cdots D_2=(0,1), ab=(0,1) \\
U_i: &  cd_1\cdots d_{m-2}C'=1,B'=1, c'D_{m-2}\cdots D_{i-2}=(0,1), abd_1\cdots d_{i-3}=(0,1) \text{ for $4\leq i<m+1$} 
\end{array}
\]

If $abd_i\cdots d_{m-2}=(0,1)$ then there are only two possibilities which satisfy the $0$-generated stability condition, namely $cd_1\cdots d_{m-2}B'A'=1$ and $cd_1\cdots d_{m-2}C'a'=1$, giving the open sets $V'_{m+2}$ and $V''_{m+3}$:
\[
\begin{array}{rl}
V'_{m+2}: & abd_i\cdots d_{m-2}=(0,1), cd_1\cdots d_{m-2}B'A'=1\\
V''_{m+3}: & abd_i\cdots d_{m-2}=(0,1), cd_1\cdots d_{m-2}C'a'=1
\end{array}
\]
Again, by using the relations in every case we conclude that $\Hilb{G}{\C^3}$ is covered by the union of $m+3$ open sets isomorphic to $\C^3$.

{\em Step 2.} As in the proof of \ref{OpensDnOdd}, the local coordinates in (1) are obtained using the quiver structure of the CM $S^G$-modules $S_\rho$, which in this case is given in Figure \ref{McKayCMDneven}. 

\begin{figure}[h]
\begin{center}
\begin{pspicture}(-1,-0.75)(10,2)
	\psset{arcangle=15,nodesep=2pt}
\rput(0,0.5){
\scalebox{0.85}{
	\rput(-0,1.7){\rnode{0}{$S_0$}}
	\rput(-0,-1.7){\rnode{0'}{$S_{0'}$}}
	\rput(2,0){\rnode{1}{$S_1$}}
	\rput(4,0){\rnode{2}{$S_2$}}
	\rput(6,0){\rnode{m-2}{\footnotesize$S_{m-2}$}}
	\rput(8,0){\rnode{m-1}{\footnotesize$S_{m-1}$}}
	\rput(10,1.7){\rnode{m}{$S_m$}}
	\rput(10,-1.7){\rnode{m'}{$S_{m'}$}}	
	\rput(2,0.3){\rnode{u1}{}}	
	\rput(4,0.3){\rnode{u2}{}}	
	\rput(8,0.3){\rnode{um-1}{}}	
	\ncarc{->}{0}{0'}\Aput[0.05]{\footnotesize $z$}		
	\ncarc{->}{0'}{0}\Aput[0.05]{\footnotesize $z$}		
	\ncarc{->}{0'}{1}\mput*[0.05]{\scriptsize $(x,\text{-}y)$}		
	\ncarc{->}{1}{0'}\Aput[0.05]{\large $\left(\begin{smallmatrix}y\\\text{-}x\end{smallmatrix}\right)$}	
	\ncarc{->}{0}{1}\Aput[0.05]{\scriptsize $(x,y)$}		
	\ncarc{->}{1}{0}\mput*[0.05]{\large $\left(\!\begin{smallmatrix}y\\x\end{smallmatrix}\!\right)$}	
	\ncline[linestyle=dashed]{-}{2}{m-2}
	\ncarc{->}{1}{2}\Aput[0.05]{$\left(\begin{smallmatrix}x&0\\0&y\end{smallmatrix}\right)$}	
	\ncarc{->}{2}{1}\Aput[0.05]{$\left(\begin{smallmatrix}y&0\\0&x\end{smallmatrix}\right)$}		
	\ncarc{->}{m-2}{m-1}\Aput[0.05]{$\left(\begin{smallmatrix}x&0\\0&y\end{smallmatrix}\right)$}	
	\ncarc{->}{m-1}{m-2}\Aput[0.05]{$\left(\begin{smallmatrix}y&0\\0&x\end{smallmatrix}\right)$}	
	\ncarc{->}{m-1}{m}\aput[0.05](0.65){$\left(\!\begin{smallmatrix}x\\y\end{smallmatrix}\!\right)$}		
	\ncarc{->}{m}{m-1}\mput*[0.05]{\scriptsize $(y,x)$}		
	\ncarc{->}{m-1}{m'}\mput*[0.05]{\large $\left(\begin{smallmatrix}x\\\text{-}y\end{smallmatrix}\right)$}	
	\ncarc{->}{m'}{m-1}\Aput[0.05]{\scriptsize $(y,\text{-}x)$}	
	\ncarc{<-}{m}{m'}\Aput[0.05]{\footnotesize $z$}	
	\ncarc{<-}{m'}{m}\Aput[0.05]{\footnotesize $z$}	
	\nccircle[angleA=-25,nodesep=3pt]{->}{u1}{.3cm}\Bput[0.025]{$\left(\begin{smallmatrix}z&0\\0&\text{-}z\end{smallmatrix}\right)$}		
	 \nccircle[angleA=0,nodesep=3pt]{->}{u2}{.3cm}\Bput[0.075]{$\left(\begin{smallmatrix}z&0\\0&\text{-}z\end{smallmatrix}\right)$}		
	 \nccircle[angleA=17.5,nodesep=3pt]{->}{um-1}{.3cm}\Bput[0.025]{$\left(\begin{smallmatrix}z&0\\0&\text{-}z\end{smallmatrix}\right)$}	
	}}
\end{pspicture}
\caption{Quiver of the CM-modules $S_\rho$ for $G=D_{2n}$, $n$ even.}
\label{McKayCMDneven}.
\end{center}
\end{figure}
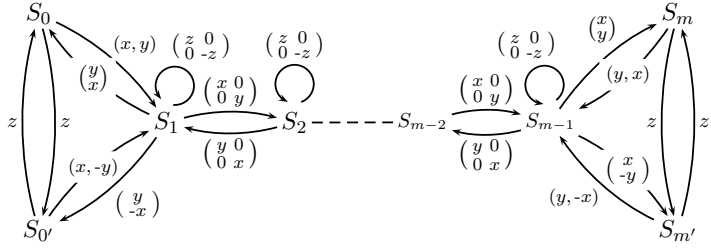

If we call $\pi:\Hilb{G}{\C^3}\to\C^3/G$, it follows that $\pi^{-1}(0)=\bigcup_{i=0}^{m+1}E_i$ where $E_i\cong\mathbb{P}^1$ intersect according to the following dual graph:

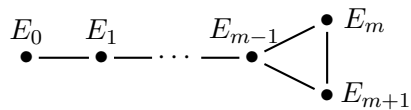
\begin{center}
\begin{pspicture}(0,-.5)(5,0.75)
	\psset{arcangle=15,nodesep=2pt}
\rput(0,0){\rnode{0}{$\bullet$}}\rput(0,0.35){$E_0$}
\rput(1,0){\rnode{1}{$\bullet$}}\rput(1,0.35){$E_1$}
\rput(2,0){\rnode{2}{$\cdots$}}
\rput(3,0){\rnode{3}{$\bullet$}}\rput(2.9,0.35){$E_{m-1}$}
\rput(4,0.5){\rnode{4}{$\bullet$}}\rput(4.5,0.5){$E_{m}$}
\rput(4,-0.5){\rnode{4'}{$\bullet$}}\rput(4.65,-0.5){$E_{m+1}$}
\ncline{-}{0}{1}\ncline{-}{1}{2}\ncline{-}{2}{3}\ncline{-}{3}{4}\ncline{-}{3}{4'}\ncline{-}{4}{4'}
\end{pspicture}
\end{center}

The curves $E_i$ have coordinates $(xy)^iz:x^{2m}-y^{2m}$ for $i<m$, $E_m$ has coordinates $x^m+y^m:z(x^m-y^m)$ and $E_{m'}$ has coordinates $x^m-y^m:z(x^m+y^m)$. The rational curves $E_0$, $E_m$ and $E_{m+1}$ are $(-1,-1)$-curves, and the rest of $E_i$'s are $(-2,0)$-curves. By Lemma \ref{floppable} only the flop of $E_0$, $E_m$ and $E_{m+1}$ gives us new crepant resolutions. As in Section \ref{SubSect:MutDnEven}, by the symmetry of the curves $E_m$ and $E_{m+1}$ it is enough to consider flops from $\Hilb{G}{\C^3}$ at $E_0$ and $E_m$. 

{\em Step 3} is analogous to the proof of Theorem \ref{OpensDnOdd}.

{\em Step 4.} Using the same method as in the proof of Theorem \ref{OpensDnOdd} we see that we can flop consecutively the curves $E_0,\ldots,E_{m-1}$ (in this order) to obtain the chain of flops \[
\Hilb{G}{\C^3}=X\dashrightarrow X_0\dashrightarrow\ldots\dashrightarrow X_{0\ldots m-1}
\] 
At every step we perform the flop as in the proof of Theorem \ref{OpensDnOdd}, producing each time two new open sets. The dual graph of the fibre over the origin of these crepant resolutions are the same as the dual graph for $\Hilb{G}{\C^3}$ except for $X_{0\ldots m-1}$ which is

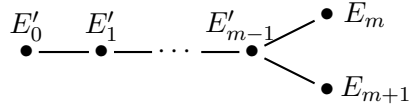
\begin{center}
\begin{pspicture}(0,-.5)(5,0.75)
	\psset{arcangle=15,nodesep=2pt}
\rput(0,0){\rnode{0}{$\bullet$}}\rput(0,0.35){$E'_0$}
\rput(1,0){\rnode{1}{$\bullet$}}\rput(1,0.35){$E'_1$}
\rput(2,0){\rnode{2}{$\cdots$}}
\rput(3,0){\rnode{3}{$\bullet$}}\rput(2.85,0.35){$E'_{m-1}$}
\rput(4,0.5){\rnode{4}{$\bullet$}}\rput(4.5,0.5){$E_{m}$}
\rput(4,-0.5){\rnode{4'}{$\bullet$}}\rput(4.65,-0.5){$E_{m+1}$}
\ncline{-}{0}{1}\ncline{-}{1}{2}\ncline{-}{2}{3}\ncline{-}{3}{4}\ncline{-}{3}{4'}
\end{pspicture}
\end{center}

In any of the crepant resolutions $X_{0\ldots i}$ except for $i=m-1$ we can also flop the rational curve $E_m$ to obtain the crepant resolution $X_{0\ldots i}^m$, where now $E_{m-1}$ and the flopped curve $E'_m$ are $(-1,-1)$-curves. Flopping $E'_m$ again takes us back to $X_{0\ldots i}$ and flopping $E_{m-1}$ leads us to $X_{0\ldots i}^{m(m-1)}$. In the same way we obtain the sequence of flops 
\[
X_{0\ldots i}\dashrightarrow X_{0\ldots i}^m\dashrightarrow X_{0\ldots i}^{m(m-1)}\dashrightarrow\ldots\dashrightarrow X_{0\ldots i}^{m\ldots i+2} 
\]
and continuing the process in the same fashion we construct the $(m+1)(m+2)/2$ crepant resolutions $\pi_{ij}:X_i^j\to\C^3/G$. Except for $X_{0\ldots i}$ which is described above, the dual graph of the fibre $\pi^{-1}(0)$ for the rest of crepant resolutions is

\begin{center}
\begin{pspicture}(0,0)(4,0.25)
	\psset{arcangle=15,nodesep=2pt}
\rput(0,0){\rnode{0}{$\bullet$}}
\rput(1,0){\rnode{1}{$\bullet$}}
\rput(2,0){\rnode{2}{$\cdots$}}
\rput(3,0){\rnode{3}{$\bullet$}}
\ncline{-}{0}{1}\ncline{-}{1}{2}\ncline{-}{2}{3}
\end{pspicture}
\end{center}

which finishes the proof of (3).

For (2), the degrees of the normal bundles in each case are obtained by using the gluings between the open sets shown below, and the result follows.
\[
\small
\begin{array}{l|l}
U_i\to U_{i+1}:  (d,D,u)\mapsto(d^2D,d^{-1},u), \text{\scriptsize $i\leq m-1$} 
		& U''_i\to U'_{i+1}: (a,d,D)\mapsto(ad,d^{-1},dD), \text{\scriptsize $i\leq m$}  \\
U_{m}\to U_{m+1}:  (d,D,u)\mapsto(u-4d^2D,d^{-1},u) 
		&  U''_i\to U''_{i+1}: \\
U_{m+1}\to V'_{m+2}:  (C',c',B')\mapsto(B'C',c'C',(C')^{-1}) 
		& \qquad	 (a,d,D)\mapsto(a,d^2D,d^{-1}), \text{\scriptsize $i\leq m-2$} \\
U_{m+1}\to V''_{m+3}:  (C',c',B')\mapsto(B'C',B'c',(B')^{-1}) 
		& \qquad (a,d,D)\mapsto(a,d^{-1},1-4d^2D), \text{\scriptsize $i=m-1$} \\	
U'_i\to U_{i+1}:  
		& U_i\to V'_{i+1}: 	 (d,D,u)\mapsto(du,dD,d^{-1}), \text{\scriptsize $i\leq m-1$} \\
\qquad (a,d,D)\mapsto(aD,a^{-1},ad), \text{\scriptsize $i\leq m-1$} 
		& \qquad	 (d,D,u)\mapsto(du,d^{-1}u-4dD,d^{-1}), \text{\scriptsize $i=m$} \\
\qquad (a,d,D)\mapsto(ad-4aD,a^{-1},ad), \text{\scriptsize $i=m$} 
		& V'_i\to V''_{i+1}: 	 (a',u,C')\mapsto(a'u,(a')^{-1},a'C'), \text{\scriptsize $i\leq m+1$} \\	
U'_i\to V'_{i+1}: 	 
		& V''_i\to V''_{i+1}: 	 (d,D,C')\mapsto(d^2D,d^{-1},C'), \text{\scriptsize $i\leq m$} \\	
\qquad (a,d,D)\mapsto(a^2d,D,a^{-1}), \text{\scriptsize $i\leq m-1$} 
		& \qquad (d,D,C')\mapsto((C')^2-4d^2D,d^{-1},C'), \text{\scriptsize $i=m+1$} \\
\qquad (a,d,D)\mapsto(a^2d,d-4D,a^{-1}), \text{\scriptsize $i=m$} 	
		& \qquad (A,c',C')\mapsto(A,c'(C')^2,(C')^{-1}), \text{\scriptsize $i=m+2$} \\
\end{array}
\]

\end{proof}

\subsection{The tetrahedral group}

Let $G$ be the tetrahedral group of order 12. In this case the McKay quiver and the relations derived by the potential $W$ are the following:

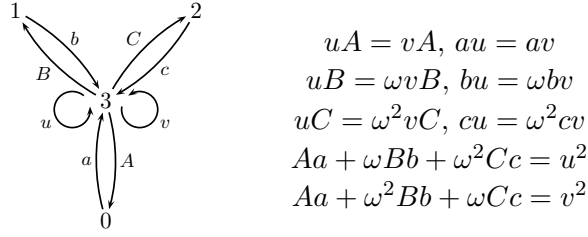
\begin{center}
\begin{pspicture}(0,-1.5)(10,1.75)
	\psset{arcangle=15,nodesep=2pt}
\rput(2.5,0.25){
\scalebox{0.8}{
	\rput(0,-2){\rnode{0}{$0$}}
	\rput(0,0){\rnode{3}{$3$}}
	\rput(-1.5,1.5){\rnode{1}{$1$}}
	\rput(1.5,1.5){\rnode{2}{$2$}}
	\rput(-0.3,0){\rnode{u}{}}	
	\rput(0.3,0){\rnode{v}{}}	
	\ncarc{->}{0}{3}\Aput[0.05]{\footnotesize $a$}	
	\ncarc{->}{3}{0}\Aput[0.05]{\footnotesize $A$}	
	\ncarc{->}{1}{3}\Aput[0.05]{\footnotesize $b$}	
	\ncarc{->}{3}{1}\Aput[0.05]{\footnotesize $B$}	
	\ncarc{->}{2}{3}\Aput[0.05]{\footnotesize $c$}	
	\ncarc{->}{3}{2}\Aput[0.05]{\footnotesize $C$}	
	\nccircle[angleA=120,nodesep=3pt]{->}{u}{.3cm}\Bput[0.05]{\footnotesize $u$}
	 \nccircle[angleA=240,nodesep=3pt]{->}{v}{.3cm}\Bput[0.05]{\footnotesize $v$}
	}}
\rput(7,1){$uA=vA$, $au=av$}
\rput(7,0.5){$uB=\omega vB$,  $bu=\omega bv$}
\rput(7,0){$uC=\omega^2 vC$, $cu=\omega^2 cv$}
\rput(7,-0.5){$Aa+\omega Bb+\omega^2 Cc=u^2$}
\rput(7,-1){$Aa+\omega^2 Bb+\omega Cc=v^2$}
\end{pspicture}
\end{center}

Considering the arrows as linear maps between vector spaces we denote them by $a:=(a_1,a_2,a_3)$, $A:=\left(\begin{smallmatrix}A_1\\A_2\\A_3\end{smallmatrix}\right)$, $b:=(b_1,b_2,b_3)$, $B:=\left(\begin{smallmatrix}B_1\\B_2\\B_3\end{smallmatrix}\right)$, $c:=(c_1,c_2,c_3)$, $C:=\left(\begin{smallmatrix}C_1\\C_2\\C_3\end{smallmatrix}\right)$, $u:=\left(\begin{smallmatrix}u_{11}&u_{12}&u_{13}\\u_{21}&u_{22}&u_{23}\\u_{31}&u_{32}&u_{33}\end{smallmatrix}\right)$ and $v:=\left(\begin{smallmatrix}v_{11}&v_{12}&v_{13}\\v_{21}&v_{22}&v_{23}\\v_{31}&v_{32}&v_{33}\end{smallmatrix}\right)$. 

The quiver structure of the CM $S^G$-modules $S_\rho$ in this case is the following:

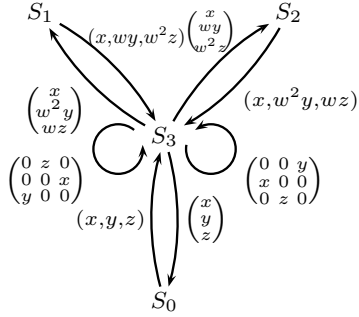
\begin{figure}[htbp]
\begin{center}
\begin{pspicture}(0,-1.1)(2,3)
	\psset{arcangle=15,nodesep=2pt}
\rput(1,1){
\scalebox{1.1}{
	\rput(0,-2){\rnode{0}{\scriptsize $S_{0}$}}
	\rput(0,0){\rnode{3}{\scriptsize $S_{3}$}}
	\rput(-1.5,1.5){\rnode{1}{\scriptsize $S_{1}$}}
	\rput(1.5,1.5){\rnode{2}{\scriptsize $S_{2}$}}
	\rput(-0.3,0){\rnode{u}{}}	
	\rput(0.3,0){\rnode{v}{}}	
	\ncarc{->}{0}{3}\Aput[0.05]{\tiny $(x,\!y,\!z)$}	
	\ncarc{->}{3}{0}\Aput[0.05]{\footnotesize $\left(\!\begin{smallmatrix}x\\y\\z\end{smallmatrix}\!\right)$}	
	\ncarc{->}{1}{3}	
	\ncarc{->}{3}{1}\Aput[0.05]{\footnotesize $\left(\!\!\begin{smallmatrix}x\\w^2y\\wz\end{smallmatrix}\!\!\right)$}	
	\ncarc{->}{2}{3}\Aput[0.05]{\tiny $(x,\!w^2y,\!wz)$}	
	\ncarc{->}{3}{2} 	
	\nccircle[angleA=120,nodesep=3pt]{->}{u}{.3cm}\Bput[0.05]{\footnotesize $\left(\!\begin{smallmatrix}0&z&0\\0&0&x\\y&0&0\end{smallmatrix}\!\right)$} 
	 \nccircle[angleA=240,nodesep=3pt]{->}{v}{.3cm}\Bput[0.05]{\footnotesize $\left(\!\begin{smallmatrix}0&0&y\\x&0&0\\0&z&0\end{smallmatrix}\!\right)$} 
	\scalebox{0.75}{\rput(-0.45,1.6){\footnotesize $(x,\!wy,\!w^2z)$}}
	 \scalebox{0.85}{\rput(0.5,1.5){\footnotesize $\left(\!\!\begin{smallmatrix}x\\wy\\w^2z\end{smallmatrix}\!\!\right)$}}
	}}
\end{pspicture}
\caption{The quiver structure of the CM $S^G$-modules for $G=\mathbb{T}$}
\label{xyQuiverE6}
\end{center}
\end{figure}

Let us define the following polynomials which appear frequently in the rest of this section.
\[
\begin{array}{ll}
f_0:=x^2+y^2+z^2, & R_0:=y^2z^2+x^2z^2+x^2y^2, \\
f_1:=x^2+\omega^2 y^2+\omega z^2, & R_1:=y^2z^2+\omega x^2z^2+\omega^2x^2y^2, \\
f_2:=x^2+\omega y^2+\omega^2 z^2, & R_2:=y^2z^2+\omega^2x^2z^2+\omega x^2y^2. \\
f_3:=xyz, \\
f_4:=(x^2-y^2)(y^2-z^2)(z^2-x^2).  
\end{array}
\]

Notice that we have $3R_0=f_0^2-f_1f_2$, $3R_1=f_1^2-f_0f_2$, $3R_2=f_2^2-f_0f_1$, $R_0^3=R_1R_2+3f_0f^2_3$, $R_1^2=R_0R_2+3f_1f_3^2$, $R_2^2=R_0R_1+3f_2f_3^2$ and $R_2^3-R_1^3=3f_3^2(f_2R_2-f_1R_1)$, as some of the relations among these polynomials. The invariant ring $S^G$ is generated by $f_0,f_3,f_1f_2$ and $f_4$ (See \cite[\S 2]{GNS1}) but $f_1f_2 = f_0^2 - 3R_0$ holds, hence one can take $R_0$ as a generator of $S^G$ instead of $f_1f_2$. There is only one relation between these polynomials:
\[
f_4^2+4R_0^3-f_0^2R_0^2-18f_0R_0f_3^2+4f_0^3f_3^2+27f_3^4.
\]

\begin{thm}\label{OpensE6} Let $G$ be the tetrahedral group of order 12 and let $\pi:Y\to\C^3/G$ be a crepant resolution. Then, 
\begin{enumerate}
\item There exist 5 crepant resolutions of $\pi:X_i\to\C^3/G$ related by flops with the following configuration:

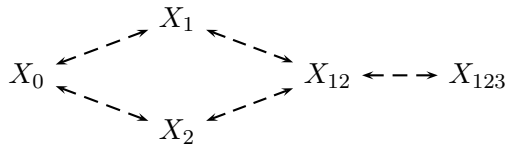
\begin{center}
\begin{pspicture}(0,-1)(8,1)
	\psset{arcangle=15,nodesep=4pt}
	\rput(0,0){\rnode{1}{$X_0$}}
	\rput(2,0.75){\rnode{2}{$X_1$}}
	\rput(2,-0.75){\rnode{3}{$X_2$}}
	\rput(4,0){\rnode{4}{$X_{12}$}}
	\rput(6,0){\rnode{5}{$X_{123}$}}
	\ncline[linestyle=dashed]{<->}{1}{2}
	\ncline[linestyle=dashed]{<->}{1}{3}
	\ncline[linestyle=dashed]{<->}{2}{4}
	\ncline[linestyle=dashed]{<->}{3}{4}
	\ncline[linestyle=dashed]{<->}{4}{5}
\end{pspicture}
\end{center}
Moreover, $X_0\cong\Hilb{G}{\C^3}$ and every $X_i$ is described as the union of 4 open sets isomorphic to $\C^3$. The open covers are 
$X_0=U_0\cup U_1\cup U_2\cup U_3$, $X_1=U'_0\cup U'_1\cup U_2\cup U_3$, $X_2=U_0\cup U_1\cup U'_2\cup U'_3$, $X_3=U'_0\cup U'_1\cup U'_2\cup U'_3$ and $X_4=U'_0\cup U''_1\cup U''_2\cup U'_3$, where the local coordinates in each open set are
{\renewcommand{\arraystretch}{1.5}
\[
\begin{array}{l|l}
U_0\cong\C^3_{c_2,c_3,C_1}=\Spec[\frac{f_1^2f_3}{R_1}, \frac{f_1R_2}{R_1}, \frac{f_2}{f_1^2}] & 
	U'_0\cong\C^3_{c_3,C_1,C_3}=\Spec[\frac{R_2}{f_1f_3}, \frac{f_2f_3}{R_1}, \frac{f_1^2f_3}{R_1}] \\

U_1\cong\C^3_{c_2,c_3,C_3}=\Spec[\frac{f_2f_3}{R_1}, \frac{f_2R_2}{f_1R_1}, \frac{f_1^2}{f_2}] & 
	U'_1\cong\C^3_{c_2,C_1,C_3}=\Spec[\frac{f_1f_3}{R_2}, \frac{f_2R_2}{f_1R_1}, \frac{f_1R_2}{R_1}] \\
U_2\cong\C^3_{b_2,b_3,B_3}=\Spec[\frac{f_1f_3}{R_2}, \frac{f_1R_1}{f_2R_2}, \frac{f_2^2}{f_1}] & 
	U'_2\cong\C^3_{b_2,B_1,B_3}=\Spec[\frac{f_2f_3}{R_1},\frac{f_1R_1}{f_2R_2},\frac{f_2R_1}{R_2}] \\ 
U_3\cong\C^3_{b_2,b_3,B_1}=\Spec[\frac{f_2^2f_3}{R_2}, \frac{f_2R_1}{R_2}, \frac{f_1}{f_2^2}] & 
	U'_3\cong\C^3_{b_3,B_1,B_3}=\Spec[\frac{R_1}{f_2f_3},\frac{f_1f_3}{R_2},\frac{f_2^2f_3}{R_2}] \\
U''_1\cong\C^3_{B_1,c_1,C_1}=\Spec[\frac{f_1f_3}{R_2}, \frac{R_0}{f_3}, \frac{f_2f_3}{R_1}] &
	U''_2\cong\C^3_{B_1,c_2,C_1}=\Spec[\frac{f_1R_0}{R_2}, \frac{f_3}{R_0}, \frac{f_2R_0}{R_1}] 
\end{array}
\]}
\item The dual graphs of $\pi^{-1}(0)$ in each crepant resolution with the corresponding degrees for the normal bundles are:

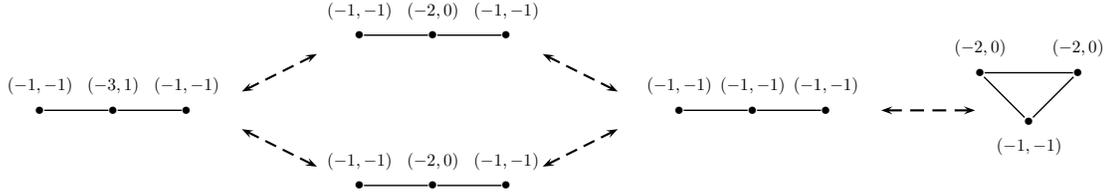
\begin{center}
\begin{pspicture}(0,-1.5)(14,1.75)
\rput(0,0){\
\scalebox{0.65}{
	\rput(0,0){\rnode{1}{$\bullet$}}
	\rput(1.5,0){\rnode{2}{$\bullet$}}
	\rput(3,0){\rnode{3}{$\bullet$}}
	\ncline{-}{1}{2}
	\ncline{-}{2}{3}
	\rput(0,0.5){\small $(-1,-1)$}
	\rput(1.5,0.5){\small $(-3,1)$}
	\rput(3,0.5){\small $(-1,-1)$}
}}
\rput(4.25,1){
\scalebox{0.65}{
	\rput(0,0){\rnode{1}{$\bullet$}}
	\rput(1.5,0){\rnode{2}{$\bullet$}}
	\rput(3,0){\rnode{3}{$\bullet$}}
	\ncline{-}{1}{2}
	\ncline{-}{2}{3}
	\rput(0,0.5){\small $(-1,-1)$}
	\rput(1.5,0.5){\small $(-2,0)$}
	\rput(3,0.5){\small $(-1,-1)$}
}}
\rput(4.25,-1){
\scalebox{0.65}{
	\rput(0,0){\rnode{1}{$\bullet$}}
	\rput(1.5,0){\rnode{2}{$\bullet$}}
	\rput(3,0){\rnode{3}{$\bullet$}}
	\ncline{-}{1}{2}
	\ncline{-}{2}{3}
	\rput(0,0.5){\small $(-1,-1)$}
	\rput(1.5,0.5){\small $(-2,0)$}
	\rput(3,0.5){\small $(-1,-1)$}
}}
\rput(8.5,0){
\scalebox{0.65}{
	\rput(0,0){\rnode{1}{$\bullet$}}
	\rput(1.5,0){\rnode{2}{$\bullet$}}
	\rput(3,0){\rnode{3}{$\bullet$}}
	\ncline{-}{1}{2}
	\ncline{-}{2}{3}
	\rput(0,0.5){\small $(-1,-1)$}
	\rput(1.5,0.5){\small $(-1,-1)$}
	\rput(3,0.5){\small $(-1,-1)$}
}}
\rput(12.5,0.5){
\scalebox{0.65}{
	\rput(0,0){\rnode{1}{$\bullet$}}
	\rput(2,0){\rnode{2}{$\bullet$}}
	\rput(1,-1){\rnode{3}{$\bullet$}}
	\ncline{-}{1}{2}
	\ncline{-}{2}{3}
	\ncline{-}{3}{1}
	\rput(0,0.5){\small $(-2,0)$}
	\rput(2,0.5){\small $(-2,0)$}
	\rput(1,-1.5){\small $(-1,-1)$}
}}
\psline[linestyle=dashed]{<->}(2.75,0.25)(3.75,0.75)
\psline[linestyle=dashed]{<->}(2.75,-0.25)(3.75,-0.75)
\psline[linestyle=dashed]{<->}(6.75,0.75)(7.75,0.25)
\psline[linestyle=dashed]{<->}(6.75,-0.75)(7.75,-0.25)
\psline[linestyle=dashed]{<->}(11.25,0)(12.5,0)
\end{pspicture}
\end{center}

\end{enumerate}

\end{thm}

\begin{proof} We start by calculating explicitly $\Hilb{G}{\C^3}$ as a moduli space of representations of the McKay QP. Let $\theta^0$ be a $0$-generated stability condition. Then $X_0:=\Hilb{G}{\C^3}$ is covered by $U_0$, $U_1$, $U_2$ and $U_3$, where
\[
\left.\begin{array}{cccccc}
U_0: & aB=1 & aBbC=1 & a=(1,0,0) & au=(0,1,0) & b=(0,0,1)\\
U_1: & aB=1 & aC=1 & a=(1,0,0) & au=(0,1,0) & b=(0,0,1)\\
U_2: & aB=1 & aC=1 & a=(1,0,0) & au=(0,1,0) & c=(0,0,1)\\
U_3: & aCcB=1 & aC=1 & a=(1,0,0) & au=(0,1,0) & c=(0,0,1)
\end{array}\right.
\]

First notice that by using the relations we have that $auB = avB = \omega^2 auB$, which implies $auB=0$. Similarly we obtain that following paths vanish:   
\[
auB=avB=auC=avC=buA=bvA=buC=bvC=cuA=cvA=cuB=cvB=0 \ \ (\#)
\]
We split the calculation this time in 5 steps:
\begin{enumerate}
\item[(i)] By changing basis we can assume that $a= (1,0,0)$.
\item[(ii)] If $aB=aC=0$, then it follows $au^iv^jB=au^iv^jC=0$ by the relations of the middle vertex, which contradicts the 0-generated stability condition $\theta$. Therefore either $aB\neq0$ or $aC\neq0$. Moreover we may assume that $aB=1$ or $aC=1$ by change of basis.
\item[(iii)] We consider the case $aB=1$ and $aC=0$. If $aBbC=0$, then it turns out that any path through $C$ is zero by the relations. This contradicts the 0-generated condition. So it must be $aBbC\not=0$ and $b$ not a linear multiple of $a$. We may assume $aBbC=1$ and $b=(0,0,1)$. Next assume that $au=(\lambda,0,\eta)$ for some $\lambda, \eta \in \mathbb C$. Since $auC=0$ by (\#), and $C_1=0, C_3=1$ by $aC=0,ABbC=1$, it follows that $\eta=0$. Moreover since $auB=0$ by (\#), and $B_1=1$ by $aB=1$, it follows that $\lambda=0$ hence $au=0$, which leads to $aBbC = aAaC + \omega aBbC + \omega^2 aCcC = au^2 = 0$. This contradicts $aBbC=1$, hence $au$ is linear independent of $(1,0,0)$ and $(0,0,1)$. Therefore we can take $au=(0,1,0)$ by change of basis. These are the conditions for $U_0$.
\item[(iv)] The case $aB=0$ and $aC=1$ is similar to (iii). This case leads to $U_3$.
\item[(v)] Consider the case $aB=aC=1$. If $au=(\lambda,0,0)$ for some $\lambda \in \mathbb C$, because $auB=0$ and $B_1=1$, we must have $\lambda=0$, hence $au=0$. The relations $aAa+\omega aBb + \omega^2 aCc = au^2$ and $aAa+\omega^2 aBb + \omega aCc = av^2$ means
\begin{align*}
\left\{\begin{array}{r}
A_1 + \omega b_1 + \omega^2 c_1 = 0 \\
\omega b_2 + \omega^2 c_2 = 0 \\
\omega b_3 + \omega^2 c_3 = 0
\end{array}\right.
\qquad
\left\{\begin{array}{r}
A_1 + \omega^2 b_1 + \omega c_1 = 0 \\
\omega^2 b_2 + \omega c_2 = 0 \\
\omega^2 b_3 + \omega c_3 = 0
\end{array}\right.
\end{align*}
hence it follows $b_1 = c_1$ and $b_2=c_2=b_3=c_3=0$, that is, $b = c = (b_1,0,0)$. This means we cannot generate the middle vertex, which contradicts the 0-generated condition. Therefore $au$ is not a linear multiple of $a$, hence we can assume $au =(0,1,0)$. 

We claim that if both of $b$ and $c$ are linear multiple of $a$ and $au$, then it contradicts the 0-generated condition. Indeed, if we assume $b=(b_1,b_2,0)$ and $c=(c_1,c_2,0)$, the relations $aAa+\omega aBb + \omega^2 aCc = au^2$ and $aAa+\omega^2 aBb + \omega aCc = av^2$ are equivalent to 
\begin{align*}
\left\{\begin{array}{r}
A_1 + \omega b_1 + \omega^2 c_1 = u_{21} \\
\omega b_2 + \omega^2 c_2 = u_{22} \\
0 = u_{23}
\end{array}\right.
\qquad
\left\{\begin{array}{r}
A_1 + \omega^2 b_1 + \omega c_1 = v_{21} \\
\omega^2 b_2 + \omega c_2 = v_{22} \\
0 = v_{23}
\end{array}\right.
\end{align*}
Therefore 
\begin{align*}
au^2&=(0,1,0)u=(u_{21},u_{22},u_{23})=(u_{21},u_{22},0), \\
av^2&=(0,1,0)u=(v_{21},v_{22},v_{23})=(v_{21},v_{22},0),
\end{align*} 
which are linear combinations of $a$ and $au$. Therefore we can not generate the middle vertex. Consequently it must be $b=(0,0,1)$ or $c=(0,0,1)$. These conditions give $U_1$ and $U_2$ respectively.

We show in Figure \ref{RepSpT} the representation spaces for each open set in $\Hilb{G}{\C^3}$.

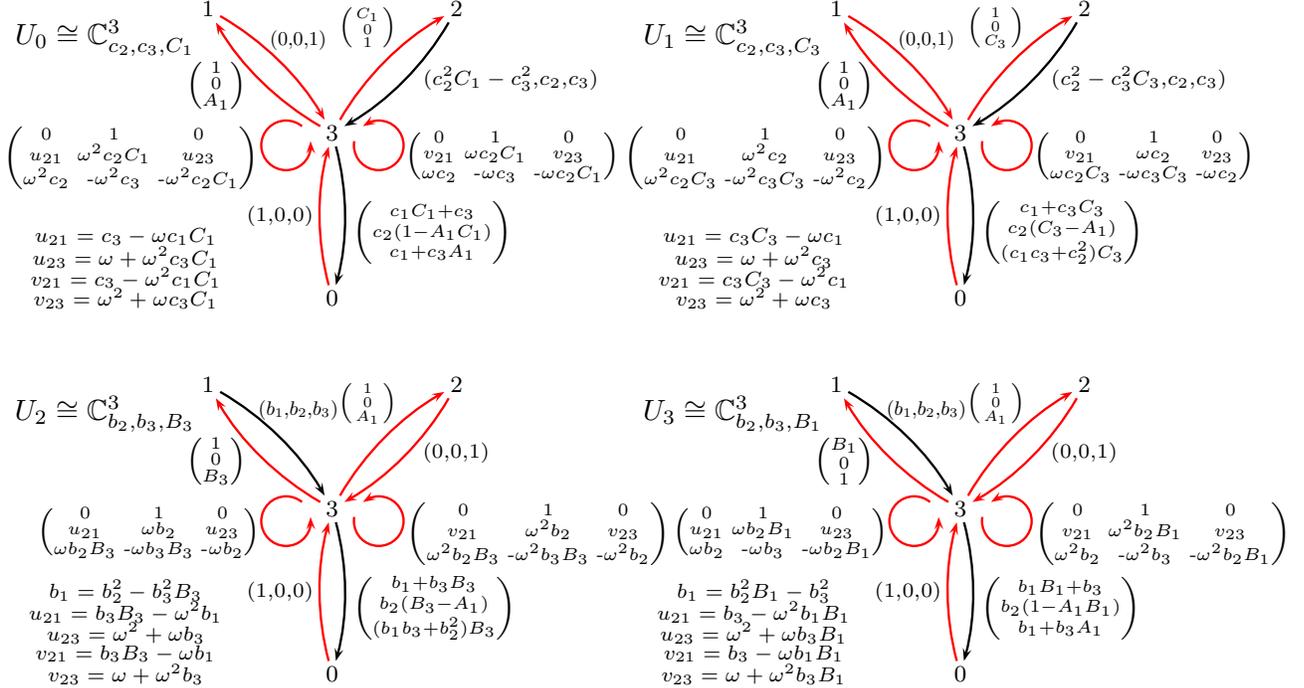
\begin{figure}[h]
\begin{center}
\begin{pspicture}(0,-1.5)(10,7.5)
	\psset{arcangle=15,nodesep=2pt}
	

\rput(1.15,5.5){
\scalebox{1.1}{
	\rput(0,-2){\rnode{0}{\scriptsize $0$}}
	\rput(0,0){\rnode{3}{\scriptsize $3$}}
	\rput(-1.5,1.5){\rnode{1}{\scriptsize $1$}}
	\rput(1.5,1.5){\rnode{2}{\scriptsize $2$}}
	\rput(-0.3,0){\rnode{u}{}}	
	\rput(0.3,0){\rnode{v}{}}	
	\ncarc[linecolor=red]{->}{0}{3}\Aput[0.05]{\tiny $(1,\!0,\!0)$}	
	\ncarc{->}{3}{0}\aput[0.05](0.6){\footnotesize $\left(\!\begin{smallmatrix}c_1C_1 + c_3 \\c_2(1-A_1C_1) \\c_1+c_3A_1 \end{smallmatrix}\!\right)$}	
	\ncarc[linecolor=red]{->}{1}{3}	
	\scalebox{0.75}{\rput(-0.55,1.5){\footnotesize $(0,\!0,\!1)$}} 	
	\ncarc[linecolor=red]{->}{3}{1}\aput[-0.05](0.65){\footnotesize $\left(\!\begin{smallmatrix}1 \\0 \\ A_1\end{smallmatrix}\!\right)$}	
	\ncarc{->}{2}{3}\aput[0.025](0.35){\tiny $(c_2^2C_1-c_3^2,\!c_2,\!c_3)$}	
	\ncarc[linecolor=red]{->}{3}{2}\scalebox{0.85}{\rput(0.5,1.5){\footnotesize $\left(\begin{smallmatrix}C_1 \\0 \\ 1\end{smallmatrix}\right)$}}	
	\nccircle[linecolor=red,angleA=120,nodesep=3pt]{->}{u}{.3cm}\bput[-0.05](0.45){\footnotesize $\left(\!\begin{smallmatrix}0&1&0 \\u_{21}&\omega^2 c_2C_1&u_{23} \\ \omega^2c_2&\text{-}\omega^2c_3&\text{-}\omega^2c_2C_1\end{smallmatrix}\!\right)$} 
	 \nccircle[linecolor=red,angleA=240,nodesep=3pt]{->}{v}{.3cm}\bput[-0.05](0.55){\footnotesize $\left(\!\begin{smallmatrix}0&1&0 \\v_{21}&\omega c_2C_1&v_{23} \\ \omega c_2&\text{-}\omega c_3&\text{-}\omega c_2C_1\end{smallmatrix}\!\right)$} 
\rput(-2.5,-1.25){\tiny $u_{21}=c_3-\omega c_1C_1$}
\rput(-2.5,-1.5){\tiny $u_{23}=\omega+\omega^2c_3C_1$}
\rput(-2.5,-1.75){\tiny $v_{21}=c_3-\omega^2 c_1C_1$}
\rput(-2.5,-2){\tiny $v_{23}=\omega^2 +\omega c_3C_1$}
\rput(-2.75,1.15){\small $U_0\cong\C^3_{c_2,c_3,C_1}$}

	}}
\rput(9.5,5.5){
\scalebox{1.1}{
	\rput(0,-2){\rnode{0}{\scriptsize $0$}}
	\rput(0,0){\rnode{3}{\scriptsize $3$}}
	\rput(-1.5,1.5){\rnode{1}{\scriptsize $1$}}
	\rput(1.5,1.5){\rnode{2}{\scriptsize $2$}}
	\rput(-0.3,0){\rnode{u}{}}	
	\rput(0.3,0){\rnode{v}{}}	
	\ncarc[linecolor=red]{->}{0}{3}\Aput[0.05]{\tiny $(1,\!0,\!0)$}	
	\ncarc{->}{3}{0}\aput[0.05](0.6){\footnotesize $\left(\!\begin{smallmatrix}c_1 + c_3C_3 \\c_2(C_3-A_1) \\ (c_1c_3+c_2^2)C_3 \end{smallmatrix}\!\right)$}	
	\ncarc[linecolor=red]{->}{1}{3}	
	\scalebox{0.75}{\rput(-0.55,1.5){\footnotesize $(0,\!0,\!1)$}}  
	\ncarc[linecolor=red]{->}{3}{1}\aput[-0.05](0.65){\footnotesize $\left(\!\begin{smallmatrix}1 \\0 \\ A_1\end{smallmatrix}\!\right)$}	
	\ncarc{->}{2}{3}\aput[0.025](0.35){\tiny $(c_2^2-c_3^2C_3,\!c_2,\!c_3)$}	
	\ncarc[linecolor=red]{->}{3}{2} \scalebox{0.85}{\rput(0.5,1.5){\footnotesize $\left(\begin{smallmatrix}1 \\0 \\ C_3\end{smallmatrix}\right)$}}	
	\nccircle[linecolor=red,angleA=120,nodesep=3pt]{->}{u}{.3cm}\bput[-0.05](0.45){\footnotesize $\left(\!\begin{smallmatrix}0&1&0 \\u_{21}&\omega^2 c_2&u_{23} \\ \omega^2 c_2C_3&\text{-}\omega^2 c_3C_3 & \text{-}\omega^2 c_2 \end{smallmatrix}\!\right)$} 
	 \nccircle[linecolor=red,angleA=240,nodesep=3pt]{->}{v}{.3cm}\bput[-0.05](0.55){\footnotesize $\left(\!\begin{smallmatrix}0&1&0 \\v_{21}&\omega c_2&v_{23} \\ \omega c_2C_3&\text{-}\omega c_3C_3 & \text{-}\omega c_2 \end{smallmatrix}\!\right)$} 

\rput(-2.5,-1.25){\tiny $u_{21}=c_3C_3-\omega c_1$}
\rput(-2.5,-1.5){\tiny $u_{23}=\omega +\omega^2 c_3$}
\rput(-2.5,-1.75){\tiny $v_{21}=c_3C_3-\omega^2 c_1$}
\rput(-2.5,-2){\tiny $v_{23}=\omega^2 +\omega c_3$}
\rput(-2.75,1.15){\small $U_1\cong\C^3_{c_2,c_3,C_3}$}
	}}


\rput(1.15,0.5){
\scalebox{1.1}{
	\rput(0,-2){\rnode{0}{\scriptsize $0$}}
	\rput(0,0){\rnode{3}{\scriptsize $3$}}
	\rput(-1.5,1.5){\rnode{1}{\scriptsize $1$}}
	\rput(1.5,1.5){\rnode{2}{\scriptsize $2$}}
	\rput(-0.3,0){\rnode{u}{}}	
	\rput(0.3,0){\rnode{v}{}}	
	\ncarc[linecolor=red]{->}{0}{3}\Aput[0.05]{\tiny $(1,\!0,\!0)$}	
	\ncarc{->}{3}{0}\aput[0.05](0.6){\footnotesize $\left(\!\begin{smallmatrix}b_1 + b_3B_3 \\b_2(B_3-A_1) \\ (b_1b_3+b_2^2)B_3 \end{smallmatrix}\!\right)$}	
	\ncarc{->}{1}{3}	
	\scalebox{0.75}{\rput(-0.55,1.6){\footnotesize $(b_1,\!b_2,\!b_3)$}} 	
	\ncarc[linecolor=red]{->}{3}{1}\aput[-0.05](0.65){\footnotesize $\left(\!\begin{smallmatrix}1 \\0 \\ B_3\end{smallmatrix}\!\right)$}	
	\ncarc[linecolor=red]{->}{2}{3}\aput[0.025](0.35){\tiny $(0,\!0,\!1)$}	
	\ncarc[linecolor=red]{->}{3}{2}\scalebox{0.85}{\rput(0.5,1.5){\footnotesize $\left(\begin{smallmatrix}1 \\0 \\ A_1\end{smallmatrix}\right)$}}	
	\nccircle[linecolor=red,angleA=120,nodesep=3pt]{->}{u}{.3cm}\bput[-0.05](0.45){\footnotesize $\left(\!\begin{smallmatrix}0&1&0 \\u_{21}&\omega b_2&u_{23} \\ \omega b_2B_3&\text{-}\omega b_3B_3 & \text{-}\omega b_2 \end{smallmatrix}\!\right)$} 
	 \nccircle[linecolor=red,angleA=240,nodesep=3pt]{->}{v}{.3cm}\bput[-0.05](0.55){\footnotesize $\left(\!\begin{smallmatrix}0&1&0 \\v_{21}&\omega^2 b_2&v_{23} \\ \omega^2 b_2B_3 & \text{-}\omega^2 b_3B_3 & \text{-}\omega^2 b_2 \end{smallmatrix}\!\right)$} 

\rput(-2.5,-1){\tiny $b_{1}=b_2^2-b_3^2B_3$}
\rput(-2.5,-1.25){\tiny $u_{21}=b_3B_3-\omega^2 b_1$}
\rput(-2.5,-1.5){\tiny $u_{23}=\omega^2 +\omega b_3$}
\rput(-2.5,-1.75){\tiny $v_{21}=b_3B_3-\omega b_1$}
\rput(-2.5,-2){\tiny $v_{23}=\omega +\omega^2 b_3$}
\rput(-2.75,1.15){\small $U_2\cong\C^3_{b_2,b_3,B_3}$}

	}}
\rput(9.5,0.5){
\scalebox{1.1}{
	\rput(0,-2){\rnode{0}{\scriptsize $0$}}
	\rput(0,0){\rnode{3}{\scriptsize $3$}}
	\rput(-1.5,1.5){\rnode{1}{\scriptsize $1$}}
	\rput(1.5,1.5){\rnode{2}{\scriptsize $2$}}
	\rput(-0.3,0){\rnode{u}{}}	
	\rput(0.3,0){\rnode{v}{}}	
	\ncarc[linecolor=red]{->}{0}{3}\Aput[0.05]{\tiny $(1,\!0,\!0)$}	
	\ncarc{->}{3}{0}\aput[0.05](0.6){\footnotesize $\left(\!\begin{smallmatrix}b_1B_1 + b_3 \\b_2(1-A_1B_1) \\ b_1+b_3A_1 \end{smallmatrix}\!\right)$}	
	\ncarc{->}{1}{3}	
	\scalebox{0.75}{\rput(-0.55,1.6){\footnotesize $(b_1,\!b_2,\!b_3)$}}  
	\ncarc[linecolor=red]{->}{3}{1}\aput[-0.05](0.65){\footnotesize $\left(\!\begin{smallmatrix}B_1 \\0 \\ 1\end{smallmatrix}\!\right)$}	
	\ncarc[linecolor=red]{->}{2}{3}\aput[0.025](0.35){\tiny $(0,\!0,\!1)$}	
	\ncarc[linecolor=red]{->}{3}{2} \scalebox{0.85}{\rput(0.5,1.5){\footnotesize $\left(\begin{smallmatrix}1 \\0 \\ A_1\end{smallmatrix}\right)$}}	
	\nccircle[linecolor=red,angleA=120,nodesep=3pt]{->}{u}{.3cm}\bput[-0.05](0.45){\footnotesize $\left(\!\begin{smallmatrix}0&1&0 \\u_{21}&\omega b_2B_1&u_{23} \\ \omega b_2& \text{-}\omega b_3 & \text{-}\omega b_2B_1 \end{smallmatrix}\!\right)$} 
	 \nccircle[linecolor=red,angleA=240,nodesep=3pt]{->}{v}{.3cm}\bput[-0.05](0.55){\footnotesize $\left(\!\begin{smallmatrix}0&1&0 \\v_{21}&\omega^2 b_2B_1&v_{23} \\ \omega^2 b_2 & \text{-}\omega^2 b_3 & \text{-}\omega^2 b_2B_1 \end{smallmatrix}\!\right)$} 

\rput(-2.5,-1){\tiny $b_{1}=b_2^2B_1-b_3^2$}
\rput(-2.5,-1.25){\tiny $u_{21}=b_3-\omega^2 b_1B_1$}
\rput(-2.5,-1.5){\tiny $u_{23}=\omega^2 +\omega b_3B_1$}
\rput(-2.5,-1.75){\tiny $v_{21}=b_3-\omega b_1B_1$}
\rput(-2.5,-2){\tiny $v_{23}=\omega +\omega^2 b_3B_1$}
\rput(-2.75,1.15){\small $U_3\cong\C^3_{b_2,b_3,B_1}$}
	}}
\end{pspicture}
\end{center}
\caption{Representation spaces for the open covering of \Hilb{$\mathbb{T}$}{$\C^3$}.}
\label{RepSpT}
\end{figure}
\end{enumerate}

We calculate now the local coordinates along the exceptional curves using the quiver shown in Figure \ref{xyQuiverE6}. For example in the open set $U_0\cong\mathbb C^3_{c_2,c_3,C_1}$, we have that
\[
\text{$aC = C_1\cdot$(the basis of $\rho_2$)}
\]
which implies that $f_2 = C_1f_1^2$, thus $C_1 = f_2/f_1^2$. Similarly,
\begin{align*}
aA&=(c_1C_1+c_3)\cdot1\Longrightarrow c_1= (f_0-c_3)f_1^2/f_2 \\
aBbA&=(c_1+c_3A_1)\cdot1\Longrightarrow f_1f_2=c_1+c_3f_0
\end{align*}
which gives $c_3=-f_1R_2/R_1$. Finally $auA=3\sqrt{3}f_3 = c_2(1-A_1C_1)\cdot1$, so that $c_2=\sqrt{3}f_3f_1^2/R_1$. Therefore (after rescaling) the coordinate ring of $U_0$ is given by $\C[c_2,c_3,C_1] = \C\left[\frac{f_3f_1^2}{R_1},\frac{f_1R_2}{R_1},\frac{f_2}{f_1^2} \right]$. The rest of the cases are done similarly.

It follows that the fibre over the origin $\pi^{-1}(0)\subset X_0$ consists of 3 rational curves $E_1$, $E_2$ and $E_3$ intersecting pairwise as 

\begin{center}
\begin{pspicture}(0,-0.1)(4,0.5)
	\psset{arcangle=15,nodesep=2pt}
\rput(0,0){\rnode{0}{$\bullet$}}\rput(0,0.35){$E_1$}
\rput(1,0){\rnode{1}{$\bullet$}}\rput(1,0.35){$E_3$}
\rput(2,0){\rnode{2}{$\bullet$}}\rput(2,0.35){$E_2$}
\ncline{-}{0}{1}\ncline{-}{1}{2}
\end{pspicture}
\end{center}

The explicit open cover shows that the curves $E_1$ and $E_2$ have degree $(-1,-1)$ while $E_3$ has degree $(-3,1)$. By Lemma \ref{floppable} we can flop $E_1$ and $E_2$, giving rise to $X_{1}$ and $X_{2}$ respectively. By symmetry we only explain the flop of $E_2$.

{\em First Flop $X_2$.} In the flop of the rational curve $E_2$ we only need to change the open sets $U_2$ and $U_3$. By the same method as in the dihedral case we produce the rational curve $E'_2$ covered by open sets $U_{2}'$ and $U_{3}'$, both of them isomorphic to $\C^3$, and given by
\[
\left.\begin{array}{cccccc}
U_2' & b_2=1 & aC=1 & a=(1,0,0) & au=(0,1,0) & c=(0,0,1)\\
U_3' & b_3=1 & aC=1 & a=(1,0,0) & au=(0,1,0) & c=(0,0,1)
\end{array}\right.
\]

\begin{figure}[h]
\begin{center}
\begin{pspicture}(0,-1.1)(10,2.5)
	\psset{arcangle=15,nodesep=2pt}

\rput(1.15,1){
\scalebox{1.1}{
	\rput(0,-2){\rnode{0}{\scriptsize $0$}}
	\rput(0,0){\rnode{3}{\scriptsize $3$}}
	\rput(-1.5,1.5){\rnode{1}{\scriptsize $1$}}
	\rput(1.5,1.5){\rnode{2}{\scriptsize $2$}}
	\rput(-0.3,0){\rnode{u}{}}	
	\rput(0.3,0){\rnode{v}{}}	
	\ncarc[linecolor=red]{->}{0}{3}\Aput[0.05]{\tiny $(1,\!0,\!0)$}	
	\ncarc{->}{3}{0}\aput[0.05](0.6){\footnotesize $\left(\!\begin{smallmatrix}b_1B_1 + B_3 \\b_2(B_3-B_1A_1) \\ (b_1+b_2^2)B_1B_3 \end{smallmatrix}\!\right)$}	
	\ncarc[linecolor=red]{->}{1}{3}	
	\scalebox{0.75}{\rput(-0.55,1.5){\footnotesize $(b_1,\!b_2,\!1)$}} 	
	\ncarc{->}{3}{1}\aput[-0.05](0.65){\footnotesize $\left(\!\begin{smallmatrix}B_1 \\0 \\ B_3\end{smallmatrix}\!\right)$}	
	\ncarc[linecolor=red]{->}{2}{3}\aput[0.025](0.35){\tiny $(0,\!0,\!1)$}	
	\ncarc[linecolor=red]{->}{3}{2}\scalebox{0.85}{\rput(0.5,1.5){\footnotesize $\left(\begin{smallmatrix}1 \\0 \\ A_1\end{smallmatrix}\right)$}}	
	\nccircle[linecolor=red,angleA=120,nodesep=3pt]{->}{u}{.3cm}\bput[-0.05](0.45){\footnotesize $\left(\!\begin{smallmatrix}0&1&0 \\u_{21}& \omega b_2B_1&u_{23} \\ \omega b_2B_3&\text{-}\omega B_3 & \text{-}\omega b_2B_1 \end{smallmatrix}\!\right)$} 
	 \nccircle[linecolor=red,angleA=240,nodesep=3pt]{->}{v}{.3cm}\bput[-0.075](0.55){\footnotesize $\left(\!\begin{smallmatrix}0&1&0 \\ v_{21}&\omega^2 b_2B_1&v_{23} \\ \omega^2 b_2B_3 & \text{-}\omega^2 B_3 & \text{-}\omega^2 b_2B_1 \end{smallmatrix}\!\right)$} 

\rput(-2.5,-1){\tiny $b_{1}=b_2^2B_1-B_3$}
\rput(-2.5,-1.25){\tiny $u_{21}=-\omega^2 b_1B_1+B_3$}
\rput(-2.5,-1.5){\tiny $u_{23}=\omega^2 +\omega B_1$}
\rput(-2.5,-1.75){\tiny $v_{21}=-\omega b_1B_1+B_3$}
\rput(-2.5,-2){\tiny $v_{23}=\omega +\omega^2 B_1$}
\rput(-2.75,1.15){\small $U'_2\cong\C^3_{b_2,B_1,B_3}$}

	}}
\rput(9.5,1){
\scalebox{1.1}{
	\rput(0,-2){\rnode{0}{\scriptsize $0$}}
	\rput(0,0){\rnode{3}{\scriptsize $3$}}
	\rput(-1.5,1.5){\rnode{1}{\scriptsize $1$}}
	\rput(1.5,1.5){\rnode{2}{\scriptsize $2$}}
	\rput(-0.3,0){\rnode{u}{}}	
	\rput(0.3,0){\rnode{v}{}}	
	\ncarc[linecolor=red]{->}{0}{3}\Aput[0.05]{\tiny $(1,\!0,\!0)$}	
	\ncarc{->}{3}{0}\aput[0.05](0.6){\footnotesize $\left(\begin{smallmatrix}b_1B_1 + b_3B_3 \\B_3-B_1A_1 \\ (b_1b_3+1)B_1B_3 \end{smallmatrix}\right)$}	
	\ncarc[linecolor=red]{->}{1}{3}	
	\scalebox{0.75}{\rput(-0.55,1.5){\footnotesize $(b_1,\!1,\!b_3)$}}  
	\ncarc{->}{3}{1}\aput[-0.05](0.65){\footnotesize $\left(\begin{smallmatrix}B_1 \\0 \\ B_3\end{smallmatrix}\right)$}	
	\ncarc[linecolor=red]{->}{2}{3}\aput[0.025](0.35){\tiny $(0,\!0,\!1)$}	
	\ncarc[linecolor=red]{->}{3}{2} \scalebox{0.85}{\rput(0.5,1.5){\footnotesize $\left(\begin{smallmatrix}1 \\0 \\ A_1\end{smallmatrix}\right)$}}	
	\nccircle[linecolor=red,angleA=120,nodesep=3pt]{->}{u}{.3cm}\bput[-0.075](0.45){\footnotesize $\left(\!\begin{smallmatrix}0&1&0 \\ u_{21}& \omega B_1 & u_{23} \\ \omega B_3 & \text{-}\omega b_3B_3 & \text{-}\omega B_1 \end{smallmatrix}\!\right)$} 
	 \nccircle[linecolor=red,angleA=240,nodesep=3pt]{->}{v}{.3cm}\bput[-0.05](0.55){\footnotesize $\left(\!\begin{smallmatrix}0&1&0 \\ v_{21}& \omega^2 B_1 & v_{23} \\ \omega^2 B_3 & \text{-}\omega^2 b_3B_3 & \text{-}\omega^2 B_1 \end{smallmatrix}\!\right)$} 

\rput(-2.5,-1){\tiny $b_{1}=B_1-b_3^2B_3$}
\rput(-2.5,-1.25){\tiny $u_{21}=-\omega^2 b_1B_1+b_3B_3$}
\rput(-2.5,-1.5){\tiny $u_{23}=\omega^2 +\omega b_3B_1$}
\rput(-2.5,-1.75){\tiny $v_{21}=-\omega b_1B_1+b_3B_3$}
\rput(-2.5,-2){\tiny $v_{23}=\omega +\omega^2 b_3B_1$}
\rput(-2.75,1.15){\small $U'_3\cong\C^3_{b_3,B_1,B_3}$}
	}}
	
\end{pspicture}
\end{center}
\end{figure}

{\em Second flop $X_{12}$.} In $X_1$ we can flop $E'_2$ obtaining $X_0$ back, or $E_1$. In the latter case we get the new curve $E'_1$ covered by $U'_{0}$ and $U'_{1}$, both of the isomorphic to $\C^3$. The conditions for the new open sets are:
\[
\left.\begin{array}{cccccc}
U_0' & c_2=1 & aB=1 & a=(1,0,0) & au=(0,1,0) & b=(0,0,1)\\
U_1' & c_3=1 & aB=1 & a=(1,0,0) & au=(0,1,0) & b=(0,0,1)
\end{array}\right.
\]

\begin{figure}[h]
\begin{center}
\begin{pspicture}(0,-1.1)(10,2.5)
	\psset{arcangle=15,nodesep=2pt}

\rput(1.15,0.75){
\scalebox{1.1}{
	\rput(0,-2){\rnode{0}{\scriptsize $0$}}
	\rput(0,0){\rnode{3}{\scriptsize $3$}}
	\rput(-1.5,1.5){\rnode{1}{\scriptsize $1$}}
	\rput(1.5,1.5){\rnode{2}{\scriptsize $2$}}
	\rput(-0.3,0){\rnode{u}{}}	
	\rput(0.3,0){\rnode{v}{}}	
	\ncarc[linecolor=red]{->}{0}{3}\Aput[0.05]{\tiny $(1,\!0,\!0)$}	
	\ncarc{->}{3}{0}\aput[0.05](0.6){\footnotesize $\left(\!\begin{smallmatrix}c_1C_1 + C_3 \\C_3-C_1A_1 \\ (c_1c_3+1)C_1C_3 \end{smallmatrix}\!\right)$}	
	\ncarc[linecolor=red]{->}{1}{3}	
	\scalebox{0.75}{\rput(-0.55,1.5){\footnotesize $(0,\!0,\!1)$}} 	
	\ncarc[linecolor=red]{->}{3}{1}\aput[-0.05](0.65){\footnotesize $\left(\!\begin{smallmatrix}1 \\0 \\ A_1\end{smallmatrix}\!\right)$}	
	\ncarc[linecolor=red]{->}{2}{3}\aput[0.025](0.35){\tiny $(C_1-c_3^2C_3,\!1,\!c_3)$}	
	\ncarc{->}{3}{2}\scalebox{0.85}{\rput(0.5,1.5){\footnotesize $\left(\begin{smallmatrix}C_1 \\0 \\ C_3\end{smallmatrix}\right)$}}	
	\nccircle[linecolor=red,angleA=120,nodesep=3pt]{->}{u}{.3cm}\bput[-0.05](0.45){\footnotesize $\left(\!\begin{smallmatrix}0&1&0 \\u_{21}& \omega^2 C_1 & u_{23} \\ \omega^2 C_3 & \text{-}\omega^2 c_3C_3 & \text{-}\omega^2 C_1 \end{smallmatrix}\!\right)$} 
	 \nccircle[linecolor=red,angleA=240,nodesep=3pt]{->}{v}{.3cm}\bput[-0.05](0.55){\footnotesize $\left(\!\begin{smallmatrix}0&1&0 \\ v_{21} & \omega C_1 & v_{23} \\ \omega C_3 & \text{-}\omega c_3C_3 & \text{-}\omega C_1 \end{smallmatrix}\!\right)$} 
	 
\rput(-2.5,-1.25){\tiny $u_{21}=-\omega c_1C_1+c_3C_3 $}
\rput(-2.5,-1.5){\tiny $u_{23}=\omega +\omega^2 c_3C_1$}
\rput(-2.5,-1.75){\tiny $v_{21}=-\omega^2 c_1C_1+c_3C_3$}
\rput(-2.5,-2){\tiny $v_{23}=\omega^2 +\omega c_3C_1$}
\rput(-2.75,1.15){\small $U'_0\cong\C^3_{c_3,C_1,C_3}$}

	}}
\rput(9.5,0.75){
\scalebox{1.1}{
	\rput(0,-2){\rnode{0}{\scriptsize $0$}}
	\rput(0,0){\rnode{3}{\scriptsize $3$}}
	\rput(-1.5,1.5){\rnode{1}{\scriptsize $1$}}
	\rput(1.5,1.5){\rnode{2}{\scriptsize $2$}}
	\rput(-0.3,0){\rnode{u}{}}	
	\rput(0.3,0){\rnode{v}{}}	
	\ncarc[linecolor=red]{->}{0}{3}\Aput[0.05]{\tiny $(1,\!0,\!0)$}	
	\ncarc{->}{3}{0}\aput[0.05](0.6){\footnotesize $\left(\!\begin{smallmatrix}c_1C_1 + C_3 \\c_2(C_3-C_1A_1) \\ (c_1+c_2^2)C_1C_3 \end{smallmatrix}\!\right)$}	
	\ncarc[linecolor=red]{->}{1}{3}	
	\scalebox{0.75}{\rput(-0.55,1.5){\footnotesize $(0,\!0,\!1)$}}  
	\ncarc[linecolor=red]{->}{3}{1}\aput[-0.05](0.65){\footnotesize $\left(\!\begin{smallmatrix}1 \\0 \\ A_1\end{smallmatrix}\!\right)$}	
	\ncarc[linecolor=red]{->}{2}{3}\aput[0.025](0.35){\tiny $(c_2^2C_1-C_3,\!c_2,\!1)$}	
	\ncarc{->}{3}{2} \scalebox{0.85}{\rput(0.5,1.5){\footnotesize $\left(\begin{smallmatrix}C_1 \\0 \\ C_3\end{smallmatrix}\right)$}}	
	\nccircle[linecolor=red,angleA=120,nodesep=3pt]{->}{u}{.3cm}\bput[-0.05](0.45){\footnotesize $\left(\!\begin{smallmatrix}0&1&0 \\u_{21} & \omega^2 c_2C_1 & u_{23} \\ \omega^2 c_2C_3 & \text{-}\omega^2 C_3 & \text{-}\omega^2 c_2C_1 \end{smallmatrix}\!\right)$} 
	 \nccircle[linecolor=red,angleA=240,nodesep=3pt]{->}{v}{.3cm}\bput[-0.05](0.55){\footnotesize $\left(\!\begin{smallmatrix}0&1&0 \\ v_{21} & \omega c_2C_1 & v_{23} \\ \omega c_2C_3 & \text{-}\omega C_3 & \text{-}\omega c_2C_1 \end{smallmatrix}\!\right)$} 

\rput(-2.5,-1.25){\tiny $u_{21}=-\omega c_1C_1+C_3$}
\rput(-2.5,-1.5){\tiny $u_{23}=\omega +\omega^2 C_1$}
\rput(-2.5,-1.75){\tiny $v_{21}=-\omega^2 c_1C_1+C_3$}
\rput(-2.5,-2){\tiny $v_{23}=\omega^2 +\omega C_1$}
\rput(-2.75,1.15){\small $U'_1\cong\C^3_{c_2,C_1,C_3}$}
	}}
\end{pspicture}
\end{center}
\end{figure}

{\em Third flop $X_{123}$.} The degree of the normal bundle of the curve $E_3$ in $X_{12}$ is now $(-1,-1)$ so we can perform the last flop. We obtain the open sets $U''_{1}$ and $U''_{2}$ given by:
\[
\left.\begin{array}{cccccc}
U''_1 & c_2=c_3=1 & a=(1,0,0) & au=(0,1,0) & b=(0,0,1)\\
U''_2 & c_1=c_3=1 & a=(1,0,0) & au=(0,1,0) & b=(0,0,1)
\end{array}\right.
\]

\begin{figure}[h]
\begin{center}
\begin{pspicture}(0,-1.25)(10,3)
	\psset{arcangle=15,nodesep=2pt}

\rput(1.15,1.25){
\scalebox{1.1}{
	\rput(0,-2){\rnode{0}{\scriptsize $0$}}
	\rput(0,0){\rnode{3}{\scriptsize $3$}}
	\rput(-1.5,1.5){\rnode{1}{\scriptsize $1$}}
	\rput(1.5,1.5){\rnode{2}{\scriptsize $2$}}
	\rput(-0.3,0){\rnode{u}{}}	
	\rput(0.3,0){\rnode{v}{}}	
	\ncarc[linecolor=red]{->}{0}{3}\Aput[0.05]{\tiny $(1,\!0,\!0)$}	
	\ncarc{->}{3}{0}\aput[0.05](0.6){\footnotesize $\left(\begin{smallmatrix}c_1C_1 + C_3 \\-C_1A_1+B_1C_3 \\ C_1(A_1-c_1^2) \end{smallmatrix}\right)$}	
	\ncarc[linecolor=red]{->}{1}{3}	
	\scalebox{0.75}{\rput(-0.55,1.5){\footnotesize $(0,\!0,\!1)$}} 	
	\ncarc{->}{3}{1}\aput[-0.05](0.65){\footnotesize $\left(\begin{smallmatrix}B_1 \\0 \\ A_1\end{smallmatrix}\right)$}	
	\ncarc[linecolor=red]{->}{2}{3}\aput[0.025](0.35){\tiny $(c_1,\!1,\!1)$}	
	\ncarc{->}{3}{2}\scalebox{0.85}{\rput(0.5,1.5){\footnotesize $\left(\begin{smallmatrix}C_1 \\0 \\ C_3\end{smallmatrix}\right)$}}	
	\nccircle[linecolor=red,angleA=120,nodesep=3pt]{->}{u}{.3cm}\bput[-0.05](0.45){\footnotesize $\left(\!\begin{smallmatrix}0&1&0 \\ C_3\text{-}\omega C_1c_1 & \omega^2 C_1 & u_{23} \\ \omega^2(A_1\text{-}C_1c_1) & u_{32} & \text{-}\omega^2 C_1 \end{smallmatrix}\!\right)$} 
	 \nccircle[linecolor=red,angleA=240,nodesep=3pt]{->}{v}{.3cm}\bput[-0.05](0.55){\footnotesize $\left(\!\begin{smallmatrix}0&1&0 \\ v_{21} & \omega C_1 & v_{23} \\ v_{31} & v_{32} & \text{-}\omega C_1 \end{smallmatrix}\!\right)$} 
	 
\rput(-2.5,-1){\tiny $C_{3}=B_1(C_1-c_1)$}	 
\rput(-2.5,-1.25){\tiny $u_{23}=\omega^2C_1+\omega B_1$}
\rput(-2.5,-1.5){\tiny $u_{32}=-\omega^2(C_1-c_1)$}
\rput(-2.5,-1.75){\tiny $v_{21}=C_3-\omega^2C_1c_1$}
\rput(-2.5,-2){\tiny $v_{23}=\omega C_1+\omega^2B_1$}
\rput(-2.5,-2.25){\tiny $v_{31}=\omega(A_1-C_1c_1)$}
\rput(-2.5,-2.5){\tiny $v_{32}=-\omega(C_1-c_1)$}

\rput(-2.75,1.15){\small $U''_1\cong\C^3_{B_1,c_1,C_1}$}
	}}
	
\rput(9.5,1.25){
\scalebox{1.1}{
	\rput(0,-2){\rnode{0}{\scriptsize $0$}}
	\rput(0,0){\rnode{3}{\scriptsize $3$}}
	\rput(-1.5,1.5){\rnode{1}{\scriptsize $1$}}
	\rput(1.5,1.5){\rnode{2}{\scriptsize $2$}}
	\rput(-0.3,0){\rnode{u}{}}	
	\rput(0.3,0){\rnode{v}{}}	
	\ncarc[linecolor=red]{->}{0}{3}\Aput[0.05]{\tiny $(1,\!0,\!0)$}	
	\ncarc{->}{3}{0}\aput[0.05](0.6){\footnotesize $\left(\!\begin{smallmatrix}C_1+C_3 \\-c_2C_1A_1+c_2B_1C_3 \\ C_1(c_2^2A_1-1) \end{smallmatrix}\!\right)$}	
	\ncarc[linecolor=red]{->}{1}{3}	
	\scalebox{0.75}{\rput(-0.55,1.5){\footnotesize $(0,\!0,\!1)$}}  
	\ncarc{->}{3}{1}\aput[-0.05](0.65){\footnotesize $\left(\!\begin{smallmatrix}B_1 \\0 \\ A_1\end{smallmatrix}\!\right)$}	
	\ncarc[linecolor=red]{->}{2}{3}\aput[0.025](0.35){\tiny $(1,\!c_2,\!1)$}	
	\ncarc{->}{3}{2} \scalebox{0.85}{\rput(0.5,1.5){\footnotesize $\left(\begin{smallmatrix}C_1 \\0 \\ C_3\end{smallmatrix}\right)$}}	
	\nccircle[linecolor=red,angleA=120,nodesep=3pt]{->}{u}{.3cm}\bput[-0.05](0.45){\footnotesize $\left(\!\begin{smallmatrix}0&1&0 \\ u_{21} & \omega^2 c_2C_1 & u_{23} \\ \omega^2c_2C_3 & u_{32} & \text{-}\omega^2 c_2C_1 \end{smallmatrix}\!\right)$} 
	 \nccircle[linecolor=red,angleA=240,nodesep=3pt]{->}{v}{.3cm}\bput[-0.05](0.55){\footnotesize $\left(\!\begin{smallmatrix}0&1&0 \\ v_{21} & \omega c_2C_1 & v_{23} \\ \omega c_2C_3 & v_{32} & \text{-}\omega c_2C_1 \end{smallmatrix}\!\right)$} 

\rput(-2.5,-1){\tiny $C_{3}=B_1(c_2^2C_1-1)$}	  
\rput(-2.5,-1.25){\tiny $u_{21}=C_3-\omega C_1$}
\rput(-2.5,-1.5){\tiny $u_{23}=\omega^2C_1+\omega B_1$}
\rput(-2.5,-1.75){\tiny $u_{32}=-\omega^2(c_2^2C_1-1)$}
\rput(-2.5,-2){\tiny $v_{21}=-\omega^2C_1+C_3$}
\rput(-2.5,-2.25){\tiny $v_{23}=\omega C_1+\omega^2B_1$}
\rput(-2.5,-2.5){\tiny $v_{32}=-\omega(c_2^2C_1-1)$}

\rput(-2.75,1.15){\small $U''_2\cong\C^3_{B_1,c_2,C_1}$}

	}}

\end{pspicture}
\end{center}
\end{figure}

The normal bundles of the rational curves in the fibre over the origin are obtained by the explicit gluings among the open sets covering the curves. These gluings are given below and the result follows.
\[
\begin{small}
\begin{array}{l}
U_0\to U_{1}: (c_2,c_3,C_1) \mapsto (c_2C_1, c_3C_1, C_1^{-1})  \\
U_1\to U_2: (c_2,c_3,C_3) \mapsto (c_2c_3^{-1}, c_3^{-1}, c_3^3C_3-3c_2^2(c_3-1)) \\
U_2\to U_3: (b_2,b_3,B_3) \mapsto (b_2B_3, b_3B_3, B_3^{-1}) \\
U_0\to U_1: (c_2,c_3,C_1) \mapsto (c_2C_1, c_3C_1, C_1^{-1}) \\
U_1\to U'_2: (c_2,c_3,C_3) \mapsto (c_2, c_3^{-1}, c_3^2C_3-3c_2^2(1-c_3^{-1})) \\
U'_2\to U'_3: (b_2,B_1,B_3) \mapsto (b_2^{-1}, b_2B_1, b_2B_3)  \\
U'_0\to U'_1: (c_3,C_1,C_3) \mapsto (c_3^{-1},c_3C_1,c_3C_3) \\
U'_1\to U'_2: (c_2,C_1,C_3) \mapsto (c_2C_1,C_1^{-1},C_1C_3-3c_2^2C_1(C_1-1))) \\
U'_2\to U'_3: (b_2,B_1,B_3) \mapsto (b_2^{-1},b_2B_1,b_2B_3) \\
U'_0\to U''_1: (c_3,C_1,C_3) \mapsto (c_3^{-1},c_3^2C_3-3C_1,C_1) \\
U''_2\to U''_1: (B_1,c_2,C_1) \mapsto (B_1c_2,c_2^{-1},c_2C_1) \\
U'_3\to U''_1: (b_3,B_1,B_3) \mapsto (B_1,b_3^2B_3-3B_1,b_3^{-1}) \\
\end{array}
\end{small}
\]
\end{proof}

\subsection{Proof of Theorem \ref{thm:main}}\label{Proof:Main} The proof is explicit and it follows from the direct comparison between every mutation of $(Q,W)$ at non-trivial vertices with no loops and the description of every crepant resolution of $\C^3/G$ given in Sections \ref{Sect:Mut-SO(3)} and \ref{sect:opens} respectively. 

The case $G\cong\Z/n\Z$ is immediate, there are no flops of $X$ and no mutations of $(Q,W)$ since every vertex has a loop. 

For the rest of the cases, note that in every projective crepant resolution $\pi:X\to\C^3/G$ the dual graph of the exceptional fibre $\pi^{-1}(0)=\bigcup E_i$ (described in the part (3) of Theorems \ref{OpensDnOdd}, \ref{OpensDnEven} and part (2) in Theorem \ref{OpensE6}) coincide with the graph associated to the corresponding mutated quiver in Section \ref{Sect:Mut-SO(3)} removing the trivial vertex. Recall that the {\em graph} of a quiver $Q$ is obtained by forgetting the direction of the arrows. More precisely, 
\[ 
\begin{array}{ll}
\text{For $G\cong D_{2n}$, $n$ odd:} & \text{Dual graph of $\pi_i^{-1}(0)$ = Graph of $Q_i\backslash0$} \\
\text{For $G\cong D_{2n}$, $n$ even:} & \text{Dual graph of $\pi_{ij}^{-1}(0)$ = Graph of $Q_{0\ldots i}^{m\ldots(m-j)}\backslash0$} \\
\text{For $G\cong\mathbb{T}$:} & \text{Dual graph of $\pi^{-1}(0)$ = Graph of $Q_i\backslash0$}
\end{array}
\]
and it follows that flopping the curve $E_i$ corresponds to mutate with respect to the vertex $i$. This also proves Corollary \ref{cor:DualGraph}. 

\begin{rmk}\label{rem:LoopDeg}
Part (ii) of Corollary \ref{cor:DualGraph} also follows by direct comparison, although it is an expected fact since the dimension of the fibres is one. Indeed, since $D^b(X)\cong D^b(\Lambda)$ with $\Lambda:=\End_R(M)$ we have that
\begin{align*}
\text{loops at a vertex $k$} &= \dim_\C\Ext^1_A(S_k,S_k) \\
	&= \dim_\C\Hom_{D^b(\Lambda)}(S_i,S_i[1]) \\
	&= \dim_\C\Hom_{D^b(X)}(\mathcal{O}_{E_k}(-1),\mathcal{O}_{E_k}(-1)[1]) \\
	&= \dim_\C\Ext^1_X(\OO_{E_k},\OO_{E_k}) \\
	&= \dim_\C H^0(\mathcal{N}_{E_k}|_X)
\end{align*}
and the three possible cases give $H^0(\OO(-1)\oplus\OO(-1))=0$, $H^0(\OO(-2)\oplus\OO)=\C$ and $H^0(\OO(-3)\oplus\OO(1))=\C^2$. We want to thank M. Wemyss for explaining this fact to us. 
\end{rmk}

\begin{ex}
In the case $D_{2n}$ for $n=7$ there are 4 non-equivalent QPs, as shown in Figure \ref{D-14}.

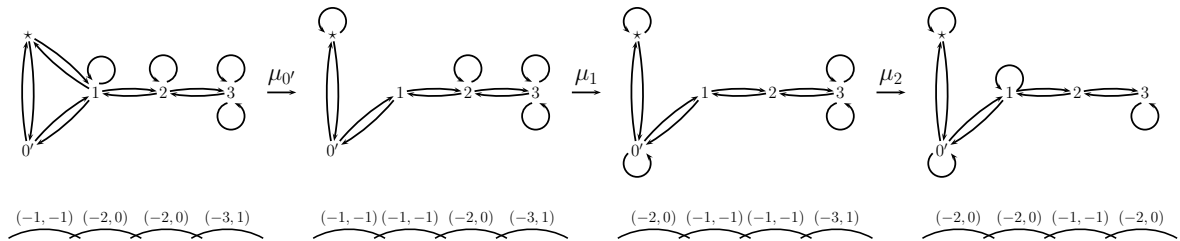
\begin{figure}[htbp]
\begin{center}
\scalebox{0.45}{
\begin{pspicture}(-1,-4.5)(34,2)
	\psset{arcangle=10,nodesep=2pt,linewidth=1.5pt}

\rput(0,0){
	\rput(-0,1.7){\rnode{0}{\Large$\star$}}
	\rput(-0,-1.7){\rnode{0'}{\Large$0'$}}
	\rput(2,0){\rnode{1}{\Large$1$}}
	\rput(4,0){\rnode{2}{\Large$2$}}
	\rput(6,0){\rnode{3}{\Large$3$}}
	\rput(2,0.3){\rnode{u1}{}}
	\rput(4,0.3){\rnode{u2}{}}
	\rput(6,0.3){\rnode{u3}{}}	
	\rput(6,-0.3){\rnode{u3'}{}}		
	\ncarc{->}{0}{0'}\ncarc{->}{0'}{0}\ncarc{->}{0'}{1}\ncarc{->}{1}{0'}
	\ncarc{->}{0}{1}\ncarc{->}{1}{0}\ncarc{->}{1}{2}\ncarc{->}{2}{1}
	\ncarc{->}{2}{3}\ncarc{->}{3}{2}
	\nccircle[angleA=-25,nodesep=3pt]{->}{u1}{.4cm}	\nccircle[angleA=0,nodesep=3pt]{->}{u2}{.4cm}
	\nccircle[angleA=0,nodesep=3pt]{->}{u3}{.4cm}
	\nccircle[angleA=180,nodesep=3pt]{->}{u3'}{.4cm}
}

\rput(9,0){
	\rput(-0,1.7){\rnode{0}{\Large$\star$}}
	\rput(-0,-1.7){\rnode{0'}{\Large$0'$}}
	\rput(2,0){\rnode{1}{\Large$1$}}
	\rput(4,0){\rnode{2}{\Large$2$}}
	\rput(6,0){\rnode{3}{\Large$3$}}
	\rput(2,0.3){\rnode{u1}{}}
	\rput(4,0.3){\rnode{u2}{}}
	\rput(6,0.3){\rnode{u3}{}}	
	\rput(6,-0.3){\rnode{u3'}{}}		
	\ncarc{->}{0}{0'}\ncarc{->}{0'}{0}\ncarc{->}{0'}{1}\ncarc{->}{1}{0'}
	\ncarc{->}{1}{2}\ncarc{->}{2}{1}\ncarc{->}{2}{3}\ncarc{->}{3}{2}
	\nccircle[angleA=0,nodesep=3pt]{->}{0}{.4cm}
	 \nccircle[angleA=0,nodesep=3pt]{->}{u2}{.4cm}
	\nccircle[angleA=0,nodesep=3pt]{->}{u3}{.4cm}
	\nccircle[angleA=180,nodesep=3pt]{->}{u3'}{.4cm}
}

\rput(18,0){
	\rput(-0,1.7){\rnode{0}{\Large$\star$}}
	\rput(-0,-1.7){\rnode{0'}{\Large$0'$}}
	\rput(2,0){\rnode{1}{\Large$1$}}
	\rput(4,0){\rnode{2}{\Large$2$}}
	\rput(6,0){\rnode{3}{\Large$3$}}
	\rput(2,0.3){\rnode{u1}{}}
	\rput(4,0.3){\rnode{u2}{}}
	\rput(6,0.3){\rnode{u3}{}}	
	\rput(6,-0.3){\rnode{u3'}{}}		
	\ncarc{->}{0}{0'}\ncarc{->}{0'}{0}\ncarc{->}{0'}{1}\ncarc{->}{1}{0'}
	\ncarc{->}{1}{2}\ncarc{->}{2}{1}\ncarc{->}{2}{3}\ncarc{->}{3}{2}
	\nccircle[angleA=0,nodesep=3pt]{->}{0}{.4cm}
	\nccircle[angleA=180,nodesep=3pt]{->}{0'}{.4cm}
	\nccircle[angleA=0,nodesep=3pt]{->}{u3}{.4cm}
	\nccircle[angleA=180,nodesep=3pt]{->}{u3'}{.4cm}
}

\rput(27,0){
	\rput(-0,1.7){\rnode{0}{\Large$\star$}}
	\rput(-0,-1.7){\rnode{0'}{\Large$0'$}}
	\rput(2,0){\rnode{1}{\Large$1$}}
	\rput(4,0){\rnode{2}{\Large$2$}}
	\rput(6,0){\rnode{3}{\Large$3$}}
	\rput(2,0.3){\rnode{u1}{}}
	\rput(4,0.3){\rnode{u2}{}}
	\rput(6,0.3){\rnode{u3}{}}	
	\rput(6,-0.3){\rnode{u3'}{}}		
	\ncarc{->}{0}{0'}\ncarc{->}{0'}{0}\ncarc{->}{0'}{1}\ncarc{->}{1}{0'}
	\ncarc{->}{1}{2}\ncarc{->}{2}{1}\ncarc{->}{2}{3}\ncarc{->}{3}{2}
	\nccircle[angleA=0,nodesep=3pt]{->}{0}{.4cm}
	\nccircle[angleA=180,nodesep=3pt]{->}{0'}{.4cm}
	\nccircle[angleA=180,nodesep=3pt]{->}{u3'}{.4cm}
	\nccircle[angleA=0,nodesep=3pt]{->}{1}{.4cm}
}
\rput(2,-3){
	\scalebox{0.9}{
	\begin{pspicture}(-1,-0.5)(7,1)
	\psset{arcangle=30,nodesep=2pt}
	\rput(0,-1.25){
		\rput(0,0){\rnode{0}{}}
		\rput(2.5,0){\rnode{1}{}}
		\rput(2,0){\rnode{2}{}}
		\rput(4.5,0){\rnode{3}{}}
		\rput(4,0){\rnode{4}{}}
		\rput(6.5,0){\rnode{5}{}}	
		\rput(6,0){\rnode{6}{}}
		\rput(8.5,0){\rnode{7}{}}	

		\ncarc{-}{0}{1}\Aput[0.05,npos=0.5]{\Large $(-1,-1)$}
		\ncarc{-}{2}{3}\Aput[0.05,npos=0.5]{\Large $(-2,0)$}
		\ncarc{-}{4}{5}\Aput[0.05,npos=0.5]{\Large $(-2,0)$}
		\ncarc{-}{6}{7}\Aput[0.05,npos=0.5]{\Large $(-3,1)$}
	}
	\end{pspicture}
	}
}
\rput(11,-3){
	\scalebox{0.9}{
	\begin{pspicture}(-1,-0.5)(7,1)
	\psset{arcangle=30,nodesep=2pt}
	\rput(0,-1.25){
		\rput(0,0){\rnode{0}{}}
		\rput(2.5,0){\rnode{1}{}}
		\rput(2,0){\rnode{2}{}}
		\rput(4.5,0){\rnode{3}{}}
		\rput(4,0){\rnode{4}{}}
		\rput(6.5,0){\rnode{5}{}}	
		\rput(6,0){\rnode{6}{}}
		\rput(8.5,0){\rnode{7}{}}	

		\ncarc{-}{0}{1}\Aput[0.05,npos=0.5]{\Large $(-1,-1)$}
		\ncarc{-}{2}{3}\Aput[0.05,npos=0.5]{\Large $(-1,-1)$}
		\ncarc{-}{4}{5}\Aput[0.05,npos=0.5]{\Large $(-2,0)$}
		\ncarc{-}{6}{7}\Aput[0.05,npos=0.5]{\Large $(-3,1)$}
	}
	\end{pspicture}
	}
}
\rput(20,-3){
	\scalebox{0.9}{
	\begin{pspicture}(-1,-0.5)(7,1)
	\psset{arcangle=30,nodesep=2pt}
	\rput(0,-1.25){
		\rput(0,0){\rnode{0}{}}
		\rput(2.5,0){\rnode{1}{}}
		\rput(2,0){\rnode{2}{}}
		\rput(4.5,0){\rnode{3}{}}
		\rput(4,0){\rnode{4}{}}
		\rput(6.5,0){\rnode{5}{}}	
		\rput(6,0){\rnode{6}{}}
		\rput(8.5,0){\rnode{7}{}}	

		\ncarc{-}{0}{1}\Aput[0.05,npos=0.5]{\Large $(-2,0)$}
		\ncarc{-}{2}{3}\Aput[0.05,npos=0.5]{\Large $(-1,-1)$}
		\ncarc{-}{4}{5}\Aput[0.05,npos=0.5]{\Large $(-1,-1)$}
		\ncarc{-}{6}{7}\Aput[0.05,npos=0.5]{\Large $(-3,1)$}
	}
	\end{pspicture}
	}
}
\rput(29,-3){
	\scalebox{0.9}{
	\begin{pspicture}(-1,-0.5)(7,1)
	\psset{arcangle=30,nodesep=2pt}
	\rput(0,-1.25){
		\rput(0,0){\rnode{0}{}}
		\rput(2.5,0){\rnode{1}{}}
		\rput(2,0){\rnode{2}{}}
		\rput(4.5,0){\rnode{3}{}}
		\rput(4,0){\rnode{4}{}}
		\rput(6.5,0){\rnode{5}{}}	
		\rput(6,0){\rnode{6}{}}
		\rput(8.5,0){\rnode{7}{}}	

		\ncarc{-}{0}{1}\Aput[0.05,npos=0.5]{\Large $(-2,0)$}
		\ncarc{-}{2}{3}\Aput[0.05,npos=0.5]{\Large $(-2,0)$}
		\ncarc{-}{4}{5}\Aput[0.05,npos=0.5]{\Large $(-1,-1)$}
		\ncarc{-}{6}{7}\Aput[0.05,npos=0.5]{\Large $(-2,0)$}

	}
	\end{pspicture}
	}
}

\rput(7,0){\rput(0,0){\rnode{0}{}}\rput(1,0){\rnode{1}{}}\ncline{->}{0}{1}\Aput{\huge $\mu_{0'}$}}
\rput(16,0){\rput(0,0){\rnode{0}{}}\rput(1,0){\rnode{1}{}}\ncline{->}{0}{1}\Aput{\huge $\mu_1$}}
\rput(25,0){\rput(0,0){\rnode{0}{}}\rput(1,0){\rnode{1}{}}\ncline{->}{0}{1}\Aput{\huge $\mu_2$}}
\end{pspicture}
}
\caption{Mutations $Q_C$ and the corresponding fibre over the origin in $\M_C$ for the dihedral group of type $D_{14}\subset\SO(3)$.}
\label{D-14}
\end{center}
\end{figure}

\end{ex}

\section{The space of stability conditions}
\label{Sect:Stability}

Let $G\subset\SO(3)$ of type $\Z/n\Z$, $D_{2n}$ or $\mathbb{T}$, let $Q$ be the McKay quiver, take $X\cong\mathcal{M}_C$ to be a projective crepant resolution of $\C^3/G$ for some chamber $C\subset\Theta$, and fix the open cover of $X$ given in Section \ref{sect:opens}. Then every open set $U\subset X$ is isomorphic to $\C^3_{a,b,c}$ where $a$, $b$ and $c$ are the local coordinates of $\C^3$ given in Theorems \ref{OpensDnOdd}, \ref{OpensDnEven} and \ref{OpensE6}. For a given point $(a,b,c)\in U$ we denote the corresponding representation by $M_{a,b,c}\in\mathcal{M}_C$. 

Fix $\theta\in C$. For any $\theta$-stable representation $M_{a,b,c}\in U$, the explicit knowledge of the representation space of $U$ gives every possible subrepresentation of $M_{a,b,c}$. In other words, the analysis of the matrices in the representation space of every open set in the open cover of $X$ give the inequalities defining the chamber $C\subset\Theta$. In order to do this analysis we encode the structure of the representation space of an open set by using its {\em skeleton}.

\begin{df} Under the above conditions, we call {\em skeleton $sk(U)$ of $U$} the representation of $Q$ corresponding to the origin ${\bf{0}}\in U$.
\end{df}

Once we choose basis for the vector spaces at every vertex of $Q$, the skeleton is obtained by setting $a=b=c=0$, i.e.\ $sk(U)=M_{0,0,0}$.

\begin{ex} Let $G=\frac{1}{3}(1,2,0)$ and consider $X=\Hilb{G}{\C^3}\cong\mathcal{M}_{C_0}$, where $C_0$ contains the $0$-generated stability condition. Then $X$ is covered by 3 open sets $U_i\cong\C^3$ for $i=1,2,3$ with skeletons
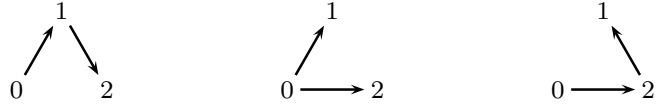
\begin{center}
\begin{scriptsize}
\begin{pspicture}(0,-0.15)(8,1.35)
	\psset{arcangle=15,nodesep=2pt}
\scalebox{1.2}{
\rput(0,0){
	\rput(0,0){\rnode{0}{0}}
	\rput(0.5,0.87){\rnode{1}{1}}
	\rput(1,0){\rnode{2}{2}}
	\ncline{->}{0}{1}\ncline{->}{1}{2}
	}
\rput(3,0){
	\rput(0,0){\rnode{0}{0}}
	\rput(0.5,0.87){\rnode{1}{1}}
	\rput(1,0){\rnode{2}{2}}
	\ncline{->}{0}{1}\ncline{->}{0}{2}
	}
\rput(6,0){
	\rput(0,0){\rnode{0}{0}}
	\rput(0.5,0.87){\rnode{1}{1}}
	\rput(1,0){\rnode{2}{2}}
	\ncline{->}{0}{2}\ncline{->}{2}{1}
	}}
\end{pspicture}
\end{scriptsize}
\end{center}
where only the non-zero arrows in $sk(U_i)$, for $i=1,2,3$, are represented in the figure. The chamber $C_0$ is therefore defined by $\theta_1,\theta_2>0$. 
\end{ex}

As a consequence of the next lemma, if there exists a finite open cover of $\mathcal{M}_C=\bigcup_{i=1}^NU_i$ and we define $C_{sk}:=\{ \theta\in\Theta | \theta(N)>0 \text{ for every $0\subsetneq N\subsetneq sk(U_i)$ and every $i$}\}$, then $C=C_{sk}$.

\begin{lem}
Let $\theta\in\Theta$ be a generic parameter. If $M_{0,0,0}$ is $\theta$-stable then $M_{a,b,c}$ is $\theta$-stable. 
\end{lem}

\begin{proof}
Let $(a,b,c)\in U$ and let $M_{a,b,c}$ be the corresponding representation. For every proper subrepresentation $N_{a,b,c}\subset M_{a,b,c}$ , the dimension vectors for $N_{a,b,c}$ and $N_{0,0,0}$ coincide. Therefore since $N_{0,0,0}\subset M_{0,0,0}$ and $M_{0,0,0}$ is $\theta$-stable we have that $\theta(N_{a,b,c})=\theta(N_{0,0,0})>0$.
\end{proof}

\begin{thm}\label{stability}
(i) Let $G=D_{2n}$ with $n$ odd and let $X_i$ be a crepant resolution of $\C^3/G$. The chamber $C_i\subset\Theta$ for which $X_i\cong\mathcal{M}_{C_i}$ is given by the inequalities:
\[
\begin{array}{l}
\theta_k>0 \text{ for $k\neq0,1$,} \\
\theta_1<0, \\
\theta_1+\theta_2<0, \\
~~~\vdots \\
\theta_1+\theta_2+\ldots+\theta_i<0, \\
\theta_1+\theta_2+\ldots+\theta_i+\theta_{i+1}>0.
\end{array}
\]
The wall between $C_i$ and $C_{i+1}$ is defined by $\theta_1+\theta_2+\ldots+\theta_i+\theta_{i+1}=0$.

(ii) Let $G=D_{2n}$ with $n$ even and let $X_{0\ldots i}^{m\ldots(m-j)}$ be a crepant resolution of $\C^3/G$. The chamber $C_{ij}\subset\Theta$ for which $X^{m..(m-j)}_{0..i}\cong\mathcal{M}_{C_{ij}}$ is given by the inequalities:

{\renewcommand{\arraystretch}{1.25}
\[
\begin{array}{c|c}
\theta_k>0 \text{ for $k\neq0,1,m+1$,}  & \theta_1<0, \theta_1+\theta_2<0, \ldots, \sum_{k=1}^i\theta_k<0, \\
\sum_{k=1}^{m+1}\theta_k>0,  & \theta_{m+1}<0, \theta_{m}+\theta_{m+1}<0, \ldots, \sum_{k=m-j+1}^{m+1}\theta_k<0, \\
\sum_{k=1}^{m}\theta_k+\theta_{m+2}>0,  & \sum_{k=1}^{i+1}\theta_k>0, \\

\theta_{m+1}+\theta_{m+2}>0,  & \sum_{k=m-j}^{m+1}\theta_k>0. 
\end{array}
\]}
The wall between $C_{i,j}$ and $C_{i+1,j}$ is defined by $\sum_{k=1}^{i+2}\theta_k=0$, and the wall between $C_{i,j}$ and $C_{i,j+1}$ is given by $\sum_{k=m-j+1}^{m+1}\theta_k=0$.

(iii) Let $G$ be the tetrahedral group of order 12 and let $X_i$ be a crepant resolution of $\C^3/G$. The chamber $C_{i}\subset\Theta$ for which $X_i\cong\mathcal{M}_{C_{i}}$ is given by the inequalities:
\[
\begin{array}{rl}
C_0: & \theta_i>0, i\neq0 \\
C_1: & \theta_1<0, \theta_2>0, \theta_1+\theta_3>0 \\
C_2: & \theta_1>0, \theta_2<0, \theta_2+\theta_3>0 \\
C_{12}: & \theta_1<0, \theta_2<0, \theta_1+\theta_2+\theta_3>0 \\
C_{123}: & \theta_1+\theta_3>0, \theta_2+\theta_3>0, \theta_1+\theta_2+\theta_3<0
\end{array}
\]
\end{thm}

\begin{proof} (i) Consider the open cover of $X_i$ given in Theorem \ref{OpensDnOdd} and let $M\in X_i$ be a representation of $Q$. We calculate for which parameters $\theta\in\Theta$ the representation $M$ is $\theta$-stable. By the representation spaces of every open set we can see that the skeletons for the open sets $U_i$, $U'_i$ and $U''_i$ are:

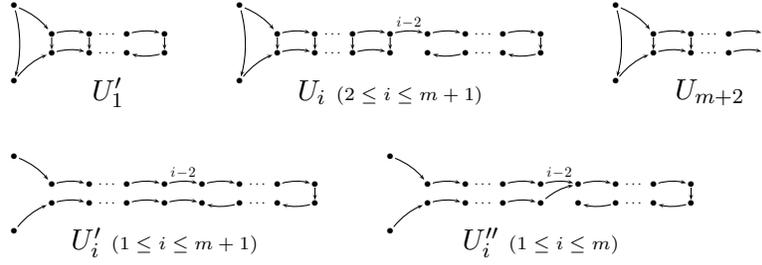
\begin{center}
\begin{pspicture}(0,-1)(10,2.5)

\rput(0,2){
\rput(1.25,-0.65){$U'_{1}$}
\scalebox{0.5}{
	\psset{arcangle=15,nodesep=1pt}
	\rput(0,1){\rnode{0}{$\bullet$}}		
	\rput(0,-1){\rnode{0'}{$\bullet$}}
	\rput(1,0.25){\rnode{1u}{$\bullet$}}
	\rput(1,-0.25){\rnode{1d}{$\bullet$}}
	\rput(2,0.25){\rnode{2u}{$\bullet$}}
	\rput(2,-0.25){\rnode{2d}{$\bullet$}}
	\rput(2.5,0.25){$\cdots$}
	\rput(2.5,-0.25){$\cdots$}
	\rput(3,0.25){\rnode{m-1u}{$\bullet$}}	
	\rput(3,-0.25){\rnode{m-1d}{$\bullet$}}
	\rput(4,0.25){\rnode{mu}{$\bullet$}}
	\rput(4,-0.25){\rnode{md}{$\bullet$}}
	
	\ncarc{->}{0}{0'} 	
	\ncarc{->}{0}{1u}
	\ncarc{->}{0'}{1d}
	\ncarc{->}{1u}{2u}
	\ncarc{->}{1d}{2d}
	\ncline[nodesep=0]{->}{1u}{1d}
	\ncline[nodesep=0]{->}{2u}{2d}
	
	\ncarc{->}{m-1u}{mu}	
	\ncline[nodesep=0]{->}{mu}{md}
	\ncarc{->}{md}{m-1d}	
	}}

\rput(3,2){
\rput(2,-0.65){$U_i$~\text{\tiny ($2\leq i\leq m+1$)}}
\scalebox{0.5}{
	\psset{arcangle=15,nodesep=1pt}
	\rput(0,1){\rnode{0}{$\bullet$}}	
	\rput(0,-1){\rnode{0'}{$\bullet$}}
	\rput(1,0.25){\rnode{1u}{$\bullet$}}
	\rput(1,-0.25){\rnode{1d}{$\bullet$}}
	\rput(2,0.25){\rnode{2u}{$\bullet$}}
	\rput(2,-0.25){\rnode{2d}{$\bullet$}}
	\rput(2.5,0.25){$\cdots$}
	\rput(2.5,-0.25){$\cdots$}
	\rput(3,0.25){\rnode{3u}{$\bullet$}}	
	\rput(3,-0.25){\rnode{3d}{$\bullet$}}
	\rput(4,0.25){\rnode{4u}{$\bullet$}}
	\rput(4,-0.25){\rnode{4d}{$\bullet$}}
	\rput(5,0.25){\rnode{5u}{$\bullet$}}
	\rput(5,-0.25){\rnode{5d}{$\bullet$}}
	\rput(6,0.25){\rnode{6u}{$\bullet$}}
	\rput(6,-0.25){\rnode{6d}{$\bullet$}}
	\rput(6.5,0.25){$\cdots$}
	\rput(6.5,-0.25){$\cdots$}
	\rput(7,0.25){\rnode{m-1u}{$\bullet$}}	
	\rput(7,-0.25){\rnode{m-1d}{$\bullet$}}
	\rput(8,0.25){\rnode{mu}{$\bullet$}}
	\rput(8,-0.25){\rnode{md}{$\bullet$}}
	
	\ncarc{->}{0}{0'}	
	\ncarc{->}{0}{1u}
	\ncarc{->}{0'}{1d}
	\ncarc{->}{1u}{2u}
	\ncarc{->}{1d}{2d}
	\ncline[nodesep=0]{->}{1u}{1d}
	\ncline[nodesep=0]{->}{2u}{2d}
	
	\ncarc{->}{3u}{4u}	
	\ncarc{->}{3d}{4d}
	\ncline[nodesep=0]{->}{3u}{3d}
	\ncarc{->}{4u}{5u}\Aput[0.05]{$i\!-\!2$}
	\ncline[nodesep=0]{->}{4u}{4d}
	\ncarc{->}{5u}{6u}
	\ncarc{->}{6d}{5d}
	
	\ncarc{->}{m-1u}{mu}	
	\ncline[nodesep=0]{->}{mu}{md}
	\ncarc{->}{md}{m-1d}	

	}}

\rput(8,2){
\rput(1.25,-0.65){$U_{m+2}$}
\scalebox{0.5}{
	\psset{arcangle=15,nodesep=1pt}
	\rput(0,1){\rnode{0}{$\bullet$}}		
	\rput(0,-1){\rnode{0'}{$\bullet$}}
	\rput(1,0.25){\rnode{1u}{$\bullet$}}
	\rput(1,-0.25){\rnode{1d}{$\bullet$}}
	\rput(2,0.25){\rnode{2u}{$\bullet$}}
	\rput(2,-0.25){\rnode{2d}{$\bullet$}}
	\rput(2.5,0.25){$\cdots$}
	\rput(2.5,-0.25){$\cdots$}
	\rput(3,0.25){\rnode{m-1u}{$\bullet$}}	
	\rput(3,-0.25){\rnode{m-1d}{$\bullet$}}
	\rput(4,0.25){\rnode{mu}{$\bullet$}}
	\rput(4,-0.25){\rnode{md}{$\bullet$}}
	
	\ncarc{->}{0}{0'} 	
	\ncarc{->}{0}{1u}
	\ncarc{->}{0'}{1d}
	\ncarc{->}{1u}{2u}
	\ncarc{->}{1d}{2d}
	\ncline[nodesep=0]{->}{1u}{1d}
	\ncline[nodesep=0]{->}{2u}{2d}
	
	\ncarc{->}{m-1u}{mu}	
	\ncarc{->}{m-1d}{md}	
	}}

\rput(0,0){
\rput(2,-0.65){$U'_i$~\text{\tiny ($1\leq i\leq m+1$)}}
\scalebox{0.5}{
	\psset{arcangle=15,nodesep=1pt}
	\rput(0,1){\rnode{0}{$\bullet$}}	
	\rput(0,-1){\rnode{0'}{$\bullet$}}
	\rput(1,0.25){\rnode{1u}{$\bullet$}}
	\rput(1,-0.25){\rnode{1d}{$\bullet$}}
	\rput(2,0.25){\rnode{2u}{$\bullet$}}
	\rput(2,-0.25){\rnode{2d}{$\bullet$}}
	\rput(2.5,0.25){$\cdots$}
	\rput(2.5,-0.25){$\cdots$}
	\rput(3,0.25){\rnode{3u}{$\bullet$}}	
	\rput(3,-0.25){\rnode{3d}{$\bullet$}}
	\rput(4,0.25){\rnode{4u}{$\bullet$}}
	\rput(4,-0.25){\rnode{4d}{$\bullet$}}
	\rput(5,0.25){\rnode{5u}{$\bullet$}}
	\rput(5,-0.25){\rnode{5d}{$\bullet$}}
	\rput(6,0.25){\rnode{6u}{$\bullet$}}
	\rput(6,-0.25){\rnode{6d}{$\bullet$}}
	\rput(6.5,0.25){$\cdots$}
	\rput(6.5,-0.25){$\cdots$}
	\rput(7,0.25){\rnode{m-1u}{$\bullet$}}	
	\rput(7,-0.25){\rnode{m-1d}{$\bullet$}}
	\rput(8,0.25){\rnode{mu}{$\bullet$}}
	\rput(8,-0.25){\rnode{md}{$\bullet$}}
	
	\ncarc{->}{0}{1u}
	\ncarc{->}{0'}{1d}
	\ncarc{->}{1u}{2u}
	\ncarc{->}{1d}{2d}
	
	\ncarc{->}{3u}{4u}	
	\ncarc{->}{3d}{4d}
	\ncarc{->}{4u}{5u}\Aput[0.05]{$i\!-\!2$}
	\ncarc{->}{4d}{5d}
	\ncarc{->}{5u}{6u}
	\ncarc{->}{6d}{5d}
	
	\ncarc{->}{m-1u}{mu}	
	\ncline[nodesep=0]{->}{mu}{md}
	\ncarc{->}{md}{m-1d}	
	}}

\rput(5,0){
\rput(2,-0.65){$U''_i$~\text{\tiny ($1\leq i\leq m$)}}
\scalebox{0.5}{
	\psset{arcangle=15,nodesep=1pt}
	\rput(0,1){\rnode{0}{$\bullet$}}	
	\rput(0,-1){\rnode{0'}{$\bullet$}}
	\rput(1,0.25){\rnode{1u}{$\bullet$}}
	\rput(1,-0.25){\rnode{1d}{$\bullet$}}
	\rput(2,0.25){\rnode{2u}{$\bullet$}}
	\rput(2,-0.25){\rnode{2d}{$\bullet$}}
	\rput(2.5,0.25){$\cdots$}
	\rput(2.5,-0.25){$\cdots$}
	\rput(3,0.25){\rnode{3u}{$\bullet$}}	
	\rput(3,-0.25){\rnode{3d}{$\bullet$}}
	\rput(4,0.25){\rnode{4u}{$\bullet$}}
	\rput(4,-0.25){\rnode{4d}{$\bullet$}}
	\rput(5,0.25){\rnode{5u}{$\bullet$}}
	\rput(5,-0.25){\rnode{5d}{$\bullet$}}
	\rput(6,0.25){\rnode{6u}{$\bullet$}}
	\rput(6,-0.25){\rnode{6d}{$\bullet$}}
	\rput(6.5,0.25){$\cdots$}
	\rput(6.5,-0.25){$\cdots$}
	\rput(7,0.25){\rnode{m-1u}{$\bullet$}}	
	\rput(7,-0.25){\rnode{m-1d}{$\bullet$}}
	\rput(8,0.25){\rnode{mu}{$\bullet$}}
	\rput(8,-0.25){\rnode{md}{$\bullet$}}
	
	\ncarc{->}{0}{1u}
	\ncarc{->}{0'}{1d}
	\ncarc{->}{1u}{2u}
	\ncarc{->}{1d}{2d}
	
	\ncarc{->}{3u}{4u}	
	\ncarc{->}{3d}{4d}
	\ncarc{->}{4u}{5u}\Aput[0.05]{$i\!-\!2$}
	\ncarc{->}{4d}{5u}
	\ncarc{->}{5u}{6u}
	\ncarc{->}{6d}{5d}
	
	\ncarc{->}{m-1u}{mu}	
	\ncline[nodesep=0]{->}{mu}{md}
	\ncarc{->}{md}{m-1d}	
	}}
\end{pspicture}
\end{center}

Every dot in the above picture corresponds to a basis element in the corresponding vector space in a representation of $Q$. Notice that the dimension vector is $\begin{smallmatrix}1\\1\end{smallmatrix}\text{\small{2\ldots2}}$ so that there is one dot for each 1-dimensional vertex and two dots for each 2-dimensional vertex.

We order the subindices of the stability condition $\theta:=(\theta_i)_{0\leq i\leq m+1}\in\Q^{|Q_0|}$ by the sequence $\begin{smallmatrix}0\\1\end{smallmatrix}\text{\small{2\ldots m+1}}$ along the vertices of $Q$. Let $s_i:=\begin{smallmatrix}0\\0\end{smallmatrix}\text{\small{0\ldots010\ldots0}}$ be the dimension vector with entry 1 at the position $i$. With this notation, we see that every $\theta$-stable submodule in the open sets $U''_1,\ldots,U''_i$ contains a submodule with dimension vector $s_2,\ldots,s_{i+2}$ respectively. Similarly, there exist a submodule of dimension vector $s_{i+3},\ldots,s_{m+2}$ in any $\theta$-stable module contained in $U_{i+2},\ldots,U_{m+2}$ respectively. Therefore we have that $\theta_i>0$ for $i\geq2$.

The rest of the condition follows by examining the remaining submodules. If $M_i\in U''_i$ then there exist a submodule $W_i\subset M_i$ with $\underline{\dim}(W_i)=\begin{smallmatrix}1\\0\end{smallmatrix}\text{\small{1\ldots12\ldots2}}$ where the first 2 is located in the position $i+1$. This imply that $\theta_1+\ldots+\theta_i<0$. Finally if $N_{i+1}\in U'_{i+1}$ then there exist a submodule $V_{i+1}\subset N_{i+1}$ with $\underline{\dim}(V_{i+1})=\begin{smallmatrix}0\\1\end{smallmatrix}\text{\small{1\ldots10\ldots0}}$ where the last 1 is located in the position $i+1$, which means that $\sum_{i=1}^{i+1}\theta_i>0$.    

Any other inequalities coming from the submodules of $M\in\mathcal{M}_C$ are implied by the ones we have just described, so the chamber $C$ is defined by the inequalities of the statement. By comparing the chamber conditions of $X_{0\ldots i}$ and $X_{0\ldots(i+1)}$ we obtain the equation of the wall.

(ii) This time we order the subindices of the stability condition $\theta:=(\theta_i)_{0\leq i\leq m+2}\in\Q^{|Q_0|}$ by the sequence ${\Large{\begin{smallmatrix}0\\1\end{smallmatrix}}}2\ldots m{\Large{\begin{smallmatrix}m\!+\!1\\m\!+\!2\end{smallmatrix}}}$ along the vertices of $Q$. The skeletons of the open sets in this case are as follows:

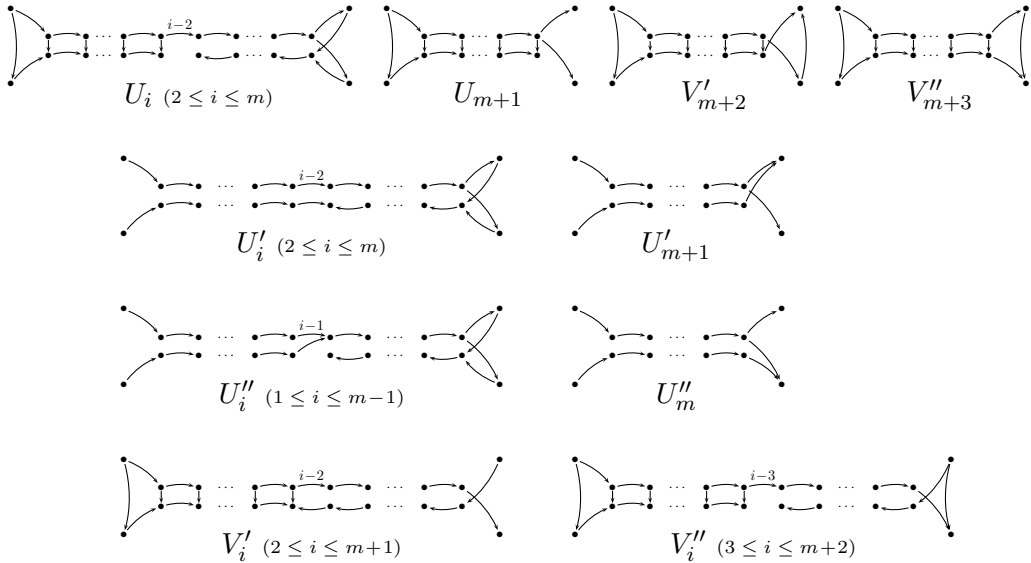
\begin{center}
\begin{pspicture}(0,-1)(14,6.5)

\rput(0,6){
\rput(2.5,-0.65){$U_i$~\text{\tiny ($2\leq i\leq m$)}}
\scalebox{0.5}{
	\psset{arcangle=15,nodesep=1pt}
	\rput(0,1){\rnode{0}{$\bullet$}}	
	\rput(0,-1){\rnode{0'}{$\bullet$}}
	\rput(1,0.25){\rnode{1u}{$\bullet$}}
	\rput(1,-0.25){\rnode{1d}{$\bullet$}}
	\rput(2,0.25){\rnode{2u}{$\bullet$}}
	\rput(2,-0.25){\rnode{2d}{$\bullet$}}
	\rput(2.5,0.25){$\cdots$}
	\rput(2.5,-0.25){$\cdots$}
	\rput(3,0.25){\rnode{3u}{$\bullet$}}	
	\rput(3,-0.25){\rnode{3d}{$\bullet$}}
	\rput(4,0.25){\rnode{4u}{$\bullet$}}
	\rput(4,-0.25){\rnode{4d}{$\bullet$}}
	\rput(5,0.25){\rnode{5u}{$\bullet$}}
	\rput(5,-0.25){\rnode{5d}{$\bullet$}}
	\rput(6,0.25){\rnode{6u}{$\bullet$}}
	\rput(6,-0.25){\rnode{6d}{$\bullet$}}
	\rput(6.5,0.25){$\cdots$}
	\rput(6.5,-0.25){$\cdots$}
	\rput(7,0.25){\rnode{m-2u}{$\bullet$}}	
	\rput(7,-0.25){\rnode{m-2d}{$\bullet$}}
	\rput(8,0.25){\rnode{m-1u}{$\bullet$}}
	\rput(8,-0.25){\rnode{m-1d}{$\bullet$}}
	\rput(9,1){\rnode{m}{$\bullet$}}
	\rput(9,-1){\rnode{m'}{$\bullet$}}
	
	\ncarc{->}{0}{0'}	
	\ncarc{->}{0}{1u}
	\ncarc{->}{0'}{1d}
	\ncarc{->}{1u}{2u}
	\ncarc{->}{1d}{2d}
	\ncline[nodesep=0]{->}{1u}{1d}
	\ncline[nodesep=0]{->}{2u}{2d}
	
	\ncarc{->}{3u}{4u}	
	\ncarc{->}{3d}{4d}
	\ncline[nodesep=0]{->}{3u}{3d}
	\ncarc{->}{4u}{5u}\Aput[0.05]{$i\!-\!2$}
	\ncline[nodesep=0]{->}{4u}{4d}
	\ncarc{->}{5u}{6u}
	\ncarc{->}{6d}{5d}
	
	\ncarc{->}{m-2u}{m-1u}	
	\ncarc{->}{m-1u}{m}
	\ncarc{->}{m-1u}{m'}
	\ncarc{->}{m}{m-1d}
	\ncarc{->}{m'}{m-1d}
	\ncarc{->}{m-1d}{m-2d}
	}}

\rput(5,6){
\rput(1.35,-0.65){$U_{m+1}$}
\scalebox{0.5}{
	\psset{arcangle=15,nodesep=1pt}
	\rput(0,1){\rnode{0}{$\bullet$}}		
	\rput(0,-1){\rnode{0'}{$\bullet$}}
	\rput(1,0.25){\rnode{1u}{$\bullet$}}
	\rput(1,-0.25){\rnode{1d}{$\bullet$}}
	\rput(2,0.25){\rnode{2u}{$\bullet$}}
	\rput(2,-0.25){\rnode{2d}{$\bullet$}}
	\rput(2.5,0.25){$\cdots$}
	\rput(2.5,-0.25){$\cdots$}
	\rput(3,0.25){\rnode{m-2u}{$\bullet$}}	
	\rput(3,-0.25){\rnode{m-2d}{$\bullet$}}
	\rput(4,0.25){\rnode{m-1u}{$\bullet$}}
	\rput(4,-0.25){\rnode{m-1d}{$\bullet$}}
	\rput(5,1){\rnode{m}{$\bullet$}}
	\rput(5,-1){\rnode{m'}{$\bullet$}}
	
	\ncarc{->}{0}{0'} 	
	\ncarc{->}{0}{1u}
	\ncarc{->}{0'}{1d}
	\ncarc{->}{1u}{2u}
	\ncarc{->}{1d}{2d}
	\ncline[nodesep=0]{->}{1u}{1d}
	\ncline[nodesep=0]{->}{2u}{2d}
	
	\ncarc{->}{m-2u}{m-1u}	
	\ncarc{->}{m-2d}{m-1d}
	\ncline[nodesep=0]{->}{m-2u}{m-2d}
	\ncarc{->}{m-1u}{m}
	\ncarc{->}{m-1u}{m'}
	\ncline[nodesep=0]{->}{m-1u}{m-1d}
	}}

\rput(8,6){
\rput(1.35,-0.65){$V'_{m+2}$}
\scalebox{0.5}{
	\psset{arcangle=15,nodesep=1pt}
	\rput(0,1){\rnode{0}{$\bullet$}}		
	\rput(0,-1){\rnode{0'}{$\bullet$}}
	\rput(1,0.25){\rnode{1u}{$\bullet$}}
	\rput(1,-0.25){\rnode{1d}{$\bullet$}}
	\rput(2,0.25){\rnode{2u}{$\bullet$}}
	\rput(2,-0.25){\rnode{2d}{$\bullet$}}
	\rput(2.5,0.25){$\cdots$}
	\rput(2.5,-0.25){$\cdots$}
	\rput(3,0.25){\rnode{m-2u}{$\bullet$}}	
	\rput(3,-0.25){\rnode{m-2d}{$\bullet$}}
	\rput(4,0.25){\rnode{m-1u}{$\bullet$}}
	\rput(4,-0.25){\rnode{m-1d}{$\bullet$}}
	\rput(5,1){\rnode{m}{$\bullet$}}
	\rput(5,-1){\rnode{m'}{$\bullet$}}
	
	\ncarc{->}{0}{0'} 	
	\ncarc{->}{0}{1u}
	\ncarc{->}{0'}{1d}
	\ncarc{->}{1u}{2u}
	\ncarc{->}{1d}{2d}
	\ncline{->}{1u}{1d}
	\ncline[nodesep=0]{->}{2u}{2d}
	
	\ncarc{->}{m-2u}{m-1u}	
	\ncarc{->}{m-2d}{m-1d}
	\ncline[nodesep=0]{->}{m-2u}{m-2d}
	\ncarc{->}{m-1u}{m'}
	\ncarc{->}{m-1d}{m}
	\ncline[nodesep=0]{->}{m-1u}{m-1d}
	\ncarc[arcangle=-15]{->}{m'}{m}
	}}

\rput(11,6){
\rput(1.35,-0.65){$V''_{m+3}$}
\scalebox{0.5}{
	\psset{arcangle=15,nodesep=1pt}
	\rput(0,1){\rnode{0}{$\bullet$}}		
	\rput(0,-1){\rnode{0'}{$\bullet$}}
	\rput(1,0.25){\rnode{1u}{$\bullet$}}
	\rput(1,-0.25){\rnode{1d}{$\bullet$}}
	\rput(2,0.25){\rnode{2u}{$\bullet$}}
	\rput(2,-0.25){\rnode{2d}{$\bullet$}}
	\rput(2.5,0.25){$\cdots$}
	\rput(2.5,-0.25){$\cdots$}
	\rput(3,0.25){\rnode{m-2u}{$\bullet$}}	
	\rput(3,-0.25){\rnode{m-2d}{$\bullet$}}
	\rput(4,0.25){\rnode{m-1u}{$\bullet$}}
	\rput(4,-0.25){\rnode{m-1d}{$\bullet$}}
	\rput(5,1){\rnode{m}{$\bullet$}}
	\rput(5,-1){\rnode{m'}{$\bullet$}}
	
	\ncarc{->}{0}{0'} 	
	\ncarc{->}{0}{1u}
	\ncarc{->}{0'}{1d}
	\ncarc{->}{1u}{2u}
	\ncarc{->}{1d}{2d}
	\ncline[nodesep=0]{->}{1u}{1d}
	\ncline[nodesep=0]{->}{2u}{2d}
	
	\ncarc{->}{m-2u}{m-1u}	
	\ncarc{->}{m-2d}{m-1d}
	\ncline[nodesep=0]{->}{m-2u}{m-2d}
	\ncarc{->}{m-1u}{m}
	\ncarc{->}{m-1d}{m'}
	\ncline[nodesep=0]{->}{m-1u}{m-1d}
	\ncarc[arcangle=-15]{->}{m}{m'}
	}}

\rput(1.5,4){
\rput(2.5,-0.65){$U'_i$~\text{\tiny $(2\leq i\leq m)$}}
\scalebox{0.5}{
	\psset{arcangle=15,nodesep=1pt}
	\rput(0,1){\rnode{0}{$\bullet$}}	
	\rput(0,-1){\rnode{0'}{$\bullet$}}
	\rput(1,0.25){\rnode{1u}{$\bullet$}}
	\rput(1,-0.25){\rnode{1d}{$\bullet$}}
	\rput(2,0.25){\rnode{2u}{$\bullet$}}
	\rput(2,-0.25){\rnode{2d}{$\bullet$}}
	\rput(2.75,0.25){$\cdots$}
	\rput(2.75,-0.25){$\cdots$}
	\rput(3.5,0.25){\rnode{3u}{$\bullet$}}	
	\rput(3.5,-0.25){\rnode{3d}{$\bullet$}}
	\rput(4.5,0.25){\rnode{4u}{$\bullet$}}
	\rput(4.5,-0.25){\rnode{4d}{$\bullet$}}
	\rput(5.5,0.25){\rnode{5u}{$\bullet$}}
	\rput(5.5,-0.25){\rnode{5d}{$\bullet$}}
	\rput(6.5,0.25){\rnode{6u}{$\bullet$}}
	\rput(6.5,-0.25){\rnode{6d}{$\bullet$}}
	\rput(7.25,0.25){$\cdots$}
	\rput(7.25,-0.25){$\cdots$}
	\rput(8,0.25){\rnode{m-2u}{$\bullet$}}	
	\rput(8,-0.25){\rnode{m-2d}{$\bullet$}}
	\rput(9,0.25){\rnode{m-1u}{$\bullet$}}
	\rput(9,-0.25){\rnode{m-1d}{$\bullet$}}
	\rput(10,1){\rnode{m}{$\bullet$}}
	\rput(10,-1){\rnode{m'}{$\bullet$}}
	
	\ncarc{->}{0}{1u}
	\ncarc{->}{0'}{1d}
	\ncarc{->}{1u}{2u}
	\ncarc{->}{1d}{2d}
	
	\ncarc{->}{3u}{4u}	
	\ncarc{->}{3d}{4d}
	\ncarc{->}{4u}{5u}\Aput[0.05]{$i\!-\!2$}
	\ncarc{->}{4d}{5d}
	\ncarc{->}{5u}{6u}
	\ncarc{->}{6d}{5d}
	
	\ncarc{->}{m-2u}{m-1u}	
	\ncarc{->}{m-1u}{m}
	\ncarc{->}{m-1u}{m'}
	\ncarc{->}{m}{m-1d}
	\ncarc{->}{m'}{m-1d}
	\ncarc{->}{m-1d}{m-2d}
	}}

\rput(7.5,4){
\rput(1.35,-0.65){$U'_{m+1}$}
\scalebox{0.5}{
	\psset{arcangle=15,nodesep=1pt}
	\rput(0,1){\rnode{0}{$\bullet$}}		
	\rput(0,-1){\rnode{0'}{$\bullet$}}
	\rput(1,0.25){\rnode{1u}{$\bullet$}}
	\rput(1,-0.25){\rnode{1d}{$\bullet$}}
	\rput(2,0.25){\rnode{2u}{$\bullet$}}
	\rput(2,-0.25){\rnode{2d}{$\bullet$}}
	\rput(2.75,0.25){$\cdots$}
	\rput(2.75,-0.25){$\cdots$}
	\rput(3.5,0.25){\rnode{m-2u}{$\bullet$}}	
	\rput(3.5,-0.25){\rnode{m-2d}{$\bullet$}}
	\rput(4.5,0.25){\rnode{m-1u}{$\bullet$}}
	\rput(4.5,-0.25){\rnode{m-1d}{$\bullet$}}
	\rput(5.5,1){\rnode{m}{$\bullet$}}
	\rput(5.5,-1){\rnode{m'}{$\bullet$}}
	
	\ncarc{->}{0}{1u}
	\ncarc{->}{0'}{1d}
	\ncarc{->}{1u}{2u}
	\ncarc{->}{1d}{2d}
	
	\ncarc{->}{m-2u}{m-1u}	
	\ncarc{->}{m-2d}{m-1d}
	\ncarc{->}{m-1u}{m}
	\ncarc{->}{m-1u}{m'}
	\ncarc{->}{m-1d}{m}
	}}

\rput(1.5,2){
\rput(2.5,-0.65){$U''_i$~\text{\tiny $(1\leq i\leq m\!-\!1)$}}
\scalebox{0.5}{
	\psset{arcangle=15,nodesep=1pt}
	\rput(0,1){\rnode{0}{$\bullet$}}	
	\rput(0,-1){\rnode{0'}{$\bullet$}}
	\rput(1,0.25){\rnode{1u}{$\bullet$}}
	\rput(1,-0.25){\rnode{1d}{$\bullet$}}
	\rput(2,0.25){\rnode{2u}{$\bullet$}}
	\rput(2,-0.25){\rnode{2d}{$\bullet$}}
	\rput(2.75,0.25){$\cdots$}
	\rput(2.75,-0.25){$\cdots$}
	\rput(3.5,0.25){\rnode{3u}{$\bullet$}}	
	\rput(3.5,-0.25){\rnode{3d}{$\bullet$}}
	\rput(4.5,0.25){\rnode{4u}{$\bullet$}}
	\rput(4.5,-0.25){\rnode{4d}{$\bullet$}}
	\rput(5.5,0.25){\rnode{5u}{$\bullet$}}
	\rput(5.5,-0.25){\rnode{5d}{$\bullet$}}
	\rput(6.5,0.25){\rnode{6u}{$\bullet$}}
	\rput(6.5,-0.25){\rnode{6d}{$\bullet$}}
	\rput(7.25,0.25){$\cdots$}
	\rput(7.25,-0.25){$\cdots$}
	\rput(8,0.25){\rnode{m-2u}{$\bullet$}}	
	\rput(8,-0.25){\rnode{m-2d}{$\bullet$}}
	\rput(9,0.25){\rnode{m-1u}{$\bullet$}}
	\rput(9,-0.25){\rnode{m-1d}{$\bullet$}}
	\rput(10,1){\rnode{m}{$\bullet$}}
	\rput(10,-1){\rnode{m'}{$\bullet$}}
	
	\ncarc{->}{0}{1u}
	\ncarc{->}{0'}{1d}
	\ncarc{->}{1u}{2u}
	\ncarc{->}{1d}{2d}
	
	\ncarc{->}{3u}{4u}	
	\ncarc{->}{3d}{4d}
	\ncarc{->}{4u}{5u}\Aput[0.05]{$i\!-\!1$}
	\ncarc{->}{4d}{5u}
	\ncarc{->}{5u}{6u}
	\ncarc{->}{6d}{5d}
	
	\ncarc{->}{m-2u}{m-1u}	
	\ncarc{->}{m-1d}{m-2d}
	\ncarc{->}{m-1u}{m}
	\ncarc{->}{m-1u}{m'}
	\ncarc{->}{m}{m-1d}
	\ncarc{->}{m'}{m-1d}
	}}

\rput(7.5,2){
\rput(1.35,-0.65){$U''_{m}$}
\scalebox{0.5}{
	\psset{arcangle=15,nodesep=1pt}
	\rput(0,1){\rnode{0}{$\bullet$}}		
	\rput(0,-1){\rnode{0'}{$\bullet$}}
	\rput(1,0.25){\rnode{1u}{$\bullet$}}
	\rput(1,-0.25){\rnode{1d}{$\bullet$}}
	\rput(2,0.25){\rnode{2u}{$\bullet$}}
	\rput(2,-0.25){\rnode{2d}{$\bullet$}}
	\rput(2.75,0.25){$\cdots$}
	\rput(2.75,-0.25){$\cdots$}
	\rput(3.5,0.25){\rnode{m-2u}{$\bullet$}}	
	\rput(3.5,-0.25){\rnode{m-2d}{$\bullet$}}
	\rput(4.5,0.25){\rnode{m-1u}{$\bullet$}}
	\rput(4.5,-0.25){\rnode{m-1d}{$\bullet$}}
	\rput(5.5,1){\rnode{m}{$\bullet$}}
	\rput(5.5,-1){\rnode{m'}{$\bullet$}}
	
	\ncarc{->}{0}{1u}
	\ncarc{->}{0'}{1d}
	\ncarc{->}{1u}{2u}
	\ncarc{->}{1d}{2d}
	
	\ncarc{->}{m-2u}{m-1u}	
	\ncarc{->}{m-2d}{m-1d}
	\ncarc{->}{m-1u}{m}
	\ncarc{->}{m-1u}{m'}
	\ncarc{->}{m-1d}{m'}
	}}

\rput(1.5,0){
\rput(2.5,-0.65){$V'_i$~\text{\tiny $(2\leq i\leq m\!+\!1)$}}
\scalebox{0.5}{
	\psset{arcangle=15,nodesep=1pt}
	\rput(0,1){\rnode{0}{$\bullet$}}	
	\rput(0,-1){\rnode{0'}{$\bullet$}}
	\rput(1,0.25){\rnode{1u}{$\bullet$}}
	\rput(1,-0.25){\rnode{1d}{$\bullet$}}
	\rput(2,0.25){\rnode{2u}{$\bullet$}}
	\rput(2,-0.25){\rnode{2d}{$\bullet$}}
	\rput(2.75,0.25){$\cdots$}
	\rput(2.75,-0.25){$\cdots$}
	\rput(3.5,0.25){\rnode{3u}{$\bullet$}}	
	\rput(3.5,-0.25){\rnode{3d}{$\bullet$}}
	\rput(4.5,0.25){\rnode{4u}{$\bullet$}}
	\rput(4.5,-0.25){\rnode{4d}{$\bullet$}}
	\rput(5.5,0.25){\rnode{5u}{$\bullet$}}
	\rput(5.5,-0.25){\rnode{5d}{$\bullet$}}
	\rput(6.5,0.25){\rnode{6u}{$\bullet$}}
	\rput(6.5,-0.25){\rnode{6d}{$\bullet$}}
	\rput(7.25,0.25){$\cdots$}
	\rput(7.25,-0.25){$\cdots$}
	\rput(8,0.25){\rnode{m-2u}{$\bullet$}}	
	\rput(8,-0.25){\rnode{m-2d}{$\bullet$}}
	\rput(9,0.25){\rnode{m-1u}{$\bullet$}}
	\rput(9,-0.25){\rnode{m-1d}{$\bullet$}}
	\rput(10,1){\rnode{m}{$\bullet$}}
	\rput(10,-1){\rnode{m'}{$\bullet$}}
	
	\ncarc{->}{0}{0'}
	\ncarc{->}{0}{1u}
	\ncarc{->}{0'}{1d}
	\ncarc{->}{1u}{2u}
	\ncarc{->}{1d}{2d}
	\ncline[nodesep=0]{->}{1u}{1d}
	\ncline[nodesep=0]{->}{2u}{2d}
	
	\ncarc{->}{3u}{4u}	
	\ncarc{->}{3d}{4d}
	\ncline[nodesep=0]{->}{3u}{3d}
	\ncarc{->}{4u}{5u}\Aput[0.05]{$i\!-\!2$}
	\ncarc{->}{5d}{4d}
	\ncline[nodesep=0]{->}{4u}{4d}
	\ncarc{->}{5u}{6u}
	\ncarc{->}{6d}{5d}
	
	\ncarc{->}{m-2u}{m-1u}	
	\ncarc{->}{m-1d}{m-2d}
	\ncarc{->}{m-1u}{m'}
	\ncarc{->}{m}{m-1d}
	}}

\rput(7.5,0){
\rput(2.5,-0.65){$V''_i$~\text{\tiny $(3\leq i\leq m\!+\!2)$}}
\scalebox{0.5}{
	\psset{arcangle=15,nodesep=1pt}
	\rput(0,1){\rnode{0}{$\bullet$}}	
	\rput(0,-1){\rnode{0'}{$\bullet$}}
	\rput(1,0.25){\rnode{1u}{$\bullet$}}
	\rput(1,-0.25){\rnode{1d}{$\bullet$}}
	\rput(2,0.25){\rnode{2u}{$\bullet$}}
	\rput(2,-0.25){\rnode{2d}{$\bullet$}}
	\rput(2.75,0.25){$\cdots$}
	\rput(2.75,-0.25){$\cdots$}
	\rput(3.5,0.25){\rnode{3u}{$\bullet$}}	
	\rput(3.5,-0.25){\rnode{3d}{$\bullet$}}
	\rput(4.5,0.25){\rnode{4u}{$\bullet$}}
	\rput(4.5,-0.25){\rnode{4d}{$\bullet$}}
	\rput(5.5,0.25){\rnode{5u}{$\bullet$}}
	\rput(5.5,-0.25){\rnode{5d}{$\bullet$}}
	\rput(6.5,0.25){\rnode{6u}{$\bullet$}}
	\rput(6.5,-0.25){\rnode{6d}{$\bullet$}}
	\rput(7.25,0.25){$\cdots$}
	\rput(7.25,-0.25){$\cdots$}
	\rput(8,0.25){\rnode{m-2u}{$\bullet$}}	
	\rput(8,-0.25){\rnode{m-2d}{$\bullet$}}
	\rput(9,0.25){\rnode{m-1u}{$\bullet$}}
	\rput(9,-0.25){\rnode{m-1d}{$\bullet$}}
	\rput(10,1){\rnode{m}{$\bullet$}}
	\rput(10,-1){\rnode{m'}{$\bullet$}}
	
	\ncarc{->}{0}{0'}
	\ncarc{->}{0}{1u}
	\ncarc{->}{0'}{1d}
	\ncarc{->}{1u}{2u}
	\ncarc{->}{1d}{2d}
	\ncline[nodesep=0]{->}{1u}{1d}
	\ncline[nodesep=0]{->}{2u}{2d}
	
	\ncarc{->}{3u}{4u}	
	\ncarc{->}{3d}{4d}
	\ncline[nodesep=0]{->}{3u}{3d}
	\ncarc{->}{4u}{5u}\Aput[0.05]{$i\!-\!3$}
	\ncline[nodesep=0]{->}{4u}{4d}
	\ncarc{->}{5u}{6u}
	\ncarc{->}{6d}{5d}
	
	\ncarc{->}{m-2u}{m-1u}	
	\ncarc{->}{m-1d}{m-2d}
	\ncarc{->}{m-1u}{m'}
	\ncarc{->}{m}{m-1d}
	\ncarc[arcangle=-15]{->}{m}{m'}
	}}
\end{pspicture}
\end{center}

Consider the open cover of $X_{0\ldots i}^{m\ldots(m-j)}$ given in \ref{OpensDnEven} and let $M\in X_{0\ldots i}^{m\ldots(m-j)}$ be a representation of $Q$. We now define the dimension vectors which are relevant in the proof, together with the corresponding inequality that any submodule $N\subset M$ with one of these dimension vectors produce:
{\renewcommand{\arraystretch}{1.5}
\[
\begin{array}{cc}
\text{Dimension vector} & \text{Inequality in $\Theta$} \\
\begin{pspicture}(0,0)(0.5,0.5)
\rput(-1.5,0){$s_i:=$}
\scalebox{0.3}{
\rput(-2.25,0){
{\Huge
	\psset{arcangle=15,nodesep=1pt}
	\rput(0,0.5){\rnode{0}{$0$}}	
	\rput(0,-0.5){\rnode{0'}{$0$}}
	\rput(1,0){\rnode{1u}{$0$}}
	\rput(2,0){$\cdots$}
	\rput(3,0){\rnode{3u}{$0$}}	
	\rput(4,0){\rnode{4u}{$1$}}	
	\rput(5,0){\rnode{5u}{$0$}}
	\rput(6,0){$\cdots$}
	\rput(7,0){\rnode{m-1u}{$0$}}	
	\rput(8,0.5){\rnode{m}{$0$}}
	\rput(8,-0.5){\rnode{m'}{$0$}}
	\rput(4,1){$i$}	
	}}}
\end{pspicture}
&  \theta_i>0 \\

\begin{pspicture}(0,0)(0.5,0.5)
\rput(-1.5,0){$r_i:=$}
\scalebox{0.3}{
\rput(-2.25,0){
{\Huge
	\psset{arcangle=15,nodesep=1pt}
	\rput(0,0.5){\rnode{0}{$0$}}	
	\rput(0,-0.5){\rnode{0'}{$1$}}
	\rput(1,0){\rnode{1u}{$1$}}
	\rput(2,0){$\cdots$}
	\rput(3,0){\rnode{3u}{$1$}}	
	\rput(4,0){\rnode{4u}{$1$}}
	\rput(5,0){\rnode{5u}{$0$}}
	\rput(6,0){$\cdots$}
	\rput(7,0){\rnode{m-1u}{$0$}}	
	\rput(8,0.5){\rnode{m}{$0$}}
	\rput(8,-0.5){\rnode{m'}{$0$}}
	\rput(4,1){$i$}	
	}}}
\end{pspicture}
&  \sum_{k=1}^i\theta_k>0 \\

\begin{pspicture}(0,0)(0.5,0.5)
\rput(-1.5,0){$n_i:=$}
\scalebox{0.3}{
\rput(-2.25,0){
{\Huge
	\psset{arcangle=15,nodesep=1pt}
	\rput(0,0.5){\rnode{0}{$0$}}	
	\rput(0,-0.5){\rnode{0'}{$0$}}
	\rput(1,0){\rnode{1u}{$0$}}
	\rput(2,0){$\cdots$}
	\rput(3,0){\rnode{3u}{$0$}}	
	\rput(4,0){\rnode{4u}{$1$}}
	\rput(5,0){\rnode{5u}{$1$}}
	\rput(6,0){$\cdots$}
	\rput(7,0){\rnode{m-1u}{$1$}}	
	\rput(8,0.5){\rnode{m}{$1$}}
	\rput(8,-0.5){\rnode{m'}{$0$}}
	\rput(4,1){$i$}	
	}}}
\end{pspicture}
&  \sum_{k=i}^{m+1}\theta_k>0 \\

\begin{pspicture}(0,0)(0.5,0.5)
\rput(-1.5,0){$e_i:=$}
\scalebox{0.3}{
\rput(-2.25,0){
{\Huge
	\psset{arcangle=15,nodesep=1pt}
	\rput(0,0.5){\rnode{0}{$0$}}	
	\rput(0,-0.5){\rnode{0'}{$0$}}
	\rput(1,0){\rnode{1u}{$0$}}
	\rput(2,0){$\cdots$}
	\rput(3,0){\rnode{3u}{$0$}}	
	\rput(4,0){\rnode{4u}{$1$}}
	\rput(5,0){\rnode{5u}{$1$}}
	\rput(6,0){$\cdots$}
	\rput(7,0){\rnode{m-1u}{$1$}}	
	\rput(8,0.5){\rnode{m}{$1$}}
	\rput(8,-0.5){\rnode{m'}{$1$}}
	\rput(4,1){$i$}	
	}}}
\end{pspicture}
&  \sum_{k=i}^{m+2}\theta_k>0 \\

\begin{pspicture}(0,0)(0.5,0.5)
\rput(-1.5,0){$c_i:=$}
\scalebox{0.3}{
\rput(-2.25,0){
{\Huge
	\psset{arcangle=15,nodesep=1pt}
	\rput(0,0.5){\rnode{0}{$1$}}	
	\rput(0,-0.5){\rnode{0'}{$0$}}
	\rput(1,0){\rnode{1u}{$1$}}
	\rput(2,0){$\cdots$}
	\rput(3,0){\rnode{3u}{$1$}}	
	\rput(4,0){\rnode{4u}{$1$}}
	\rput(5,0){\rnode{5u}{$2$}}
	\rput(6,0){$\cdots$}
	\rput(7,0){\rnode{m-1u}{$2$}}	
	\rput(8,0.5){\rnode{m}{$1$}}
	\rput(8,-0.5){\rnode{m'}{$1$}}
	\rput(4,1){$i$}	
	}}}
\end{pspicture}
&  \sum_{k=1}^i\theta_k<0 \\

\begin{pspicture}(0,0)(0.5,0.5)
\rput(-1.5,0){$d_i:=$}
\scalebox{0.3}{
\rput(-2.25,0){
{\Huge
	\psset{arcangle=15,nodesep=1pt}
	\rput(0,0.5){\rnode{0}{$1$}}	
	\rput(0,-0.5){\rnode{0'}{$1$}}
	\rput(1,0){\rnode{1u}{$2$}}
	\rput(2,0){$\cdots$}
	\rput(3,0){\rnode{3u}{$2$}}	
	\rput(4,0){\rnode{4u}{$1$}}
	\rput(5,0){\rnode{5u}{$1$}}
	\rput(6,0){$\cdots$}
	\rput(7,0){\rnode{m-1u}{$1$}}	
	\rput(8,0.5){\rnode{m}{$0$}}
	\rput(8,-0.5){\rnode{m'}{$1$}}
	\rput(4,1){$i$}	
	}}}
\end{pspicture}
&  \sum_{k=i}^{m+1}\theta_k<0 \\

\begin{pspicture}(0,0)(0.5,0.5)
\rput(-1.5,0){$j:=$}
\scalebox{0.3}{
\rput(-2.25,0){
{\Huge
	\psset{arcangle=15,nodesep=1pt}
	\rput(0,0.5){\rnode{0}{$0$}}	
	\rput(0,-0.5){\rnode{0'}{$1$}}
	\rput(1,0){\rnode{1u}{$1$}}
	\rput(2,0){$\cdots$}
	\rput(3,0){\rnode{m-1u}{$1$}}	
	\rput(4,0.5){\rnode{m}{$0$}}
	\rput(4,-0.5){\rnode{m'}{$1$}}
	}}}
\end{pspicture}
&  \sum_{k=1}^m\theta_k+\theta_{m+2}>0
\end{array}
\]}

Note that $r_1=s_1$, $n_{m+1}=s_{m+1}$ and $r_{m+1}=n_1$. \\

The following Lemma shows the presence of the first two inequalities of the Theorem, namely $\theta_k>0$ for $k\neq0,1,m+1$, and $\sum_{k=1}^{m+1}\theta_k>0$.

\begin{lem}\label{lem-Dn-even} (i) $\theta_i>0$ for all $i\neq0,1,m+1$. \\
\noindent (ii) There always exists a submodule $N_1\subset M$ with $\underline{\dim}(N_1)=n_1$.
\end{lem}

\begin{proof} (i) By the open covers given in Theorem \ref{OpensDnEven} (1) and the corresponding skeletons, any crepant resolution of $\C^3/G$ has at least one open set containing a submodule $S_i$ with $\underline{\dim}(S_i)=s_i$ for $i=2,\ldots,m,m+2$. In the cases $i=0,1,m+1$ note that $S_0$ do not belong to any open set so that there's no condition of the form $\theta_0>0$. The submodule $S_1$ is only contained in $U'_1$ and $U_2$, which implies that only $X^{m\ldots(m-j)}$ for any $j$ have the condition $\theta_1>0$. Finally, only $U_{m+1}$ and $V'_{m+2}$ contain the submodule $S_{m+1}$, so that the condition $\theta_{m+1}>0$ only is valid in $X_{0\ldots k}$ for any $k$.

\noindent (ii) Notice that for every $k$ the submodule $N_1\in U'_k,V'_k$, and every $X_{0\ldots i}^{m\ldots(m-j)}$ contains at least one of these affine sets. This finishes the proof of the lemma.
\end{proof}

The dimension vectors that we have to consider in every open set are the following:
{\renewcommand{\arraystretch}{1}
\[
\begin{array}{rl|rl}
\text{Open set} & \text{Dimension vectors} & \text{Open set} & \text{Dimension vectors} \\
	\text{\tiny ($2i\leq m$)}, U_i & s_{i-1}, s_{i}, r_{i-1}, n_{i} 	& \text{\tiny ($i\leq m-1$)}, U''_i & s_{i+1}, n_{i+1}, c_1,\ldots, c_i \\
	U_{m+1} & s_m, s_{m+1}, s_{m+2}, r_m & U''_m & s_{m+1},s_{m+2}, j , c_m \\
V'_{m+2} & s_{m+1}, n_1 & V'_i & s_{i-1}, s_{m+2}, r_{i-1}, n_{i-1}, n_1, j, d_{i}, \ldots,d_{m+1} \\
V''_{m+3} & s_{m+2}, j, e_{m+1} & \text{\tiny ($i\leq m+1$)}, V''_i & s_{i-2}, s_{i-1}, s_{m+2}, r_{i-2}, e_{i-1}, d_{i-1}, \ldots,d_{m+1} \\
\text{\tiny ($1i\leq m$)}, U'_i & s_i, r_i, n_i, n_1, c_1,\ldots,c_{i-1} & V''_{m+2} & s_{m},s_{m+2}, r_m, e_m, d_{m+1} \\
U'_{m+1} & s_{m+1}, s_{m+2}, n_m, n_1, c_m & 
\end{array}
\]}
The result follows by going through the open cover of $X^{m..(m-j)}_{0..i}\cong\mathcal{M}_{C_{ij}}$ given in Theorem \ref{OpensDnEven}, and writing down the corresponding inequalities.

(iii) The skeletons in this case are:

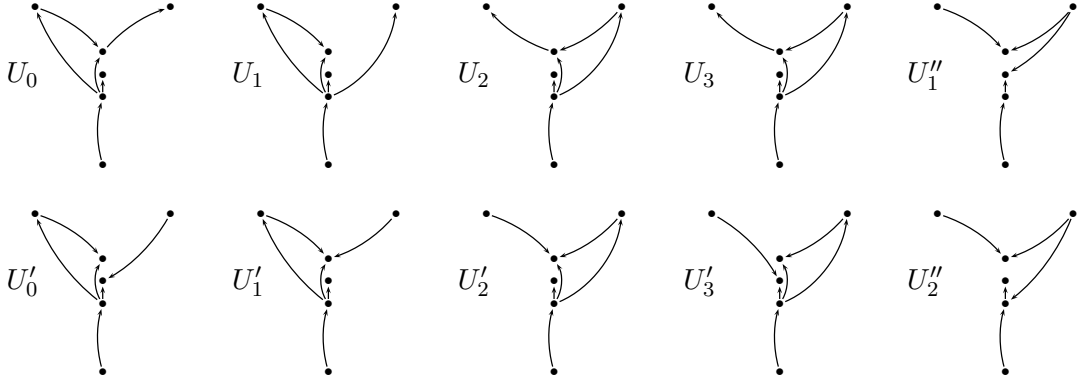
\begin{center}
\begin{pspicture}(0,-1.5)(12,4)
	\psset{arcangle=15,nodesep=0.75pt}
	
\rput(0,2.75){
\scalebox{0.6}{
	\rput(0,-2){\rnode{0}{$\bullet$}}
	\rput(0,0.5){\rnode{33}{$\bullet$}}	
	\rput(0,0.0){\rnode{32}{$\bullet$}} 	
	\rput(0,-0.5){\rnode{31}{$\bullet$}}
	\rput(-1.5,1.5){\rnode{1}{$\bullet$}}
	\rput(1.5,1.5){\rnode{2}{$\bullet$}}
	\ncarc{->}{0}{31}
	\ncline{->}{31}{32}
	\ncarc[arcangle=30]{->}{31}{33}
	\ncarc{->}{31}{1}
	\ncarc{->}{1}{33}
	\ncarc{->}{33}{2}
	}}
	
\rput(3,2.75){
\scalebox{0.6}{
	\rput(0,-2){\rnode{0}{$\bullet$}}
	\rput(0,0.5){\rnode{33}{$\bullet$}}	
	\rput(0,0.0){\rnode{32}{$\bullet$}} 	
	\rput(0,-0.5){\rnode{31}{$\bullet$}}
	\rput(-1.5,1.5){\rnode{1}{$\bullet$}}
	\rput(1.5,1.5){\rnode{2}{$\bullet$}}
	\ncarc{->}{0}{31}
	\ncline{->}{31}{32}
	\ncarc[arcangle=30]{->}{31}{33}
	\ncarc{->}{31}{1}
	\ncarc{->}{1}{33}
	\ncarc[arcangle=-30]{->}{31}{2}
	}}
	
\rput(6,2.75){
\scalebox{0.6}{
	\rput(0,-2){\rnode{0}{$\bullet$}}
	\rput(0,0.5){\rnode{33}{$\bullet$}}	
	\rput(0,0.0){\rnode{32}{$\bullet$}} 	
	\rput(0,-0.5){\rnode{31}{$\bullet$}}
	\rput(-1.5,1.5){\rnode{1}{$\bullet$}}
	\rput(1.5,1.5){\rnode{2}{$\bullet$}}
	\ncarc{->}{0}{31}
	\ncline{->}{31}{32}
	\ncarc[arcangle=-30]{->}{31}{33}
	\ncarc{->}{33}{1}
	\ncarc{->}{2}{33}
	\ncarc[arcangle=-30]{->}{31}{2}
	}}

\rput(9,2.75){
\scalebox{0.6}{
	\rput(0,-2){\rnode{0}{$\bullet$}}
	\rput(0,0.5){\rnode{33}{$\bullet$}}	
	\rput(0,0.0){\rnode{32}{$\bullet$}} 	
	\rput(0,-0.5){\rnode{31}{$\bullet$}}
	\rput(-1.5,1.5){\rnode{1}{$\bullet$}}
	\rput(1.5,1.5){\rnode{2}{$\bullet$}}
	\ncarc{->}{0}{31}
	\ncline{->}{31}{32}
	\ncarc[arcangle=-30]{->}{31}{33}
	\ncarc{->}{33}{1}
	\ncarc{->}{2}{33}
	\ncarc[arcangle=-30]{->}{31}{2}
	}}

\rput(12,2.75){
\scalebox{0.6}{
	\rput(0,-2){\rnode{0}{$\bullet$}}
	\rput(0,0.5){\rnode{33}{$\bullet$}}	
	\rput(0,0.0){\rnode{32}{$\bullet$}} 	
	\rput(0,-0.5){\rnode{31}{$\bullet$}}
	\rput(-1.5,1.5){\rnode{1}{$\bullet$}}
	\rput(1.5,1.5){\rnode{2}{$\bullet$}}
	\ncarc{->}{0}{31}
	\ncline{->}{31}{32}
	\ncarc{->}{1}{33}
	\ncarc{->}{2}{32}
	\ncarc{->}{2}{33}
	}}

\rput(0,0){
\scalebox{0.6}{
	\rput(0,-2){\rnode{0}{$\bullet$}}
	\rput(0,0.5){\rnode{33}{$\bullet$}}	
	\rput(0,0.0){\rnode{32}{$\bullet$}} 	
	\rput(0,-0.5){\rnode{31}{$\bullet$}}
	\rput(-1.5,1.5){\rnode{1}{$\bullet$}}
	\rput(1.5,1.5){\rnode{2}{$\bullet$}}
	\ncarc{->}{0}{31}
	\ncline{->}{31}{32}
	\ncarc[arcangle=30]{->}{31}{33}
	\ncarc{->}{31}{1}
	\ncarc{->}{1}{33}
	\ncarc{->}{2}{32}
	}}
	
\rput(3,0){
\scalebox{0.6}{
	\rput(0,-2){\rnode{0}{$\bullet$}}
	\rput(0,0.5){\rnode{33}{$\bullet$}}	
	\rput(0,0.0){\rnode{32}{$\bullet$}} 	
	\rput(0,-0.5){\rnode{31}{$\bullet$}}
	\rput(-1.5,1.5){\rnode{1}{$\bullet$}}
	\rput(1.5,1.5){\rnode{2}{$\bullet$}}
	\ncarc{->}{0}{31}
	\ncline{->}{31}{32}
	\ncarc[arcangle=30]{->}{31}{33}
	\ncarc{->}{31}{1}
	\ncarc{->}{1}{33}
	\ncarc{->}{2}{33}
	}}
	
\rput(6,0){
\scalebox{0.6}{
	\rput(0,-2){\rnode{0}{$\bullet$}}
	\rput(0,0.5){\rnode{33}{$\bullet$}}	
	\rput(0,0.0){\rnode{32}{$\bullet$}} 	
	\rput(0,-0.5){\rnode{31}{$\bullet$}}
	\rput(-1.5,1.5){\rnode{1}{$\bullet$}}
	\rput(1.5,1.5){\rnode{2}{$\bullet$}}
	\ncarc{->}{0}{31}
	\ncline{->}{31}{32}
	\ncarc[arcangle=-30]{->}{31}{33}
	\ncarc[arcangle=-30]{->}{31}{2}
	\ncarc{->}{1}{33}
	\ncarc{->}{2}{33}
	}}

\rput(9,0){
\scalebox{0.6}{
	\rput(0,-2){\rnode{0}{$\bullet$}}
	\rput(0,0.5){\rnode{33}{$\bullet$}}	
	\rput(0,0.0){\rnode{32}{$\bullet$}} 	
	\rput(0,-0.5){\rnode{31}{$\bullet$}}
	\rput(-1.5,1.5){\rnode{1}{$\bullet$}}
	\rput(1.5,1.5){\rnode{2}{$\bullet$}}
	\ncarc{->}{0}{31}
	\ncline{->}{31}{32}
	\ncarc[arcangle=-30]{->}{31}{33}
	\ncarc[arcangle=-30]{->}{31}{2}
	\ncarc{->}{1}{32}
	\ncarc{->}{2}{33}
	}}
	
\rput(12,0){
\scalebox{0.6}{
	\rput(0,-2){\rnode{0}{$\bullet$}}
	\rput(0,0.5){\rnode{33}{$\bullet$}}	
	\rput(0,0.0){\rnode{32}{$\bullet$}} 	
	\rput(0,-0.5){\rnode{31}{$\bullet$}}
	\rput(-1.5,1.5){\rnode{1}{$\bullet$}}
	\rput(1.5,1.5){\rnode{2}{$\bullet$}}
	\ncarc{->}{0}{31}
	\ncline{->}{31}{32}
	\ncarc{->}{1}{33}
	\ncarc{->}{2}{31}
	\ncarc{->}{2}{33}
	}}

\rput(-1,2.75){$U_0$}
\rput(2,2.75){$U_1$}
\rput(5,2.75){$U_2$}
\rput(8,2.75){$U_3$}
\rput(11,2.75){$U''_1$}

\rput(-1,0){$U'_0$}
\rput(2,0){$U'_1$}
\rput(5,0){$U'_2$}
\rput(8,0){$U'_3$}
\rput(11,0){$U''_2$}
\end{pspicture}
\end{center}

As expected, only the skeletons for $U_i, i=0,\ldots,3$ are generated from the vertex 0. Indeed, this is equivalent to the $0$-generated stability condition which only $\Hilb{G}{\C^3}$ satisfies.

Now take the open covers of $X_i$ given in Theorem \ref{OpensE6}. Then the inequalities defining the chambers $C_i$ for which $X_i\cong\mathcal{M}_{C_i}$ are given by the submodules of the above skeletons, and the result follows.
\end{proof}

\begin{rmk} The set of inequalities in Theorem \ref{stability} does not give the reduce description of the chamber $C\subset\Theta$. Nevertheless, for any crepant resolution the minimum number of walls or inequalities defining $C$ is precisely $|Q_0|-1$, which coincides with the number of components of the fibre over the origin (or the number of non-trivial irreducible representations of $G$).
\end{rmk}

If for ${\bf d}:=(\dim\rho_i)_{i\in Q_0}$ we consider the dual graph $\mathcal{T}$ of $\Theta_{\bf{d}}$ (that is, one vertex for each chamber and an edge between two vertices if the corresponding chambers are separated by a wall), as a consequence of the previous theorem we have the following corollary.

\begin{cor}\label{FRegion} There exists a path in $\mathcal{T}$ where every crepant resolution of $\C^3/G$ can be found and such that every wall crossing in $\mathcal{T}$ corresponds to a flop.
\end{cor}

This nice distribution contrast for example with the general case for Abelian groups in $\SL(3,\C)$, where it can happen that finitely many wall crossings (of {\em Types} $0$ or $III$) are needed to connect two crepant resolutions related by a single flop. See \cite{CI} for more details.

\subsection{Stability conditions and mutations}

In this section we compare the classical approach of changing the stability condition on the representations of the McKay QP to obtain all crepant resolutions of $\C^3/G$ with the mutation approach, which change the QP but not the stability.

Let $G\subset\SO(3)$ of type $\Z/n\Z$, $D_{2n}$ or $\mathbb{T}$, and let $(Q,W)$ be the Mckay QP.
Let $\mathbb Z^{Q_0}$ be the space of dimension vectors, with canonical basis $\{ \e_0, \e_1,\ldots,\e_n \}$. Let $\Hom_{\mathbb Z}(\mathbb Z^{Q_0},\mathbb Z)$ be the dual space with the dual basis $\e_0^*,\e_1^*,\ldots,\e_n^*$ and define $\Theta := {\Hom_{\mathbb Z}(\mathbb Z^{Q_0}},\mathbb Z) \otimes \mathbb Q$ the whole parameter space. Let ${\bf{d}}=\sum_{i \in Q_0} (\dim\rho_i)\e_i$ with $\rho_i\in\Irr G$.

Let $\mu(Q,W)$ a QP obtained by a sequence of mutations $\mu = \mu_{i_1}\cdots\mu_{i_m}$ from the McKay QP. We denote by $\Lambda = \mathcal P(\mu(Q,W))$ the Jacobian algebra. We fix a vertex $i \in Q_0$ with no loops and let $P_i$ be the projective $\Lambda$-module and $S_i$ the simple module associated to the vertex $i$.
Then, as in \cite{BIRS} Proposition 4.2, there is an exact sequence of the form
\begin{equation}\label{ex-Lambda}
0 \to P_i \to  X_2 = \!\!\!\bigoplus_{a \in Q_1, ha=i}\!\!\!P_{ta} \to X_1 = \!\!\!\bigoplus_{a \in Q_1, ta=i}\!\!\!P_{ha} \stackrel{f}{\to} P_i \to S_i \to 0.
\end{equation}

Let $(-,-)$ be a symmetric bilinear form on $\mathbb Z^{Q_0}$ defined by
\[
(\e_i,\e_j) = \left\{\begin{array}{cl} 2 & i=j \\ -\#(i \to j) & i\neq j \end{array}\right.
\]
In our case, if $i$ and $j$ are adjacent, then we can see that there is only one arrow from $i \to j$, so $(\e_i,\e_j)$ is $-1$.

We define $(M,N) := (\udim M,\udim N) := (\udim M,\udim N)$ for any finite dimensional $\Lambda$-modules $M,N$. We denote by $s_i$ the reflection with respect to a vertex $i$, which is defined by 
\[
s_i\alpha := \alpha - (\alpha,\e_i)\e_i
\] 
for any dimension vector $\alpha \in \mathbb Z^{Q_0}$ and dually 
\[
s_i\theta := \theta - \theta_i \sum_{j=0}^n (\e_i,\e_j)\e_j^*.\] 
Trivially, for any dimension vector $\alpha$, $\theta(\alpha)=0$ if and only if $(s_{i}\theta)(s_{i}\alpha)=0$. 
For a sequence $\mu=\mu_{i_1}\cdots \mu_{i_m}$ of mutations, we consider the corresponding sequence of reflections $\omega = s_{i_1}\cdots s_{i_m}$.
Then dimension vectors $\omega \mathbf d$ determine parameter spaces $\Theta_{\omega\mathbf d}$.
Let $\theta^0 \in \Theta_{\omega\mathbf d}$ be the $0$-generated stability condition and $C_0$ the chamber in $\Theta_{\omega \mathbf d}$ defined by the inequalities of $\theta^0_i>0$ for $i \neq 0$. 

\begin{lem}\label{0-gen}
The chamber of $\mathcal{M}_{\theta^0,\omega d}(\Lambda)$ is $C_0$.
\end{lem}

\begin{proof}
It follows from direct calculations that all simple modules associated to vertices can be a subrepresentation of some point in $\mathcal{M}_{\theta^0,\omega d}(\Lambda)$
\end{proof}

Recall that by the one-to-one correspondence between flops of \Hilb{$G$}{$\C^3$} and mutations of the McKay QP, for any projective crepant resolution $X\cong\mathcal{M}_C$ for some $C\subset\Theta$ there exists a corresponding iterated QP $(Q_C,W_C)$ obtained by a sequence of mutations from the McKay QP. The goal of this subsection is the next result. 

\begin{thm}\label{horizontal}
Let $X\to \mathbb C^3/G$ be an arbitrary projective crepant resolution, that is $X\cong\mathcal{M}_C$ for some $C\subset\Theta$.
Then $X \cong \mathcal{M}_{\theta^0,\omega \mathbf d}(\Gamma)$ for the Jacobian algebra $\Gamma := \mathcal P(Q_C,W_C)$ and the 0-generated stability condition $\theta^0$. Moreover, there exists a corresponding sequence of wall crossings from $\Hilb{G}{\C^3}$ which leads to $X \cong \mathcal M_{\theta,\mathbf d}(\Lambda)$ where $\Lambda = \mathcal{P}(Q,W)$, ${\bf d}=(\dim\rho_i)_{\rho_i\in\Irr G}$ and the chamber $C \subset \Theta_{\mathbf d}$ containing $\theta$ is given by the inequalities $\theta(\omega^{-1}\e_i)>0$ for any $i \neq 0$.
\end{thm}

The rest of the section is dedicated to prove the above theorem. 

\begin{df}
For any parameter $\theta \in \Theta$, we define the full subcategory $\Stab{\theta}{\Lambda}{}$ of $\Mod \Lambda$ consisting of $\theta$-semistable finite dimensional $\Lambda$-modules. Moreover we denote by $\Stab{\theta}{\Lambda}{,\alpha}$ the full subcategory of $\Stab{\theta}{\Lambda}{}$ consisting of $\theta$-semistable $\Lambda$-modules of dimension vector $\alpha$ if $\Stab{\theta}{\Lambda}{,\alpha}$ is not empty. 
\end{df}

In the exact sequence \eqref{ex-Lambda}, let $K_i$ be the kernel of $f$ fitting in the exact sequence
\begin{equation}\label{ex-K}
0 \to K_i \stackrel{g}{\to} X_1 \stackrel{f}{\to} P_i \to S_i \to 0.
\end{equation}

Then it can be checked that $T_i :=\Lambda/P_i\oplus K_i$ is a tilting $\Lambda$-module of projective dimension one.
We put $\Gamma = \End_{\Lambda}(T_i)$. By a similar strategy as in \cite{BIRS}, it follows that $\Gamma \simeq \mathcal P(\mu_i\omega(Q,W))$. 

\begin{lem}\label{dim-ref}
Let $M$ be a finite dimensional $\Lambda$-module of dimension vector $\alpha=(\alpha_k)$. Then the alternating sum of the dimension vector of $\RHom_{\Lambda}(T_i,M)$ is given by the following formula:
\[
\udim_{\Gamma} \Hom_{\Lambda}(T_i,M) - \udim_{\Gamma} \Ext^1_{\Lambda}(T_i,M) = s_{i}\alpha.
\]
\end{lem}

\begin{proof}
For each $j \in Q_0$, $e_j$ denotes the corresponding idempotent of $\Lambda$. The following hold:
\[
\Hom_{\Lambda}(T_i,M)e_j \simeq \Hom_{\Lambda}(e_jT_i,M) = 
\begin{cases}
\Hom_{\Lambda}(K_i,M) & \text{ if } j=i \\
\Hom_{\Lambda}(P_j,M) & \text{ if } j\neq i 
\end{cases}
\]
and 
\[
\Ext^1_{\Lambda}(T_i,M)e_j \simeq \Ext^1_{\Lambda}(e_jT_i,M) =
\begin{cases}
\Ext^1_{\Lambda}(K_i,M) & \text{ if } j=i \\
0 & \text{ if } j\neq i. 
\end{cases}
\]
By applying $\Hom_{\Lambda}(-,M)$ to the exact sequence $0 \to P_i \to X_2 \to K_i \to 0$, we have 
\[
0 \to \Hom_{\Lambda}(K_i,M) \to \Hom_{\Lambda}(X_2,M) \to \Hom_{\Lambda}(P_i,M) \to \Ext^1_{\Lambda}(K_i,M) \to 0.
\]
Hence we have
\[
\dim_{\C} \Hom_{\Lambda}(K_i,M) - \dim_{\C} \Ext^1_{\Lambda}(K_i,M) = \dim_{\C}  \Hom_{\Lambda}(X_2,M) - \dim_{\C} \Hom_{\Lambda}(P_i,M).
\]
so that $\udim_{\Gamma} \Hom_{\Lambda}(T_i,M) - \udim_{\Gamma} \Ext^1_{\Lambda}(T_i,M)$ is equal to
\begin{align*}
& \sum_{j\neq i} \alpha_j\e_j + (\dim_{\C} \Hom_{\Lambda}(K_i,M) - \dim_{\C} \Ext^1_{\Lambda}(K_i,M))\e_i \\
= &\sum_{j\neq i} \alpha_j\e_j + (\!\!\!\!\sum_{a\in Q_1, ha=i}\!\!\!\!\alpha_{ta} -\alpha_i)\e_i 
= \alpha - (2\alpha_i - \!\!\!\!\sum_{a \in Q_1, ha=i}\!\!\!\!\alpha_{ta})\e_i
= \alpha - (\alpha,\e_i)\e_i. 
\end{align*}
\end{proof}

We have a similar result as in the two dimensional case treated in \cite{SY}.

\begin{thm}\label{ref}
If $\theta_i>0$, then there is an equivalence
\[\xymatrix{
\Stab{\theta}{\Lambda}{} \ar@<0.5ex>[rr]^{\Hom_{\Lambda}(T_i,-)}  && \Stab{s_{i}\theta}{\Gamma}{} \ar@<0.5ex>[ll]^{-\otimes_{\Gamma}T_i}
}\]
of categories which preserves $S$-equivalence classes. Moreover this equivalence induces an isomorphism 
\[
\mathcal M_{\theta,\alpha}(\Lambda) \cong \mathcal M_{s_i\theta,s_i\alpha}(\Gamma)
\]
 of varieties for any $\alpha \in \mathbb Z^{Q_0}$.
\end{thm}

\begin{proof}
Since $T_i$ is a tilting module, there is a derived equivalence
\[\xymatrix{
\mathcal D(\Mod \Lambda) \ar@<0.5ex>[rr]^{\RHom_{\Lambda}(T_i,-)}  && \mathcal D(\Mod \Gamma) \ar@<0.5ex>[ll]^{-\Ltensor_{\Gamma}T_i}.
}\]
The functor $\RHom_{\Lambda}(T_i,-)$ induces a functor $\Hom_{\Lambda}(T_i,-)$ from $\Stab{\theta}{\Lambda}{}$ to $\mod \Gamma$.  We show that $\Hom_{\Lambda}(T_i,-)$ is well-defined. Let $M$ be a $\theta$-semistable $\Lambda$-module of dimension $\alpha$. 
By applying $\Hom_{\Lambda}(-,M)$ to the exact sequence \eqref{ex-K} and using the fact that $\Lambda$ is 3-CY, we have 
\[
\Ext^1_{\Lambda}(T_i,M) \simeq \Ext^3(S_i,M) \simeq D\Hom_{\Lambda}(M,S_i).
\]
Since $\theta_i>0$, $M$ doesn't have $S_i$ as a factor. So we have $\Hom_{\Lambda}(M,S_i)=0$, hence $\Ext^1_{\Lambda}(T_i,M)=0$. 
Next we show that $M' = \Hom_{\Lambda}(T_i,M)$ is $s_{i}\theta$-semistable.
By Lemma \ref{dim-ref} we have
\[
(s_{i}\theta)(M')=(s_{i}\theta)(\udim_{\Gamma} \Hom_{\Lambda}(T_i,M)) = (s_{i}\theta)(s_{i}\alpha) = \theta(\alpha)=0.
\]
We take any proper submodule $N'$ of $M'$ and consider the following exact sequence
\[
0 \to N' \to M' \to C \to 0.
\]
By applying $-\otimes_{\Gamma}T_i$ to the above, since $\Tor_1^{\Gamma}(M',T_i)=0$ we have an exact sequence
\[
0 \to \Tor_1^{\Gamma}(C,T_i) \to N'\otimes_{\Gamma}T_i \stackrel{f}{\to} M'\otimes_{\Gamma}T_i \simeq M.
\]
We have $\Tor_1^{\Gamma}(N,T_i)e_j = \Tor_1^{\Gamma}(N,T_ie_j) = 0$
since $T_ie_j \simeq \Hom_{\Lambda}(\Lambda,T_i)e_j \simeq \Hom_{\Lambda}(P_j,T_i)$ is a projective $\Gamma^{\op}$-module. So $\udim_{\Gamma} \Tor_1^{\Gamma}(C,T_i) = \e_i^m$ for some non-negatve integer $m$. Since $\Tor_1^{\Gamma}(N',T_i)=0$ we have $\udim_{\Lambda} N' \otimes_{\Gamma}T_i = s_{i} \udim_{\Gamma}N'$. Thus since $\Im f$ is a submodule of $M$, we have
\begin{align*}
(s_{i}\theta)(N')&=(s_{i}\theta)(s_{i}\udim_{\Lambda} N' \otimes_{\Gamma}T_i) \\
&= \theta(\udim_{\Lambda} N' \otimes_{\Gamma}T_i) \\
&= \theta(\e_i^m) + \theta(\Im f) \geq 0.
\end{align*}
Note that if $M$ is $\theta$-stable, then $M'$ is also $s_{i}\theta$-stable since $\theta(\Im f)>0$. The converse is proved similarly. Moreover, one can easily check that $S$-equivalent classes are preserved. 

As in \cite[Theorem 5.6]{SY}, for any dimension vector $\alpha$ the functors $\Hom_{\Lambda}(T_i,-)$ and $-\otimes_{\Gamma}T_i$ induce inverse morphisms $f : \mathcal M_{\theta,\alpha}(\Lambda) \to \mathcal M_{s_i\theta,s_i\alpha}(\Gamma)$ and  $g : \mathcal M_{s_i\theta,s_i\alpha}(\Gamma) \to \mathcal M_{\theta,\alpha}(\Lambda)$ so that there is an isomorphism between the moduli spaces.
\end{proof}

\begin{cor}\label{chamber}
A chamber $C \subset \Theta_{\omega \mathbf d}$ is mapped to a chamber $s_iC \subset \Theta_{s_i\omega \mathbf d}$. More precisely, if $C$ is defined by inequalities $\theta(\alpha) > 0$ for a set of vectors $\{ \alpha \}$, $s_iC$ is defined by inequalities $\theta(s_i\alpha) > 0$.
\end{cor}

\begin{proof}
For any $\theta \in \Theta_{\omega\mathbf d}$ by Theorem \ref{ref}, a $\Lambda$-module $M$ of dimension vector $\omega\mathbf d$ is $\theta$-(semi)stable if and only if $\Hom_{\Lambda}(T_i,M)$ is $s_i\theta$-(semi)stable, so the first assertion follows. The second assertion follows from the fact that $s_i\theta(s_i\alpha)>0$ is equivalent to $\theta(\alpha)>0$ for any $\theta \in C$.
\end{proof}

\begin{proof}[Proof of Theorem \ref{horizontal}]
Let $X$ be any crepant resolution and $\mu(Q,W)$ the corresponding QP given in the previous sections. By Theorem \ref{ref} it follows that $\mathcal M_{\theta^0,\omega\mathbf d}(\mu(Q,W)) \cong \mathcal M_{\omega^{-1}\theta^0,\mathbf d}(Q,W)$ and it can be checked that $X \simeq \mathcal M_{\omega^{-1}\theta^0,\mathbf d}(Q,W)$. Also by combining Lemma \ref{0-gen} and Corollary \ref{chamber}, the chamber containing $\omega^{-1}\theta^0$ is given by the equalities $\theta(\omega^{-1}\e_i)>0$.
\end{proof}

\begin{ex} Let $G=D_{14}$. The rows in the following diagram correspond to the different chambers in the three mutated algebras for which the crepant resolution of $\C^3/G$ shown in the left column can be realized. Note that in any mutated algebra we can find the corresponding crepant resolution in the chamber containing the 0-generated parameter. \\

\begin{center}
\begin{tabular}{|c|c|c|c|}
	\hline
		& 
		\scalebox{0.45}{
			\begin{pspicture}(-1,-2.2)(6,2)
			\psset{arcangle=10,nodesep=2pt,linewidth=1.5pt}

			\rput(-0.5,0){\Huge $Q:$}
			\rput(1,0){
				\rput(-0,1.25){\rnode{0}{\Large$\star$}}
				\rput(-0,-1.25){\rnode{0'}{\Large$0$}}
				\rput(2,0){\rnode{1}{\Large$1$}}
				\rput(4,0){\rnode{3}{\Large$3$}}
				\rput(2,0.3){\rnode{u1}{}}
				\rput(4,0.3){\rnode{u3}{}}	
				\rput(4,-0.3){\rnode{u3'}{}}		
				\ncarc{->}{0}{0'}\ncarc{->}{0'}{0}\ncarc{->}{0'}{1}\ncarc{->}{1}{0'}
				\ncarc{->}{0}{1}\ncarc{->}{1}{0}\ncarc{->}{1}{3}\ncarc{->}{3}{1}
				\nccircle[angleA=-25,nodesep=3pt]{->}{u1}{.4cm}	
				\nccircle[angleA=0,nodesep=3pt]{->}{u3}{.4cm}
				\nccircle[angleA=180,nodesep=3pt]{->}{u3'}{.4cm}
				}
		 \rput(3,-1.75){\Huge ${\bf{d}}=\begin{smallmatrix}1\\1\end{smallmatrix}\!\text{\huge 22}$}

			\end{pspicture}}
		& 
		\scalebox{0.45}{
			\begin{pspicture}(-1,-2.2)(6,2)
			\psset{arcangle=10,nodesep=2pt,linewidth=1.5pt}

			\rput(-0.5,0){\Huge $Q_{0}:$}
			\rput(1,0){
				\rput(-0,1.25){\rnode{0}{\Large$\star$}}
				\rput(-0,-1.25){\rnode{0'}{\Large$0$}}
				\rput(2,0){\rnode{1}{\Large$1$}}
				\rput(4,0){\rnode{3}{\Large$3$}}
				\rput(-0.15,1.25){\rnode{u1}{}}
				\rput(4,0.3){\rnode{u3}{}}	
				\rput(4,-0.3){\rnode{u3'}{}}		
				\ncarc{->}{0}{0'}\ncarc{->}{0'}{0}\ncarc{->}{0'}{1}\ncarc{->}{1}{0'}
				\ncarc{->}{1}{3}\ncarc{->}{3}{1}
				\nccircle[angleA=75,nodesep=3pt]{->}{u1}{.4cm}	
				\nccircle[angleA=0,nodesep=3pt]{->}{u3}{.4cm}
				\nccircle[angleA=180,nodesep=3pt]{->}{u3'}{.4cm}
				}
		 \rput(3,-1.75){\Huge ${\bf{d}}=\begin{smallmatrix}1\\2\end{smallmatrix}\!\text{\huge 22}$}

			\end{pspicture}}

		& 
		\scalebox{0.45}{
			\begin{pspicture}(-1,-2.2)(6,2)
			\psset{arcangle=10,nodesep=2pt,linewidth=1.5pt}

			\rput(-0.4,0){\Huge $Q_{01}\!:$}
			\rput(1.25,0){
				\rput(-0,1.25){\rnode{0}{\Large$\star$}}
				\rput(-0,-1.25){\rnode{0'}{\Large$0$}}
				\rput(2,0){\rnode{1}{\Large$1$}}
				\rput(4,0){\rnode{3}{\Large$3$}}
				\rput(-0.2,1.25){\rnode{u1}{}}
				\rput(-0.2,-1.25){\rnode{u3}{}}	
				\rput(4,-0.3){\rnode{u3'}{}}		
				\ncarc{->}{0}{0'}\ncarc{->}{0'}{0}\ncarc{->}{0'}{1}\ncarc{->}{1}{0'}
				\ncarc{->}{1}{3}\ncarc{->}{3}{1}
				\nccircle[angleA=75,nodesep=3pt]{->}{u1}{.4cm}	
				\nccircle[angleA=75,nodesep=3pt]{->}{u3}{.4cm}
				\nccircle[angleA=180,nodesep=3pt]{->}{u3'}{.4cm}
				}

		 \rput(3,-1.75){\Huge ${\bf{d}}=\begin{smallmatrix}1\\2\end{smallmatrix}\!\text{\huge 22}$}
			\end{pspicture}}
		 \\
	\hline
		\scalebox{0.5}{
		\begin{pspicture}(-1,-0.5)(7,1)
			\psset{arcangle=30,nodesep=2pt}
		\rput(0,-1.25){
			\rput(0,0){\rnode{0}{}}
			\rput(2.5,0){\rnode{1}{}}
			\rput(2,0){\rnode{2}{}}
			\rput(4.5,0){\rnode{3}{}}
			\rput(4,0){\rnode{4}{}}
			\rput(6.5,0){\rnode{5}{}}	
			\ncarc{-}{0}{1}\Aput[0.05,npos=0.4]{\large $(-1,-1)$}
			\ncarc{-}{2}{3}\Aput[0.05,npos=0.5]{\large $(-2,0)$}\Aput[0.75]{\huge $\Hilb{G}{\C^3}$}
			\ncarc{-}{4}{5}\Aput[0.05,npos=0.6]{\large $(-3,1)$}
		}
		\end{pspicture}
		}

		&	
		\begin{tabular}{c}
		$\theta_0>0$ \\
		$\theta_1>0$ \\
		$\theta_2>0$
		\end{tabular}
		& 
		\begin{tabular}{c}
		$\theta_0<0$ \\
		$\theta_0+\theta_1>0$ \\
		$\theta_2>0$
		\end{tabular} 
		& 
		\begin{tabular}{c}
		$\theta_0+\theta_1<0$ \\
		$\theta_0>0$ \\
		$\theta_1+\theta_2>0$
		\end{tabular} \\
	\hline
		\scalebox{0.5}{
		\begin{pspicture}(-1,-0.5)(7,1)
			\psset{arcangle=30,nodesep=2pt}
		\rput(0,-1.25){
			\rput(0,0){\rnode{0}{}}
			\rput(2.5,0){\rnode{1}{}}
			\rput(2,0){\rnode{2}{}}
			\rput(4.5,0){\rnode{3}{}}
			\rput(4,0){\rnode{4}{}}
			\rput(6.5,0){\rnode{5}{}}	
			\ncarc{-}{0}{1}\Aput[0.05,npos=0.4]{\large $(-1,-1)$}
			\ncarc{-}{2}{3}\Aput[0.05,npos=0.5]{\large $(-1,-1)$}\Aput[0.75]{\huge $X_0$}
			\ncarc{-}{4}{5}\Aput[0.05,npos=0.6]{\large $(-3,1)$}
		}
		\end{pspicture}
		}

		&	
		\begin{tabular}{c}
		$\theta_0<0$ \\
		$\theta_0+\theta_1>0$ \\
		$\theta_2>0$
		\end{tabular}
		& 
		\begin{tabular}{c}
		$\theta_0>0$ \\
		$\theta_1>0$ \\
		$\theta_2>0$
		\end{tabular} 
		& 
		\begin{tabular}{c}
		$\theta_0+\theta_1>0$ \\
		$\theta_1<0$ \\
		$\theta_1+\theta_2>0$
		\end{tabular}  \\
	\hline
		\scalebox{0.5}{
		\begin{pspicture}(-1,-0.5)(7,1)
			\psset{arcangle=30,nodesep=2pt}
		\rput(0,-1.25){
			\rput(0,0){\rnode{0}{}}
			\rput(2.5,0){\rnode{1}{}}
			\rput(2,0){\rnode{2}{}}
			\rput(4.5,0){\rnode{3}{}}
			\rput(4,0){\rnode{4}{}}
			\rput(6.5,0){\rnode{5}{}}	
			\ncarc{-}{0}{1}\Aput[0.05,npos=0.4]{\large $(-2,0)$}
			\ncarc{-}{2}{3}\Aput[0.05,npos=0.5]{\large $(-1,-1)$}\Aput[0.75]{\huge $X_{01}$}
			\ncarc{-}{4}{5}\Aput[0.05,npos=0.6]{\large $(-2,0)$}
		}
		\end{pspicture}
		}

		&	
		\begin{tabular}{c}
		$\theta_1>0$ \\
		$\theta_0+\theta_1<0$ \\
		$\theta_0+\theta_1+\theta_2>0$
		\end{tabular}
		& 
		\begin{tabular}{c}
		$\theta_0+\theta_1>0$ \\
		$\theta_1<0$ \\
		$\theta_1+\theta_2>0$
		\end{tabular} 
		& 
		\begin{tabular}{c}
		$\theta_0>0$ \\
		$\theta_1>0$ \\
		$\theta_2>0$
		\end{tabular}  \\
	\hline		
\end{tabular}
\end{center}

\end{ex}

\section{Floppable curves in $\mathcal{M}_C$}
\label{Sect:Floppable}

Let $G\subset\SO(3)$ of type $\Z/n\Z$, $D_{2n}$ or $\mathbb{T}$. Let $\pi:X\to\C^3/G$ be a crepant resolution and $E\subset X$ be a rational curve. In this section we prove that the only rational curves in a crepant resolution of $\C^3/G$ that can be flopped are the $(-1,-1)$-curves. 

There are three possible degrees for the normal bundle $\mathcal{N}_{X|E}$ over a curve $E\cong\mathbb{P}^1$ in $X$, namely $(-1,-1)$, $(-2,0)$ and $(-3,1)$, and all three types appear in the families treated in this paper. For every $(-1,-1)$-curve there always exists a flop $X\dashrightarrow X'$ of $E$ where $X$ and $X'$ are isomorphic in codimension one. If $E$ is a $(-2,0)$-curve then we use the {\em width of $E$} defined in \cite{Pagoda} to conclude that $E$ is always contained on a scroll, which implies that it does not exist a small contraction of $X$ which contracts $E$. 

There are only two $(-3,1)$-curves: $E_{m}\subset\Hilb{D_{2n}}{\C^3}$ when $n$ is odd and $E_2\subset\Hilb{\mathbb{T}}{\C^3}$. In both cases we use the fact that $X\cong\mathcal{M}_C$ for some chamber $C\in\Theta$ and we consider the contraction of $E$ as the map $\mathcal{M}_C\to\mathcal{M}_{\overline{\theta}}$ where $\overline{\theta}\in\overline{C}$ lies on a wall of the chamber $C$ (cf.\ \cite{CI} \S3.2). By the study of $S$-equivalence classes we are able to describe explicitly the contracted locus and conclude that such a contraction is divisorial, i.e.\ the curve is not floppable.

We finish the section giving an alternative proof of the fact that $E_2\subset\Hilb{\mathbb{T}}{\C^3}$ is not floppable using {\em contraction algebras}.

\begin{lem}\label{floppable} Let $G\subset\SO(3)$ of type $\Z/n\Z$, $D_{2n}$ or $\mathbb{T}$, and let $\pi:X\to\C^3/G$ be a crepant resolution. Then only the rational curves $E\subset X$ with degree of normal bundle $(-1,-1)$ are floppable.
\end{lem}

\begin{proof} Let $(Q,R)$ the McKay quiver with relations, $\Lambda=\C Q/R$, ${\bf{d}}:=(\dim\rho)_{\rho\in\Irr G}$ and $\theta\in\Theta$ be the $0$-generated parameter. Denote by $\mathcal{M}_\theta:=\mathcal{M}_{\theta,{\bf{d}}}(\Lambda)$. By the part $(3)$ in Theorems \ref{OpensDnOdd}, \ref{OpensDnEven} and \ref{OpensE6} only $E_{m}\subset\Hilb{D_{2n}}{\C^3}$ when $n$ is odd and $E_2\subset\Hilb{\mathbb{T}}{\C^3}$ are $(-3,1)$-curves. Since the open sets covering these curves do not change under the flop of any other curve, it is enough to prove that they are not floppable in $\Hilb{G}{\C^3}$. Thus, it is enough to show the following three claims:
\begin{itemize}
\item[(i)] If $E$ is a $(-2,0)$-curve then $E$ is contained on a scroll.
\item[(ii)] Let $G=D_{2n}\subset SO(3)$ with $n$ odd. Then the $(-3,1)$-curve on $\mathcal M_{\theta}$ is not floppable.
\item[(iii)] Let $G \subset SO(3)$ be the tetrahedral group. Then the $(-3,1)$-curve on $\mathcal M_{\theta}$ is not floppable.
\end{itemize}

{\em Proof of (i).} By the covering of $X$ given in Section \ref{sect:opens} we know that $E\subset X$ is covered by two open sets $U$ and $U'$ where $U,U'\cong\C^3$. First notice that for every curve $E$ of type $(-2,0)$ we can make a suitable change of basis on $U$ or $U'$ to obtain the gluing to be of the form $U\backslash\{a=0\}\ni(a,b,c)\mapsto(a^{-1},a^2b,c)\in U'\backslash\{a'=0\}$. It is straightforward in most cases, although we give here some of them:

In $D_{2n}$ with $n$ odd have $U'_{m+1}\backslash\{a=0\}\ni(a,b,B)\mapsto(a^2(d^4-D^2/4),d,a^{-1})\in U_{m+2}\backslash\{u=0\}$, so we can change of coordinates in $U'_{m+1}$ by $(\bar{a},\bar{d},\bar{D})=(a,d,d^2-D)$.

In $D_{2n}$ with $n$ even have $V''_{m+1}\backslash\{d=0\}\ni(d,D,C')\mapsto(C'^2+d^2D,d^{-1},C')\in V''_{m+2}\backslash\{c'=0\}$, so we can change of coordinates in $V''_{m+2}$ by $(\bar{A},\bar{c'},\bar{C'})=(A-C'^2,c',C')$.

In $\mathbb{T}$ have $U_{1}\backslash\{c_3=0\}\ni(c_2,c_3,C_3)\mapsto(-c_2,c_3^{-1},c_2^2(1+c_3^{-1})-c_3^2C_3)\in U'_{2}\backslash\{B_1\}$, so we can change of coordinates in $U'_{2}$ by $(\bar{b_1},\bar{B_1},\bar{B_3})=(b_1,B_1,b_2^2(1+B_1)-B_3)$.

The {\em width} of a $(-2,0)$-curve $E\in X$ is defined in \cite{Pagoda} as
\[n:=\sup\{n | \text{$\exists$ scheme $E_n$ with $E\subset E_n\subset X$ s.t.\ $E_n\cong E\times\Spec\C[\varepsilon]/\varepsilon^n$} \}
\] 
Once we have the gluing in the form $(a,b,c)\mapsto(a^{-1},a^2b,c)$, the curve $E\subset U$ is defined by the ideal $I=(b,c)$ and for any $k>0$ the ideal $J_k=(b,c^k)$ satisfy the conditions of the criteria in Proposition 5.10 in \cite{Pagoda}, so the curve $E$ has infinity width. Thus $E$ moves in a scroll $S\subset X$ so there is no small contraction of $X$ which contracts only $E$.

{\em Proof of (ii).} First we note that, if $\bar\theta$ is a parameter with $\theta_i >0$ for $i \neq 0,m$ and $\theta_m =0$, then there is a morphism $f : \mathcal M_{\theta} \to \mathcal M_{\bar\theta}$ which cannot be further factored into birational morphisms between normal varieties. If $M$ is a point on $\mathcal M_{\theta}$, then the image $[M] := f(M)$ is an $S$-equivalence class of $M$ with respect to $\bar\theta$.

In the open cover $U_{m+2}$, put $x=u,y=v,z=V$. We consider the hypersurface $X \subset U_{m+2}$ defined by $y^2-xz^2=0$. We prove that the surface $X$ is contracted to a curve by calculating $S$-equivalence classes.
Fix $x = \alpha$ and denote by $X_{\alpha}$ the curve on $X$ determined by $x=\alpha$. Note that $X_{0}$ is contained in the $(-3,1)$-curve.
Take any representation $M$ on $X_{\alpha}$. 
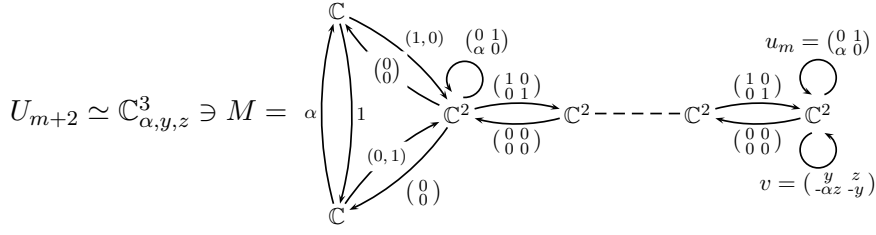
\begin{center}
\begin{pspicture}(-2,-1.5)(8,1.5)
	\psset{arcangle=15,nodesep=2pt}
\rput(-2.5,0){$U_{m+2} \simeq \C^3_{\alpha,y,z} \ni M = $}
\scalebox{0.8}{
	\rput(-0,1.7){\rnode{0}{$\C$}}
	\rput(-0,-1.7){\rnode{1}{$\C$}}
	\rput(2,0){\rnode{2}{$\C^2$}}
	\rput(4,0){\rnode{3}{$\C^2$}}
	\rput(6,0){\rnode{5}{$\C^2$}}
	\rput(8,0){\rnode{6}{$\C^2$}}
	\rput(2,0.3){\rnode{u1}{}}	
	\rput(4,0.3){\rnode{u2}{}}	
	\rput(6,0.3){\rnode{u3}{}}	
	\rput(8,0.3){\rnode{u5}{}}	
	\rput(8,-0.3){\rnode{u6}{}}	
	\ncarc{->}{0}{1}\Aput[0.05]{\footnotesize $1$}	
	\ncarc{->}{1}{0}\Aput[0.05]{\footnotesize $\alpha$}	
	\ncarc{->}{1}{2}\mput*[0.05]{\scriptsize $(0,1)$}	
	\ncarc{->}{2}{1}\Aput[0.05]{\large $\left(\begin{smallmatrix}0\\0\end{smallmatrix}\right)$}		
	\ncarc{->}{0}{2}\Aput[0.05]{\scriptsize $(1,0)$}		
	\ncarc{->}{2}{0}\mput*[0.05]{\large $\left(\!\begin{smallmatrix}0\\0\end{smallmatrix}\!\right)$}	
	\ncarc{->}{2}{3}\Aput[0.05]{$\left(\begin{smallmatrix}1&0\\0&1\end{smallmatrix}\right)$}		
	\ncarc{->}{3}{2}\Aput[0.05]{$\left(\begin{smallmatrix}0&0\\0&0\end{smallmatrix}\right)$}		
	\ncline[linestyle=dashed]{-}{3}{5}
	\ncarc{->}{5}{6}\Aput[0.05]{$\left(\begin{smallmatrix}1&0\\0&1\end{smallmatrix}\right)$}		
	\ncarc{->}{6}{5}\Aput[0.05]{$\left(\begin{smallmatrix}0&0\\0&0\end{smallmatrix}\right)$}		
	 \nccircle[angleA=-25,nodesep=3pt]{->}{u1}{.3cm}\Bput[0.05]{$\left(\!\begin{smallmatrix}0&1\\ \alpha&0\end{smallmatrix}\!\right)$}	
	 \nccircle[nodesep=3pt]{->}{u5}{.3cm}\Bput[0.05]{$u_m = \left(\!\begin{smallmatrix}0&1\\ \alpha&0\end{smallmatrix}\!\right)$}		
	 \nccircle[angleA=180,nodesep=3pt]{->}{u6}{.3cm}\Bput[0.05]{$v = \left(\begin{smallmatrix}y&z\\\text{-} \alpha z&\text{-}y\end{smallmatrix}\right)$}	
}
\end{pspicture}
\end{center}
Then there is a submodule $M'$ of $M$ whose dimension vector is $2\e_m$. 
One can check that the eigenvalue of $u_m$ is $0$ and eigenvectors are $(\sqrt{\alpha},1)$ and $(-\sqrt{\alpha},1)$.
Since $y^2 - \alpha z^2 = (y + \sqrt{\alpha}z)(y - \sqrt{\alpha}z)$, we put $X_{\alpha}^+ = X \cap (y + \sqrt{\alpha}z=0)$ and $X_{\alpha}^- = X \cap (y - \sqrt{\alpha}z=0)$. Then $X_{\alpha} = X_{\alpha}^+ \cup X_{\alpha}^-$. If $M$ is a point on $X_{\alpha}^+$, we consider the subspace $M''$ of $M'$ spanned by $(-\sqrt{\alpha},1)$. The actions of $u_m,v$ are zero on $M''$, so it becomes subrepresentation of $M'$, and we have a filtration of $\bar\theta$-semistable representations
\[
0 \subsetneq M'' \subsetneq M' \subsetneq M.
\]
One can check that 
\[
[M'/M''] \simeq
[
\begin{pspicture}(-2.3,-0.2)(2,0.3)
	\psset{arcangle=15,nodesep=2pt}
\rput(0,0){
\scalebox{1}{
	\rput(0,0){\rnode{1}{$\C$}}
	\rput(-0.3,0){\rnode{u1}{}}	
	\rput(0.3,0){\rnode{u2}{}}	
	\nccircle[angleA=90,nodesep=3pt]{->}{u1}{.3cm}\Bput[0.05]{\footnotesize $ \sqrt{\alpha}=u_m$}
	\nccircle[angleA=270,nodesep=3pt]{->}{u2}{.3cm}\Bput[0.05]{\footnotesize $v = 0$}
	}}
\end{pspicture}
]
,
[M''] \simeq
[
\begin{pspicture}(-2.1,-0.2)(2,0.3)
	\psset{arcangle=15,nodesep=2pt}
\rput(0,0){
\scalebox{1}{
	\rput(0,0){\rnode{1}{$\C$}}
	\rput(-0.3,0){\rnode{u1}{}}	
	\rput(0.3,0){\rnode{u2}{}}	
	\nccircle[angleA=90,nodesep=3pt]{->}{u1}{.3cm}\Bput[0.05]{\footnotesize $ 0=u_m$}
	\nccircle[angleA=270,nodesep=3pt]{->}{u2}{.3cm}\Bput[0.05]{\footnotesize $v = 0$}
	}}
\end{pspicture}
]
\]
where the vector spaces lie only on the vertex $m$.
Thus the factor modules $M/M'$, $M'/M''$, $M''$ do not depend on $y,z$.
Hence any representation on $X_{\alpha}^+$ is $S$-equivalent to the representation $[M/M'] \oplus [M'/M''] \oplus [M'']$.
Similarly any representation on $X_{\alpha}^-$ is also $S$-equivalent to it. Thus the surface $X$ is contracted to a curve. Therefore, there is no small contraction and the $(-3,1)$-curve is not floppable.

{\em Proof of (iii).} It is proved by the same strategy, however the computation is not so obvious, so we show it. 
We take the open set $U_1$ and put $x = C_3, y = c_2, c_3 = z$.
We consider the hypersurface $X$ defined by $y^2 = xz^2-xz+x = x(z+\omega)(z+\omega^2)$ where $\omega$ is a primitive 3rd root of unity.
We fix $x=\alpha$ and put $X_{\alpha} = X \cap (x=\alpha)$. 
Take any representation $M$ on $X_{\alpha}$.
Then we have $A_1=\alpha$, so the matrices in $M$ become as follows:\\
$\begin{array}{cl}
& a=(1,0,0), b=(0,0,1), c=(y^2-\alpha z^2,y,z), A=\alpha\left(\begin{smallmatrix} 1 \\ 0 \\ \alpha \end{smallmatrix}\right), B=\left(\begin{smallmatrix}1 \\0 \\ \alpha \end{smallmatrix}\right), C=\left(\begin{smallmatrix}1 \\0 \\ \alpha\end{smallmatrix}\right), \\
& u=\left(\begin{smallmatrix}0&1&0 \\\alpha z-\omega (y^2-\alpha z^2)&\omega^2 y&\omega +\omega^2 z \\ \omega^2 \alpha y & -\omega^2 \alpha z & -\omega^2 y \end{smallmatrix}\right), v=\left(\begin{smallmatrix}0&1&0 \\ \alpha z-\omega^2 (y^2-\alpha z^2)&\omega y&\omega^2 +\omega z \\ \omega \alpha y & -\omega \alpha z & -\omega y \end{smallmatrix}\right).
\end{array}$

Thus we see that there is a subrepresentation $M'$ of $M$ generated by $(\alpha,0,-1)$ and $(0,1,0)$:
\begin{pspicture}(-5,-0.2)(4.3,0.5)
	\psset{arcangle=15,nodesep=2pt}
\rput(0,0){
\scalebox{1}{
	\rput(0,0){\rnode{1}{$\C^2$}}
	\rput(-0.3,0){\rnode{u1}{}}	
	\rput(0.3,0){\rnode{u2}{}}	
	\nccircle[angleA=90,nodesep=3pt]{->}{u1}{.3cm}\Bput[0.05]{$\omega^2\left(\!\begin{smallmatrix} -y & \alpha (z+\omega)\\ -(z+\omega^2) & y \end{smallmatrix} \!\right) =u$}
	\nccircle[angleA=270,nodesep=3pt]{->}{u2}{.3cm}\Bput[0.05]{\footnotesize $v = \omega\left(\!\begin{smallmatrix} -y & \alpha (z+\omega^2)\\ -(z+\omega) & y \end{smallmatrix} \!\right)$}
}}
\end{pspicture}.
The eigenvalues of $u,v$ are $0$ and the eigenvectors are $(y,z+\omega^2),(\alpha(z+\omega^2),y)$ and $(y,z+\omega),(\alpha(z+\omega),y)$ respectively.
Here we put $X_{\alpha}^{+}=X_{\alpha} \cap (y=\sqrt{\alpha}(z+w))$ and $X_{\alpha}^{-}=X_{\alpha} \cap (y=-\sqrt{\alpha}(z+w))$.
Then if $M$ is on $X_{\alpha}^+$, all eigenvectors are multiples of the vector $(\sqrt{\alpha},1)$. So by taking the subrepresentation $M''^+$ of $M'$ generated by $(\sqrt{\alpha},1)$, we have a filtration of $\bar\theta$-semistable representations:
\[
0 \subsetneq M'' \subsetneq M' \subsetneq M
\]
and one can check that factor modules $M/M'$,$M'/M''^+$,$M''^+$ do not depend on $y,z$. So any point on $X_{\alpha}^+$ is S-equivalent to $M/M' \oplus M'/M''^+ \oplus M''^+$. Similarly one can check that any point on $X_{\alpha}^-$ is S-equivalent to $M/M' \oplus M'/M''^- \oplus M''^-$ where $M''^-$ is the subrepresentation of $M'$ generated by $(-\sqrt{\alpha},1)$.
Therefore $X$ is contracted to a curve, so the $(-3,1)$-curve is not floppable.
\end{proof}

\subsection{A contraction algebra}
\label{sect:contralg}

In this section we give an alternative proof of Lemma \ref{floppable} by using {\em contraction algebras}. Let $G$ be the tetrahedral group $\mathbb{T}$ of order 12, $S:=\C[x,y,z]$ as usual. Let $M_i:=(S\otimes\rho)^G$ for $i=0,\ldots,3$ be the non-isomorphic CM $S^G$-modules with $M_0\cong R$, and let $M:=\bigoplus_{i=0}^3M_i$. Then the algebra $\Lambda:=\End_{S^G}(M)$ is isomorphic to the Jacobian algebra $\mathcal{P}(Q,W)$ for the McKay QP $(Q,W)$ given in the previous sections. 

Let $X:=\Hilb{\mathbb{T}}{\C^3}\cong\mathcal{M}_{\theta^0,{\bf{d}}}(\Lambda)$ with $\theta^0=(-5,1,1,1)$ and ${\bf{d}}=(1,1,1,3)$. Let $E_3\subset X$ the rational curve of type $(-3,1)$, which corresponds with the vertex $3\in Q_0$. If $E_3$ is floppable there exists a small contraction $\tau:X\to Y$, where we can realize $Y$ as the moduli space $\mathcal{M}_{\overline{\theta^0},\overline{d}}(\Gamma)$ with $\overline{\theta^0}=(-3,1,1)$, $\overline{\bf{d}}=(1,1,1)$ and $\Gamma:=\End_{S^G}(M/M_3)$ is obtained by removing the module $M_3$ from $M$. We call $\Gamma$ a {\em contraction algebra} by its analogy with what is happening geometrically. 

It turns out that $\Gamma$ is isomorphic to the path algebra $\C\overline{Q}/\overline{R}$ for the following quiver $(\overline{Q},\overline{R})$ with relations:

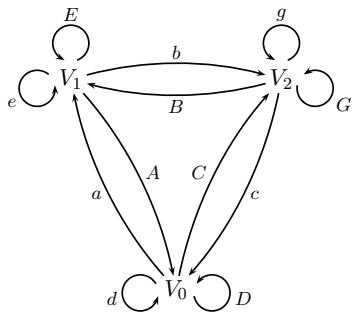
\begin{center}
\begin{pspicture}(0,-2)(10,2.25)
	\psset{arcangle=15,nodesep=2pt}
\rput(1,0){
\scalebox{0.8}{
	\rput(0,-2){\rnode{0}{$V_0$}}
	\rput(-1.75,1.5){\rnode{1}{$V_1$}}
	\rput(1.75,1.5){\rnode{2}{$V_2$}}
	\rput(-0.3,-2){\rnode{d}{}}	
	\rput(0.3,-2){\rnode{D}{}}	
	\rput(-2.05,1.5){\rnode{e}{}}	
	\rput(-1.75,1.8){\rnode{E}{}}	
	\rput(1.75,1.8){\rnode{g}{}}	
	\rput(2.05,1.5){\rnode{G}{}}	

	\ncarc{->}{0}{1}\Aput[0.05]{\footnotesize $a$}	
	\ncarc{->}{1}{0}\Aput[0.05]{\footnotesize $A$}	
	\ncarc{->}{1}{2}\Aput[0.05]{\footnotesize $b$}	
	\ncarc{->}{2}{1}\Aput[0.05]{\footnotesize $B$}	
	\ncarc{->}{2}{0}\Aput[0.05]{\footnotesize $c$}	
	\ncarc{->}{0}{2}\Aput[0.05]{\footnotesize $C$}	
	\nccircle[angleA=100,nodesep=3pt]{->}{d}{.3cm}\Bput[0.05]{\footnotesize $d$}
	 \nccircle[angleA=260,nodesep=3pt]{->}{D}{.3cm}\Bput[0.05]{\footnotesize $D$}
 	\nccircle[angleA=120,nodesep=3pt]{->}{e}{.3cm}\Bput[0.05]{\footnotesize $e$}
	 \nccircle[angleA=0,nodesep=3pt]{->}{E}{.3cm}\Bput[0.05]{\footnotesize $E$}
	\nccircle[angleA=0,nodesep=3pt]{->}{g}{.3cm}\Bput[0.05]{\footnotesize $g$}
	 \nccircle[angleA=-120,nodesep=3pt]{->}{G}{.3cm}\Bput[0.05]{\footnotesize $G$}
	}}
\scalebox{0.8}{
	\rput(9,1.5){$aA=Cc$, $bB=Aa$, $cC=Bb$}
	\rput(9,1){$da=ae$, $eb=bg$, $gc=cd$}
	\rput(9,0.5){$dC=Cg$, $eA=Ad$, $gB=Be$}
	\rput(9,0){$Da=\omega aE$, $Ed=\omega bG$, $Gc=\omega cD$}
	\rput(9,-0.5){$DC=\omega^2CG$, $EA=\omega^2AD$, $GB=\omega^2BE$}
	\rput(9,-1){$D^2=d^3+abc+CBA-3Ccd$}
	\rput(9,-1.5){$\omega^2E^2=e^3+bca+ACB-3Aae$}
	\rput(9,-2){$\omega G^2=g^3+cab+BAC-3Bbg$}
	}
\end{pspicture}
\end{center}
Thus we obtain $Y\cong\mathcal{M}_{\overline{\theta^0},\overline{\bf{d}}}(\overline{Q},\overline{R})=U_1\cup U_2\cup U_3$ where $U_i$ are hypersurfaces given by equations:
\begin{align*}
U_1: &  ~(wG^2=d^3+c+c^2C^3-3cCd)\subset\C^4_{c,C,d,G} \\
U_2: &  ~(wG^2=d^3+b^2B+bB^2-3bBd)\subset\C^4_{c,B,d,G} \\
U_3: &  ~(wG^2=d^3+A+A^2a^3-3aAd)\subset\C^4_{c,C,d,G} 
\end{align*} 
Therefore $Y$ has a singular line $L$ which in $U_2$ is given by the points $(d,d,d,0)$. As in Lemma \ref{floppable} that the preimage of $L$ is precisely the equation $y^2=xz^2-xz+x$ by setting $C_3=x$, $c_2=y$ and $c_1=y^2-xz^2$. Therefore $\tau$ is not a small contraction, therefore $E_3$ is not floppable. Moreover, this construction coincides with the contraction map $\mathcal{M}_{\theta^0}\to\mathcal{M}_{\overline{\theta^0}}$ described in the previous lemma where $\overline{\theta^0}$ is a stability condition at the wall $\theta_3=0$.






\end{document}